%% file: ccro2.tex
\documentclass[11pt]{amsart}

\def\flnm{ccro2}

\input{ccro2-switches}
\input{ccro2-not}

\usepackage[backend=biber,style=mathrefnum, dashed=true]{biblatex}
\begin{document}

\input{ccro2-titlepage}

\setcounter{tocdepth}{1}
\tableofcontents

\input{ccro2-newintro}


\input{ccro2-2-group-symsp}

\input{ccro2-3-autf}

\input{ccro2-4-cohom}

\input{ccro2-5-ps}

\input{ccro2-6-cf-coh}

\input{ccro2-7-asp}

\input{ccro2-8-bgrms}

\input{ccro2-9-cohbg}

\input{ccro2-10-coh-cf}

\input{ccro2-11-inj}

\input{ccro2-12-gef}

\input{ccro2-13-concl}


\defbibheading{bibliography}[\bibname]{%
\section*{References}%
}

\printbibliography
\input{ccro2-ind}

\setlength{\parindent}{0pt}

\end{document}

%% file: ccro2-not.tex
\newcommand{\Fpb}[1]{F^\pb_{\!\!#1}}
\newcommand{\rmetric}{\hbox{$g^{(R)}$}} 

\usepackage{amssymb,amsmath,amsthm}

\usepackage[all]{xy}
\usepackage{graphicx}

\usepackage{enumitem}

\def\be#1\ee{\begin{equation}#1\end{equation}}
\def\bad#1\ead{\be\begin{aligned}#1\end{aligned}\ee}
\def\badl#1#2\eadl{\be\label{#1}\begin{aligned}#2\end{aligned}\ee}

\newtheorem{thm}{Theorem}[section]
\newtheorem{prop}[thm]{Proposition}
\newtheorem{cor}[thm]{Corollary}
\newtheorem{lem}[thm]{Lemma}

\newtheorem{mainthm}{Theorem}

\theoremstyle{definition}
\newtheorem{defn}[thm]{Definition}
\newtheorem{remark}[thm]{Remark}

\newcommand\rmrk[1]{\medskip\par\noindent\emph{#1. }}

\theoremstyle{remark}

\usepackage{color}
\definecolor{blue}{rgb}{0,0,1}
\definecolor{red}{rgb}{1,0,0}
\definecolor{green}{rgb}{0,.6,.2}
\definecolor{purple}{rgb}{1,0,1}
\definecolor{brown}{rgb}{.59,.29,0}

\long\def\red#1\endred{{\color{red}#1}}
\long\def\blue#1\endblue{{\color{blue}#1}}
\long\def\purple#1\endpurple{{\color{purple} #1}}
\long\def\green#1\endgreen{{\color{green}#1}}
\long\def\brown#1\endbrown{{\color{brown}#1}}

\newcommand\al{\alpha}

\newcommand\bt{\beta}
\newcommand\Gm{\Gamma}
\newcommand\gm{\gamma}
\newcommand\Dt{\Delta}
\newcommand\dt{\delta}

\newcommand\e{\varepsilon}
\newcommand\z{{\zeta}}

\newcommand\tht{\theta}
\newcommand\thv{\vartheta}
\newcommand\kp{\kappa}
\newcommand\Ld{\Lambda}
\newcommand\ld{\lambda}
\newcommand\x{\xi}

\newcommand\s{\sigma}

\newcommand\Ph{\Phi}
\newcommand\ph{\varphi}
\newcommand\ch{\chi}
\newcommand\Ps{\Psi}
\newcommand\ps{\psi}
\newcommand\Om{\Omega}
\newcommand\om{\omega}

\DeclareMathOperator{\re}{Re}
\DeclareMathOperator{\im}{Im}
\DeclareMathOperator{\SL}{SL}
\DeclareMathOperator{\PSL}{PSL}
\DeclareMathOperator{\GL}{GL}

\newcommand\PSU{\mathrm{PSU}}
\DeclareMathOperator{\PU}{PU}

\newcommand\CC{\mathbb{C}}

\newcommand\NN{\mathbb{N}}

\newcommand\RR{\mathbb{R}}
\newcommand\ZZ{\mathbb{Z}}

\makeatletter
\newcommand\rmatc[4]{\left( {#1\@@atop #3}{#2\@@atop #4}\right)}
\newcommand\rmatr[4]{\left( {\hfill #1\@@atop\hfill #3}{\hfill
#2\@@atop\hfill #4}\right)}
\newcommand\matc[4]{\left[ {#1\@@atop #3}{#2\@@atop #4}\right]}
\newcommand\matr[4]{\left[ {\hfill #1\@@atop\hfill #3}{\hfill
#2\@@atop\hfill #4}\right]}
\makeatother

\let\setminus\smallsetminus

\newcommand\lijn{\bigskip\par\hrule\bigskip\par}

\let\uhp\hp

\newcommand\hypg[2]{{}_{#1}F_{\!#2}}

\DeclareMathOperator{\supp}{Supp}
\newcommand\Gf{\Gamma}
\newcommand\oh{{\mathrm{O}}}
\newcommand\proj[1]{\mathbb{P}^1_{\!#1}}

\newcommand\Lie{\mathrm{Lie}}

\newcommand\intitle[1]{\texorpdfstring{$#1$}{}}

\numberwithin{equation}{section}

\ifhyperref
\newcommand\tocrgl[2]{
\addtocontents{toc}{\string\contentsline{section}
{\blue\quad ${}^{\thefootnote}$ #1, #2\endblue} {\thepage}
{}}}
\else
\newcommand\tocrgl[2]{\addtocontents{toc}%
{\string\contentsline{section}%
{\blue\quad ${}^{\thefootnote}$ #1, #2 \endblue}
{\thepage}%
{}}}
\fi

\numberwithin{footnote}{section}
\numberwithin{table}{section}
\numberwithin{figure}{section}
\renewcommand\thefootnote{\arabic{section}.\arabic{footnote}}

\usepackage{multicol}
\makeatletter
\@ifundefined{ifindexindication}{\let\ifindexindication=\iffalse}{}
\ifindexindication
 \newcommand\indexindication[2]
 {\raisebox{1.6ex}{\makebox[.4em][#2]{\tiny #1}}}
 \newcommand\il[2]{\label{i-#1}\indexindication{#2}{l}}
 \newcommand\ir[2]{\indexindication{$#2$}{r}}
\else
 \newcommand\il[2]{\label{i-#1}}
 \newcommand\ir[2]{}
\fi
\makeatother

\ifhyperref
  \usepackage[colorlinks,breaklinks,pdfpagelabels,hypertexnames=false]{hyperref}
\else
  \newcommand\texorpdfstring[2]{#1}
\fi

\newcommand\am{\mathrm{a}}

\newcommand\nm{\mathrm{n}}
\newcommand\wm{\mathrm{w}}

\newcommand\ii{\mathbf{i}}

\DeclareMathOperator{\diag}{diag}
\newcommand\X{\mathfrak{X}}
\newcommand\bu{{\mathsf b}}
\newcommand\cu{{\mathsf c}}
\newcommand\du{{\mathsf d}}
\DeclareMathOperator{\Ad}{Ad}
\DeclareMathOperator{\ad}{ad}
\newcommand\nul{\text{\textit{\textbf o}}}

\newcommand\orgn{\text{\textit{\textbf e}}}
\newcommand\inft{{\pmb\infty}}

\newcommand\alie{{\mathfrak a}}
\newcommand\glie{{\mathfrak g}}
\newcommand\klie{{\mathfrak k}}
\newcommand\mlie{{\mathfrak m}}
\newcommand\nlie{{\mathfrak n}}
\newcommand\plie{{\mathfrak p}}

\renewcommand\H{{\mathbf H}}
\newcommand\XX{{\mathbf X}}
\newcommand\YY{{\mathbf Y}}
\newcommand\Z{{\mathbf Z}}
\newcommand\UU{{\mathbf U}}

\newcommand\WW[1]{{\mathbf{W}_{\!#1}}}
\newcommand\VV[1]{{\mathbf{V}_{\!#1}}}

\newcommand\A{{ \mathcal A}}

\newcommand\Gmm[1]{\Gm_{\!#1}}
\newcommand\Sie{{\mathfrak S}}
\newcommand\tess{{\mathcal T}}

\newcommand\sh{{\mathcal F}}
\newcommand\Cu{{\mathrm{Cu}(\Gm)}}
\newcommand\fd{{\mathfrak{F}}}
\newcommand\st{{\mathrm{st}}}
\newcommand\pb{\mathrm{pb}}
\newcommand\Sast{S^{\!\ast}}
\newcommand\dist{{\mathrm{d}_S\!}}
\newcommand\dtst{{\dt_S}}
\newcommand\Ta[1]{\mathrm{T}_{\!#1}}
\newcommand\Tr{\mathrm{Tr}}

\newcommand\Cas{{\mathsf{Cas}}}

\newcommand\SY{S_{\!\Gm}(Y)}

\newcommand\V[2]{\mathcal{V}^{#1}_{\!\!#2}}
\newcommand\W[2]{\mathcal{W}^{#1}_{\!\!#2}}
\newcommand\omi{{\om^0,\infty}}
\newcommand\MS{{\mathrm{MS}}}
\DeclareMathOperator{\grad}{grad}

\renewcommand\div{\mathrm{div}}

\newcommand\C{{\mathcal C}_{\Gm\backslash S}}
\newcommand\E{{\mathcal E}}
\newcommand\Q{{\mathcal Q}}
\newcommand\Y{{\mathcal Y}}
\newcommand\exc{{\mathrm{exc}}}
\newcommand\excg{{\mathrm{excg}}}

\newcommand\F[1]{{\mathcal F}_{\!#1}}

\newcommand\alw{\al^{\!\scriptscriptstyle{\mathcal W}}}
\newcommand\btv{\bt^{\!\scriptscriptstyle{\mathcal V}}}
\newcommand\btw{\bt^{\!\scriptscriptstyle{\mathcal W}}}
\newcommand\omv{\om^{\!\scriptscriptstyle{\mathcal V}}}
\newcommand\omw{\om^{\!\scriptscriptstyle{\mathcal W}}}

\newcommand\DS{D_{S,\nu}}

\newcommand\bsing{\mathsf{bdSing}}
\newcommand\sing{\mathsf{Sing}}
\newcommand\bdc{{\mathrm{bdc}}}

\newcommand\N[2]{\mathcal{N}^{#1}_{#2}}
\newcommand\G[2]{{\mathcal{G}^{#1}_{#2}}}

\newcommand\kI{\mathbf{k}_{\mathrm{o}}}
\newcommand\tI{\mathbf{t}_{\mathrm{o}}}
\newcommand\nI{\mathbf{n}_{\mathrm{o}}}
\newcommand\kJ{\mathbf{k}}
\newcommand\tJ{\mathbf{t}}
\newcommand\nJ{\mathbf{n}}

\newcommand\pr{{\mathrm{pr}}}

\newcommand\Utess{\mathcal{U}}

\newcommand\sttess{{\tess_{\!\st}}}

\newcommand\Lt{\mathcal{L}}

\newcommand\deqref[2]{(\ref{#1}:\ref{#2})}

\newcommand\poiss[1]{\mathsf{P}_{\!#1}}

\newcommand{\prj}{\mathrm{prj}}
\newcommand{\lnm}{\mathrm{lnm}}

\newcommand\PI[1]{P_{\!#1}}

\newcommand\tlref[2]{(\ref{#1}-\ref{#2})}

\newcommand \anal {\mathsf{An}(0)}

\newcommand\ds{{\mathrm d}}

%% file: ccro2-titlepage.tex

\author{Roelof Bruggeman}
\address{Roelof Bruggeman, Mathematisch Instituut, Universiteit Utrecht, Postbus 80010,
3508 TA Utrecht, Nederland}
\email{r.w.bruggeman@uu.nl}

\author{YoungJu Choie}
\address{YoungJu Choie, Department of Mathematics and PMI, Postech, Pohang, Korea
790--784}
\email{yjc@postech.ac.kr}

\author{Roberto  Miatello}
\address{Roberto  Miatello, FAMAF-CIEM, Universidad Nacional de C\'or\-do\-ba,
C\'or\-do\-ba~5000, Argentina}
\email{miatellorj@gmail.com}

\author{Anke Pohl}
\address{Anke Pohl, University of Bremen, Department 3: Mathematics and Computer Science, Institute
for Dynamical Systems, Bibliothekstr.~5, 28359 Bremen, Germany}
\email{apohl@uni-bremen.de}

\title[Cusp forms and cohomology for rank one]{Cusp forms and parabolic
cohomology classes for symmetric spaces of rank one}

\begin{abstract}
For any rank-one Riemannian symmetric space~$S$ of non-compact type and any
discrete, cofinite, non-cocompact, torsion-free group~$\Gm$ of
orienta\-tion-pre\-serv\-ing Riemannian isometries on~$S$, we develop a
cohomological interpretation for the cusp forms of~$\Gm$.
To that end, we identify certain $\Gm$-submodules of smooth semi-analytic vectors in the spherical principal series representation with spectral parameter~$\nu$ as well as certain subspaces of parabolic cohomology spaces of~$\Gm$ of degree $\dim S-1$ with these $\Gm$-submodules.
We provide explicit isomorphisms between the spaces of cusp forms of spectral parameter~$\nu$ and these specific cohomology subspaces. The isomorphisms from cusp forms to cohomology are given by an integral transform, and the explicit form of the inverse isomorphism takes advantage of a certain reproducing property of the integral transform.
The result is uniform for all these symmetric spaces and does not rely
on their classification.
\end{abstract}

\date{\today}\subjclass[2020]{Primary: 11F75, 32N15; Secondary: 22E30, 22E40
}

\keywords{rank one symmetric space, cusp form, spherical principal series, parabolic cohomology, analytic boundary germs,  extension of symmetric spaces 
}

\subjclass[2020]{Primary: 11F75, 32N15; Secondary: 11F55, 22E40}

\maketitle

%% file: ccro2-newintro.tex

\def\flnm{ccro2-newintro}

\section{Introduction}\label{sect-intro}

Let $S$ be a Riemannian symmetric space of non-compact type and of rank
one, and let $\Gm$ be a discrete group of orientation-preserving
Riemannian isometries of~$S$ that is torsion-free and cofinite, but not
cocompact. In a nutshell, our main result is a cohomological interpretation of cusp forms for~$\Gm$.
In what follows, we will provide a more detailed description of this result. By its very nature, the proof of this result is rather lengthy and involved. For the convenience of the reader we will therefore first deliver a reduced explanation that allows us to provide, with Theorem~\ref{thm:main1}, a partial statement of our main result that can be stated and surveyed (but not proved) with a reduced number of entities used for intermediate steps. We will then deepen this explanation by surveying the key steps in our approach, therefore giving a glimpse on some of the additional entities used, and finally state, with Theorem~\ref{mainthm}, the full version of our main result. Thereafter we will discuss our motivation and provide a brief survey of the relation between our main result to previous results in this research area. Afterwards we will give an overview of the structure of this article.

\subsection{Main result} Throughout let $S$ denote an arbitrary rank-one
Riemannian symmetric space of non-compact type.

According to \'E.\@ Cartan's magnificent classification of Lie groups
and symmetric spaces, there are four types of rank-one
(irreducible) Riemannian symmetric spaces of non-compact type, namely the three families of the hyperbolic spaces
over the real numbers, the complex numbers, and the quaternions, and
the one exceptional space, often considered to be the two-dimensional
hyperbolic space over the Cayley numbers.

For our purposes here, we will use the presentation of Riemannian
symmetric spaces as quotients of connected Lie groups by certain
compact subgroups. We refer to Section~\ref{sect2} for details. In our
situation it means that we may and shall identify $S$ with the
Riemannian symmetric space~$G/K$, where $G$ is the connected component
of the identity element of the group of Riemannian isometries on~$S$,
and $K$ is the stabilizer group in~$G$ of an arbitrarily chosen point
$\orgn\in S$. Then $G$ is the group of orientation-preserving Riemannian isometries of~$S$ and a connected simple Lie group of real rank one
with trivial center, $K$ is a maximal compact subgroup of~$G$, and the
Riemannian structure on $G/K$ as a symmetric space is essentially
unique. Using this presentation, our approach and our proofs do not rely on the classification of Riemannian symmetric spaces mentioned above.

Throughout, we let $\Gm$ be a discrete, torsion-free subgroup of~$G$
that is cofinite, but not cocompact, and we let $\Delta$ denote the (positive) Laplacian on~$S$.  The cusp forms for~$\Gamma$ are the $\Gamma$-invariant eigenfunctions of~$\Delta$ that have quick decay at each cusp of~$\Gamma$. (Of course, we could define these forms as being the eigenfunctions of the Laplacian on the Riemannian locally symmetric space $\Gm\backslash S = \Gm\backslash G/K$ with quick decay towards each end. However, working with ``lifted'' functions and asking for invariance has some advantages in our considerations.)

For each cusp form~$u$ we parametrize its $\Delta$-eigenvalue~$\lambda=\lambda(u)$ as $\rho^2-\nu^2$, where $\rho$ is
essentially half the sum of the positive roots of~$S$. More precisely, for convenience, we deviate here slightly from standard conventions and define $\rho$ to be the positive half-integer that corresponds to the half sum of the positive roots using standard normalizations (see Section~\ref{subs2-Lie-alg}). The element~$\nu$ is a complex number depending on~$u$, unique only up to sign. Each choice of~$\nu$ is called a \emph{spectral parameter} of~$u$. We denote the space of cusp forms for~$\Gamma$ with spectral parameter~$\nu$ by~$\A^0(\Gm;\nu)$.

For the cohomological interpretation of the cusp forms, the set of cusps~$\Cu$ of~$\Gm$, considered as a set of points in the geodesic boundary $\partial S$ of~$S$, plays a prominent role. Using the resolution~$F^\pb_\bullet$ with $\Fpb j = \CC\bigl[\Cu^{j+1}\bigr]$ for $j\in\NN_0$, the parabolic cohomology space~$H^i_\pb(\Gm;V)$ of degree~$i\in\NN_0$ and $\CC[\Gm]$-module~$V$ arises in the same manner as the standard cohomology spaces of~$\Gm$ arise from the standard resolution. See Section~\ref{sec:parabcohom}. However, as the action of~$\Gm$ on~$\Cu$ is not free, parabolic cohomology spaces are in general not isomorphic to the standard cohomology spaces. For our purposes we require a realization of the parabolic cohomology spaces using a resolution based on a $\Gm$-invariant tessellation of~$\Sast=S\cup \Cu$ into cells of different dimensions such that the boundary components of higher dimensional cells are made of cells of lower dimensions in a highly controlled way.
With this somewhat refined resolution, the geometry of~$\Sast$ and the action of~$\Gm$ on~$\Sast$ can be sufficiently tracked via the resolution and the parabolic cohomology spaces. We refer to Sections~\ref{sect4-tess}-\ref{prf-pcshc}, and in particular to Proposition~\ref{prop-pcht}, for detailed discussions. In what follows, the tessellation is denoted $\tess = (\X^\tess_i)_{i=0}^{\ds}$ with $\X^\tess_i$ being the collection of $i$-dimensional cells for $i=0,\ldots, \ds$ where $\ds=\dim S$, and the associated resolution is denoted $F^\tess_\bullet$.

For the choice of suitable $\CC[\Gm]$-modules~$V$, we take advantage of the relation between cusp forms for~$\Gm$ and principal series representations. We initially use the $\CC[\Gm]$-module $\V\omi\nu$ of smooth semi-analytic vectors in the spherical principal series representation with spectral parameter~$\nu$. We realize these vectors as (germs of) smooth, semi-analytic functions on the geodesic boundary~$\partial S$ of~$S$.
The points of (potential) non-analyticity of elements of~$\V\omi\nu$ are contained in~$\Cu$. See Section~\ref{sect6} for a full definition.

We develop a certain differential form on~$S$ defined in detail in~\eqref{omnu}, which is denoted here~$\omv_\nu(f; b, \cdot)$, depending on~$f\in C^\infty(S)$, $\nu\in \CC$ and $b\in \partial S$. It is built by means of a generalization of Green's form and of a kernel function on~$(\partial S)\times G$ of  Poisson-type defined in~\eqref{rnudef}, where the latter is based on a real-analytic vector in the spherical principal series representation.
For $f$ being a cusp form (i.e., for $f\in \A^0(\Gm;\nu)$) the form~$\omv_\nu(f; b, \cdot)$ is closed. Using the absolutely convergent integral
\begin{equation}\label{eq:intro_integral}
I(f,X;b) = \int_{x\in X} \omv_\nu(f;b,x)
\end{equation}
for $f\in \A^0(\Gm;\nu)$, $X\in  \X^\tess_{\ds-1}$ and $b\in\partial S$, the map
\begin{equation}\label{eq:intro_psi}
\psi^f \colon \X^\tess_{\ds-1} \to \V\infty\nu(\partial S)\,,\quad X \mapsto I(f,X;\cdot)\,,
\end{equation}
assigns to $f$ the cocycle~$\ps^f$ in the space~$Z^{\ds-1}\bigl( F^\tess_\bullet;\V\infty\nu(\partial S)\bigr)$. Here, $\V\infty\nu(\partial S)$ denotes the $G$-module of smooth vectors on~$\partial S$ in the principal series.
The assignment in~\eqref{eq:intro_psi} descends to the linear map
\begin{equation}\label{eq:intro_beta1}
\btv_\nu\colon \A^0(\Gm;\nu) \rightarrow H^{\ds-1}_\pb\bigl( \Gm; \V \omi \nu \bigr)\,,\quad f\mapsto [\ps^f]\,,
\end{equation}
where $[\ps^f]$ denotes the cohomology class of~$\ps^f$. See Section~\ref{sect6}.

The map~$\btv_\nu$ typically is not bijective, and hence we are interested in better understanding its image and under which conditions it is injective. We first observe that the strong decay of cusp forms and their derivatives towards the cusps of~$\Gm$ yields that for any $f\in \A^0(\Gm;\nu)$, the cocycle~$\ps^f$ maps into the sheaf of smooth semi-analytic vectors~$\V\omi\nu$ in the specific manner that
\[
\ps^f(X) \in \V\om\nu\bigl( (\partial S) \setminus (X\cap \Cu) \bigr)
\]
for each $X\in\X^\tess_{\ds-1}$.
Thus, the points of potential non-analyticity of~$\ps^f(X)$ are precisely those cusps that are ``seen'' by~$X$. We say that cocycles with this property satisfy the \emph{boundary condition}~$\bdc$. Thus, the map~$\btv_\nu$ in~\eqref{eq:intro_beta1} descends to the map
\begin{equation}\label{eq:intro_beta2}
\btv_\nu\colon \A^0(\Gm;\nu) \rightarrow H^{\ds-1}_\pb\bigl( \Gm; \V \omi \nu \bigr)^\bdc\,,\quad f\mapsto [\ps^f]\,,
\end{equation}
for any $\nu\in\CC$. See Theorem~\ref{thm-btv}. We emphasize that we continue to denote all specifications of~$\btv_\nu$ from~\eqref{eq:intro_beta1} by the same symbol.

Furthermore, at least for spectral parameters~$\nu\in\CC$ obeying $\nu\notin\tfrac12\ZZ_{\leq -1}$ and $\re\nu>-\rho$,
the regularity of cusp forms or, equivalently, of smooth semi-analytic vectors in combination with the contraction property of the action of~$\Gm$ towards cusps yields that the cocycles~$\ps^f$ for $f\in\A^0(\Gm;\nu)$ extend real-analytically and with highly controlled growth to the part in~$S$ of so-called excised neighborhoods of~$(\partial S)\setminus E$, where $E$ is the finite set of cusps of non-analyticity of~$\ps^f$. We refer to Section~\ref{sect-spacesgerms}, in particular to Definitions~\ref{exc-nbh} and \ref{def-Woe} as well as Proposition~\ref{prop-omfintW}, for details.
We denote by~$\V{\om^0,\infty,\excg}\nu$ the $\CC[\Gm]$-module (or the sheaf) of elements in~$\V{\omi}\nu$ with these additional extension properties. Then the map~$\btv_\nu$ in~\eqref{eq:intro_beta2} restricts to the map
\begin{equation}\label{eq:intro_beta3}
\btv_\nu\colon \A^0(\Gm;\nu) \rightarrow H^{\ds-1}_\pb\bigl( \Gm; \V{\om^0,\infty,\excg}\nu \bigr)^\bdc\,,\quad f\mapsto [\ps^f]\,.
\end{equation}
We emphasize that in~\eqref{eq:intro_beta3} we changed our initial choice of the $\CC[\Gm]$-module~$V$  being $\V\omi\nu$ to its submodule $\V{\om^0,\infty,\excg}\nu$.

With further restrictions on the values for~$\nu$, we can show that the map~$\btv_\nu$ in~\eqref{eq:intro_beta3} is indeed an isomorphism, as stated in the following partial statement of our main theorem.

\begin{mainthm}[partial statement of main theorem]\label{thm:main1}
Let $S$ be a rank-one Riemannian symmetric of
non-compact type and let $\Gm$ be a cofinite, torsion-free discrete
group of orientation-preserving Riemannian isometries of~$S$ that is
not cocompact. For any (spectral) parameter
$\nu \in\CC\setminus \frac12\ZZ$ with $|\re \nu| < \rho$, the map
\[
\btv_\nu\colon \A^0(\Gm;\nu) \rightarrow H^{\ds-1}_\pb\bigl(\Gm;\V{\om^0,\infty,\excg}\nu\bigr)^\bdc\,,\quad  f\mapsto [\ps^f]\,,
\]
is an isomorphism of vector spaces.
\end{mainthm}

For the proof of Theorem~\ref{thm:main1} we employ several further entities and establish intermediate isomorphisms, as surveyed in what follows. In this way, we prove a stronger version of Theorem~\ref{thm:main1} in the sense that we provide an explicit formulation of the inverse isomorphism.

To establish the map~$\btv_\nu$ as in~\eqref{eq:intro_beta2}, i.e., as a map into~$H^{\ds-1}_\pb\bigl( \Gm; \V \omi \nu \bigr)^\bdc$, we proceed via a careful analysis of the properties of cusp forms and the integral in~\eqref{eq:intro_integral}. Additional entities are used for determining the image of~$\btv_\nu$, i.e., establishing it as the map in~\eqref{eq:intro_beta3}, and showing that it is an isomorphism for a suitable restricted regime for~$\nu$.

In a first step, we relate the sheaf~$\V \om \nu$ to $\nu$-analytic boundary germs. In a nutshell, the space of boundary germs with spectral parameter~$\nu$ forms a sheaf on~$\partial S$, which is the direct limit of smooth Laplace eigenfunctions with spectral parameter~$\nu$. A $\nu$-analytic boundary germ on an open subset~$I$ of~$\partial S$ is a boundary germ that extends analytically across~$\partial S$ in the following sense: We use an analytic extension~$\hat S$ of~$S$ in which $\partial S$, and hence $I$, is in the interior. The $\nu$-analytic boundary germ is then represented by a smooth Laplace eigenfunction~$u$ on~$U\cap S$ for some open neighborhood~$U$ of~$I$ in~$\hat S$ with the additional property that there is an analytic function~$A_u$ on~$U$ such that
\[
u(x) = t(x)^{-(\rho+\nu)}A_u(x)
\]
for all $x\in U\cap S$, where $t(x)$ is, in a certain sense, the vertical distance of $x$ from~$\partial S$. For details see Section~\ref{sect8}. We denote by $\W \om \nu (I)$ the space of $\nu$-analytic boundary germs on~$I$. The restriction map $\rho_\nu\colon u\mapsto A_u$ is a $G$-equivariant sheaf morphism $\W \om \nu\rightarrow \V \om \nu$ for all~$\nu\in\CC$; for $\nu\notin\tfrac 12\ZZ_{\leq -1}$ it is an isomorphism. See Theorem~\ref{thm-isoVW}. These properties are preserved for the extension to $\nu$-semi-analytic entities and induce morphisms $\W {\om^0} \nu \to \V {\om^0} \nu$ as well as $\W \omi \nu \to \V \omi \nu$ (isomorphisms for $\nu\notin\tfrac12\ZZ_{\leq -1}$), both denoted $\rho_\nu$.

For $\nu\notin \tfrac12\ZZ_{\leq -1}$, the isomorphism $\rho_\nu$ between $\W \omi \nu$ and $\V \omi \nu$ in combination with the map~$\btv_\nu$ in~\eqref{eq:intro_beta2} yields the map
\begin{equation}\label{eq:intro_beta4}
\btw_\nu = \rho_\nu^{-1}\circ\btv_\nu\colon \A^0(\Gm;\nu) \rightarrow H^{\ds-1}_\pb\bigl( \Gm; \W \omi \nu \bigr)^\bdc\,,\quad f\mapsto [\rho_\nu^{-1}(\ps^f)]\,.
\end{equation}
For $\btv_\nu$ in this ``lifted'' formulation, a thorough analysis of the preimage of the differential form~$\omv_\nu(f; \cdot, x)$ and the integral~$I(f,X;\cdot)$ from~\eqref{eq:intro_integral} under~$\rho_\nu$ allows us to detect that the image of~$\btw_\nu$ is indeed in $H^{\ds-1}_\pb\bigl(\Gm;\W{\om^0,\infty,\excg}\nu\bigr)^\bdc$, at least for $\nu\in\CC$ obeying $\nu\notin \tfrac12\ZZ_{\leq -1}$ and $\re \nu > -\rho$. See Section~\ref{sect9}, in particular Propositions~\ref{prop-omfintW} and~\ref{prop-btimg}. As $\W{\om^0,\infty,\excg}\nu$ is, by definition, the preimage of~$\V{\om^0,\infty,\excg}\nu$ under the isomorphism~$\rho_\nu$, at this stage we have established the cohomology space in the image of~$\btv_\nu$ in Theorem~\ref{thm:main1}.

In the next step, to show that the map~$\btv_\nu$ is an isomorphism between $\A^0(\Gm;\nu)$ and $H^{\ds-1}_\pb\bigl( \Gm; \V{\om^0,\infty,\excg}\nu \bigr)^\bdc$ for $\nu\in\CC\setminus\tfrac12\ZZ_{\leq 0}$ obeying $|\re \nu|< \rho$, we work again with the (isomorphic) map~$\btw_\nu$ and show that it is an isomorphism between~$\A^0(\Gm;\nu)$ and $H^{\ds-1}_\pb\bigl( \Gm; \W{\om^0,\infty,\excg}\nu \bigr)^\bdc$. To that end we  construct an explicit inverse map to~$\btw_\nu$ for $\nu\in\CC\setminus\tfrac12\ZZ_{\leq 0}$ obeying $|\re \nu|< \rho$.

We start with providing, for
$\nu\in\CC\setminus\tfrac12\ZZ_{\leq-1}$ with $|\re\nu|<\rho$, a map
\begin{equation}\label{eq:intro_alpha}
\alw_\nu\colon H^{\ds-1}_\pb\bigl( \Gm; \W{\om^0,\excg}\nu\bigr)^\bdc\to \A^0(\Gm;\nu)
\end{equation}
that is left inverse to $\ii\circ\btw_\nu$, where $\ii$ is the inclusion map from $H^{\ds-1}_\pb\bigl( \Gm; \W{\om^0,\infty,\excg}\nu\bigr)^\bdc$ to $H^{\ds-1}_\pb\bigl( \Gm; \W{\om^0,\excg}\nu\bigr)^\bdc$, and where $ \W{\om^0,\excg}\nu$ denotes the sheaf of $\nu$-analytic boundary germs that satisfy the extension property mentioned above but for which the extension map is not necessarily smooth at the points of non-analyticity.

To establish this map we provide a specific construction that assigns to each cocycle class $[\psi]\in H^{\ds-1}_\pb\bigl( \Gm; \W{\om^0,\excg}\nu\bigr)^\bdc$ a Laplace eigenfunction~$u_\psi$ with spectral parameter~$\nu$ that is defined on all of~$S$, not only on some subset of~$S$ as \emph{a priori} requested by the definition of boundary germs, such that the $\W{}{}$-version of the integral transform~$I$ in~\eqref{eq:intro_integral} is locally reproducing for~$u_\psi$. See Proposition~\ref{prop-cohE}. This is the most crucial result for building the inverse map. The definition of~$\alw_\nu$ is constructive, rather explicit and takes advantage of the realization of the cohomology spaces using a tessellation of~$\Sast$. Further, under our restrictions on~$\nu$, the map $u_\psi$ is indeed in~$\A^0(\Gm;\nu)$, and if $[\psi]$ arises from an element $u\in\A^0(\Gm;\nu)$, then $u_\psi = u$. See Propositions~\ref{prop-cohE} and \ref{prop-pi}.

While these results show that $\alw_\nu$ is surjective and a left-inverse to~$\ii\circ\btw_\nu$, it still admits the possibility that two different cocycle classes in~$H^{\ds-1}_\pb\bigl( \Gm; \W{\om^0,\excg}\nu\bigr)^\bdc$, say $[\psi_1], [\psi_2]$ with $[\psi_1]\not=[\psi_2]$, are mapped to the same element in~$\A^0(\Gm;\nu)$, thus $u_{\psi_1}=u_{\psi_2}$. At this point in the argumentation, we cannot even rule out this possibility for cocycle classes in~$H^{\ds-1}_\pb\bigl( \Gm; \W{\om^0,\infty,\excg}\nu\bigr)^\bdc$. Hence, this possibility still prevents us from concluding if~$\btw_\nu$ is an isomorphism or not. Therefore, in Section~\ref{sect11} we explicitly determine the kernel of~$\alw_\nu$ for $\nu\in\CC$, $|\re\nu|<\rho$, satisfying $\nu\notin\tfrac12\ZZ_{\leq-1}$. See Proposition~\ref{prop-qIwe}.

By means of sheaf cohomology we relate this kernel, which uses a certain space of global representatives for the $\nu$-analytic boundary germs with sufficient regularity properties, to a cohomology space with a certain module, $\E^ {\om^0, \excg} _\nu$ resp.\@ $\E^ {\om^0,\infty,\excg}_ \nu$, of Laplace eigenfunctions. See~\eqref{HisoEG} and Proposition~\ref{prop-dsc}. In Section~\ref{sect12} we show that under the additional condition that $\nu\notin\tfrac12\ZZ$, the kernel is indeed trivial. See Theorem~\ref{thm-E}. Combining these results with the isomorphism~$\rho_\nu$ gives the full statement of our main result, which we can state now. We emphasize that \eqref{mnthm:nsp} in this theorem handles more values of~$\nu$ than~\eqref{mnthm:nugrl}.

\begin{mainthm}[full statement of main result]\label{mainthm}
Let $S$ be a rank-one Riemannian symmetric of
non-compact type and let $\Gm$ be a cofinite, torsion-free discrete
group of orientation-preserving Riemannian isometries of~$S$ that is
not cocompact. Let $\nu \in \CC$, $|\re\nu|<\rho$.
\begin{enumerate}[label=$\mathrm{(\roman*)}$, ref=$\mathrm{\roman*}$]
\item \label{mnthm:nsp} If $\nu \not \in \frac12\ZZ_{\leq -1}$, then the map
\[
\bt_\nu\colon\A^0(\Gm;\nu) \rightarrow H^{\ds-1}_\pb\bigl(
\Gm;\V{\om^0,\infty,\excg}\nu\bigr)^\bdc
\]
from~\eqref{eq:intro_beta3} induces a linear isomorphism between $\A^0(\Gm;\nu)$ and the quotient
\[
H^{\ds-1}_\pb\bigl( \Gm;\V{\om^0,\infty,\excg}\nu\bigr)^\bdc \bigm/
H^{\ds-1}_\pb\bigl( \Gm;\rho_\nu\E^{\om^{0},\infty,\excg}_\nu\bigr)^\bdc\,.
\]
\item \label{mnthm:nugrl} If $\nu \not \in  \frac12\ZZ$, then $\bt_\nu$
is a linear bijection.
\end{enumerate}
\end{mainthm}

\subsection{Motivation and relation to previous results}

Our investigations in this article are a natural continuation of the studies of cohomological interpretations of Maass cusp forms and more general automorphic forms in, e.g., \cite{Ei57,LZ01, BLZ15,Bruggeman_lewiseq, BCD, tahA, ABC_rwt}. We refer the reader to the introductional sections of these references for a survey of the development of this research field and for further references, and we restrict ourselves here to a brief explanation of the relation of the present article to this previous results. The results in~\cite{Ei57,LZ01, BLZ15,Bruggeman_lewiseq, BCD, tahA} are for the situation of Fuchsian groups and hyperbolic orbisurfaces, whereas \cite{ABC_rwt} is the (to our knowledge) first result beyond hyperbolic orbisurfaces, namely for Jacobi Maass cusp forms. In the present article we vastly go beyond hyperbolic orbisurfaces (hence two-dimensional spaces) and study arbitrary rank-one locally symmetric Riemannian symmetric spaces of non-compact type that are non-compact but of finite volume. In view of the length of this article, we here restrict to torsion-free fundamental groups to avoid to spend space on additional, necessary discussions.

In our approach, we follow as much as possible the approach in~\cite{BLZ15}. However, the presence of more dimensions in the considered spaces as well as the more involved geometry in particular for hyperbolic spaces over the complex numbers or the quaternions as well as for the exceptional space yield new obstacles to overcome and require to adapt several steps.

In the situation of the hyperbolic plane and Fuchsian groups, as handled in~\cite{BLZ15} and also in~\cite{Ei57}, the passage from (Maass) cusp forms to cocycle classes is given by an integral transform for which the considered cusp form is integrated in a suitably chosen differential form against a suitably chosen kernel function along a well-chosen path (here, a geodesic) with a certain geometric meaning. See, in particular, \cite[\S2.1]{Ei57} and \cite[\S5]{BLZ15}. This structure is also visible in~\cite{LZ01} (also a cohomological interpretation is discussed in there) and it is highly present in essentially all further articles discussing cohomological interpretations of automorphic forms. We could preserve this structure also in the much more general rank-one situations that we consider here by developing a suitable kernel and exchanging the geodesics used as path of integration by tesselation elements of dimension $\ds-1$.

Moreover, the cohomology space in the image of the map~$\btv_\nu$ is more involved than in dimension~$\ds=2$. See Theorem~\ref{mainthm}.
In particular, for $\ds=2$, the kernel of the map~$\alw_\nu$ in~\eqref{eq:intro_alpha} is trivial for all $\nu\notin\tfrac12\ZZ_{\leq -1}$,
whereas for higher dimensions this is not necessarily the case. This issue leads to a much more involved argumentation than in the case of hyperbolic surfaces. In Section~\ref{sect13} we provide a more elaborate comparison.

\subsection{Outline of article}
In Section~\ref{sect2} we provide a survey of the necessary background material on Lie theory and symmetric spaces. In Section~\ref{sect3} we discuss the class of discrete groups that we consider here and we present the notion of automorphic forms and cusp forms that we will use. With Section~\ref{sect4} the main bulk of work in this article starts. In Section~\ref{sect4} we present the necessary concept of subgroup-invariant tessellation and develop parabolic cohomology based on a tessellation. In Section~\ref{sect5} we discuss the spherical principal series and their realization on~$\partial S$. These preparations allow us to develop, in Section~\ref{sect6}, a kernel function that generalizes the kernel function from~\cite{BLZ15} and further to construct an integral transform that associates a cocycle with values in the spherical principal series to any given cusp form. This provides the passage from cusp forms to cocycle classes. In Sections~\ref{sect7}--\ref{sect12} we study the image of this map and discuss its inverse map under the restrictions on the spectral parameter also presented above.
In the concluding Section~\ref{sect13} we comment on our approach and its relation to~\cite{BLZ15} and discuss as well some potential continuations of this work.

\subsection{Acknowledgements}
YJC was partially supported by the grants NRF-2022\-R1A2B5B0100187111  and BSRI-2021R1A6A1A1004294412. Further, we would like to thank the Mathematisches Forschungsinstitut Oberwolfach (MFO) for the 2025 workshop ``Cohomology theories for Automorphic forms and Enumerative Algebra'' and its hospitality, which provided us with an opportunity to meet in person for an intensive period of research time.

%% file: ccro2-2-group-symsp.tex

\def\flnm{ccro2-2-group-symsp}
\bigskip

\section{Lie groups and symmetric spaces}\label{sect2}
In this section we provide a review of background information on
Riemannian symmetric spaces and Lie groups needed for our purpose. For
proofs and further details, we usually refer to~\cite{Wal88}.
Alternative standard references are, e.g., \cite{Hel00, Hel62,
Eberlein}.

We start with a survey of the most essential concepts. We then provide
further discussion of technical facts concerning the Riemannian metric
as well as of notations and methods that we will use.
\smallskip

\subsection{Lie groups and symmetric spaces of rank
one}\label{sect2-gr-smsp}
By de Rham decomposition, any Riemannian symmetric space of
non-compact type decomposes into a direct product of irreducible
Riemannian symmetric spaces of non-compact type. Any irreducible
Riemannian symmetric space~$S$ of non-compact type can be presented as a
quotient space $G/K$ with a non-compact connected simple real Lie group
$G$ and a maximal compact subgroup $K$. For any given~$S$, the choices
of~$G$ and~$K$ are not unique, not even up to isomorphisms of Lie
groups. However, this non-uniqueness is highly controlled. Namely, the
Lie group $G$ may have a nontrivial center, which is then automatically
contained in $K$. By choosing \il{G}{$G$}$G$ connected and with a
trivial center, which we may and shall do here, $G$ as well as $K$ are
unique up to isomorphisms of Lie groups.

Regarding this choice, $G$ is isomorphic to the group of
orientation-preserving isometries of~$S\cong G/K$, and it is also
isomorphic to a subgroup of $\GL(\glie)$ for the Lie algebra
\il{glie}{$\glie$}$\glie$ of $G$. The rank of the Riemannian symmetric
space~$S$ coincides with the real rank of the Lie group~$G$. Thus,
throughout we will choose $G$ to be a non-compact simple connected real
Lie group of real rank one with trivial center.

We will take advantage of several standard decompositions of~$G$, which
we will survey in what follows.

\rmrk{The Iwasawa decomposition of \intitle{G}} We fix an
\il{Id}{Iwasawa decomposition}{\em Iwasawa decomposition} \il{N}{$N$}
\il{A}{$A$} \il{K}{$K$} $G=NAK$, with $N$ maximal unipotent in $G$, $K$
maximal compact in $G$ and $A\cong \RR^\ast_{>0}$. All Iwasawa
decompositions are conjugate to each other. See,
e.g.,~\cite[\S2.1]{Wal88}.

The \il{ssp}{symmetric space}{\em symmetric space} of $G$ is the
quotient \il{S}{$S=G/K$} $G/K=:S$. The group $G$ acts on $S$ by
multiplication on the left. We denote its real dimension by
\il{d}{$\ds=\dim G/K$}$\ds$.

The dimension of $A$ is the real rank of $G$. In this paper, this rank
is $1$, and $A$ is isomorphic to $\RR^\ast_{>0}$. \il{ro}{rank one}We
choose a continuous group homomorphism
\il{am}{$\am(t)\in A$}$\am\colon\RR_{>0}^\ast \rightarrow A$. Then
there is a unique element \il{H0}{$\H_0 \in \alie$}$\H_0$ in the Lie
algebra $\alie$ of $A$ such that $\am(e^x) = \exp( x\H_0)$.

The unipotent group $N$ is isomorphic to $\RR^{\ds-1}$ as an analytic
manifold. For many groups $G$ of rank one the subgroup $N$ is not
abelian. Then $N$ is not isomorphic to the additive group
$\RR^{\ds-1}$. In all cases, the commutator subgroup of $N$ is equal to
the center of $N$.

An important subgroup is \il{M}{$M$\ centralizer of
$A$}$M = \bigl\{ m\in K \;:\; m a m^{-1}= a \text{ for all } a\in A\bigr\}$.
We can choose an element $\wm\in K$ such that
\il{wm}{$\wm\in K$}$\wm \am(t)\wm^{-1}=\am(1/t)$ for all $t>0$. We can
and shall arrange the choice in $\wm M$ such that $\wm^{-1}=\wm$. See
\cite[Lemma 1.5]{Hel70}.

The Iwasawa decomposition is unique: it gives a diffeomorphism
\[ N \times A \times K \rightarrow G\colon (n,a,k) \mapsto nak\,. \]

\rmrk{Two more decompositions for \intitle{G}}The \il{Cd}{Cartan
decomposition}{\em Cartan decomposition} $G=KAK$ leads to a
diffeomorphism $(K/M) \times A^+ \times K\cong G\setminus K$ with
\il{A+}{$A^+$}$A^+ = \bigl\{\am(t)\;:\; t> 1 \bigr\}$. See
\cite[\S~2.4.2]{Wal88} for these facts.

For rank one groups the \il{Bd}{Bruhat decomposition}{\em Bruhat
decomposition} takes the form
\be\label{Bd} G = NAM \sqcup N \wm MAN\,.\ee
See \cite[Theorem 2.2.10]{Wal88}. We note that $M$ and $A$ commute, and
that both normalize $N$. So we have $MAN=NAM=NMA$.
\rmrk{Analytic structure} A function on $G$ is \il{anal}{analytic}{\em
analytic} if it can locally be given as an analytic function of the
matrix entries in the realization of $G \subset \GL(\glie)$ (with a
choice of a basis). We use `analytic' to mean `real-analytic': locally
given by power series in analytic coordinates. This provides $G$ with
the structure of an analytic manifold.

This induces an analytic structure on closed subgroups of $G$. For
instance the coordinate function $t\leftrightarrow \am(t)$ on $A$ is
analytic. The diffeomorphism $S\cong NA$ determines an analytic
structure on~$S$.

\rmrk{The symmetric space and its boundary}We denote by
\il{orgn}{$\orgn\in S$}$\orgn\in S$ the image of the unit element
$e\in G$.

The map $n a\mapsto n a \,\orgn$ (based on the Iwasawa decomposition)
gives a diffeomorphism $NA \rightarrow S$. Similarly the Cartan
decomposition gives the map $(k,t) \mapsto k \am(t) \,\orgn$, giving a
diffeomorphism $(K/M) A^+ \rightarrow S\setminus\{\orgn\}$. As
$t\uparrow \infty$ the point $k \am(t)\, \orgn$ approaches the boundary
$\partial S=K/M$
of $S$. \il{bdS}{$\partial S$}
We name two special points of $\partial S$:\il{inftnul}{$\inft$, \; \nul
\text{ points of $\partial S$}}
\be \label{inft-nul}
\inft = \lim_{t\rightarrow\infty } \am(t)\,\orgn\quad\text{and}\quad
\nul = \lim_{t\rightarrow \infty} \wm \am(t)\,\orgn= \wm\,\inft =
\lim_{t\downarrow 0} \am(t) \orgn\,. \ee
We use the boldface version $\inft$ for the particular point of
$\partial S$ fixed by $MAN$; by $\infty $ we mean the standard infinity
in the two-point compactification of $\RR$.

The Iwasawa decomposition can be inverted to give the opposite Iwasawa
decomposition $G=KAN$. This implies that $\partial S= K/M$ is
isomorphic to $G/MAN$. The subgroup of $G$ fixing $\inft$ is $MAN$. The
Bruhat decomposition~\eqref{Bd} shows that
\be\label{Brh-pS} \partial S = \bigl\{\inft\}\sqcup N \,\nul\,.\ee

\rmrk{Action of \intitle{G} on functions on \intitle{G} and on
\intitle{S}}\label{sect2-trdf} The group $G$ acts on itself by left and
by right translations. That induces representations of $G$ in
$C^\infty(G)$:\il{L}{$L$} \il{R}{$R$}
\badl{RLg} R(g_1) f &\colon g \mapsto f(gg_1)\,,\\
L(g_1) f &\colon g\mapsto f(g_1^{-1}g)\,.
\eadl
The inverse in the definition of $L$ is necessary if one works with left
representations, which satisfy $L(g_1 g_2) = L(g_1) L(g_2)$.

\rmrk{ Action of the Lie algebra} The Lie algebra $\glie$, which will be
discussed in~\S\ref{subs2-Lie-alg} in more detail, acts on functions in
$C^\infty(G)$ by left and right differentiation: For each
$\XX\in \glie$ there is a group homomorphism
$\RR \rightarrow G\colon t\mapsto \exp(t\XX) $ that has at $t=0$ the
direction corresponding to $\XX$. We denote the corresponding derived
actions of $\XX\in \glie$ by the same symbols $R$ and $L$:
\badl{RLa} R(\XX) f &\colon g\mapsto \partial_t f\bigl( g \exp(t\XX)
\bigr)\bigm|_{t=0}\,,\\
L(\XX) f &\colon g\mapsto \partial_t f \bigl( \exp(-t\XX) g\bigr)
\bigm|_{t=0}\,.
\eadl
These actions extend $\CC$-linearly to elements of the complex Lie
algebra \il{gliec}{$\glie_c$}$ \glie_c = \CC\otimes_\RR \glie$. We note
that the actions $R$ by right translation and $L$ by left translation
commute.

Functions $f$ on $S$ correspond to functions on $G$ that satisfy
$R(k)f = f$ for all $k\in K$. Thus, we can use $L$ to denote the action
of $G$ on functions on $S$ and the action of $\glie$ on differentiable
functions on~$S$.

A function $f\in C^\infty(G)$ is \emph{right}
\il{Kfin}{$K$-finite}\emph{$K$-finite} if the functions $R(k)f$ with
$k\in K$ generate a finite-dimensional space, which is then
$R(K)$-invariant. By \il{Gi_K}{$C^\infty(G)_K$}$C^\infty(G)_K$ we
denote the subspace of right $K$-finite functions in $C^\infty(G)$.
This subspace is invariant under $R(\XX)$ for all $\XX\in \glie$, but
not under $R(g)$ for all $g\in G$.

\rmrk{The Riemannian structure of \intitle{S}}The Lie algebra $\glie$ of
$G$ can be defined as the tangent space $\Ta e G$ at the unit element
$e$ of $G$, and the tangent space \il{Ta}{$\Ta x S$ tangent space of
$S$ at $x$}$\Ta \orgn S$ at $\orgn\in S$ can be identified with a
linear subspace of $\glie$. The Lie algebra structure provides
$\Ta \orgn S$ with a positive definite bilinear form, which can be
transported by the action of $G$ on $S$ to all tangent spaces
$\Ta x S$, thus providing $S$ with a $G$-invariant linear Riemannian
structure.  We will specify it in \eqref{gRi}. 
\begin{description}

\item[Laplace operator]\il{Dt0}{$\Dt$}The Riemannian structure
determines the Laplace operator. This is a second order differential
operator on $S$. It commutes with the action of $G$:
\be \Dt L(g) f = L(g) \Dt f \ee
for $g\in G$ and $f\in C^2(U)$ with $U\subset S$ open.

The operator $\Dt$ is an elliptic differential operator with analytic
coefficients. Hence its eigenfunctions are analytic.

\item[Distance function]The Riemannian metric determines a distance
function
\[ \dist\colon S\times S\rightarrow \RR_{\geq 0} \]
which turns $S$ into a metric space. It satisfies
$\dist (g x, g y) = \dist(x,y)$ for all $g\in G$, $x,y\in S$.

\item[Volume form]The Riemannian structure also determines a $\ds$-form
\il{ztS}{$\z_s$}$\z_S$ on $S$ that satisfies $ \z_S \circ l_g = \z_S$
for all $g\in G$, where $l_g: x\mapsto g\, x$. By
$f\mapsto \int_S f\, \z_S$, $f\in C_c(S)$, we obtain an invariant
positive measure \il{dmu}{$d\mu$ invariant measure}$d\mu$ on $S$ that
is invariant under $f\mapsto L(g) f$ for all $g\in G$.
\end{description}

\rmrk{Some references}The facts mentioned above concern the common
structure of Lie groups of real rank one, as discussed in \cite{Wal88},
\cite{JoWa77}, \cite{Johns76} and also \cite{Hel00}. One may consult
also \cite{CDKR1,CDKR2}, and \cite{Hof79} for alternative approaches.

\subsection{The Lie algebra of \intitle{G}}\label{subs2-Lie-alg}The
tangent space $\Ta e G = \glie$ has the structure of a Lie algebra,
given by a bilinear operation
$\glie\times\glie\rightarrow\glie: (\XX,\YY) \mapsto [\XX,\YY]$. To
closed subgroups correspond Lie subalgebras, which we denote by the
corresponding lower case fraktur letter: $\alie = \Lie(A)$,
$\nlie=\Lie(N)$, etc. \il{nlie}{$\nlie$}
\il{alie}{$\alie$}\il{klie}{$\klie$} The Iwasawa decomposition implies
that the vector space $\glie $ is the direct sum
$ \nlie \oplus \alie\oplus \klie$ of vector spaces.

The choice of the isomorphism $\am\colon \RR^\ast_{>0}\rightarrow A$
determines an element $\H_0 \in \alie$. The operator
$\ad(\H_0):\XX\mapsto [\H_0,\XX]$ in $\glie$ can be brought into
diagonal form, with eigenvalues in $\{-2,-1,0,1,2\}$. This corresponds
to the fact that $(\glie,\alie)$ has a one-dimensional root system of
the form $\{\al,0,-\al\}$ or $\{ 2\al,\al,0,-\al,-2\al\}$, where $\al$
is the linear form on $\alie$ that satisfies $\al(\H_0)=1$. For any
$j\in\{0,\pm 1, \pm2\}$, the root space $\glie_{j\al} \subset \glie$
consists of the elements $\XX\in \glie$ such that
$\bigl[\H,\XX\bigr] = j \,\al(\H)\XX$ for all $\H \in \alie$. This
gives a root space decomposition \il{glieal}{$\glie_{j\al}$}
\be \label{rtdecomp} \glie = \glie_{-2\al}\oplus \glie_{-\al} \oplus
\glie_0 \oplus \glie_\al \oplus \glie_{2\al}\,. \ee

The simple root \il{al}{$\al$ simple root for $(\glie,\alie)$}$\al$ is a
linear form on the one-dimensional space~$\alie$. Each $\nu \in \CC$
determines a character of $A$ given by
\be\label{am-char} \am(t)\mapsto \am(t)^{\nu\al} \coloneqq t^\nu\,.\ee

We have $\nlie = \glie_\al\oplus \glie_{2\al}$. If $N$ is abelian then
$\glie_{2\al}=\{0\}$, otherwise $\glie_{2\al}$ is the Lie algebra of
the center \il{Z}{$Z_X$ center of $X$}$Z_N$ of~$N$. The \il{Wg}{Weyl
group}{\em Weyl group} of $(\glie,\alie)$ is a group of order two; the
non-trivial element can be represented by $\wm$. This element $\wm$
normalizes the group $M$. We have
\il{Ad}{$\Ad(g)\colon \X \mapsto g \X g^{-1}$}
$\Ad(\wm) \glie_{j\al} = \glie_{-j\al}$. The sum
$\glie_{-\al}\oplus \glie_{-2\al}$ is the Lie algebra
$\nlie^- = \Ad(\wm) \nlie$ of the opposite unipotent group $N^-$, which
is equal to $\wm N \wm^{-1}$. Furthermore,
$\glie_0 = \alie \oplus \mlie$.

We put \il{rho}{$\rho$}\il{pp}{$p$}\il{qq}{$q$}
\be\label{qq} \rho=\rho_G = \Bigl(\frac p2+q\Bigr)\,,\qquad p=\dim_\RR
\glie_\al\,,\quad q= \dim_\RR \glie_{2\al}\,. \ee
The linear form $\rho \al$ on $\alie$ is half the sum of the positive
roots (with multiplicities).

\subsubsection{Killing form and choice of basis}\label{sect2-Kfbg}For
each $\YY\in \glie$ we have the operator
\il{ad}{$\ad(\YY)\colon \XX\mapsto [\YY,\XX]$}$\ad(\YY) : \XX\mapsto [\YY,\XX]$
in $\glie$. The Killing form is the bilinear form on $\glie$ given by
$(\XX,\YY) \mapsto \Tr\bigl( \ad(\XX) \ad(\YY)\bigr)$. For the
semisimple group $G$ this form is non-degenerate. We multiply it by a
positive quantity such that $(\H_0,\H_0)\mapsto 1$, and denote this
\il{Kf}{Killing form, normalized}{\em normalized Killing form}
by~$B$.\il{B}{$B(\cdot,\cdot)$ normalized Killing form}

A \il{Ci}{Cartan involution}{\em Cartan involution}
$\glie\rightarrow \glie$ is an involution of $\glie$ that preserves the
Lie algebra structure. See \cite[\S2.1.1, \S2.1.4]{Wal88}. The choice
of a Cartan involution corresponds to the choice of a maximal compact
subgroup, on which the Cartan involution has eigenvalue $1$. We use the
Cartan involution \il{tht}{$\tht$ Cartan involution}$\tht$ such that
$\klie$ is the $1$-eigenspace of $\tht$. By $\plie$ we denote the
$(-1)$-eigenspace. It contains $\RR\H_0 = \alie$. We can find a basis
of $\glie$\il{XXi}{$\XX_i \in \nlie$}\il{UUi}{$\UU_j\in \mlie$}
\be \label{basX} \bigl\{ \H_0,
\XX_1,\ldots,\XX_{p+q},\tht\XX_1,\ldots,\tht\XX_{p+q}, \UU_1,\ldots,
\UU_{\dim(\mlie)} \bigr\}\ee
with the following properties
\begin{enumerate}[label=$\mathrm{(\roman*)}$, ref=$\mathrm{\roman*}$]
\item $\XX_1,\ldots,\XX_p$ span $ \glie_\al\subset \nlie$, and
$\XX_{p+1},\ldots, \XX_{p+q} $ span $\glie_{2\al}=\Lie(Z_N)$;
\item $\tht\XX_1,\ldots,\tht\XX_p$ span $ \glie_{-\al}\subset \nlie^-$,
and $\tht\XX_{p+1},\ldots, \tht\XX_{p+q} $ span
$\glie_{-2\al}=\Lie(Z_{N^-})$;
\item $\UU_1,\ldots,\UU_{\dim(\mlie)}$ span $\mlie\subset \glie_0$;
\item the basis in \eqref{basX} and the basis
\be\label{basXd} \bigl\{ \H_0,
-\tht\XX_1,\ldots,-\tht\XX_{p+q},-\XX_1,\ldots,-\XX_{p+q},
-\UU_1,\ldots,
-\UU_{\dim(\mlie)} \bigr\} \ee
are dual with respect to $B$. (This means that $B(a,b)=1$ if $a$ and $b$
have the same position in both bases, and $B(a,b)=0$ otherwise.)
\item \label{bchXX} $\bigl[\XX_i,\tht \XX_i\bigr] = \begin{cases} - \H_0
&\text{ for }1\leq i \leq p\,,\\
-2 \H_0 &\text{ for }p+1\leq i \leq p+q\,.
\end{cases}$
\end{enumerate}

Based on the choice of the basis in \eqref{basX} we define for
$i=1,\ldots,p+q$\il{Vi}{$\VV i \in \plie$}\il{Wi}{$\WW i \in \klie$}
\be \label{VWdef}\VV i = \frac1{\sqrt 2}(\XX_i - \tht \XX_i)\,,
\qquad \WW i = \frac1{\sqrt 2}(\XX_i+\tht\XX_i)\,.\ee
The elements $\VV 1,\ldots,\VV{p+q}, \H_0$ form an orthonormal basis for
the positive definite bilinear form $(X,Y) \mapsto -B(X,\tht Y)$ on
$\plie$. The elements
\be \WW 1,\ldots,\WW{p+q}, \UU_1,\ldots, \UU_{\dim(\mlie)}\ee
form an orthonormal basis of $\klie$ for the positive definite bilinear
form $-B$.

The fact that $\glie=\klie\oplus\plie$ corresponds to the splitting in
eigenspaces of the Cartan involution $\tht$ for the eigenvalues $1$ and
$-1$, implies that $[\Z,\XX] \in \plie$ for $\Z\in \klie$ and
$\XX\in \plie$. In other words $\Z\mapsto \ad(\Z)$ determines a Lie
algebra representation of $\klie$ on $\plie$. The form $B$ satisfies
$ B\bigl(\ad(\Z)\XX,\YY\bigr) = - B \bigl(\XX,\ad(\Z) \YY\bigr)$, and
hence
\be -B\bigl( \XX,\ad(\Z)\YY\bigr)-B\bigl( \ad(\Z)\XX,\YY\bigr)=0\qquad
\text{ for }\Z\in \klie,\; \XX,\YY \in \plie\,.\ee
See, eg, \cite[Chap~II, \S6, (1)]{Hel62}. This shows that the Lie
algebra representation $\ad$ of $\klie$ on $\plie$ preserves the
Hilbert space structure of $\plie$ corresponding to $-B$, and that the
corresponding representation $\Ad$ of $K$ on~$\plie$ is orthogonal for
$-B$.

\subsection{Coordinates on \intitle{S}}\label{sect2-hsphco} We consider
two systems of coordinates on $S$ and give their relation.

\rmrk{Polar coordinates}\il{polco}{polar coordinates}The Cartan
decomposition gives a diffeomorphism between $S\setminus\{\orgn\}$ and
$(K/M) \times(1,\infty)$.

\rmrk{Coordinates on \intitle{N}} Let us take coordinates
$(x_1,\ldots,x_{p+q}) $ on $N$ such that the derivation
$(\partial_{x_i})_ e $ at the unit element $e$ is equal to the
derivation $\sqrt 2\, R(\XX_i)_e$. By $\nm(x_1,\ldots,x_{p+q})$ we
denote the corresponding element of $N$. We put
\il{x12}{$x^{(1)}\in \RR^p,\; x^{(2)} \in \RR^q$}
\be x^{(1)}=(x_1,\ldots, x_p)\,,\qquad x^{(2)}=(x_{p+1},\ldots,
x_{p+q})\,,\ee
and write \il{nm}{$\nm(x) = \nm(x^{(1)},x^{(2)})\in
N$}$\nm(x_1,\ldots,x_{p+q})=\nm(x^{(1)},x^{(2)})$. The coordinates can
be chosen such that
\begin{align}
\label{anai} &\am(t) \nm(x^{(1)},x^{(2)}) \am(t)^{-1} =
\nm(tx^{(1)},t^2x^{(2)})\,,\\
\label{nprhc}
&\nm(x^{(1)},x^{(2)}) \, \nm(u^{(1)},u^{(2)}) = \nm\bigl(
x^{(1)}+u^{(1)}, x^{(2)} + u^{(2)} + b(x^{(1)},u^{(1)})\bigr)\,,
\end{align}
where $b$ is a bilinear map $\RR^p \times\RR^p \rightarrow  \RR^q $.
This can be arranged by letting
\be\label{nc-exp} \nm(x^{(1)},x^{(2)}) = \exp\Bigl( \sqrt 2
\sum_{i=1}^{p+q} x_i \, \XX_i \Bigr)\,.\ee
This can also be written as $\nm(x^{(1)}, x^{(2)}) = \exp(\XX)\exp(\YY)$
with $\XX=\sqrt 2\sum_{i=1}^p x_i \XX_i$ and
$\YY= \sqrt2 \sum_{i=p+1}^{p+q} x_j\XX_j$. (Note that $\exp(\XX)$ and
$\exp(\YY)$ commute.)
A consequence is
\be \nm(-x^{(1)},-x^{(2)}) = \nm(x^{(1)},x^{(2)})^{-1}\,.\ee

\rmrk{Horospherical coordinates} \il{nhorsco}{normalized horospherical
coordinates} We take coordinates $(x_1,\ldots,x_{p+q},y)$ on $S$ such
that
\be (x_1,\ldots,x_{p+q},y) \in \RR^{p+q}\times\RR_{>0}
\;\leftrightarrow\; \nm(x_1,\ldots,x_{p+q}) \am(y)\, \orgn \in S\,.\ee
We call these coordinates \emph{normalized horospherical coordinates}.

We call $\bigl( (x')^{(1)},(x')^{(2)}\bigr)$ \il{horsco}{horospherical
coordinates}\emph{horospherical coordinates} if $(x')^{(1)}$,
respectively $(x')^{(2)}$, depends on $x^{(1)}$, respectively on
$x^{(2)}$, by a bijective linear transformation. The properties
\eqref{anai} and \eqref{nprhc} stay valid for general horospherical
coordinates.
\smallskip

The Lie algebra $\glie$ can be viewed as the tangent space $\Ta e G$ at
the unit element $e$. Under the natural projection $G
\rightarrow G/K=S$ the subspace $\klie \subset \glie$ is mapped to
zero, and $\plie$ is mapped isomorphically onto the tangent space $\Ta
\orgn S$. We use that $\VV i = \sqrt 2\, \XX_i - \WW i$ to see that for
$f\in C^\infty(G/K)$, on which $R(\klie)$ acts trivially,
\[ R(\VV i ) f = \sqrt 2\, R(\XX_i) f - R(\WW i )f = \sqrt 2\, R(\XX_i)
f\,.\]
We have arranged the definitions so that the derivation $R(\VV i)_\orgn$
at $\orgn$ corresponds to $\partial_{x_i} $
at the point $(0, \cdots,0,1)$ in normalized horospherical coordinates.
This explains the use of $\sqrt 2$ in the introduction of the
coordinates on~$N$. The derivation $R(\H_0)_\orgn$ corresponds to
$\partial_y$ at the point $(0,\cdots,0,1)$.

\rmrk{Relation between coordinatizations}

\begin{prop}\label{prop-Id} {\rm Projection functions for Iwasawa
decomposition}
\begin{enumerate}[label=$\mathrm{(\roman*)}$, ref=$\mathrm{\roman*}$]
\item \label{Id:nak}
There are unique analytic maps
\bad \nJ: G &\rightarrow N\,,& \qquad \tJ: G&\rightarrow (0,\infty)\,,\\
\kJ:G&\rightarrow K\,,
\ead
such that $g=\nJ(g)\, \am\bigl( \tJ(g)\bigr) \kJ(g)$ for all $g\in G$.
There are also unique analytic maps for the opposite Iwasawa
decomposition
\bad \nI: G &\rightarrow N\,,& \qquad \tI: G&\rightarrow (0,\infty)\,,\\
\kI:G&\rightarrow K\,,
\ead
such that $g=\kI(g)\, \am\bigl( \tI(g)\bigr) \nI(g)$ for all $g\in G$.

\item \label{Id:expl}
The following explicit expression does not depend on the choice of
$\wm$:
\be \label{tJexpl}\tJ\bigl ( \wm \nm(x^{(1)},x^{(2)}) \am(y) \bigr)
= \frac y {\sqrt{ \bigl(y^2+c \|x^{(1)}\|^2 \bigr)^2 + 4c \|x^{(2)}\|^2
}}\,. \ee
By $\|\cdot\|$ we denote the norm $\|x\|^2=\sum_{i=1}^n x_i^2$ on
$\RR^n$ and $c$ is a positive constant.\il{c}{$c>0$ explicit constant
in Iwasawa decomposition}
\end{enumerate}
\end{prop}
\il{tIJ}{$\tJ(g),\; tI(g)$ determine $A$ comp. in Iw. decomp}
\il{kIJ}{$K$-component in Iw.dec.} \il{nIJ}{$N$-component in Iw.dec.}
\begin{proof}The existence of the functions in part~\eqref{Id:nak} is a
direct consequence of the Iwasawa decomposition. We note that
\be \label{ktni}
\kI(g) = \kJ(g^{-1})^{-1}\,,\ee and analogously for $\tI$ and~$\nI$.

Part~\eqref{Id:expl} depends on \cite[Theorem 1.14]{Hel70}
\[ \tI\bigl( \nm^-(x^{(1)},x^{(2)}) \bigr)= \sqrt{ (1+c\|x^{(1)}\|^2)^2
+ 4 c\|x^{(2)}\|^2 }\,.\]
Helgason describes the element $\nm(x^{(1)},x^{(2)})$ as
$\exp(\XX)\exp(\YY)$ with $\XX\in \glie_\al$, $\YY \in \glie_{2\al}$.
It turns out that $ |\XX|=2 \|x^{(1)}\|^2$ and $|\YY|= 2\|x^{(2)}\|^2$.
So our $c$ differs from the constant given by Helgason. Conjugation of
$\nm(x^{(1)},x^{(2)})$ by elements of $M$ and taking the inverse does
not change the norms $\|x^{(j)}\|$.

To get the relation between $\nm^-(x)$ and $\nm(x)$ we write these
elements as $\exp(\XX_-)$ and $\exp(\XX)$ with $\XX\in \nlie $ and
$\XX_- = \tht \XX\in \nlie^-$. We have chosen $\wm\in K$ such that
$\wm \am(t)\wm^{-1}= \am(1/t)$. Conjugation by $\wm$ gives an
isomorphism $N \cong N^-$ as well. The element
$\wm \nm(x^{(1)},x^{(2)}) \wm^{-1}$ has the form
$\nm^- ( \tilde x^{(1)}, \tilde x^{(2)})$ with orthogonal
transformations $x^{(j)}\mapsto \tilde x^{(j)}$. So for formulas
involving only these norms on $\RR^p$ and $\RR^q$, we can work with
$\wm \nm(x) \wm^{-1}$ as well as with $\nm^-(x)$.

In the following computation we use that $\tI(g)$ is invariant under
$g\mapsto kg$ with $k\in K$.
\begin{align*}
\tJ \bigl( \wm \nm(x^{(1)},x^{(2)}) \bigr) &= \tI\bigl(
\nm(x^{(1)},x^{(2)})^{-1} \wm^{-1}\bigr)^{-1} = \tI\bigl( \wm
\nm(x^{(1)},x^{(2)})^{-1} \wm^{-1}\bigr)^{-1}\\
&= \tI\bigl( \nm^-(x^{(1)}, x^{(2)}) ^{-1} \bigr)^{-1} = \Bigl(
(1+c\|x^{(1)}\|^2)^2+ 4 c \|x^{(2)}\|^2\Bigr)^{-1/2}\,;
\displaybreak[0]\\
\tJ\bigl( \wm \nm(x^{(1)},x^{(2)}) \am(y) \bigr) &= \tJ\bigl(
\am(y)^{-1}\wm \nm(y^{-1}x^{(1)}, y^{-2} x^{(2)}\bigr)
\displaybreak[0]\\
&= y^{-1}\, \Bigl( \bigl( 1+ y^{-2}\|x^{(1)}\|^2\bigr)^2+ 4 c y^{-4}
\|x^{(2)}\|^2\Bigr)^{-1/2}\displaybreak[0]\\&= y \,\Bigl( \bigl( y^2+c
\|x^{(1)}\|^2\bigr)^2+ 4 c \|x^{(2)}\|^2\Bigr)^{-1/2}\,.
\end{align*}
This is relation \eqref{tJexpl}.
\end{proof}

\rmrk{Relation in \intitle{\partial S}}We have
$N\,\nul \subset K/M=\partial S$ with
\be n\wm\,\inft = n \,\nul =\kI(n\wm) \,\inft\quad\text{ for }n\in
N\,.\ee
Hence the embedding $N\,\nul \rightarrow \partial S$ is given by
\be \label{N0->pS} \nm(x)\,\nul \mapsto \kI\bigl(\nm(x)
w\bigr)\,\inft\,.\ee

\subsection{Riemannian structure of \intitle{S}}\label{sect2-Rie}We
recall that the Lie algebra $\glie$ can be viewed as the tangent space
$\Ta e G$ at the unit element \il{e}{$e\in G$ \text{ unit element}}$e$.
The action of $K$ by conjugation preserves this tangent space, and is
orthogonal on $\plie \subset \glie$ for the scalar product induced by
$(\XX,\YY) = -B(\XX,\tht Y)$.

The subspace $\plie$ can be identified with the tangent space
$\Ta \orgn S$. It inherits the scalar product, and $K$ acts
  orthogonally on $\Ta\orgn S$. The derivations
  $\sqrt 2 \,R(\XX_i)_\orgn$, $1\leq i \leq p+q$, and $R(\H_0)_\orgn$
form an orthonormal basis of $\Ta \orgn S$. In normalized horospherical
coordinates these derivations correspond to $(\partial_{x_i})_\orgn$
and $(\partial_y)_\orgn$.

We transport the scalar product on $\Ta \orgn S$ to $\Ta x S$ for all
$x\in S$ by the translation $L(g)$ with $g\in G$ such that
$g\orgn = x$. Since the scalar product on $\Ta\orgn S$ is preserved by
the action of $K$, any choice of $g$ in $gK$ determines the same scalar
product. The image of an orthonormal basis is an orthonormal basis
which does depend on the choice of $g$.\smallskip

Let $g = \nm(x^{(1)},x^{(2)})\am(y)$ (normalized horospherical
coordinates). Since
\[g\, \am(e^\xi) = \nm(x^{(1)},x^{(2)})\am(ye^\xi)\,,\]
the tangent vector $(\partial_y)_\orgn$ is transported to
$y(\partial_y)_{g\,\orgn}$.

Let $\e_j=(0,\ldots,0,1,0,\ldots,0)$ with $1$ at position $j$ be the
$j$-th unit vector in $ \RR^{p+q}$. If $j\geq p+1$, then for
$\eta\in \RR$, with use of \eqref{anai},
\begin{align*}g \, \nm(\eta\e_j) &= \nm(x_1,\ldots,x_{p+q}) \, \am(y)
\,\nm(\eta\e_j)
\\
&= \nm(x_1,\ldots,x_{p+q}) \,\nm(y^2 \eta \e_j)
\,\am(y)\\
&= \nm(x_1,\cdots, x_{j-1}, x_j+y^2\eta,x_{j+1},\cdots,
x_{p+q})\,\am(y)\,,\end{align*}
and $\sqrt 2 (\XX_i)_\orgn\leftrightarrow \partial_{x_j}$ is transported
to $y^2 (\partial_{x_j})_{g\orgn}$.

If $1\leq j \leq p$, then we have to take into account the bilinear map
$b $ in \eqref{nprhc}:
\[ g \, \nm(\eta \e_j) = \nm\bigl( x_1 \cdots,x_j + y \eta,\cdots x_p,
x^{(2)} + y \xi \,b(x^{(1)}, \e_j)
\bigr)\am(y)\,.\]
So $\sqrt 2 (\XX_j)_\orgn \leftrightarrow \partial_{x_j}$ is transported
to
\[ y\, (\partial_{x_j})_{g\orgn}
+ \sum_{i=p+1}^q y\, b\bigl( x^{(1)}, \e_j)_i
\,(\partial_{x_i})_{g\orgn}
\,,\]
with linear forms \il{cji}{$c_{j,i}$ \text{ linear forms on
$\RR^p$}}$c_{j,i}: x^{(1)} \mapsto b\bigl( x^{(1)}, \e_j)_i$ on
$\RR^p$, $1\leq j \leq p$, $p+1\leq i \leq p+q$.

Thus, we have obtained an orthogonal basis of $\Ta {g\orgn}S$ at each
point $g\,\orgn$ of $S$:
\badl{mj} m_j &= y\, (\partial_{x_j})_{g\orgn}
+ \sum_{i=p+1}^{p+q} y\, c_{j,i}(x^{(1)})\, (\partial_{x_i})_{g\orgn}&
&1\leq j \leq p\,,\\
m_j &= y^2\, (\partial_{x_j})_{g\orgn}&&p+1\leq j \leq p+q\,,\\
m_\ds=m_{p+q+1} &= y\,(\partial_y)_{g\orgn}\,,
\eadl
with the linear function $c_{j,i}$ defined above. This basis depends on
the choice of $g\in gK$, but the corresponding scalar product is
independent of this choice, since $K$ acts orthogonally on
$\Ta \orgn S$.

Inverting the relations \eqref{mj} we arrive at
\badl{ir} (\partial_y)_{g\orgn} &= y^{-1} m_d\,,\\
(\partial_{x_j})_{g\orgn} &= y^{-2} m_j && \text{ for } p+1\leq j \leq
p+q\,,\\
(\partial_{x_j})_{g\orgn} &= y^{-1} m_j - \sum_{i=p+1}^{p+q} y^{-2}
c_{j,i}(x^{(1)})\, m_i&&\text{ for }1\leq j \leq p\,.
\eadl

So the scalar product $\rmetric(\cdot,\cdot)$ on $\Ta{g\orgn}S$
satisfies
\bad \rmetric\bigl( (\partial_y)_{g\orgn},(\partial_y)_{g\orgn}\bigr)
&=y^{-2}\,,\\
\rmetric\bigl( (\partial_y)_{g\orgn},(\partial_{x_j})_{g\orgn}\bigr)
&=0& & \text{ for } 1\leq j \leq p+q\,,\\
\rmetric\bigl( (\partial_{x_i})_{g\orgn},(\partial_{x_j})_{g\orgn}\bigr)
&= y^{-4} \, \dt_{i,j} &&\text{ for }p+1\leq i,j \leq p+q\,.
\ead
For $1\leq j \leq p$ and $p+1\leq l \leq p+q$
\be \rmetric\bigl(
(\partial_{x_j})_{g\orgn},(\partial_{x_l})_{g\orgn}\bigr) =
- y^{-4} c_{j,l}(x^{(1)})\,. \ee
If both $j$ and $l$ are in $[1,p]$, then
\bad \rmetric\bigl(
(\partial_{x_j})_{g\orgn},(\partial_{x_l})_{g\orgn}\bigr) = y^{-2}\,
\dt_{j,l} + y^{-4} \sum_{i=p+1}^{p+q}c_{j,i}(x^{(1)})
c_{l,i}(x^{(1)})\,.
\ead

These formulas hold at all points $g\,\orgn$ in $S$. We write the matrix
describing the scalar product in $(p,q,1)$ block
form:\il{Riemat}{$\rmetric$}
\be\label{Rmat} \rmetric =
\begin{pmatrix} y^{-2} I_p + y^{-4} C(x^{(1)})\, C(x^{(1)})^t & - y^{-4}
C(x^{(1)}) & 0\\
-y^{-4} C(x^{(1)})^t & y^{-4}I_q & 0\\
0&0&y^{-2}
\end{pmatrix}\,,\ee
with $I_n$ the unit matrix of size $n\times n$, and
\il{Cx1}{$C(x^{(1)})$}$C(x^{(1)})$ the $p\times q$ matrix with
$c_{j,l}(x^{(1)})$ at the position $(j,l-p)$.

If we use another coordinate system on an open set in $S$, then we get
other matrices describing the Riemannian structure. We will still
denote those matrices by $\rmetric$, although they need not have
exactly the structure in \eqref{Rmat}. When dealing with discrete
  subgroups of $G$ it will be useful to take coordinates
$x'=(x_1',\ldots x_{p+q}')$ on $N$ that depend linearly on
$\bigl( x^{(1)}, x^{(2)}\bigr)$, and to keep the coordinate $y$. Then
the Riemannian structure is given by a $(p+q) \times 1$ block matrix of
the form
\be \label{Rmat'} \begin{pmatrix} Q(x',y) & 0\\0&y^{-2} \end{pmatrix}\ee
with a positive definite matrix
\il{Qxay}{$Q(x',y)$}$Q(x',y) = Q_1(x')y^{-2}+ Q_2(x')y^{-4}$ where
$Q_1$ and $Q_2$ depend on $x'$ by quadratic polynomials.

\rmrk{Distance function} The Riemannian structure describes the length
of a path $t\mapsto p(t)$, $t_1\leq t\leq t_2$, as
\be\label{length} \int_{t=t_1}^{t_2} \sqrt{ \rmetric(p'(t))}\, dt\,,\ee
where $\rmetric(p'(t)) $ is a short-hand notation for
$(\rmetric)_{p(t)}(p'(t),p'(t))$, i.e., the inner product on
$\Ta {p(t)}S$ that is given by the Riemannian metric at $p(t)$,
evaluated on the tangent vector $p'(t)$. The distance
\il{dist}{$\dist(\cdot,\cdot)$}$\dist\bigl(x_1,x_2\bigr)$
is the minimum of these lengths for all paths with $p(t_1)=x_1$ and
$p(t_2) = x_2$. One can check that for $y_1, y_2>0$
\be \label{distln}\dist\bigl( \am(y_1)\,\orgn, \am(y_2)\, \orgn\bigr) =
\bigl|\log(y_2/y_1)\bigr|\,.\ee
This relation stays valid for the description of the Riemannian
structure with the matrix in~\eqref{Rmat'}.

\rmrk{Invariant \intitle{\ds}-form} The matrix elements $g_{i,j}$ are
functions on $S$, and the $\ds$-form\il{ups}{$\z_S$}
\be\label{ups} \z_S = (\det \rmetric)^{1/2} \, dx_1\wedge dx_2\wedge
\cdots
\wedge dx_d\ee
satisfies $\z_S \circ h = \z_S$ for all $h\in G$. Integration of
functions on $S$ against $\z_S$ gives a $G$-invariant measure on $S$.
In normalized horospherical coordinates
$\det (\rmetric)^{ 1/2 } = y^{-2\rho-1}$, with $\rmetric$ as in
\eqref{Rmat}.

\subsection{Laplace operator}\label{sect2-Lapl}In general, the
Laplace--Beltrami operator is described as in \cite[Chap X, \S2.1,
equation (4)]{Hel62}:\il{Laop}{Laplace operator}
\be\label{Dtp} \Dt_+ = (\det \rmetric)^{-1/2}\, \sum_k
\partial_{x_k} \sum_i (\rmetric)^{-1}_{i,k} \sqrt{\det \rmetric}\,
\partial_{x_i}\,,\ee
with respect to any choice of coordinates $x_1,\ldots x_\ds$. The matrix
$\rmetric$ describes the inner products of the basis vectors
$\partial_{x_1},\ldots,\partial_{x_\ds}$. The expression in
\eqref{Rmat} is for normalized horospherical coordinates.

By {\sl loc.~cit.}~equations (2) and
(3),\il{div}{$\div$}\il{grad}{$\grad$} we can obtain $\Dt_+$ as a
product of first order differential operators:
\badl{divgrad} \Dt_+ f &= \div\,\grad \,f\,,\\
\grad &f = \sum_{i,j} (\rmetric)^{-1}_{i,j} \frac{\partial f}{\partial
x_{ j}}\,
\partial_{x_i}\,,\\
\div \sum_i a_i \,\partial_{x_i} &= \bigl( \det \rmetric)^{-1/2} \sum_i
\frac\partial{\partial x_i} \bigl( (\det \rmetric)^{1/2}\, a_i\bigr)\,.
\eadl
We prefer to use \il{Dt}{$\Dt, \;\Dt_+$}
\be\label{Dt} \Dt = -\Dt_+\,.\ee
Since the Riemannian metric is invariant under the action of $G$ on $S$,
the Laplace operator commutes with the action of $G$ on~$S$:
\be \Dt L(g) f = L(g)\Dt f\,.\ee

\subsubsection{Laplace operator in horospherical
coordinates}\label{sect2-Lapl-horco}To apply the definition in
\eqref{Dtp} we need the inverse of the matrix $\rmetric$. In normalized
horospherical coordinates we can check that
\be \label{gRi}(\rmetric)^{-1} = \begin{pmatrix}
y^2 I_p & y^2 C(x^{(1)}) & 0\\
y^2 C(x^{(1)})^t& y^4 I_q + y^2 C(x^{(1)})^t C(x^{(1)}) & 0\\
0&0& y^2
\end{pmatrix}\,,\ee
with the matrix $C(x^{(1)})$ in~\eqref{Rmat}. Left multiplication by
\[ \begin{pmatrix} I_p & C(x^{(1)})& 0 \\
0&I_q&0\\0&0&1\end{pmatrix}\]
brings $\rmetric$ into a form that shows that
\be \label{detrm} \det \rmetric= y^{-4\rho-2}\,.\ee

The form of this matrix implies that the differentiation with respect to
the last variable $y$ does not interact with the other coordinates. We
conclude that
\be \Dt = - y^2\,\partial_y ^2+(2\rho-1) y \,\partial y - y^2 L_1 - y^4
L_2\,,\ee
with second order differential operators in the variables
$x_1, \ldots, x_{p+q}$. Actually, we have
$L_2 = - \sum_{i=p+1}^{p+q} \partial_{x_i}^2$. The operator $L_1$ is
more complicated.

\subsubsection{Ellipticity of the Laplace operator}\label{sect3-ell}
The matrix elements of $\rmetric$ are analytic functions on $S$ provided
we work with coordinates that are analytic. Furthermore, the matrix
$\rmetric$ and its inverse are positive definite. The principal symbol
of $\Dt_+$ is
\[(\det\rmetric)^{-1/2}\, \sum_{i,k}
  (\rmetric)^{-1}_{i,k} \sqrt{\det \rmetric}\, X_i X_k\,.\]
  It is obtained by taking the highest order terms and replacing partial
derivatives by variables. The resulting polynomial is the principal
symbol of the differential operator. In this way we see that the
principal symbol of $\Dt_+$ is positive definite, which implies that
$\Dt_+$, and hence $\Dt$ as well, are elliptic differential operators.

All eigenfunctions of $\Dt$ are automatically smooth by elliptic
regularity, as discussed for instance by Folland \cite[Theorem
(6.33)]{Fol95} or Lang \cite[Appendix 4]{Lng-sl2r}. By elliptic
regularity, any function or distribution that is an eigenvector of
$\Dt$ is a $C^\infty$-function. Analytic elliptic regularity implies
that if the coefficients of the elliptic operator are analytic, then
the eigenvectors are analytic functions. Lang refers to Bers and
Schechter \cite{BeSc64} for analytic elliptic regularity.

\subsection{The Casimir element}\label{sect2-Cas}
The Lie algebra $\glie$ generates the universal enveloping algebra
$U(\glie)$. We can view the elements of \il{Ug}{$U(\glie)$}$U(\glie)$
as non-commutative polynomials in elements of $\glie$. The
representations $R$ and $L$ of $\glie$ in $C^\infty(G)$ by right and
left differentiation extend as representations of $U(\glie)$ where the
differential operators $R(u)$, $u\in U(\glie)$, commute with $L(g)$,
$g\in G$, and the differential operators $L(u)$ commute with $R(g)$. If
we apply $R(u)$ for $u \in U(\glie)$ to elements of $C^\infty(G/K)$,
i.e., of $C^\infty(S)$, we end up with functions that are in general
not right $K$-invariant.

The \il{Ce}{Casimir element}Casimir element
\il{Cas}{$\Cas$}$\Cas \in U(\glie)$ is in the center of the universal
enveloping algebra $U(\glie)$. It can be described starting from any
basis $\bigl\{ \XX_j\bigr\}$ of $\glie$ and its dual basis
$\bigl\{\XX_j'\bigr\}$ with respect to the normalized Killing form as
\be \Cas=\sum_j \XX_j \XX_j' \,.\ee

In the representation $R$ of $\glie$ in $C^\infty(G)$ by right
differentiation, we call $R(\Cas)$ the {\em Casimir operator}. It is
the unique second order differential operator that commutes with all
operators $R(\XX)$ for $\XX\in \klie$, and has vanishing constant term.
Hence $R(\Cas)$ preserves the space of right $K$-invariant functions,
which we can identify with $C^\infty(S)$. Since right differentiation
commutes with left translation, we have a differential operator
$R(\Cas)$ in $C^\infty(S)$ commuting with $L(g)$.
\begin{prop}\label{prop-Cas-Lapl} The differential operators $R(\Cas)$
and $\Dt_+$ in $C^\infty(S)$ coincide.\end{prop}
This is Exercise C5 in \cite{Hel62}.

With the bases of $\glie$ in \eqref{basX} and \eqref{basXd} we find the
Casimir element
\be \Cas = \H_0^2 - \sum_j \UU_j^2 - \sum_i \XX_i(\tht\XX_i) -
\sum_i(\tht\XX_i)\XX_i\,,\ee
with $1\leq j \leq \dim \mlie$, and $1\leq i \leq p+q$. For application
on right $K$-invariant functions we can use that $R(\UU_j)$ acts as
zero. Using that $\XX_i + \tht \XX_i \in \klie$ acts as zero as well,
we arrive at the action
\be \label{RCas-Kinv}
R(\H_0)^2 + \sum_{i=1}^{p+q} 2 R(\XX_i)^2 - 2 \rho\, R(\H_0)\,.\ee

%% file: ccro2-3-autf.tex

\def\flnm{ccro2-3-autf}\bigskip

\section{Automorphic forms}\label{sect3}

After describing in \S\ref{sect3-discr-subgr} a class of discrete
subgroups of $G$, we define automorphic forms in \S\ref{sect3-autf}.
The last subsection gives some proofs.

\subsection{Discrete subgroup}\label{sect3-discr-subgr} We fix a
discrete subgroup \il{GM}{$\Gm$}$\Gm\subset G$ satisfying the following
conditions:
\begin{enumerate}[label=$\mathrm{(\roman*)}$, ref=$\mathrm{\roman*}$]
\item $\Gm$ is cofinite, but not cocompact.

Cofinite means that the quotient $\Gm\backslash S$ has finite measure
for the measure induced by $d\mu$. Since the group is not cocompact, it
has cusps, discussed below.

\item $\Gm$ is torsion-free. This means that the sole $\gm\in \Gm$ with
finite order is the unit element $e$ of~$\Gm$.

By requiring $\Gm$ to be torsion-free, we avoid some technical
difficulties.
\end{enumerate}

\rmrk{Cusps} Each point $g\,\inft \in \partial S$ is fixed by the group
$g N g^{-1}$. The element $g\in G$ is determined by the point
$g\,\inft$ up to right multiplication by $NAM$. A \il{cusp}{cusp}cusp
\il{cu}{$\cu\in \Cu$}$\cu =  g_\cu \,\inft$ of $\Gm$ is a point of
$\partial S$ for which $N_\cu:=g_\cu N g_\cu^{-1} $ intersects $\Gm$ in
a lattice.
(A discrete subgroup $\Ld\subset N$ is called a
(uniform)
\il{latt}{lattice}{\em lattice} if the quotient $\Ld\backslash N$ is
compact.)

All maximal unipotent subgroups in $G$ are conjugate to each other. So
there exist elements \il{gcu}{$g_\cu$}$g_\cu \in G$ such that
\il{Ncu}{$N_\cu$}$N_\cu= g_\cu N g_\cu^{-1}$. The requirement that the
intersection $\Gm\cap N_\cu$ is a lattice in $N_\cu$ is that the
quotient $\left(\Gm\cap N_\cu\right) \backslash N_\cu$ is compact. The
set \il{Cu}{$\Cu$}$\Cu$ of all cusps of~$\Gm$ is invariant under the
action of $\Gm $ on $\partial S$. It consists of finitely many
$\Gm$-orbits.

By \il{Gmcu}{$\Gmm\cu$}$\Gmm \cu$ we denote the subgroup $\Gmm\cu$ of
$\Gm$ stabilizing $\cu\in \Cu$.
\begin{lem}\label{lem-Gmcu}The condition that $\Gm$ is torsion-free
implies that $\Gmm \cu = \Gm \cap N_\cu$.
\end{lem}
\begin{lem}\label{lem-normgc}
We can (and do) normalize the choice of the elements $g_\cu$ such that
the lattice $\Ld_\cu= g_\cu^{-1} \Gmm \cu g_\cu$ in $N$ depends only on
the $\Gm$-orbit of $\cu$ in~$\Cu$.\end{lem}
Proofs in \S\ref{sect3-prfs}.

\subsection{Automorphic forms and cusp forms}\label{sect3-autf}
We now define spaces of various types of $\Gm$-invariant functions on
$S$: invariant eigenfunctions, automorphic forms, and cusp forms.

\begin{defn}\label{def-invei}We define an \il{inveif}{invariant
eigenfunction}{\em invariant eigenfunction} as a function
$f\in C^\infty(S)$ that satisfies
\begin{enumerate}[label=$\mathrm{(\alph*)}$, ref=$\mathrm{\alph*}$]
\item $f(\gm x)
= f(x)$ for each $x\in S$ and each $\gm\in \Gm$.

\item $\Dt f = \ld f$ for some $\ld \in \CC$\,.
\end{enumerate}
\end{defn}

Since the Laplace operator is elliptic with analytic coefficients, all
its eigenfunctions are analytic. See \S\ref{sect3-ell}.

We will parametrize the eigenvalue as $\ld = \rho^2-\nu^2$, for a reason
that we will discuss after Proposition~\ref{prop-eivCasps}. We call
$\nu$ the \il{spparm}{spectral parameter}{\em spectral parameter}, and
denote by \il{EnuGm}{$\E_\nu^\Gm(S)$ \text{ space of invariant
eigenfunctions with eigenvalue
$\rho^2-\nu^2$}}$\E_\nu(S)^\Gm =  \E_{-\nu}(S)^\Gm$ the space of
invariant eigenfunctions with spectral parameter $\nu$.

In general, the spaces $\E_\nu(S)^\Gm$ have infinite dimension. We
define interesting subspaces by imposing growth conditions at the
cusps.

\begin{defn}\label{def-grcd}
We define \il{pgc1}{polynomial growth at a cusp}\emph{polynomial growth}
and \il{qdc1}{quick decay at a cusp}\emph{quick decay} at a cusp $\cu$
by
\badl{grcp} f\bigl( g_\cu \am(t) x \bigr) &= \oh (t^b \bigr)
\qquad\text{as }t\uparrow \infty \text{ uniformly for $x$ in compact
sets} \\
&\begin{cases} \text{ for some }b\in \RR & \text{\emph{polynomial
growth},}\\
\text{ for each }b\in \RR &\text{\emph{quick decay}.}
\end{cases}
\eadl
\end{defn}
For polynomial growth large values of $b$ are relevant, and for quick
decay, negative values of $b$ with $|b|$ large. See
\S\ref{sect3-gr-cond} for a further discussion of growth conditions.

\begin{defn}\label{def-af}
An invariant eigenfunction $f\in \E_\nu(S)^\Gm$ is an
\il{autf}{automorphic form}{\em automorphic form} if it has polynomial
growth at all cusps. It suffices to impose this growth condition for a
representative of each $\Gm$-orbit of cusps.

By \il{Ausp}{$\A(\Gm;\nu)$}$\A(\Gm;\nu) $ we denote the space of all
automorphic forms for $\Gm$ with spectral parameter $\nu$.
\end{defn}

Identifying an automorphic form $f$ on~$S$ with the function
$g\mapsto f(g\, \orgn)$ on $G$ we have an automorphic form of trivial
$K$-type, as defined by Harish-Chandra. The spaces $\A(\Gm;\nu)$ have
finite dimension. See \cite[p~7, and Theorem 1, p 8]{HCh68}.
\smallskip

If $f\in \A(\Gm;\nu)$, then the integral
\be \label{fcu}f_\cu (x) = \int_{n\in (\Gm\cap N_\cu)\backslash N_\cu}
f(n\,x)\, dn\ee
is well-defined for each cusp $\cu$, since
$(\Gm\cap N_\cu)\backslash N_\cu$ is compact.
\begin{defn}\label{def-cf}The space of \il{cf}{cusp form}\emph{cusp
forms} \il{A0}{$\A^0(\Gm;\nu)= A^0_\nu(\Gm)$}$\A^0(\Gm;\nu)$ is the
linear subspace of $f \in \A(\Gm;\nu)$ for which $f_\cu=0$ for all
cusps $\cu$.
\end{defn}

It suffices to require that $f_\cu=0$ for $\cu$ in a set of
representatives of the $\Gm$-orbits of cusps.

\begin{prop}[Harish-Chandra] Cusp forms have quick decay at each cusp.
\end{prop}
\begin{proof}Let $\cu\in \Cu$. Harish-Chandra \cite[Chap I, \S~3,
4]{HCh68} shows that if a $\Gmm \cu$-invariant eigenfunction of $\Dt$
with polynomial growth at the cusps has vanishing Fourier term of order
zero (i.e., $f_\cu=0$), then it has quick decay.
\end{proof}

We can characterize cusp forms by the condition of quick decay at the
cusps. If the unipotent subgroup $N$ of $ G$ is abelian, then cusp
forms have exponential decay. In more general situations we have to be
content with quick decay.

\begin{prop}\label{prop-cusp-sp}The set of $\nu\in \CC$ such that
$\A^0(\Gm;\nu)\neq \{0\}$ is a discrete set in
$\bigl\{\nu \in \CC \;:\; \rho^2-\nu^2 \geq 0\bigr\}$.
\end{prop}
The discreteness of the spectrum follows from the compactness of
appropriate integral operators. See, for instance, \cite[Theorems 2
(p~9), 3 (p~15), and the remarks on p~16]{HCh68}. For the inequality
$\rho^2-\nu^2\geq 0$ it is essential that we work with cusp forms on
$S$, corresponding to right $K$-invariant functions on~$G$. See the
completion of the proof in~\S\ref{sect3-prfs}.

The proposition implies that $\A^0(\Gm;\nu)$ might be non-zero for
infinitely many values $\nu \in i\RR$, and for at most finitely many
real values in $[-\rho,\rho]$. For cusp forms of non-trivial $K$-type,
the eigenvalues $\rho^2-\nu^2$ of the Casimir operator on cusp forms
can be negative.

\subsection{Growth conditions on the group}\label{sect3-gr-cond}
The growth conditions at a cusp can be formulated for all functions on
the group, and can be applied to functions on $S$ by considering them
as functions on the group that are right $K$-invariant. On the group,
polynomial growth and quick decay at the cusp $\cu$ can be formulated
as
\badl{grcpg} f\bigl( g_\cu \am(t) g\bigr) &= \oh (t^b \bigr)
\qquad\text{as }t\uparrow \infty \text{ uniformly for $g$ in compact
sets in $G$} \\
&\begin{cases} \text{ for some }b\in \RR_{> 0} & \text{\emph{polynomial
growth},}\\
\text{ for each }b\in \RR_{<0} &\text{\emph{quick decay}.}
\end{cases}
\eadl
By replacing the single element $g_\cu$ in \eqref{grcpg} by an element
$k$ running over~$K$ we obtain stronger conditions: polynomial growth
on $G$, and quick decay on $G$, defined by
\badl{grbdg} f( k \am(t) g) &= \oh(t^b) \qquad\text{as }t\uparrow
\infty\\
&\qquad\text{for each $k\in K$, uniformly for $g$ in compact sets in
$G$}\\
& \qquad\begin{cases}\text{for some }b\in \RR&\text{\emph{polynomial
growth}}\,,\\
\text{for each }b\in \RR&\text{\emph{quick decay}}\,.\end{cases}
\eadl
Applied to functions on $S$, we call these conditions
\il{pg1}{polynomial growth at the boundary}\emph{polynomial growth}
(respectively \il{qd1}{quick decay at the boundary}\emph{quick decay})
\emph{ at the boundary}.

\begin{prop}\label{prop-grL}The conditions of polynomial growth and
quick decay on $G$ are preserved under left translation by elements
of~$G$.
\end{prop}
\begin{proof}Left translation by $ h^{-1}$ turns the quantity to be considered
into
\[ \bigl( L(h^{-1}) f\bigr)(k \am(t) g) = f(hk\am(t) g) \,.\]
By the Iwasawa decomposition we can write
$ hk = k_1 \am(t_1) n_1$ with $k_1\in K$, $t_1>0$, and
$n_1=\nm(x^{(1)},x^{(2)})$ (in horospherical coordinates). The quantity
to estimate is
\[ f\bigl( k_1 \am(t_1) n_1 \am(t) g \bigr) = f\bigl( k_1 \am(t_1 t)
\nm(t^{-1}x^{(1)}, t^{-2}x^{(2)}) g\bigr)\,.\]
As $t\uparrow\infty$ the element $ \nm(t^{-1}x^{(1)}, t^{-2}x^{(2)}) g$
stays in a compact subset of $G$. The influence of $t_1$ is hidden in
the implicit constant in $\oh(t)$.\end{proof}

\begin{prop}\label{prop-grwt}Let the $K$-finite $f\in C^\infty(G)_K$ be
an eigenfunction of the differential operator $R(\Cas)$ determined by
the Casimir element in \S\ref{sect2-Cas}. If $f$ satisfies one of the
growth conditions mentioned above, then so does $R(u)f$ for all
$u\in U(\glie)$.\end{prop}
\begin{proof}[Indication of the proof] The proposition is a direct
consequence of the convolution representation of Harish-Chandra in
Theorem 1 on p~18 of \cite{HCh66}. It gives
\[ f(g) = \int_G f(g h^{-1})\,\al(h)\, dh = \int_G f(h) \,\al(h^{-1}
g)\, dh\,,\]
for a suitable compactly supported function $\al\in C_c^\infty(G)$. In
the second form of the convolution, the differentiation concerns only
the factor $\al(h^{-1}g)$. This allows us to show that the growth
property of $f$ is valid for $R(\XX)f$ for any $\XX\in \glie$. These
derivatives are eigenfunctions of $R(\Cas)$ with the same eigenvalue.
So we can repeat the argument to get preservation of the growth
property under all elements of $U(\glie)$.
\end{proof}

\begin{cor}\label{cor-grwt-hc}
Suppose that $f\in C^\infty(G)_K$ is an eigenfunction of $R(\Cas)$ and
has quick decay at the cusp $\cu$. Then all derivatives of $f$ of any
order in the horospherical coordinates $x_1,\ldots, x_{p+q},y$ have
quick decay at~$\cu$.
\end{cor}
\begin{proof}The derivatives
$\bigl(R(\XX_i)f\bigr)\bigl(g_\cu\nm(x^{(1)},x^{(2)})\am(y)\bigr)$ are
given by a linear combination of the $\partial_{x_i}$ with polynomial
coefficients in the $x_j$ and $y$.
\end{proof}

Definition~\ref{def-cf} characterizes the space of cusp forms in
$\A(\Gm;\nu)$ by the vanishing of the integrals $f_\cu$ in \eqref{fcu}
for all $\cu\in \Cu$. The following lemma gives another
characterization of cusps forms. We recall that non-zero cusp forms
occur only for $\rho^2-\nu^2\geq 0$, which implies that we have
$|\re\nu| \leq \rho$.

\begin{lem}\label{lem-cusp-crit}We can characterize cusp forms within
the space of bounded invariant eigenfunctions with eigenvalue
$\rho^2-\nu^2$ by the condition
\be \label{bco} f(g_\cu \am(y) \, x) = \oh\bigl( y^{b}\bigr) \qquad
\text{as }y\uparrow \infty \text{ for all cusps }\cu\ee
for some $b< \rho-|\re\nu|$.\end{lem}
\begin{proof}Since $f$ is bounded, it is an automorphic form. The
functions $f_\cu$ inherit the boundedness from the function $f$.
Moreover, $f_\cu (g_\cu n \am(y)\,\orgn)$ does not depend on $n\in N$.
(See the integral in \eqref{fcu} and use that $f$ is left-invariant
under $\Gm\cap g_\cu N g_\cu^{-1}$.)

For any $f\in \A(\Gm;\nu)$ the functions $f_\cu$ are linear combinations
of $g_\cu n \am(t) \, \orgn\mapsto t^{\rho\pm \nu}$, or of $t^{\rho}$
and $t^\rho \log t$ if $\nu=0$. Condition \eqref{bco} implies that all
$f_\cu$ have to vanish, and $f$ is a cusp form.
\end{proof}

\subsection{Proofs}\label{sect3-prfs}
\begin{proof}[Proof of Lemma \ref{lem-Gmcu}.]We have to show that if
$\Gm$ is torsion-free, then
$\Gmm \cu = \bigl\{ \gm\in \Gm\;:\; \gm \,\cu=\cu\}$ is equal to
$\Gm \cap g_\cu N g_{\cu}^{-1}$.

Since $NMA$ is the normalizer of $N$ in $G$, any element of $\Gmm \cu$
is of the form $g_\cu nma g_\cu^{-1}$, with $n\in N$, $a\in A$ and
$m\in M$. Since
\[ \am(t)^{-1} \nm(x^{(1)}, x^{(2)}) \am(t)= \nm(t^{-1} x^{(1)}, t^{-2}
x^{(2)}) \,,\]
the discreteness of $N_\cu$ implies that $t=1$, and $\gm = n m$. Since
$M$ is compact, the group $\Gm\cap g_\cu N g_\cu^{-1}$ has finite index
in $\Gm\cap g_\cu NM g_\cu^{-1}$. So there exists $k\geq 1$ such that
$g_\cu (nm)^k g_\cu^{-1} \in  g_\cu N g_\cu^{-1}$, and then $m^k\in N$
since $M$ normalizes $N$. However, $M\cap N = \{e\}$, and $k=1$ since
$\Gm$ is torsion-free. So $\gm \in \Gm\cap g_\cu N g_\cu^{-1} $.
\end{proof}

\begin{proof}[Proof of Lemma~\ref{lem-normgc}] For each $\Gm$-orbit in
$\Cu$ we choose a representative $\du$ and choose some $g_\du\in G$
such that $N_\du= g_\du N g_\du^{-1}$. If $\cu = \gm \du$ with
$\gm\in \Gm$, then
$N_\cu = \gm g_\du N  g_\du^{-1}  \gm^{-1}= \gm N_\du \gm^{-1}$, and
also $\Gmm\cu = \gm \Gmm\du \gm^{-1}$. So we can take
$g_\cu = \gm g_\du$. Then
\be \label{Ldrel}
\Ld_\cu = (\gm g_\du)^{-1} \Gmm \cu \gm g_\du = g_\du^{-1}\gm^{-1} \gm
\Gmm \du \gm^{-1} \gm g_\du=g_\du^{-1}\Gmm\du g_\du = \Ld_\du\,. \ee

The choice of $\gm$ has the freedom $\gm\mapsto \gm \dt$ with
$\dt \in \Gm_\du$, and has no effect on the relation~\eqref{Ldrel}.
\end{proof}

\begin{proof}[Completion of the proof of Proposition
\ref{prop-cusp-sp}.] After the statement of the proposition, we already
mentioned that \cite[Theorem 2, p~9]{HCh68} implies the discreteness of
the set of $\nu\in \CC$ such that $\A^0(\Gm;\nu)\neq \{0\}$.

For functions $h_1,h_2\in C_c^\infty(\Gm\backslash S)$ we have
\be\label{Dtgrd} \int_{\Gm\backslash S} (\Dt h_1)\,\bar h_2\, \z_S =
\int_{\Gm\backslash S} \sum_{i,j} \bigl( \rmetric^{-1}\bigr)_{ j,i }\,
(\partial_{x_j} h_1)\, (\overline{\partial_{x_i} h_2})\, \z_S\,.\ee
To obtain this, we work with a fundamental domain $\fd$ of $\Gm$ in~$S$,
for instance obtained by the Dirichlet method that we will also use in
\S\ref{sect4-stdtess}. Obtaining \eqref{Dtgrd} by partial integration
over $\fd$, we use that the boundary term is zero, since the
contribution of $\Gm$-equivalent boundary components cancel each other.

We apply this with $h_1=h_2=\ch f$ for a non-zero cusp form
$f\in \A^0(\Gm;\nu)$, with a cut-off function $\ch$ that equals $1$ on
most of $\fd$, and behaves near each cusp $\cu$ in the closure of $\fd$
in $S \cup (\partial S)$ like
$\ch\bigl( g_\cu n \am(t) \,\orgn) = \ph_\cu(t)$, where $\ph_\cu(t)$
goes down from $1$ to $0$ in some interval $[T,T+1]$. The quick decay
of cusp forms implies that $\ch f$ approaches $f$ and $\Dt(\ch f) $
approaches $\Dt f$ in $L^2(\Gm\backslash S)$ as $T\uparrow\infty$. So
\be\label{fappr}\int_\fd \bigl(\Dt (\ch f) \bigr) \,\overline{ \ch f}\,
\z_S\ee
approaches $(\rho^2-\nu^2)\,\int_S |f|^2\, \z_s$. On the other hand, the
quantity in~\eqref{fappr} is equal to
\[ \int_{\fd} \sum_{ j,i} \bigl( \rmetric^{-1}\bigr)_{i,j}
(\partial_{x_j}(\ch f)) \,(\overline{\partial_{x_i}(\ch f) })\z_S\,,\]
and is non-negative since $\rmetric^{-1}$ is positive definite.
\end{proof}
With more work we can show that $\rho^2-\nu^2=0$ can happen only for
constant functions, which are not cusp forms.

%% file: ccro2-4-cohom.tex

\bigskip

\def\flnm{ccro2-4-cohom}

\section{Cohomology}\label{sect4}

In this paper we relate spaces of cusp forms to parabolic cohomology
spaces $H^{\ds-1}_\pb(\Gm;V)$ for $\Gm$-modules $V$ that are
$\CC$-linear spaces. That implies that $H^j_\pb(\Gm;V)$ is a vector
space over $\CC$ for all $j$. We prefer to speak of
\il{cohsp}{cohomology space}`cohomology spaces' instead of the more
usual name `cohomology groups'.

We use methods that worked well in the case of $\ds=\dim S=2$, and we
follow as far as possible the approach in \cite[Sections 6 and
11]{BLZ15} and \cite[Chap.~9]{tahA}.

There are many ways to describe parabolic cohomology spaces and the more
usual cohomology spaces $H^j(\Gm;V)$. The description that we use here
has a geometrical flavor. It is based on tessellations of the symmetric
space $S$.

The first three subsections give a discussion of the concepts and
results that we will need in this paper, skipping some proofs. The
subsequent subsections give more details and provide some of the proofs
skipped earlier.

\subsection{Horoballs and horospheres}\label{sect4-gorbs}
We start with some geometrical definitions.

\begin{defn}\label{def-hor}Let $Y>0$. The
\il{hors}{horosphere}\emph{horosphere} $H_\inft(Y)$ at $\inft$ with
\emph{height}~$Y$, and the corresponding
\il{horb}{horoball}\emph{horoball} $B_\inft(Y)$ are the following
subsets of the symmetric space~$S$:
\bad H_\inft(Y) &= \bigl\{ n \am(Y) \, \orgn\;:\; n \in N\bigr\}\,,\\
B_\inft(Y) &= \bigl\{ n \am(y) \, \orgn\;:\; n \in N,\; y\geq
Y\bigr\}\,.
\ead
The corresponding subsets at the cusp $\cu$ of $\Gm$
are\il{HcuY}{$H_\cu(Y)$} \il{BcuY}{$B_\cu(Y)$}
\be H_\cu(Y) = g_\cu\, H_\inft(Y) \,,
\qquad B_\cu(Y) = g_\cu \, B_\inft(Y)\,. \ee
\end{defn}
In this definition it is important to use a normalization in
Lemma~\ref{lem-normgc} of the elements $g_\cu\in G$.

\begin{defn}The \il{extsp}{extended symmetric space}\emph{extended
symmetric space} of $\Gm$ is\il{Sast}{$\Sast$}
\be \Sast = \Sast_{\!\Gm} \;:=\; S \cup \Cu \subset S \cup \partial
S\,.\ee
For each cusp $\cu$ the \il{ehb}{horoball, extended}extended horoballs
\il{BacuY}{$B^\ast_\cu(Y)$}$B^\ast_\cu(Y)
:= B_\cu(Y)\cup\{\cu\}$ with $Y >0$ form a neighborhood basis of $\cu$
in $\Sast$.
\end{defn}
\begin{figure}
\begin{center}\includegraphics[width=9cm]{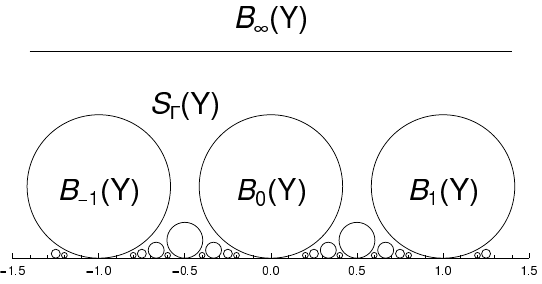}
\end{center}
\caption{Some horoballs for $\SL_2(\ZZ)\backslash \uhp$, with $Y=1.2$,
and $\uhp $ the complex upper half-plane.\\
Each horoball intersects $\partial S$ only in one cusp. Each pair of
points of $\partial S \setminus \Cu$ (i.e., a pair of irrational
numbers)
can be joined by a continuous path in $\SY$. }
\end{figure}

In this context, we say that the height \il{asl}{suitably large}$Y$ is
\emph{suitably large} if all horoballs $B_\cu(Y)$ with $\cu\in \Cu$ are
pairwise disjoint. The constructions in this section are based on the
choice of a suitably large number \il{Y}{$Y$}$Y$. We assume that
$Y\geq 4$, and we keep it fixed from here on.

\begin{defn}\label{def-SY}
For the fixed suitably large number $Y$ we define
\il{SY}{$\SY$}$\SY\subset S$ as the closure in $S$ of
$ S \setminus \bigcup_{\cu\in \Cu} B_\cu(Y)$.
\end{defn}
The set $\SY$ intersects each horoball $B_\cu(Y)$ in the horosphere
$H_\cu(Y)$. We obtain the decompositions
\badl{Sdcp} S = \SY \cup \bigcup_{\cu\in \Cu} B_\cu(Y)\,,
\\
\Sast = \SY \cup \bigcup_{\cu\in \Cu} B_\cu^\ast(Y)
\,.
\eadl

The quotient $\Gm\backslash S$ is a manifold, since we suppose $\Gm$ to
be torsion-free. The quotient $\Gm\backslash \SY$ is a manifold with
boundary. The quotient $\Gm\backslash \Sast$ is a compactification of
$\Gm\backslash S$, obtained by adding one point for each $\Gm$-orbit of
cusps.

The projection \il{pr}{$\pr\colon\Sast \rightarrow \Gm\backslash \Sast$}
$\pr\colon\Sast \rightarrow \Gm\backslash \Sast$ maps $\SY$ to a
compact part of $\Gm\backslash S$. The image of $B^\ast_\cu(Y) $ is a
closed neighborhood of the image of $\cu$ in $\Gm\backslash \Sast$. The
sets in the decomposition
\be \Gm\backslash\Sast = \pr\bigl( \SY\bigr) \cup\bigcup_{\cu \in
\Gm\backslash \Cu} \pr \bigl( B_\cu^\ast(Y) \bigr)
\ee
overlap only in their common boundary components.

\subsection{Tessellations}\label{sect4-tess}We will use a
$\Gm$-invariant tessellation $\tess$ of $\Sast$. The idea of a
tessellation is explained in \cite[\S6.2, \S11]{BLZ15}. Here we give an
informal discussion of the concept of a tessellation of $\Sast$.

We start with two tessellations in Figure~\ref{fig-tess2} of the upper
half-plane $\uhp$ (case $\ds=2$), both for the principal congruence
subgroup $\Gm(2)$ generated by $\matc 1201$ and $\matc 1021$. This is
the largest subgroup of $\Gm(1)=\PSL_2(\ZZ)$ that has no torsion.
\begin{figure}[ht]
\begin{center}\includegraphics[width=8cm]{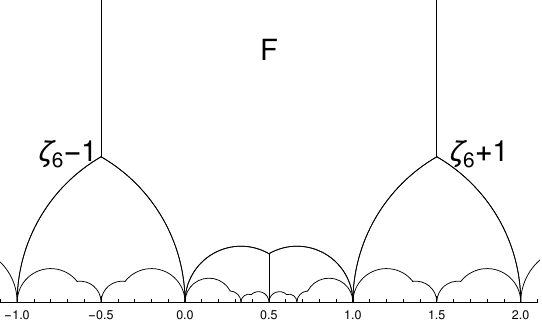}
\\
\includegraphics[width=8cm]{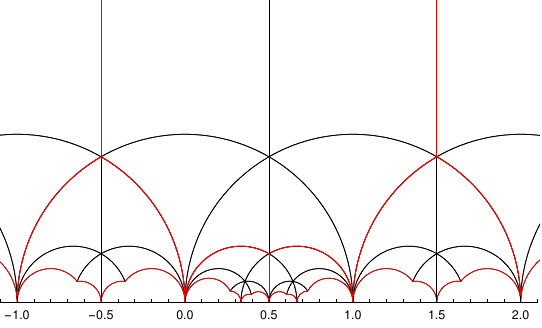}
\end{center}
\caption{Two $\Gm(2)$-invariant tessellations of
$\uhp$.}\label{fig-tess2}
\end{figure}
The upper tessellation $\tess_u$ consists of all $\Gm(2)$-translates of
a fundamental domain for $\Gm(2)$ in $\uhp$. The lower tessellation is
the refinement $\tess_l$, where each translate is divided into six
translates of the standard fundamental domain for $\Gm(1)$ in~$\uhp$.

A tessellation \il{tes2}{tessellation}\il{tess}{$\tess$}$\tess$ of
$\Sast$ provides us with collections
\il{Xtessi}{$\X^\tess_i$}$\X^\tess_i$ of \il{cell}{$i$-cell}$i$-cells,
for $1\leq i \leq \ds$. The following conditions should be satisfied:
\begin{enumerate}[label=$\mathrm{(\alph*)}$, ref=$\mathrm{\alph*}$]
\item {\it An $i$-cell, $0\leq i \leq \ds$, is a closed subset of
$\Sast$ that is homeomorphic to the closed unit disk in $\RR^i$. A
$0$-cell is a singleton in $\Sast$. }

In the upper tessellation $\tess_u$ in Figure~\ref{fig-tess2} the set
$\X^{\tess_u}_0$ consists of the cusps $\infty$, $0$ and $1$ and all
their $\Gm(2)$-translates, and all $\Gm(2)$-translates of the point
$\z_6=\frac12+\frac i2 \sqrt 3$. The set $\X^{\tess_u}_1$ consists of
the edges $e_{\z_6-1,\infty}$, $e_{0,\z_6}$, $e_{1, \z_6+1}$ and all
their $\Gm(2)$-translates. The set $\X^{\tess_u}_2$ consists of all
$\Gm(2)$-translates of the fundamental domain~$F$.

\item {\it For $1\leq i \leq \ds$ the boundary $\partial X$ of an
$i$-cell $X$ is the union
\[ \bigcup_{Y \in B(X)} Y\]
where $B(X) \subset \X^\tess_{i-1}$ is a finite collection of
$(i-1)$-cells.}

In $\tess_u$ the $2$-cell $F$ containing the cusps $\infty$ and $0$ has
\[ B(F) = \bigl\{ e_{\z_6-1,\infty} , t \, e_{\z_6-1,\infty},
e_{0,\z_6-1}, t' e_{0,\z_6-1}, e_{\z_6+1,1}, t't^{-1}\, e_{\z_6+1,1}
\bigr\}\,,\]
where $t= \matc 1201$, $t'=\matc1021$. Furthermore,
$B(e_{\z_6-1,\infty} ) = \{\z_6-1,\infty\}$.

\item {\it $\Sast = \bigcup_{X\in \X^\tess_\ds } X$.}
\item {\it For $1\leq i \leq \ds$, the pairwise intersection of
$i$-cells is equal to the intersection of their boundaries.}

\item {\it For $0\leq i \leq \ds-1$, each $i$-cell is a boundary
component of at least one $(i+1)$-cell.}

\item \label{gma}{\it For each $\gm \in \Gm$ and each $i$-cell~$X$, the
set $\gm \, X= \bigl\{\gm\, x\;:\; x\in X\bigr\}$ is an $i$-cell as
well. Furthermore $Y \in B(\gm X)$ if and only if $\gm^{-1}Y\in B(X)$.
This action of $\Gm$ on $\X^\tess_i$, $0 \leq i \leq \ds$, has finitely
many orbits.}

In the description of $\X^{\tess_u}_i$, $0\leq i \leq 2$ in (a) we have
given generating elements for the $\Gm(2)$-orbits in $\X^{\tess_u}_i$.
Condition \eqref{gma} requires that such finite choices can be made in
each tessellation. We can divide up $F$ into finitely many subsets
(intersecting only in their boundaries), obtaining a refinement of the
tessellation. A division of $F$ into infinitely many subcells is not
allowed.
\end{enumerate}

\rmrk{Orientation} We have to take into account the orientation of
$i$-cells. Each $i$-cell $X\in \X_i$ can be given two orientations.
(See \ref{sect4-or} for a further discussion.) We choose one of these
for each $X\in \X_i$, $0\leq i \leq \ds$, and denote by $-X$ the same
$i$-cell with the opposite orientation. For $i=0$ and $i=\ds$ there are
obvious choices of the orientation, for the other dimensions the choice
seems arbitrary.

For our purpose it is important that the chosen orientation is invariant
under the action of $\Gm$. If there were $\gm \in \Gm\setminus \{e\}$
such that $\gm X = X$, then it might happen that $\gm X$ and $X$ have
different orientations. Such an element $\gm$ would have finite order.
Since $\Gm$ is torsion-free, we see that $\gm=e$.

We proceed with a $\Gm$-invariant tessellations $\tess$ for which we
have chosen a $\Gm$-invariant orientation. For $0\leq i \leq \ds$, the
elements $X\in \X_i^\tess$ form an algebraic basis of the vector space
$\CC[\X_i^\tess]$. The basis elements $X$ have positive orientation,
and $-X$ is the same $i$-cell as $X$, but with negative orientation.
These spaces are $\CC[\Gm]$-modules for the action
$\gm \colon X \mapsto \gm X$.

\rmrk{Separation between \intitle{\SY} and cuspidal horoballs}The
tessellations that are handy for our purpose have the special property
that the boundary of $\SY$ in $S$ is covered by $(\ds-1)$-cells. An
example is sketched in Figure~\ref{fig-tessY}.
\begin{figure}[ht]\begin{center}
\includegraphics[width=9cm]{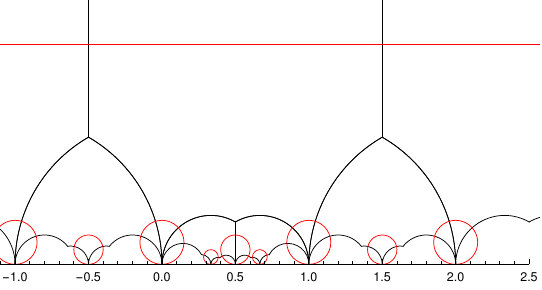}
\end{center}
\caption{A \intitle{\Gm(2)}-invariant tessellation of $\uhp$ in which
$\partial \SY$ is covered by $1$-cells.
(For clarity different values of $Y$ are used for the
horospheres.)}\label{fig-tessY}
\end{figure}

The procedure followed in this example works generally: Take a
fundamental domain $F$ that has only one cusp in each $\Gm$-orbit of
cusps and divide it up into $F\cap \SY$ and $F\cap B_\cu(Y)$ for all
cusps that occur in the closure of $F$ in~$\Sast$. This leads to a
tessellation that is a refinement of the tessellation obtained from all
$\Gm$-translates of~$F$.

For our purpose it does not matter if we move the intersections
$F\cap B_\cu(Y)$ by an element of $N_\cu$. See Figure~\ref{fig-plDYa}.
Of course, the boundary structure of the lower-dimensional cells will
become more complicated.
\begin{figure}\begin{center}
\includegraphics[width=8cm]{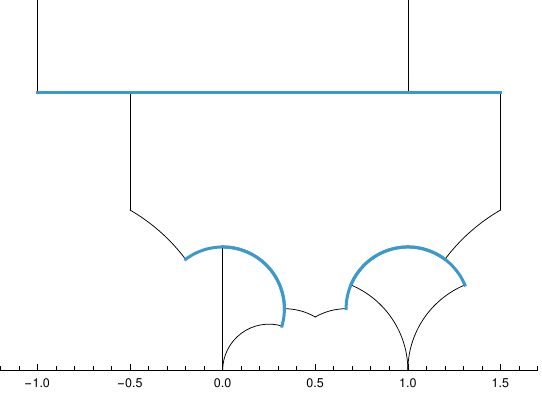}
\end{center}
\caption{The union of the four $2$-cells form a fundamental domain for
$\Gm(2)$ in $\uhp$. The $\Gm(2)$-translates of these cells generate a
tessellation for which $\partial S_{\!\Gm(2)}(Y)$ is covered by
$1$-cells. }\label{fig-plDYa}
\end{figure}

By a \il{sttess}{standard tessellation}\emph{standard tessellation}
\il{tessst}{$\tess_\st$}$\tess_\st$ we indicate a tessellation of this
type: The set $\X^{\tess_\st}_\ds$ consists of the (free) $\Gm$-orbits
of a cell \il{fd0a}{$\fd_0$}$\fd_0$ that is a fundamental domain for
$\Gm$ in~$\SY$, and the (free) $\Gm$-orbits of
\il{fdcu}{$\fd_\cu$}$\fd_\cu$ for a set of representatives of
$\Gm\backslash \Cu$. Each cell $\fd_\cu$ is a fundamental domain for
$\Gm_\cu$ in $B_\cu(Y)$ that has the product structure
$  F_{\!\cu} g_\cu A^+_Y$ with a fundamental domain $ F_{\!\cu}$ for
$\Gm_\cu$ in $N_\cu$.

For most of our discussions a standard tessellation will suffice.
Sometimes we need a refinement of a standard tessellation.

\subsection{Parabolic cohomology}\label{sec:parabcohom}
We describe parabolic cohomology spaces based on a tessellation, and
discuss their geometric relevance.
\smallskip

Let $\tess$ denote a standard tessellation, or a refinement of a
standard tessellation. We define the $\Gm$-modules
$F^\tess_i = \CC[\X^\tess_i]$ of \il{chn}{chain}$i$-chains for
$0\leq i \leq \ds$, and take $F^\tess_i=\{0\}$ for $i>\ds$.

For each $X\in \X^\tess_i$, $i\geq 1$, we define $\partial_i X$ as the
sum $\sum_{Y\in B(X)} \e_Y Y$ with $\e_Y\in \{1,-1\}$ such that the
orientation of $\e_Y Y$ is the orientation that the boundary component
$Y$ inherits from the orientation of $X$ (as will be discussed in
\S\ref{sect4-or}). The map $X \mapsto \partial_i X$ induces the
\il{bdop}{boundary operator}boundary operator $\partial_i$, which is a
$\CC[\Gm]$-equivariant map
\il{parti}{$\partial_i$}$\partial_i : F^\tess_i \rightarrow
F^\tess_{i-1}$. In this way we arrive at the resolution
\il{Ftb}{$F^\tess_\bullet$}$F^\tess_\bullet$
\be 0 \rightarrow F^\tess_\ds \stackrel{\partial_\ds}\longrightarrow
F^\tess_{\ds-1} \stackrel {\partial_{\ds-1}}\longrightarrow \cdots
F^\tess_1 \stackrel{\partial_1}\longrightarrow F^\tess_0 \stackrel
\e\longrightarrow \CC\,.\ee
The\il{augm}{augmentation} augmentation \il{e1}{$\e$ augmentation}$\e$
is determined by $\e X=1$ for all $X\in \X^\tess_0$. Giving $\CC$ the
structure of the trivial $\Gm$-module, we get a
\il{resol}{resolution}resolution of $\CC$ by $\CC[\Gm]$-modules.
(See for instance \cite[p~10]{Brown}.)
The resolution is exact: one can check that at each step
$\partial_{i}\partial
_{i+1} = 0$, and, moreover, the kernel of $\partial_{i}$ is equal to the
image of $\partial_{i+1}$, for $i < \ds$. (Take $\e$ as replacement for
$\partial_0$.) The resolution $F^\tess_\bullet $ is not a free
resolution: each cusp $\cu \in \X^\tess_0$ is fixed by the infinite
subgroup $\Gmm\cu \subset \Gm$.

The spaces of \il{coch}{cochain}cochains
\il{Cch}{$C^i(F^\tess_\bullet;\cdot)$}$C^i\bigl(F^\tess_\bullet;
V\bigr)$, $0 \leq i \leq \ds$, for a $\CC[\Gm]$-module $V$ consist of
the $\CC[\Gm]$-linear maps $F^\tess_i \rightarrow V$. An $i$-cochain
$c\in C^i\bigl(F^\tess_\bullet; V\bigr)$ is completely determined by
the $\Gm$-equivariant map $\X^\tess_i \rightarrow V$ given by
$X\mapsto c(X)$.

The boundary maps $\partial_{i+1}$ induce $\CC[\Gm]$-linear
\il{cbm}{coboundary map}coboundary maps \il{di}{$d^i$}$d^i
:C^i(F^\tess_\bullet;V) \rightarrow C^{i+1}(F^\tess_\bullet ;V)$
induced by $(d^i c)(X) = c\bigl(\partial_{i+1}
X)$ for $X\in \X^\tess_{i+1}$.

The spaces
\il{Coc}{$Z^i(F^\tess_\bullet;\cdot)$}$Z^i(F^\tess_\bullet;V)$ of
\il{coc}{cocycle}cocycles are defined as $\ker(d^i)$, and the subspaces
\il{Ccb}{$B^i(F^\tess_\bullet;\cdot)$}$B^i(F^\tess_\bullet;V)
\subset Z^i(F^\tess_\bullet;V)$ of \il{cob}{coboundaries}coboundaries
are defined as $\im(d^{i-1})$. Since there is no $C^{-1}$ we have
$B^0(F^\tess_\bullet;V)=\{0\}$.

The cohomology spaces
\il{Coh}{$H^i(F^\tess_\bullet;\cdot$)}$H^i(F^\tess_\bullet;V)$ are the
quotients $Z^i\bigl(
F^\tess_\bullet;V\bigr)/B^i\bigl(F^\tess_\bullet;V\bigr)$. For standard
tessellations $\tess_\st$ and its refinements, the cohomology spaces
are naturally isomorphic, and we can define the \il{pcs}{parabolic
cohomology space}parabolic cohomology spaces $H^i_\pb(\Gm;V)$ as
$H^i(F^{\tess_\st}_\bullet;V)$. In Subsection~\ref{sect4-tpc} we
will give a more detailed discussion.

\rmrk{Standard cohomology spaces} For each $\CC[\Gm]$-module $C$ the
space of \il{inv}{invariant}\emph{invariants} is
\il{V^Gm}{$V^\Gm$}$V^\Gm=\bigl\{
v\in V\;:\; \gm v=v \text{ for all }\gm\in \Gm\bigr\}$ , and the space
$V_\Gm$ of \il{coinv}{coinvariant}\emph{coinvariants}, is the space $V$
modulo the subspace generated by the elements $\gm\, v-v$ with
$\gm\in \Gm$ and $v\in V$.\il{V_Gm}{$V_\Gm = V \bmod
(\Gm-1)V$}

By defining \il{tessY}{$\X^{\tess,Y}_i$}$\X^{\tess,Y}_i =
\bigl\{ X\in \X^\tess_i\;:\; X\subset \SY\bigr\}$ we obtain a
tessellation of $\SY$. This leads to a resolution
\il{FtessY}{$F^{\tess,Y}_\bullet$}$F^{\tess,Y}_\bullet$ in an analogous
way. This is a free resolution, and the cohomology spaces
$H^i(F^{\tess,Y}_\bullet;V)$ are isomorphic to the usual cohomology
spaces $H^i(\Gm;V)$ for all suitably large $Y$. The restrictions of
cochains $c\in C^i(F^\tess_\bullet,V)$ to $F_i^{\tess,Y}$ induce linear
maps $H^i_\pb(\Gm;V) \rightarrow H^i(\Gm;V)$.
\begin{prop}\label{prop-H0d}The spaces $H^i_\pb(\Gm;V)$ and $H^i(\Gm;V)$
are zero for $i>\ds$, and
\badl{Hexpl} H^0\bigl( \Gm;V) &\cong H^0_\pb(\Gm;V\bigr) \cong
V^\Gm\,,\\
H^\ds_\pb\bigl( \Gm;V) &\cong V_\Gm \,,\qquad H^\ds\bigl( \Gm;V\bigr) =
\{0\}\,.
\eadl
\end{prop}
We will give a proof in \S\ref{prf-H0d}.

\rmrk{Relation to sheaf cohomology}The geometric relevance of parabolic
cohomology spaces $H^i_\pb(\Gm;V)$ is their interpretation as sheaf
cohomology spaces on the quotient $\Gm\backslash\Sast$. The global
functor $\sh \mapsto \sh(\Gm\backslash \Sast)$ has derived functors
$\sh\mapsto H^i\bigl( \Gm\backslash \Sast;\sh)$. These cohomology
groups can be computed as \u{C}ech cohomology. See, e.g.\cite[Chap III,
\S2, \S4]{Harts}.

If $V$ is a $\CC[\Gm]$-module, then $V\times\Sast$ can be considered as
a constant sheaf on $\Sast$. The group $\Gm$ acts on this sheaf by
$\gm \,(v,x) = \bigl(\gm v,\gm \, x\bigr)$. The
quotient\il{QV}{$\Q[V]$}
\be \Q[V] = \Gm \bigm\backslash (V\times \Sast)
\ee
is a locally constant sheaf on $\Gm\backslash\Sast$. The parabolic
cohomology spaces can be related by cohomology spaces for sheaves on
the quotient $\Gm\backslash \Sast$, like in \cite[Proposition
11.8]{BLZ15}.
\begin{prop}\label{prop-pcshc} In this situation, for all $j\geq 0$,
\be H^j_\pb\bigl( \Gm;V\bigr) \;\cong\; H^j\bigl( \Gm\backslash \Sast;
\Q[V]\bigr)\,.\ee
\end{prop}
In \S\ref{prf-pcshc} we will give a sketch of a proof.

\begin{remark}
\label{sect4-sexsq}
This proposition shows that the parabolic cohomology spaces are
connected to a well-known cohomology theory. It implies that there are
long exact sequences of parabolic cohomology spaces associated to short
exact sequences of $\CC[\Gm]$-modules.\end{remark}
\bigskip

This finishes a discussion of parabolic cohomology and tessellations
that seems essential for the following sections of this paper. In the
further subsections of this section we give proofs of Propositions
\ref{prop-H0d} and~\ref{prop-pcshc}, and more details on standard
tessellations.

\subsection{Further discussion of \intitle{\Gm}-invariant
tessellations}\label{sect4-tess-f}The following definition complements
the informal discussion of tessellations in \S\ref{sect4-tess}.

\begin{defn}\label{def-tess}
A \il{tessdef}{tessellation} \emph{tessellation} $\tess $ of $\Sast$
consists of $\ds+1$ collections \il{XtessiA}{$\X^\tess_i$}$\X^\tess_i$,
$0\leq i \leq \ds$, of subsets of $\Sast$ with the following
properties.
\begin{enumerate}[label=$\mathrm{(\roman*)}$, ref=$\mathrm{\roman*}$]
\item\label{tess:i} \emph{Cells. } The elements $X\in \X^\tess_i$,
$0\leq i \leq \ds$, called \il{icA}{$i$-cell}\emph{$i$-cells}, are
closed subsets of $\Sast$ that have the following properties:
\begin{enumerate}[label=$\mathrm{(\alph*)}$, ref=$\mathrm{\alph*}$]
\item\label{tess:ia}\emph{Topological structure. } There is a
homeomorphism from $X$ to the closed unit ball
$\{ x\in \RR^i\;:\; \|x\|\leq 1 \bigr\}$.

By the \il{bdic}{boundary of $i$-cell}boundary $\partial X$ of $X$,
respectively the \il{iciX}{interior of $i$-cell}interior
\il{intX}{$\mathring X$}$\mathring X$ of $X$, we understand the image
of the unit $i$-sphere in $\RR^i$, respectively the image of the open
unit ball in $\RR^i$ under this homeomorphism.

\item\label{tess:ib}\emph{Analytic structure. }Each $i$-cell $X$ can be
described as
\be X= \bigl\{ x\in \Sast\;:\; f_j(x)=0 \text{ for }1\leq j \leq
a_{X}\,,\; h_j(x) \geq 0 \text{ for }1\leq j \leq b_{X}\bigr\}\,,\ee
with a finite number of continuous functions $f_j$, with
$1\leq j \leq a_X$, and $h_j$, $1\leq j\leq b_X$, on a neighborhood of
$X$ in $\Sast$, such that the restriction of each $f_j$ to $S$ is
analytic and  the $h_j$ are analytic on $\partial X \cap S$.
\end{enumerate}

\item\label{tess:bd} \emph{Boundary decomposition. }Let
$1\leq i \leq \ds$. For each $X\in \X^\tess_i$ there is a finite subset
\il{BX}{$B(X)$}$B(X) \subset \X^\tess_{i-1}$ such that
\be \partial X = \bigcup_{Y \in B(X)} Y\,.\ee

\item\label{tess:cov} \emph{Coverings.} Let $1\leq i \leq \ds$.
\begin{enumerate}[label=$\mathrm{(\alph*)}$, ref=$\mathrm{\alph*}$]
\item\label{tess:cova} The pairwise intersections of $i$-cells are
contained in the intersection of their boundaries.
\item\label{tess:covb} We have
\be \bigcup_{X\in \X^\tess_i } X =
\begin{cases} W&\text{if }i=\ds\,,\\
\bigcup_{Y\in \X^\tess_{i+1} } \partial Y &\text{ if }1\leq i \leq
\ds-1\,.
\end{cases}
\ee
\end{enumerate}

\item\label{tess:nmtn} \emph{All cells occur in a boundary.} If
$0\leq i \leq \ds-1$, then
\be \X^\tess_i = \bigcup_{Y\in \X^\tess_{i+1}}B(Y)\,.\ee
\end{enumerate}
\end{defn}
\rmrk{Remarks}
\begin{enumerate}[label=$\mathrm{(\arabic*)}$, ref=$\mathrm{\arabic*}$]
\item The $0$-cells are points in $W$.

\item The interior and boundary of an $i$-cell $X$ with
$1\leq i\leq \ds-1$ are not the interior and boundary of $X$ as a
subset of the surrounding topological space~$W$.

\item Condition \eqref{tess:nmtn} ensures that for $j>i$ each $i$-cell
is contained in the boundary of some $j$-cell.

\end{enumerate}

\begin{defn}\label{def-itess}
A tessellation $\tess$ of $\Sast $ is \il{Gmitess}{$\Gm$-invariant
tessellation}$\Gm$-invariant if the following conditions are satisfied:
\begin{enumerate}[label=$\mathrm{(\alph*)}$, ref=$\mathrm{\alph*}$]
\item Let $0\leq i \leq \ds$. If $X\in \X^\tess_i$, then
$\gm X = \bigl\{ \gm x\;:\; x\in X$, $\gm\in \Gm\bigr\}$ is in
$\X^\tess_i$ as well.
\item For each $\gm\in \Gm$ and $X\in \X^\tess_j$ with
$1\leq i \leq \ds$
\be B(\gm \, X) = \gm\, B(X) = \bigl\{ \gm Y \;:\; Y \in B(X)\bigr\}
\,.\ee
\end{enumerate}
\end{defn}

\subsubsection{Orientation}\label{sect4-or}
There are various ways to formulate what we mean by orientation of the
cells of a tessellation $\tess$. Here we use the analytic structure of
cells.

For points $p\in \X^\tess_0$ we just view $p$ as positively oriented,
and $-p$ as negatively oriented. In the top dimension $\ds$ we consider
$X\in \X^\tess_\ds$ as positively oriented.

For $X\in \X^\tess_i$, $1\leq i \leq \ds-1$, each choice of a
homeomorphism $\ph: X \rightarrow U \subset \RR^i$ implies the choice
of coordinates $(x_1,\ldots,x_i)$ on $X$. We use differential forms
$ \om =f(x_1,\ldots,x_i) \, dx_1\wedge \cdots\wedge dx_i$ with a
real-valued analytic function~$f$.

Let $\ph_1$ be another homeomorphism as in Definition~\ref{def-tess}
\tlref{tess:i}{tess:ia}. It can be viewed as the choice of other
coordinates $y_1,\ldots,y_i$. The coordinate transformation from
$(x_1,\ldots,x_i)$ to $(y_1,\ldots,y_i)$ is described by the map
$\ps = \ph_1\circ\ph^{-1}$. We have
\[ \om \circ \ps = (f\circ \ps)\; \det J(\ps)\, dy_1\wedge \cdots
\wedge dy_i\,,\]
where $\det J(\ps)$ is the \il{Jm}{Jacobi matrix}determinant of the
Jacobian matrix $\Bigl( \frac{\partial  \ps_j }{ \partial x_n} \Bigr)$
of the vector-valued function~$\ps$. Whatever we have chosen as the
orientation of $X$, the transformation preserves the orientation if
$ \det J(\ps) >0$, and reverses it if $\det J(\ps)<0$. This is clear
for dimensions $0$ and $\ds$, but not for the intermediate dimensions.

Suppose now that for each $X\in \X^\tess_i$ we have chosen coordinates
$\bigl (x_{X,1}, \ldots, x_{X,i}\bigr)$ as corresponding to the
positive orientation. For $\gm\in \Gm$ and $1\leq i \leq \ds-1$, we
have the transformation $s(\gm,i)$ fitting in the following scheme.
\be \xymatrix{ X \ar[r]^{\ph_X} \ar[d]^\gm & U_X\ar[d]^{s(\gm,i)}&
\bigl(x_{X,1},\ldots,x_{X,i}\bigr) \text{ coord.}
\\
\gm X \ar[r]_{\ph_{\gm X}} & U_{\gm X}
& \bigl(x_{\gm X,1},\ldots,x_{\gm X,i}\bigr)
\text{ coord.} } \ee
The action of $\gm$ sending $X$ to $\gm X$ determines a transformation
$s(\gm,i)$ of $i$-forms. We have
\be \bigl( dx_{X,1}\wedge \cdots\wedge dx_{X,i} \bigr) \circ s(\gm,i)
= \det J\bigl( s(\gm,i) \bigr)\, \bigl(dx_{\gm X,1} \wedge\cdots\wedge
dx_{\gm X,i}\bigr)\,.\ee
So the action of $\gm\in \Gm$ preserves the chosen orientation if and
only if $\det J\bigl( s(\gm,i)\bigr)>0$. For $1\leq i \leq \ds$ there
are no non-trivial $\gm \in \Gm$ stabilizing $X$ (by the fact that
$\Gm$ is torsion-free). So after we choose the orientation of $X$, this
induces a unique orientation of all $\gm X$ in the $\Gm$-orbit of~$X$.

\subsubsection{Refinements}We call a $\Gm$-invariant tessellation
$\tess'$ a refinement of the $\Gm$-invariant tessellation $\tess$ if
for $0\leq i \leq \ds$ each $X\in \X^\tess_i$ is the union of a finite
set of $X_1',\ldots,X_n' \in \X^{\tess'}_i$ where different $X_l'$
overlap only in their common boundary. This induces $\CC$-linear maps
$F_i^\tess \rightarrow F^{\tess'}_i$ that we require to be
$\CC[\Gm]$-linear and to be compatible with the boundary maps
$\partial_i$ in $F^\tess_\bullet$ and $F^{\tess'}_\bullet$.

\begin{lem}[Refining a tessellation]\label{lem-refinetess} Let $\tess$
be a $\Gm$-invariant tessellation of $W\subset S$, and let
$X\in \X^\tess_i$ with $1\leq i\leq \ds$. If
$X= X_1 \cup X_2\cup \ldots\cup X_n$ with $i$-cells $X_l$ such that the
pairwise intersections $X_l\cap X_{l'}$ have a structure as indicated
in part~\eqref{tess:i} of Definition~\ref{def-tess}, then there is a
refined $\Gm$-invariant tessellation $\tess'$ of $W$ such that
$X\in \X^\tess_i$ is replaced by $X_1,\ldots,X_n$ in $\X^{\tess'}_i$.
In this refinement, $\X^{\tess'}_j = \X^\tess_j$ for $i<j\leq \ds$. The
sets $B(X)$ with $X\in \X^\tess_j$ are unchanged for
$i+1<j\leq \ds$.\end{lem}
\begin{proof}The process is straightforward. For each $\gm\in \Gm$
replace $\gm X$ by the new cells $\gm X_1,\ldots,\gm X_n$. Adapt the
boundary decomposition of all $Y\in B(X) $. Next go down dimension by
dimension, adapting $\X^\tess_j$ and the boundary decomposition in
dimension $j+1$.
\end{proof}
\remark{\em Condition on refinements. }\label{refine-cond}
  We use only refinements that respect the product form
\[\fd_\cu= g_\cu \fd(\Ld_\cu) A^+
_T \,\orgn \cong \fd(\Ld_\cu)
\times [T,\infty)\]
in the subdivisions of the generators $\fd_\cu$, $\cu \in \Cu$. 
We use the notation \il{A+T}{$A^+_T$}$ A^+_T=\bigl\{ \am(t)\;:\; t\geq T \bigr\}$. 
We split
up $\fd(\Ld_\cu)$ into a finite number of cells defined by inequalities
$a_i \leq x_i \leq b_i$ for $1\leq i \leq \ds-1$, in coordinates that
are linear forms in normalized horospherical coordinates, and split up
$[T,\infty)$ into a finite number of subintervals. This system of
coordinates need not be horospherical.

\subsubsection{Standard tessellations}\label{sect4-stdtess}On
p~\pageref{i-tessst} we mentioned standard tessellations. The idea is
illustrated in Figures \ref{fig-tess2}--\ref{fig-plDYa}.

\begin{prop}\label{prop-sttess}Let $Y$ be suitably large. There exist
standard $\Gm$-invariant tessellations $\tess=\tess_\st$ of $\Sast$
with the following properties.
\begin{enumerate}[label=$\mathrm{(\roman*)}$, ref=$\mathrm{\roman*}$]
\item\label{Ysplit} For all $i$, $0\leq i \leq \ds$, there is a subset
$\X_i^{\tess,Y}$ of $\X^\tess_i$ such that
  \begin{enumerate}[label=$\mathrm{(\alph*)}$, ref=$\mathrm{\alph*}$]
  \item \label{Ycond} For each $X\in \X^\tess_i$:
  \[ X\in \X^{\tess,Y}_i \quad\Leftrightarrow\quad X\subset \SY \,.\]
  \item \label{nYcond} If $X\in \X^\tess_i \setminus \X^{\tess,Y}_i$,
  then there is a unique cusp $\cu$ such that $X \subset B^\ast_\cu(Y)$.
  \end{enumerate}
\item The collection of sets $\X^{\tess,Y}_i$, $0\leq i\leq \ds$,
determines a $\Gm$-invariant tessellation of $\SY$.
\item \label{gen} Let $n$ denote the number of cuspidal $\Gm$-orbits.
The set $\X^\tess_\ds$ is generated by $n+1$ elements, namely
\begin{itemize}
\item A $\ds$-cell $\fd_0\in \X^{\tess,Y}_\ds$, which is a fundamental
domain for $\Gm$ in $\SY$.
\item A $\ds$-cell $\fd_c \in \X^\tess_d \setminus \X^{\tess,Y}_\ds$ for
each representative $\cu$ in a system of representatives of
$\Gm\backslash \Cu$. The set $\fd_\cu$ is a fundamental domain for
$\Gmm \cu$ in $B^\ast_\cu(Y)$.
\end{itemize}
\end{enumerate}
\end{prop}
\begin{proof}
We start with a fundamental domain of $\Gm$ in $S$. These domains can be
constructed in various ways. A nice example in the context of
$G=\PSU(2,1)$ is given by Falbel and Parker, \cite[Theorem 5.8 and
5.9]{FP06}. For the cofinite groups $\Gm$ under consideration the Ford
domains in \cite[Theorem 3.18]{Po10} also satisfy the requirements.
Here we use the Dirichlet method of construction of fundamental
domains; see the discussion of tessellations of type {\bf Dir} in
\cite[\S6.2]{BLZ15}.

\rmrk{Tessellation of \intitle{\SY}} We choose $p_0$ in the interior of
$ \SY$ for suitably large~$Y$. Its $\Gm$-orbit $\Gm p_0$ is an infinite
discrete subset of $S$. We use the distance function $\dist$ to define
the following fundamental domain:\il{Dp}{$D(p)$}
\be \label{Dp}D(p_0) = \bigl\{ x\in S\;:\; \dist(x,p_0) \leq \dist(x,p)
\text{ for all } p\in \Gm p_0 \setminus \{p_0\} \bigr\}\,.\ee
This is a fundamental domain for $\Gm$ on $ S$. Since $\Gm$ is discrete,
in the definition we need to take into account only a finite set
$\{p_1,\ldots,p_n\}\subset \Gm p_0 \setminus\{p_0\}$.

The functions $h_l(x) = \dist(x,p_l)-\dist(x,p_0)$, $1\leq l \leq n$,
are analytic on a neighborhood of the boundary of $D(p_0)$, and
\be \label{Dp0-ineq} D(p_0) = \bigl\{ x\in S\;:\; h_l(x)\geq0\text{ for
}l=1,\ldots,n\bigr\}\,.\ee

We intersect $D(p_0)$ with $\SY$, where $Y$ is such that all horoballs
$B_\cu(Y)$, $\cu \in \Cu$, are pairwise disjoint. The number of cusps
in the closure of $D(p_0)$ in $S\cup\partial S$ is finite. If we
further enlarge $Y$ sufficiently, then
\il{fd0}{$\fd_0$}$\fd_0  := D(p_0) \cap \SY$ has a description like
\eqref{Dp0-ineq}.

The $\Gm$-translates of $D(p_0) \cap \SY$ cover $\SY$, and lead to a
tessellation \il{tess-DY}{$\tess(D,Y)$}$\tess(D,Y)$ of $\SY$.\medskip

\rmrk{Tessellation of \intitle{N}} Next we construct a tessellation of
the $(\ds-1)$-dimensional group that is invariant under a lattice
$\Ld \subset N$. If $\glie_{2\al}=\{0\}$, then the group $N$ is
isomorphic to $\RR^p = \RR^{\ds-1}$ (as a group), $\Ld$ is isomorphic
to $\ZZ^p$, and $\Ld \backslash N$ is a torus, for which it is easy to
construct a fundamental domain, leading to a $\Ld$-invariant
tessellation of~$N$, for which the $i$-cells are determined by
$(n_1,\ldots,n_p) \in \ZZ^p$ and by a finite set
$J \subset \{1,\ldots,p\}$, namely
\be \bigl\{ x\in \RR^p\;:\; n_i \leq x_i \leq n_i+1 \text{ for } i \in
J\,, \text{ and } x_i=n_i\text{ for }i\not\in J \bigr\}\,.\ee

If $N$ is not abelian, we can choose coordinates
$x=(x_1,\ldots,x_{p+q})$ on $\RR^{p+q}$ such that the intersection
$Z_N \cap \Ld$ of the lattice with the center $Z_N$ of $N$ is generated
by specific elements $\ld_j$, $p+1\leq j \leq p+q$, such that the group
operation satisfies
\be\label{coord-trf1} \ld_j \, \nm[x] = \nm\bigl[x_1,\ldots,x_{j-1},
x_j+1,x_{j+1},\ldots , x_{p+q}\bigr]\,.\ee
We remind ourselves by the use of
\il{nm[]}{$\nm[x_1,\cdots,x_{p+q}]$}square
brackets that these coordinates are not necessarily horospherical
coordinates. We can complete the choice of generators of $\Ld$ by
choosing $\ld_j$, $1\leq j \leq p$, for which
\badl{coord-trf2} \ld_j\,\nm[x] &= \nm\bigl[ x_1,\ldots, x_{j-1}, x_j+1,
x_{j+1}, \ldots, x_p,\\
&\qquad\qquad\qquad\hbox{} x_{p+1}+ a_{p+1}(x^{(1)}), \ldots, x_{p+q}+
a_{p+q}( x^{(1)} ) \bigr]\,,\eadl
with linear forms $a_j$ on $\RR^p$ for $p+1\leq j \leq q$.

We can choose a fundamental domain $X(\Ld)$ for $\Ld \backslash N$ of
the form
$\bigl\{ x\;:\; 0\leq x_j \leq 1\text{ for }1\leq j \leq p+q\bigr\}$.
We can bring any $x\in \RR^{p+q}$ into $X(\Ld)$ by first applying
$\ld_j^n$ with $1\leq j \leq p$, $n\in \ZZ$, to arrive at
$0\leq x_j \leq 1$ for $1\leq j \leq p$. Next we apply $\ld_j^n$ with
$p+1\leq j \leq p+q$, $n\in \ZZ$, to adjust the coordinates with index
at least $p+1$. This fundamental domain has $(p+q-1)$-dimensional
boundary cells $B_{\e,j}$ with $1\leq j \leq p+q$ and $\e\in \{0,1\}$,
determined by $x_j = \e$ and $x_{j'}\in [0,1]$ for $j'\neq j$.

For $j \geq p+1$ we have the nice pairing of
$\ld_j\, B_{0,j} = B_{1,j}$. This does not work well if $j\leq p$. Then
the presence of the linear forms $a_{j'}$ will in general require
$\mu \ld_j$ with $\mu \in \Ld \cap Z_N$ to carry out the
transformation. The element $\mu$ depends on the values of $x$ in
$\nm[x]\in B_{0,j}$. The values that the $a_{j'}$ will take on elements
of $B_{0,j}$ are bounded, so finitely many elements of $\mu$ in the
center of $\Ld$ can occur. See Figure~\ref{fig-fdsh}.
\begin{figure}
\begin{center}\includegraphics[width=3cm]{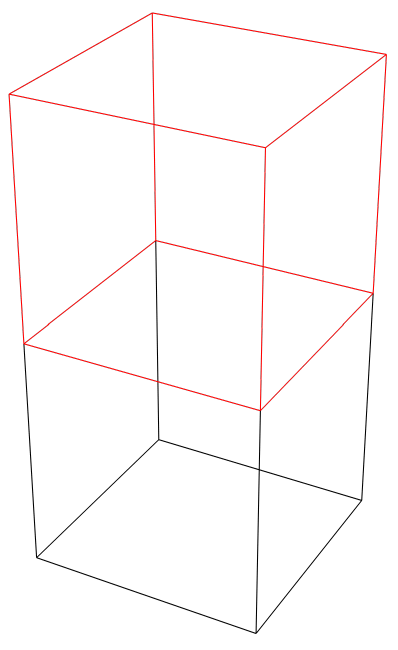}
\qquad \includegraphics[width=5cm]{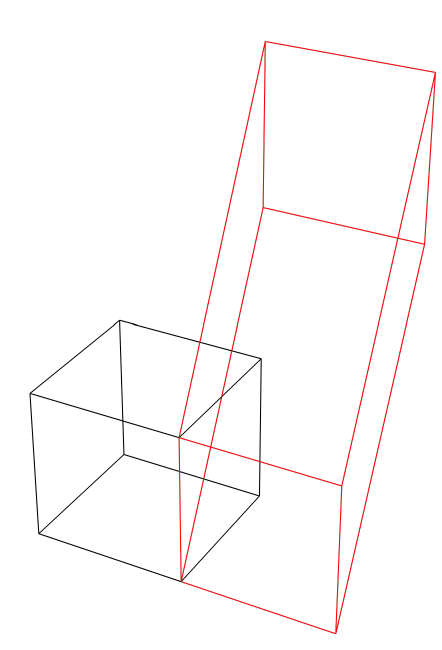}
\end{center}
\caption{The black cube is a fundamental domain $B$ for
$\Ld\backslash N$ with some lattice $\Ld$ in the context of
$G=\PU(2,1)$, for which $\ds=4$, $p=2$ and $q=1$. On the left, the red
cube is the image of $B$ under $\ld_3 = \ld_{p+1}$. On the right the
image $\ld_1 B$ is deformed. }\label{fig-fdsh}
\end{figure}
The fundamental domain $X(\Ld)$ intersects a finite number of translates
$\mu_i \ld_j X( \Ld )$ with $\mu_i \in \Ld \cap Z_N$,
$1\leq i \leq n_{0,j}$. This intersection is the union
$\bigcup_i B_{1,j}\cap \mu_i \ld_j B_{0,j}$. The elements of this union
have pairwise intersections contained in boundary components. We
replace the boundary component $B_{0,j}$ of $X(\Ld)$ by the subsets
$B_{0,j}\cap \ld_j^{-1}\mu_i^{-1} B_{1,j}$. (We recall that
$ \Ld\subset N$, and that $Z_N$ is center of $N$; so
$\ld_j,\mu_j \in N$.) We write the $\left(p+q-1\right)$-cells $B_{0,j}$
with $1\leq j \leq p$ as a union $\bigcup_i B_{0,j}^{(i)}$. The
subcells of dimension $\left(p+q-1\right)$ in this union intersect only
in their boundaries. If we increase the coordinate $x_j$ by $1$ we get
a subdivision of $B_{1,j}$. We can take the subdivisions in such a way
that $B_{1,j} $ is the union of cells $\mu_i \ld_j B_{0,j}$ with
$\mu_i\in \Ld\cap Z_N$.

This process increases the number of $(p+q-1)$-cells. The resulting
system of cells gives a $\Ld$-invariant tessellation
\il{UtessLd}{$\Utess(\Ld)$}$\Utess(\Ld)$ of $N$ that satisfies
conditions analogous to those in Definitions \ref{def-tess}
and~\ref{def-itess}. The set $\X^{\Utess(\Ld)}_{\ds-1}$ consists of one
$\Ld$-orbit. As a generator we may take the hypercube
\il{fdld}{$\fd(\Ld)$}$\fd(\Ld)=[0,1]^{\ds-1}$ in the coordinates
indicated above.
\smallskip

\rmrk{Tessellation of a horoball} We fix a cusp $\cu\in \Cu$ and
describe the horoball $B_\cu(Y)$ as
$g_\cu\bigl(N \times A_Y^+ \,\orgn\bigr)$, with
$A_T^+= \bigl\{\am(t)\in A\;: \; t\geq T\bigr\}$.

The group $\Gmm\cu$ stabilizing the cusp $\cu$ acts on $B_\cu(Y)$. This
action leaves $g_\cu$ and $y$ in
$g_\cu (n,y) \in g_\cu \bigl(N \times[Y,\infty)\bigr)$ invariant, and
changes only the element $n\in N$. Under conjugation by $g_\cu$ this
corresponds to the action of the lattice $\Ld_\cu$ on $N$ by left
translation, for which we constructed a fundamental domain in the
previous subsection. This corresponds to a fundamental domain
$g_\cu \bigl( \fd(\Ld_\cu ) \times [Y,\infty)\bigr)$ of $B_\cu(Y)$. We
close it in $B^\ast_\cu(Y)$ by adding the cusp $\cu$. This results in a
fundamental domain \il{fdcua}{$\fd_\cu$}$\fd_\cu$ for
$\Gmm \cu \backslash B^\ast_\cu(Y)$. It is a cone with top $\cu$ on the
hypercube $g_\cu \bigl(\fd(\Ld_\cu)\times \{\am(Y) \}\bigr)$. The
resulting tessellation of $B_\cu(Y)$ is a $\Gmm\cu$-invariant
tessellation.

If $\gm$ runs over $\Gm\backslash\Gmm\cu$ then $\gm\, B_\cu(Y)$ runs
over all horoballs $B_{\cu'}(Y)$ with $\cu'$ running over the
$\Gamma$-orbit $\Gm\,\cu$. Then the $\Gm$-translates $\gm \, \fd_\cu$
provide us with a $\Gm$-invariant tessellation $\tess_a$ of the union
$U_a=\bigcup_{\cu' \in a}B_{\cu'}(Y)$ for the $\Gm$-equivalence class
$a\in \Gm\backslash \Cu$ of~$\cu$.

\rmrk{Putting tessellations together} We have obtained $n+1$
tessellations: the tessellation $\tess(D,Y)$ of $\SY$ introduced at the
end of \S\ref{sect4-stdtess}, and, for each of the $n$ elements
$a\in \Gm\backslash \Cu$, the tessellation $\tess_a$ of a union $U_a$
of horoballs.

We have still to handle the intersections of $\SY$ with the $U_a$. This
intersection is a union of $(\ds-1)$-dimensional horospheres. The
common refinement of these $n+1$ tessellations forms a tessellation
$\tess_\st$ that satisfies the properties in Theorem~\ref{prop-sttess}.

As generating $\ds$-cells we can take $\fd_0 = D(p_0)\cap\SY$, and
$\fd_a$ as the generating element in $B_\cu(Y)$ for the representative
of $a$ that we used.
\end{proof}

\subsection{Tessellations and parabolic cohomology}\label{sect4-tpc} We
will justify the name parabolic cohomology by showing that the
following three ways to construct cohomology spaces lead to canonically
isomorphic results:
\begin{enumerate}[label=$\mathrm{(\roman*)}$, ref=$\mathrm{\roman*}$]
\item \label{ctess} The cohomology spaces based on the resolution
$F^\tess_\bullet$ for a standard tessellation or refinements of it.
\item \label{shcoh} Sheaf cohomology spaces for the quotient
$\Gm\backslash\Sast$.
\item \label{cuc} Use of a resolution $F^\pb_\bullet$ based on the set
of cusps $\Cu$ and the $\Gm$-action on it, to be defined in
\S\ref{sect4-def-Fpb}.
\end{enumerate}
The relation between \eqref{ctess} and \eqref{shcoh} is given by
Proposition~\ref{prop-pcshc}, which we will prove in \S\ref{prf-pcshc}.
Proposition~\ref{prop-pcht} will relate \eqref{cuc} to \eqref{ctess}.
We define the \il{pbcs}{parabolic cohomology space}parabolic cohomology
spaces \il{Hipb}{$H^i_\pb(\Gm;V)$}$H^i_\pb(\Gm;V)$ as the cohomology
spaces obtained from \eqref{ctess}--\eqref{cuc}.

The proof of Proposition~\ref{prop-H0d} is in~\S\ref{prf-H0d}.

\subsubsection{Acyclic resolutions} We consider refinements $\tess $ of
a standard tessellation $\tess_\st$ of the space $\Sast$. It leads to
the resolution $F^\tess_\bullet$ and to cohomology spaces
$H^i \bigl(F^\tess_\bullet; \cdot)$ with values in $\CC[\Gm]$-modules.
We have also the tessellation \il{tessYa}{$\tess_Y$}$\tess_Y$ of $\SY$,
connected to the resolution
$F^{\tess,Y}_\bullet = \bigl( \CC[\X^{\tess,Y}_i]\bigr)_{i\geq 0}$.

We mention some terminology and facts from \cite[Chap I]{Brown}. One
speaks of an \il{ac-res}{acyclic resolution}\emph{acyclic resolution}
if the complex is exact at each place. In each complex we have
$ \partial_{i-1}\partial_i = 0$, and hence $\im(\partial_{i})$
is contained in $\ker(\partial_{i-1})$.
The complex is acyclic if $\im(\partial_{i}) =\ker(\partial_{i-1})$ for
all $i\geq 2$, and in the case of a resolution of the trivial
$\Gm$-module $\CC$ the condition $\e\,\partial_1=0$ for $i=1$
(with the augmentation $\e$ instead of $\partial_0$).

Let $A_\bullet$ and $B_\bullet$ denote complexes. A \il{chmp}{chain
map}\emph{chain map} $f_\bullet: A_\bullet \rightarrow B_\bullet$ gives
$\CC[\Gm]$-linear maps $f_i : A_i \rightarrow B_i$ such that
$ \partial_i f_i = f_i \partial_i
$ for all $i$. (We allow ourselves to denote the boundary operator by
$\partial_i$ in all complexes.)

Two chain maps $f_\bullet: A_\bullet \rightarrow B_\bullet$ and
$g_\bullet : A_\bullet \rightarrow B_\bullet$ are
\il{hmtp}{homotopic}\emph{homotopic} if the difference
$f_\bullet-g_\bullet$ can be written as
\be \partial_{i+1} h_i + h_{i-1} \partial_i = f_i-g_i\ee
for some family $h_\bullet : A_\bullet \rightarrow B_\bullet$ of
homomorphisms $h_i\colon A_i\to B_{i+1}$ for $i\in\NN_0$. Such a family
$h_\bullet$ is called a \il{hmtpy}{homotopy}{\em homotopy}. Even if we
have a $\Gm$-action on the complexes, the homotopies $h_i$ need only be
$\CC$-linear.

\rmrk{Contractible spaces}A space $W$ is
\il{contr}{contractible}contractible if it is homotopic to a space with
one point. In that case tessellations of $W$ lead to acyclic
resolutions.

The horoballs $B^\ast_\cu (Y)$ are contractible by use of the map
\be [0,1]\times B^\ast_\cu (Y) \rightarrow B^\ast_\cu\;:\;
\begin{cases}
(t,\cu) \mapsto \cu\,,
\\
\bigl( t, g_\cu n \am(y)
\bigr) \mapsto
\begin{cases} g_\cu n \am\bigl( y/(1-t)\bigr)&\text{ if }t<1\,,\\
\cu&\text{ if }t=1\,.
\end{cases}
\end{cases}
\ee

The space $S$ is contractible as well:
\be \bigl(t,k \am(y)\,\orgn\bigr) \mapsto k
\am\bigl(y(1-t)\bigr)\,\orgn\,.\ee

\begin{prop}\label{prop-acycl}The resolutions $F^\tess_\bullet$ and
$F^{\tess,Y}_\bullet$ are exact.
\end{prop}
\begin{proof} Let $Y'< Y$ be suitably large: the horoballs $B_\cu(Y')$
for $\cu \in \Cu$ are pairwise disjoint.

We showed above that each extended horoball $B^\ast_\cu(Y)$ can be
contracted to $\{\cu\}$. In a similar way we can contract $B_\cu(Y')$
to $B_\cu(Y')\setminus B_\cu(Y)$. Doing this for all cusps and leaving
the points of $S_{\!\Gm}(Y')$ on their place, we get a continuous
deformation of $S$ into $\SY$. This implies that $\SY$ is contractible.
That implies that $F^{\tess,Y}_\bullet$ is acyclic.

The contractibility of $\SY$ implies that there is a homotopy
$h^Y_\bullet$ between the identity chain map and the zero chain map
from $F^{\tess, Y}_\bullet  $ in itself. It satisfies for all $i$
\be \partial_{i+1} h_i^Y \xi + h_{i-1}^Y \partial_i \xi = \xi\qquad
\text{for all }\xi \in F^{\tess, Y}_i\,.\ee

To show that the complex $F^\tess_ i$ is acyclic we take
$C\in F^\tess_i$ and assume that $\partial C=0$. We want to show that
$C \in \partial_{i+1} F^\tess_{i+1}$.
We can write
\be C= a_0 + \sum_\cu a_\cu \ee
with $a_0\in F^{\tess,Y}_i$, and $a_\cu \in F^{\tess(\cu)}_i$ for all
$\cu \in \Cu$, for the tessellation
\il{tesscu}{$\tess(\cu)$}$\tess(\cu)$ of $B^\ast_\cu(Y)$ induced by
$F^\tess_\bullet$. Since elements of $F^\tess_i$ are finite linear
combinations of elements of $\X^\tess_i$, all but finitely many $a_\cu$
are zero. Since
\be 0 = \partial _i a_0 + \sum_\cu \partial _i a_\cu,\ee
each chain $\partial _i a_\cu$ is a linear combination of
$X\in F^\tess_i$ that are contained in the horosphere $H_\cu(Y)$
(Definition~\ref{def-hor}), and hence
$\partial_i a_\cu \in F^{\tess,Y}_i \cap F^{\tess(\cu)}_i$.
For each cusp $\cu$ we put
\bad q_\cu &= h^Y_{i-1} \partial_i a_\cu && \in F^{\tess,Y}_i
\\
&= a_\cu- \partial_{i+1} h^Y_i a_\cu \,.\ead
Hence
$\partial_i(q_\cu-a_\cu) =- \partial_i\partial_{i+1} 
h_i^Y a_\cu=0$,
and
\be q_\cu-a_\cu\in \partial _{i+1} F^\tess_i\,.\ee
Hence
\[ C = a_0 + \sum_\cu q_\cu + \sum_c(a_\cu-q_\cu) \equiv a_0 + \sum_\cu
q_\cu \bmod \partial_{i+1} F^\tess_i\,.\]
The sum $C_1 =a_0 + \sum_\cu q_\cu$ is an element of $F^{\tess,Y}_i$,
and satisfies $\partial_i C_1 = \partial_i C=0$. Since the space $\SY$
is contractible, we conclude that
$C_1 \in \partial_{i+1}F^{\tess,Y}_i \subset \partial_{i+1} F^\tess_i$.
\end{proof}

\subsubsection{A resolution based on cusps}\label{sect4-def-Fpb}Define
\il{Fpb}{$F^\pb_\bullet$}$F^\pb_\bullet$ by
\be F^\pb_i = \CC\bigl[ \Cu^{i+1}\bigr]\,,\ee
with boundary operators $\partial_i : F^\pb_i\rightarrow \Fpb {i-1}$
determined by
\be \partial_i
(\cu_0,\cdots, \cu_i) = \sum_{j=0}^i (-1)^j\, (\cu_0,\ldots,\hat
\cu_j,\ldots ,\cu_i)\,.\ee
($\hat \cu_j$ means that the $j$-th coordinate is omitted. )
This gives a resolution
\be \label{Fpbres} \cdots \rightarrow \Fpb 2 \stackrel
{\partial_2}\rightarrow
\Fpb 1 \stackrel {\partial_1}\rightarrow \Fpb 0 \stackrel
\e\rightarrow \CC \rightarrow 0\ee of the trivial $\CC[\Gm]$-module
$\CC$. If we replace $\Cu$ by a set with a free $\Gm$-action, for
instance $\Gm$ itself, we have a free resolution, and its cohomology
spaces are the standard group cohomology spaces.

The action of $\Gm$ on $\cu$ is not free, and the resolution
$F^\pb_\bullet$ is not free. It is acyclic, which one can check with
the homotopy $h_i: \Fpb i \rightarrow \Fpb {i+1}$ determined by
\be h_i(\cu_0,\ldots, \cu_i) = (\bu, \cu_0,\ldots,\cu_i)\,,\ee
for some fixed cusp~$\bu$.

\begin{prop}{\rm(Parabolic cohomology and tessellations)
}\label{prop-pcht}If $\tess$ is a standard tessellation or a refinement
of it then
$H^i\bigl( F^\tess_\bullet;V\bigr) \cong H^i_\pb\bigl( \Gm;V\bigr)$ for
all $i \geq 0$ and all $\CC[\Gm]$-modules $V$.
\end{prop}
A consequence is that $H^i_\pb(\Gm;V)=\{0\}$ for $i>\ds$.
\begin{proof}We will choose $\CC[\Gm]$-linear chain maps
$a_\bullet : F^\pb_\bullet \rightarrow F^\tess_\bullet $ and
$b_\bullet: F^\tess_\bullet \rightarrow F^\pb_\bullet$, and
$\CC$-linear homotopies $k_\bullet$ and $h_\bullet$ showing that
$b_\bullet \circ a_\bullet$ is homotopic to the identity in
$F^\pb_\bullet$ and that $a_\bullet \circ b_\bullet$ is homotopic to
the identity in $F^\tess_\bullet$. Explicitly, these four systems of
maps have to satisfy the following relations.
\begin{align}\label{acm}
\partial_i a_i \xi &= a_{i-1} \partial_i \xi && \text{ for }\xi \in
\Fpb i,\; i\geq 0\,,
\displaybreak[0]\\
\label{bcm}
\partial_i b_i \eta &= b_{i-1}\partial_i \eta && \text{ for }\eta \in
F^\tess_i,\; i\geq 0\,,
\displaybreak[0]\\
\label{bak}
b_i a_i \xi - \xi&= \partial_{i+1} k_i \xi + k_{i-1} \partial_i \xi
&&\text{ for }\xi \in \Fpb i,\; i\geq 0\,,
\displaybreak[0]\\
\label{abh}
a_i b_i \eta -\eta &= \partial_{i+1} h_i \eta +h_{i-1} \partial_i
\eta&& \text{ for }\eta \in F^\tess_i,\; i\geq0\,.
\end{align}

We start with $\Gm$-equivariant maps $a_0$ and $b_0$. We can make the
easy choice $a_0 \cu = \cu$ for all $\cu \in \Cu$. With the
interpretation of $\partial_0$ as the augmentation, this satisfies
\eqref{acm}, if we take $a_{-1}$ as the identity map in the trivial
$\Gm$-module $\CC$.

We choose representatives $\cu_1,\ldots , \cu_n$ of the $\Gm$-orbits of
cusps, and representatives $p_1,\ldots, p_m$ of the $\Gm$-orbits in
$\X_0^{\tess,Y}$. We take $b_0 \, \cu_r = \cu_r$ for $1\leq r\leq n$
and $b_0 \, p_s = \cu_1$ for $1\leq s \leq m$. This determines
$b_0 : F^\tess_0 \rightarrow \Fpb 0$ as a $\CC[\Gm]$-linear map, which
satisfies \eqref{bcm} for $i=0$, if we take $b_{-1}$ as the identity
map in~$\CC$.

\rmrk{Start of the induction} For $i=0$ we get from \eqref{bak} for all
$\cu\in \Cu$:
\[ 0=\cu- \cu = \partial_1 k_0 \cu + k_{-1} \partial_0 \cu= \partial_1
k_0 \cu + k_{-1}\, 1\,.\]
We use that $\partial_0$ is the augmentation, and sends all cusps $\cu $ to $ 1
\in \CC$ for all cusps $\cu$. 
We can satisfy this by taking $k_0 = k_{-1} = 0$.

Not all maps in \eqref{bak} and \eqref{abh} are necessarily
$\CC[\Gm]$-linear, and we have to let $\xi$ and $\eta$ run through all
$0$-cells. From \eqref{abh} we get for $i=0$
\begin{align*}
0 = a_0b_0 (\gm \cu_r) - \gm \cu_r &= \partial_1 h_0 (\gm c_r) +
h_{-1}1\,,\\
\gm\cu_1 - \gm p_s = a_0 b_0(\gm p_s) - \gm p_s &=\partial_1 h_0 (\gm
p_s) + h_{-1} \,1\,.
\end{align*}
Here we can take $h_{-1}=0 $ and $h_0(\cu) = 0$ for all $\cu\in \Cu$. We
take $-h_0 (\gm p_s)$ as a path in $ \ZZ\bigl[ \X^\tess_1]$ along edges
  in $X^\tess_1$ from $\gm p_s$ to $\gm \cu_1$ for
$\gm p_s \in \X^\tess_0$.

\rmrk{Induction steps} Now we assume that we have made choices so that
\eqref{acm}--\eqref{abh} are satisfied for all values of $i$ satisfying
$0 \leq i \leq j$.

Relation \eqref{acm} requires that
$\partial_{i+1} a_{i+1} \xi = a_i \partial_{i+1} \xi$.
This would imply that
\[ \partial_i a_i \partial_{i+1}\xi = \partial_i \partial_{i+1} a_{i+1}
\xi = 0 \, a_{i+1}\xi = 0\,.\]
By the acyclicity of the complex, there exist $\thv \in \Fpb {i+1}$
such that \[ \partial_{i+1}a_{i+1}\xi = a_i \partial_{i+1}\xi =
\partial_{i+1} \thv\,.\]
So we can take $ a_{i+1} \xi$ equal to this solution $\thv$. To get
$\Gm$-equivariance, we do this for each $\xi$ in a system of
representatives of $\Gm \backslash \Cu^{i+2} $, and extend this to a
definition of $a_{i+1}$. We note that this procedure is not
constructive, and needs the axiom of choice.

For $b_{i+1}$ we follow the same approach. For $i+1 >\ds$ the map
$b_{i+1}$ is automatically zero. For $i+1\leq \ds$, the sets
$\Gm\backslash\X^\tess_{i+1}$ are finite. (This illustrates an
advantage of the use of tessellations.)

We turn to the relation
\[ \partial_{i+2} k_{i+1}\xi = b_{i+1}a_{i+1}\xi - \xi - k_i
\partial_{i+1}\xi\]
for all $\xi \in \Cu^{i+2}$. We check that the right-hand side is in the
kernel of $\partial_{i+1}$. Hence we can pick a solution $k_{i+1} \xi$,
and extend it $\CC$-linearly to $\Fpb {i+1}$. For $h_{i+1}$ we proceed
similarly.
\end{proof}

\subsubsection{Proof of Proposition \ref{prop-H0d}}\label{prf-H0d}
We use a standard tessellation $\tess_\st$, and compute the parabolic
cohomology as $H^i\bigl( F^{\tess_\st}_\bullet;V)$.

For $i=0$ in \eqref{Hexpl} we note that any two points
$p_1,p_2\in \X^{\tess, Y}_0$ are connected by a path in
$\ZZ[\X^{\tess,Y}_1]$ along edges of the tessellation. If
$c\in Z^0\bigl(F^{\tess,Y}_\bullet ;V\bigr)$ then $c(p_1) = c(p_2)$ if
$p_1$ and $p_2\in X^{\tess,Y}_0$ are connected by an element of
$\X^{\tess,Y}_1$. So $c$ is a constant function on $\X^{\tess,Y}_0$.
Since $\Gm$ acts freely on $\X^{\tess,Y}_0$ the constant value is
$\Gm$-invariant. In dimension $0$ the space of coboundaries is zero. So
$H^0\bigl( F^\tess_\bullet;V\bigr) \cong V^\Gm$.

For $F^\tess_\bullet$ we can give a similar argument. The inclusion
$\X^{\tess,Y}_0 \subset \X^\tess_0$ induces a map sending a constant
function on $\X^\tess_0$ to its restriction to
$\X^{\tess,Y}_0$.\smallskip

Now let $i=\ds$. We use the generating $\ds$-cells indicated in part
\eqref{gen} of Proposition~\ref{prop-sttess}.

Let $c\in Z^\ds\bigl( F^{\tess,Y}_\bullet;V)$. It is determined by its
value on the $\ds$-cell $\fd_0$, which is the sole generator of
$\X^{\tess,Y}_\ds$. We can add a coboundary $b$ by giving an arbitrary
value on one boundary component $y_0$ of the first type, extending it
$\Gm$-equivariantly to its $\Gamma$-orbit, and by taking $ b(z)=0$ for $z$ in 
all other $\Gm$-orbits of $(\ds-1)$-cells. 
In this way we can add to
$c(\fd_0)$ an arbitrary value in $V$. This implies that
$H^\ds\bigl( F^{\tess,Y}_\bullet;V\bigr)=\{0\}$.

The cocycles in $Z^\ds\bigl( F^\tess_\bullet;V)$ are determined by their
values on $\fd_0$ and by the cells $\fd_a$, $a\in \Gm\backslash\Cu$.

For each $a\in \Gm\backslash \Cu$ the cell $\fd_a$ has boundary
components in common with $\fd_0$. Let $Y\in \X^\tess_{\ds-1}$ be one
of these components. We define a coboundary $db$, giving $b(Y)$ an
arbitrary value $\bt\in V$, extending this by
$b(\gm Y) = \tau(\gm) b(Y)$ to the orbit of $Y$, and taking $b$ equal
to zero on all other $\Gm$-orbits in $\X^\tess_{\ds-1}$. (Here we use
$\tau$ to indicate the representation of $\Gm$ in~$V$.) Since $Y$
occurs in $\partial \fd_a$ and in $\partial \fd_0$ with opposite signs,
we can move in this way the value of $c(\fd_a)$ to $\fd_0$. In this way
we can replace $c$ by a cocycle $c'$ in its cohomology class that
satisfies $c'(\fd_a)=0$ for all $a\in \Gm\backslash \Cu$. So the
cohomology class of $c'$ is determined by $c'(\fd_0)$.

We still have the freedom of adding a coboundary $db$ such that $b$ has
an arbitrary value $v$ on $Y\in B(\fd_0)$ that is not contained in a
horosphere. There is a boundary identification by some $\gm\in \Gm$
which adds $\tau(\gm) v-v$ to $c'(\fd_0)$. Since $\fd_0$ is a
fundamental domain for $\Gm$ in $\SY$, this changes $c(\fd_0)$ by a
trivial coinvariant. We may also change $c'$ by adding a coboundary
based on $(\ds-1)$-cells in the intersection of $\fd_0 \cap \fd_a$ for
some class $a$ of cusps. However, if we want to keep $c'(\fd_a)$ equal
to zero, this does not change the value $c(\fd_0)$.

Thus, we conclude that the cohomology class is determined by
$c'(\fd_0)\in V \bmod \sum_\gm \bigl( \tau(\gm)-1\bigr) V = V_\Gm$.
\qed

\subsection{Proof of Proposition~\ref{prop-pcshc}}\label{prf-pcshc}For
each $\Gm$-module $V$ we consider $ V\times \Sast$ as a constant sheaf
over $\Sast$. The quotient $\Q[V]$ determines a locally constant sheaf
on $\Gm\backslash\Sast$. We have to show that $H^j_\pb (\Gm;V) $ is
isomorphic to the sheaf cohomology space
$H^j\bigl( \Gm\backslash \Sast;\Q[V]\bigr)$.

We follow the approach of the proof of Proposition 11.8 in \cite{BLZ15},
with the references given there. The fact that here we do not use mixed
parabolic cohomology gives a simplification. For each $0$-cell $p$, we
form the open set $\Om_p$ that is the union of $\{p\}$ and of the
interiors (as defined in Definition~\ref{def-tess}
\tlref{tess:i}{tess:ia}) $\mathring X$ for all
$X \in \bigcup_{1\leq i\leq \ds} \X^\tess_i$ that contain $p$ . We use
a refinement $\tess$ of the standard tessellation $\sttess$ with the
property that for all $x\in \Om_p \cap S$ and all
$\gm \in \Gm\setminus\{e\}$ the point $\gm x$ is not in~$\Om_p$. If
$p=\cu\in \Cu$ then $\Om_\cu$ is contained in the union of the open
horoball $B^\ast_\cu(Y) \setminus H_\cu(Y)$, which does not contain
other cusps.

Let $\pi:\Sast \rightarrow \Gm\backslash\Sast$ be the natural
projection. The set
$\mathfrak{A} = \bigl\{ \pi \Om_p \;:\; p\in \X^\tess_0\bigr\}$ is a
finite covering of $\Gm\backslash \Sast$. Proceeding as in \cite{BLZ15}
this gives an isomorphism with \u{C}ech cohomology spaces:
\[ H^i_\pb\bigl( \Gm;V\bigr) \cong \check H^i\bigl(\mathfrak{A};
\Q[V]\bigr)\,.\]
See, eg, \cite[Chap III, \S4]{Harts} for \u{C}ech cohomology. Leray's
theorem \cite[Exerc. 4.11]{Harts} gives
\[H^i_\pb(\mathfrak{A};\Q[V]) = \check H^i(\Gm\backslash\Sast; \Q[V])\]
if $H^k(U;\Q[V]|_U)=\{0\}$ for all $k\geq 1$ for all intersections $U$
of elements of $\mathfrak{A}$.

If $U $ is in $\pi S$, the sheaf $\Q[V]$ is the constant sheaf $V$
on~$U$, and $H^k(U, \Q[V])=0$ for all $k\geq 1$.

Finally, suppose that $U$ contains the image $\pi \cu$ of a cusp. Then
$\Q[V]$ restricted to $U\setminus\{\cu\}$ is the constant sheaf $V$,
and we have the short exact sequence
\be\label{exs}
0 \rightarrow \Q[V]|_U \rightarrow V \rightarrow j_\ast(V / V^{\Gmm\cu})
\rightarrow 0\,, \ee
where $j: \{\cu\}\rightarrow U$ is the natural injection, and $j_\ast$
is the direct image functor determined by $j$, defined on \cite[p
65]{Harts}. To see this observe that the stalks at
$u\in U\setminus\{\cu\}$ form the sequence
\[ 0 \rightarrow V \rightarrow V \rightarrow 0 \rightarrow 0 \,,\]
and the stalk at $p$ gives
\[ 0 \rightarrow V^{\Gmm \cu} \rightarrow V \rightarrow V/V^{\Gmm \cu}
\rightarrow 0\,.\]
The sheaf $j_\ast( V/V^{\Gm\cu})$ is a skyscraper sheaf with support
$\{\cu\}$. To \eqref{exs} corresponds a long exact sequence that starts
with
\[\xymatrix{0 \ar[r]& H^0(U;(\Q[V])|_U ) \ar[r]\ar[d]^=& H^0(U;V)
\ar[r]\ar[d]^=& H^0(U;j_\ast(V/V^{\Gmm \cu}) )\ar[r]\ar[d]^{\cong} &
\cdots\\
0 \ar[r] & V^{\Gmm \cu}\ar[r] &V \ar[r] &V/V^{\Gmm\cu}\ar[r]& 0} \]
So the next part starts with
\[\xymatrix{ 0 \ar[r]& H^1(U;(\Q[V])|_U ) \ar[r] & H^1(U;V) \ar[r]
\ar[d]^=& \cdots\\
&&0 }\]
The constant sheaves have zero cohomology spaces in dimension $1$ and
larger. We see that $H^1(U;(\Q[V])|_U )=0$.

For $i\geq 2$
\[\xymatrix{H^{i-1}\bigl(U;j_\ast(V/V^{\Gmm \cu}) \bigr)\ar[r] \ar[d]^=&
H^i(U;
(\Q[V])|_U ) \ar[r]& H^i(U;V) \ar[r] \ar[d]^=&
\\
0&&0 }\]
In dimensions $i\geq 1$ the cohomology spaces of skyscraper sheaves are
zero. We conclude that $H^i(U ;(\Q[V])|_U)=0$ for all $i \geq 1$. This
gives the last ingredient to get the desired isomorphism.\qed

%% file: ccro2-5-ps.tex

\bigskip

\def\flnm{ccro2-5-ps}

\section{Spherical principal series}\label{sect5}
We will use the spherical principal series representations of $G$. In
particular, we need the realization in functions on $\partial S$
discussed in \S\ref{sect5-prs-bd}.

\subsection{The spherical principal series of \intitle{G}} The
\il{ps}{principal series}{\em principal series} $H^{\xi,\nu\al }_0$
depends on a spectral parameter $\nu \in \CC$ and a representation
$\xi$ of~$M$
(the centralizer of $A$ in $K$)
in a finite-dimensional vector space $V$. The space $H^{\xi,\nu\al}_0$
consists of the continuous functions $f:G\rightarrow V$ that satisfy
\be\label{fps} f\bigl( nm\am(t) g ) = t^{\rho+\nu}\, \xi(m) \,
f(g)\qquad(n\in N,\; t>0,\; m\in M)
\,,\ee
and becomes a representation of $G$ with the action $R$ by right
translation. In $t^{\rho+\nu} =\am(t)^{(\rho+\nu) \al} $ we tacitly
identify the number $\rho+\nu \in \CC$ with
$(\rho+\nu)\al \in \CC \al $, the space of linear forms on $\alie$.

In this paper we use the spherical principal series representations
$H^{1,\nu \al}_0$, in which $\xi$ is the trivial representation of $M$
in $V=\CC$. We denote
\il{Hnu}{$H^\nu_0 = H^{1,\nu\al}_0$}$H^\nu_0 = H^{1,\nu\al}_0$. One may
impose more strict regularity conditions on the functions
$f \in H^\nu_0$. We will use $H^\nu_\infty \subset C^\infty(G)$,
$H^\nu_\om \subset C^\om(G)$
(the space of (real-)analytic functions).

The space \il{Hnuio0K}{$ H^\nu_\infty,\; H^\nu_\om,\; H^\nu_K$}$H^\nu_K$
consists of $K$-finite functions. For each $f\in H^\nu_K$ the elements
$R(k) f$ with $k\in K$ lie in a finite-dimensional subspace. The Lie
algebra $\glie$ acts in $H^\nu_K$ by right differentiation. The
restriction of this action to $\klie$ corresponds to an action of $K$
by right translation.

The Casimir element $\Cas$ in the center of the universal enveloping
algebra $U(\glie)$ was discussed in \S\ref{sect2-Cas}.
\begin{prop}\label{prop-eivCasps}
The operator $R(\Cas)$ acts in $H^\nu_\infty $ as multiplication by the
factor $\nu^2-\rho^2$. \end{prop}
This parametrization, by the spectral parameter $\nu$, explains the use
of $\rho^2-\nu^2$ as eigenvalue of the Laplace operator $\Dt$,
corresponding to $-R(\Cas)$, in Definition \ref{def-af} of automorphic
forms.

See \S\ref{sect5-prfs} for a proof of this proposition.\smallskip

\begin{prop}[Kostant, {\cite[Theorem 2]{Kost69}}] \label{prop-irr-ps}
$H^\nu_K$ is an irreducible $(\glie,K)$-mod\-ule unless
\be
\begin{cases}
\nu+ \rho\in\ZZ \quad\text{ and }\quad \nu \not\in (-\rho,\rho)&\text{
if }q=0\,,\\
\nu + \rho \in 2\ZZ \quad\text{ and }\quad \nu \not \in \bigl( -\tfrac
p2-1,\tfrac p2+1\bigr)
&\text{ if }q\geq 1\,.
\end{cases}
\ee
\end{prop}

We recall that $p$ and $q$ are the dimensions of $\glie_\al$ and
$\glie_{2\al}$, and that $\rho=\frac12p+q$; see~\S\ref{subs2-Lie-alg}.
In the region $i\RR \cup [-\rho,\rho]$ in which spectral parameters
$\nu$ of cusp forms can occur, the endpoints $\nu = \pm \rho$
correspond to reducibility of $H^\nu_K$ if $q=0$ or~$1$. If $q\geq 2$,
then reducibility occurs for $\pm \nu= \rho-2\ell$ with $\ell\in \ZZ$
and $0\leq 2\ell \leq q-1$.

Table~\ref{tab-irr-ps} may be helpful to compare the notation in
Kostant's statement with our notation.
\begin{table}[ht]\renewcommand\arraystretch{1.2}
\begin{center}\begin{tabular}{|c|c|c|}\hline
Kostant \cite[\S2]{Kost69}&\multicolumn{2}{|c|}{here}\\ \hline
$\Ld_+^1$ & \multicolumn{2}{|c|}{$\{\al\}$}\smallskip \\
$w_\al = w_{-\al}$ & \multicolumn{2}{|c|}{$\al$}\\
$\ld-\rho=\langle \ld-\rho, w_\al\rangle$ &
\multicolumn{2}{|c|}{$\nu$}\\
$m_\al $& \multicolumn{2}{|c|}{$\rho=\frac p2+q$}\\ \hline
& $q=0$ & $ q\geq 1$ \\ \hline
$n_\al $& $1$ & $2$ \\
$T_\al$&$ \bigl( -\tfrac p2,\tfrac p2 \bigr) $&
$ \bigl( -\tfrac p2-1,\tfrac p2 +1\bigr) $
\\ \hline
\end{tabular}\end{center}
\caption{Corresponding notations.}\label{tab-irr-ps}
\end{table}

\subsection{The realization on \intitle{\partial S}}\label{sect5-prs-bd}
Functions on $G$ that satisfy \eqref{fps} are determined by their values
on $K$, by the Iwasawa decomposition. In the spherical principal series
they satisfy $f(m k) = f(k)$ for $m\in M$, hence they are functions on
$M\backslash K$. So $k\mapsto f(k^{-1})$ is a function on
$K/M = \partial S$, and the transformation $\Ps_\nu f (k) = f(k^{-1})$
gives a linear bijection from the space $H^\nu_\infty$ to
$C^\infty(\partial S)$. Moreover,
$\Ps_\nu H^\nu_\om = C^\om(\partial S)$.

\rmrk{Functions on \intitle{\partial S}} When dealing with functions on
$\partial S$ we often use that $\partial S \cong K/M \cong G/MAN$.
We use $b$ as a variable on $\partial S$, and write
$f(b) = f(k\,\inft) = f(g\,\inft)$ if $k$ represents $b$ in $K/M$ or if
$g$ represents $b$ in $G/MAN$. Then
$f(g\,\inft)= f\bigl( \kI(g)\,\inft)$. By $db$ we denote the invariant
measure on $\partial S$ such that
$\int_{b\in \partial S}f(b)\, db = \int_{k\in K} f(k\, \inft)\, dk$. We
normalize the Haar measure $dk$ so that $K$ has volume $1$.
\rmrk{Action of \intitle{G}} We put for $g\in G$ and
$k\in K$\il{J}{$J(g,k)$ factor of automorphy on $\partial S$}
\be \label{Jdef}J(g,k\,\inft) =\tI(gk) \in \RR_{>0}^\ast\,,\ee
in the notation of Proposition~\ref{prop-Id}. To see that this is
well-defined we note that $m\in M$ normalizes $N$ and centralizes $A$,
and hence
\[ gk m = \kI(gk) \, \am\bigl( \tI(gk) \bigr) \, \nI(g k ) m \in K m
\am\bigl(\tI(gk) \bigr) N = K \am\bigl(\tI(gk) \bigr) N\,.\]
\begin{prop}\label{prop-autfct} $J(g,b)$ is analytic on $G \times(K/M)$,
and satisfies
\be\label{Jrel} J(g h, b) = J(g ,h\, b) J(h,b)\qquad (g,h\in G,\;
b\in\partial
S)\,.\ee
\end{prop}\begin{proof}Let $ b  =k\, \inft$ with some $k\in K$. Then for
$g\in G$ we have
\[ gk = \kI(gk) \am\bigl( \tI(gk) \bigr) \nI(gk)\in KAN\,,\]
and $\tI(gk) =J(g,k \, \inft)$. This shows, together with
Proposition~\ref{prop-Id}, that $J( g  ,k)$ is analytic in $( g ,k)$.

The action of $G$ on $K/M$ is given by $g \, (k\,\inft) = \kI(gk)$. We
apply the opposite form of the Iwasawa decomposition first to $hk$, and
then to $g\kI(hk)$:
\begin{align*}
g hk &\in g \kI(hk) \am(\tI(hk)) N = \kI\bigl( g\kI(hk) \bigr)
\am\bigl(\tI(g\kI(hk) \bigr) \nI\bigl( g\kI(hk) \bigr) \; \am(\tI(hk))
N \displaybreak[0]
\\
&\subset \kI\bigl( g\kI(hk) \bigr)
\am\bigl(\tI(g\kI(hk) \bigr)\am(\tI(hk))N \subset K \; \am\bigl(\tI(g\kI(hk) \bigr) \; \am(\tI(hk)) N\displaybreak[0]\\
& = K \,
\am\bigl( \tI(g\kI(hk) \, \tI(hk) \bigr) \, N\,,
\end{align*}
 where we use that $A$ normalizes $N$, and that $t\mapsto \am(t)$ is a group
homomorphism $\RR^\ast_{>0}\rightarrow A$. This implies that
$\am\bigl( \tI(g\kI(hk) \, \tI(hk) \bigr) $ is the component in $A$ in
the opposite Iwasawa decomposition of $ghk$, and hence that
\[ \tI(ghk) = \tI\bigl( g\kI(hk) \bigr)\, \tI( hk) \,.\]
This gives \eqref{Jrel}.
\end{proof}

Property \eqref{Jrel} shows that $J$ is a factor of automorphy, and,
since $J$ has positive values, powers $J^z$ with $z\in \CC$ are factors
of automorphy as well. We put for functions $f$ on
$\partial S$\il{taunu}{$\tau_\nu$ action of $G$ on $\V\om\nu$}
\be \label{taunu}
\tau_\nu(g) f:b\mapsto J(g^{-1},b)^{-(\rho+\nu)}\,f(g^{-1}\,
b)\qquad(g\in G,\; b\in \partial S)\,.\ee
Then $\tau_\nu$ is a representation of $G$ in the functions on
$\partial S$.

We have thus obtained the realization of the spherical principal series
in functions on $\partial S$:
\begin{prop}\label{prop-iso-VH}The operators $\tau_\nu(g)$ in
\eqref{taunu} define a representation $\tau_\nu$ of $G$ in each of the
spaces of (real-)analytic functions, and smooth functions on
$\partial S$. The resulting representations $\V \om \nu(\partial S)$
and $\V\infty\nu(\partial S)$
are isomorphic to $H^\nu_\om$ and $H^\nu_\infty$. The isomorphism is
given by the bijective linear operator \il{Psnu}{$\Ps_\nu$}$\Ps_\nu$
that associates to a function $f$ on $K/M$ the function $\Ps_\nu f$ on
$G$ given by
\be \label{iso-VH} (\Ps_\nu f)(n \am(t)k) = t^{\rho+\nu} \,
f(k^{-1})\,.\ee
\end{prop}\il{VS}{$\V\om\nu(\partial S),\;\V \infty\nu(\partial S)$
\text{princ. series on } $\partial S$}

\rmrk{A \intitle{G}-equivariant sheaf}If $f$ is an analytic function on
an open subset $I \subset \partial S$, then the formula in
\eqref{taunu} gives the analytic function $\tau_\nu(g)f$ on $ gI$.

\begin{defn}\label{def-Vomsh}
We define \il{Vomsh}{$\V\om\nu$} $\V\om\nu$ as a sheaf on $\partial S$
by defining for open $I \subset \partial S$
the linear space $\V\om \nu(I)$ of analytic functions on~$I$. We make it
into a $G$-equivariant sheaf by the operators
$\tau_\nu(g) : \V \om\nu(I) \stackrel 
{\cong}{\longrightarrow} \V\om\nu( g\, I)$ for $g\in G$. \end{defn}

\subsection{Proof of Proposition~\ref{prop-eivCasps}}\label{sect5-prfs}
We determine the eigenvalue of $R(\Cas)$ in $H^\nu_\infty$. The Casimir
element $\Cas$ is in the center of the universal enveloping algebra
$U(\glie)$. As a differential operator on $G$ it commutes with left and
right translations. The representation $H^\nu_\infty$ is irreducible
for a dense set of values of $\nu$ in $\CC$. So for general values of
$\nu$ the action of $R(\Cas)$ in $H^{\nu}_\infty$ is given by an
eigenvalue $\ld_{\nu}$. Actually, this extends to all $\nu$, since the
spaces $H^\nu_\infty $ form a holomorphic family depending on $\nu$. To
compute the eigenvalue $\ld_{\nu}$ we consider $R(\Cas) f$ for a
function $f$ as in~\eqref{fps} with $f(e)\neq 0$. Then
$R(\Cas f)( e) = \ld_{\nu} f(e)$.

Since $\exp( \XX) \in N$ for all $\XX\in \nlie$ we have
\be R(\XX_i) f(e) = \partial_x f\bigl( \exp (x\XX_i) \bigr)
\bigm|_{x=0} = 0\,.\ee
In a similar way, we have $R(\UU_j) f(e)=0$. Using this and
\eqref{bchXX} in \S\ref{sect2-Kfbg} we obtain at the unit element
$e\in G$
\begin{align*} \bigl( R(\Cas) f \bigr)(e)
&= R \Bigl( \H_0^2 - \sum_i \UU_i^2-\sum_i 2 (\tht \XX_i) \XX_i +
(p+2q) \H_0 \Bigr) f (e)\\
&=R(\H_0)^2 f(e) -2\rho \,R(\H_0) f(e)
\,.
\end{align*}
We use that
$f\bigl( \exp(x \H_0) \bigr) = f\bigl(\am(e^x) \bigr) = e^{x(\rho+
\nu)}\, f(e)$, and hence
\begin{align*} R(\H_0) f (e) &= \partial_x f\bigl( \exp(x\H_0)\bigr)
\bigm|_{x=0} = \partial_x e^{x(\rho+\nu)} f(e) \bigm|_{x=0}\\
&= (\nu+\rho)\, f(e)\end{align*}
to get
\begin{align*}
R(\H_0)^2 f(e) &-2\rho \, R(\H_0) f(e)
= (\nu+\rho)^2 \,f(e) -2 \rho(\nu+\rho) \, f(e) \\
&= (\nu^2-\rho^2)
\, f(e)\,,
\end{align*}
and
\[\ld_{\nu} =\nu^2-\rho^2\,. \qedhere\]

%% file: ccro2-6-cf-coh.tex

\bigskip

\def\flnm{ccro2-6-cf-coh}

\section{From cusp forms to cocycles}\label{sect6}
The preparations in the previous sections are sufficient to construct
from a given cusp form a cocycle with values in the spherical principal
series. See Theorem~\ref{thm-btv}.

We now generalize the Green form in \cite[\S1.3]{BLZ15}: We define a
kernel function on $\left(\partial S\right)\times G$, a generalization
of the Green form, and the definition of a $(\ds-1)$-form on $S$ with
values in a spherical principal series representation. Integrating this
differential form over $(\ds-1)$-cells in a tessellation leads to
$(\ds-1)$-cocycles that describe the map from cusp forms to cohomology
classes.

\subsection{Kernel function}\label{sect6-kf} One of the tools in
\cite{BLZ15} to go from automorphic forms to cocycles is the kernel
function $R(t;z)^s$ in \cite[(1.6)]{BLZ15}. It relates functions of
$t\in \proj\RR$ (the boundary of the symmetric space) to functions of
$z\in \uhp$ (the symmetric space).

We start with the function $\Ph_\nu\in H^\nu_\om$ given, in the notation
in \eqref{am-char}, by\il{Phinu}{$\Ph_\nu \in H^\nu$}
\be\label{Phinu} \Ph_\nu(n \am(t) k ) = \am(t)^{(\rho+\nu)\al} =
t^{\rho+\nu}
\qquad(n\in N,\; k\in K,\; t>0)\,.\ee
It behaves according to a character of $NAM$ on the left, and is
invariant under $K$ on the right.

In \S\ref{sect5-prs-bd} we realized the principal series on the boundary
$\partial S$. The function $\Ph_\nu$ corresponds to the constant
function $k\,\inft \mapsto \Ph_\nu(k^{-1}) =1$ in this realization.
Proposition~\ref{prop-eivCasps} and equation~\eqref{Dt} imply that
$\Dt \Ph_\nu = (\rho^2-\nu^2) \Ph_\nu$.

The kernel function $r_\nu$ is defined by \il{rnu}{$r_\nu(\xi;x)$}
\be \label{rnudef}r_\nu(k\,\inft;g\, \orgn) = \Ph_\nu(k ^{-1}g)
\qquad\text{ for } k\in K,\;g\in G\,.\ee
To see that this is well-defined, we check that for $m\in M$ and
$k_1\in K$
\begin{align*}
m^{-1} k^{-1} g k_1 &= m^{-1} \nJ(k^{-1}g) \am\bigl( \tJ(k^{-1} g)\bigr)
\kJ(g^{-1}k)k_1 \in m^{-1}N m \, \am\bigl( \tJ(k^{-1} g)\bigr) \,m^{-1}
\, K
\\
&= N\,\am\bigl( \tJ(k^{-1} g)\bigr)
\, K\,.
\end{align*}
\begin{prop}\label{prop-Rnu}Let $\nu\in \CC$.
\begin{enumerate}[label=$\mathrm{(\roman*)}$, ref=$\mathrm{\roman*}$]
\item\label{Rnu:ana} The function $r_\nu(\cdot;\cdot)$ is a well-defined
analytic function on $(\partial S) \times S$.

\item\label{Rnu:eiggrow} Let $b\in \partial S$ and $g\in G$.
\begin{enumerate}[label=$\mathrm{(\alph*)}$, ref=$\mathrm{\alph*}$]
\item\label{Rnu:one} $r_\nu(\cdot ;g\,\orgn)= \tau_\nu(g)\, 1$
as an equality of functions in $\V\om\nu(\partial S)$.
\item\label{Rnu:eigen} The function $r_\nu(b;\cdot)$ on $S$ is an
   eigenfunction of $\Dt$ with eigenvalue $\rho^2-\nu^2$.
\item\label{Rnu:growth} The function $r_\nu(b;\cdot)$ has polynomial
   growth at the boundary.
\end{enumerate}

\item\label{Rnu:equi} For each $g\in  G $
\be\bigl( \tau_\nu(g) r_\nu(\cdot;x) \bigr)(b) = \bigl( L( g^{-1})
r_\nu(b;\cdot)\bigr) (x)\,, \ee
which is equivalent to
\be \label{Rgg}
\bigl(\tau_\nu(g)\times L(g) \bigr) r_\nu(\cdot;\cdot)
= r_\nu(\cdot;\cdot)\,. \ee
\end{enumerate}
\end{prop}
\begin{proof}
Since $\Ph_\nu$ is analytic, part~\eqref{Rnu:ana} follows from
\eqref{rnudef}.

The invariance of $\Dt$ under left translation implies
\tlref{Rnu:eiggrow}{Rnu:eigen}. The function $\Ph_\nu$ has polynomial
growth at the boundary. This is preserved under left translation. This
gives \tlref{Rnu:eiggrow}{Rnu:growth}.

The kernel $r_\nu$ on $(\partial S) \times S$ given in \eqref{rnudef}
satisfies $\bigl(L(k) \times L(k)\bigr) r_\nu=r_\nu$ for $k\in K$. Let
$b=k_b \,\inft$, $k_b\in K$ and $x=h_x\, \orgn$, $h_x\in G$. For
$g\in G$
\begin{align*}
\Bigl(\bigl(1\times L(g^{-1})\bigr)& r_\nu \Bigr)(b;x)
= \Ph_\nu\bigl( k_b^{-1} g h_x\bigr)
= \Ph_\nu \bigl( \nJ(k_b^{-1} g) \am\bigl( \tJ(k_b^{-1}g) \bigr)
\kJ(k_b^{-1}g) h_x\bigr)
\displaybreak[0]\\
&= \tJ\bigl( k_b^{-1} g \bigr)^{\rho+\nu} \, \Ph_\nu\bigl(\kJ(k_b^{-1}g)
h_x\bigr)
= \tI(g^{-1} k_b)^{-\rho-\nu} \,\Ph_\nu \bigl( \kI(g^{-1} k_b) h_x
\bigr)
\displaybreak[0]\\
&= J(g^{-1}, b)^{-\rho-\nu}\, r_\nu\bigl(g^{-1} \, b; x\bigr)=\bigl(
\tau_\nu(g)
\times 1) r_\nu(b;x)\,.
\end{align*}
This gives \eqref{Rnu:equi}.

For \tlref{Rnu:eiggrow}{Rnu:one} we note that
$r_\nu(b;\orgn) = \Ph_\nu(k_b^{-1})=1$. Hence
\[ r_\nu( b;g\,\orgn) =\bigl(L(g^{-1}) r_\nu(b;\cdot)\bigr)(\orgn)
= \bigl( \tau_\nu(g) r_\nu( \cdot'\orgn) \bigr)(b)
= \bigl( \tau_\nu(g) 1 \bigr)(b)\,. \qedhere \]
\end{proof}

\subsection{Generalization of Green's form} We now generalize the
Green's form in \cite[\S1.3]{BLZ15}.

We define for $f_1,f_2\in C^\infty(S)$ the antisymmetric bilinear form
in $f_1$ and $f_2$ with values in $C^\infty(S)$\il{MS}{$\MS(f_1,f_2)$}
\be \label{MS}\MS(f_1,f_2) = f_1 \, \Dt f_2 - f_2\, \Dt f_1\,. \ee
This bilinear form is the starting point in the derivation of the
Maass--Selberg relations. This works for more general Lie groups; see
for instance \cite[Chap IV, \S2]{HCh68}.

The $G$-invariant measure $d\mu$ on $S$ is obtained by integration of
the invariant volume form
\[ \z_S = \sqrt{\det \rmetric(x)}\, dx_1\wedge\cdots\wedge dx_\ds\]
in any choice of coordinates on $S$. The matrix $\rmetric$ describes the
Riemannian metric; see the discussion on \S\ref{sect2-Rie}. In terms of
horospherical coordinates, $\rmetric$ has the description
in~\eqref{Rmat}. Then $\MS(f_1,f_2)\, \z_S$ is a $\ds$-form on $S$ that
satisfies
$\MS\bigl( L(h^{-1}) f_1,L(h^{-1}) f_2\bigr)\z_S = \MS(f_1,f_2)\z_S \circ h$
for all $h\in G$.

The $\ds$-form $\MS(f_1,f_2) \, \z_S$ is closed (since $\ds=\dim S$),
and hence there are $(\ds-1)$-forms $\eta$ such that
$d\eta = \MS(f_1,f_2)
\z_S$. We can choose $\eta$ in the following way.
\begin{prop}\label{prop-eta}
Let $U\subset S$ be open. For $f_1,f_2\in C^\infty(U)$ there is a
$(\ds-1)$-form $\eta(f_1,f_2)$ on $U$ such that
\begin{enumerate}[label=$\mathrm{(\alph*)}$, ref=$\mathrm{\alph*}$]
\item\label{eta:diff} $d\eta(f_1,f_2) = \MS(f_1,f_2)\,\z_S$,
\item\label{eta:anti} $\eta(f_1,f_2)$ is skew symmetric bilinear in
$(f_1,f_2)$,
\item\label{eta:equi} $\eta\bigl( L(h^{-1}) f_1,
L(h^{-1}) f_2\bigr) =\eta(f_1,f_2)\circ h$ for all $h\in G$.
\item\label{eta:closed} If $\Dt f_1 = \ld f_1$ and $\Dt f_2=\ld f_2$ for
the same eigenvalue $\ld\in \CC$, then $\eta(f_1,f_2)$ is a closed
$(\ds-1)$-form.
\end{enumerate}
\end{prop}
\begin{proof}We recall the Riemannian structure on $S$, which for a
given choice of coordinates $x=(x_1,\ldots, x_\ds)$ is described by a
$\ds\times \ds$-matrix $\rmetric(x)$, as discussed
in~\S\ref{sect2-Rie}.

We check first, with \eqref{divgrad}, that for $f \in C^\infty(S)$ and
for any $C^\infty$ vector field $X$ on $S$, we have
\be \div(f X) = f \, \div X + X f\,,\ee
and, further, that for $f_1,f_2\in C^\infty(S)$
\be (\grad f_1)(f_2) = \sum_{i,j}( \rmetric(x)^{-1})_{i,j}
(\partial_{x_i}f_1)\,(\partial_{x_j} f_2)\ee
satisfies $(\grad f_1)(f_2)=(\grad f_2)(f_1)$. With
$-\Dt F = \div\, \grad F$ we get
\be \label{MS-dg}\MS(f_1,f_2) =\div \Bigl( f_2\,\grad f_1 - f_1 \,\grad
f_2\Bigr)\,.\ee

We note that $f_1\,\grad f_2-f_2\,\grad f_1$ is a vector field on $S$.
We can write it in the form $\sum_j a_j \partial_{x_j}$ with
$a_j \in C^\infty(S)$. The divergence looks like
\be \div \sum_j a_j \partial_{x_j} = (\det \rmetric (x))^{-1/2} \sum_j
\partial_{x_j}\bigl( (\det \rmetric(x))^{1/2} a_j\bigr)\,.\ee

With the differential forms of degree $\ds-1$:\il{upsj}{$\z[j]$}
\be \z[j] = dx_1 \wedge dx_2\wedge \cdots \wedge \widehat{dx_j} \wedge
\cdots
\wedge dx_\ds\ee
($dx_j$ omitted), we form for any vector
$b=(b_1,\ldots, b_\ds) \in C^\infty(S)^\ds$
\be \eta_b = \sum_j (-1)^{j+1} \, (\det \rmetric(x))^{1/2} b_j \,
\z[j]\,.\ee
The exterior derivative is
\bad d\eta_b &= \sum_j (-1)^{j+1}\, \partial_{x_j} \bigl((\det
\rmetric)^{1/2} b_j \bigr) \, dx_j\wedge \z[j]\\
&= \sum_j \partial_{x_j} \bigl((\det \rmetric(x))^{1/2} b_j \bigr)
\,dx_1\wedge \cdots \wedge dx_\ds\\
&= (\det \rmetric(x))^{-1/2} \sum_j \partial_{x_j} \bigl((\det
\rmetric(x))^{1/2} b_j \bigr) \, \z_S \,.
\ead

Taking $b$ with components
$b_j = \sum_i ( \rmetric(x)^{-1})_{i,j} \Bigl( f_2 \partial_{x_i} f_1- f_1
\partial_{x_i} f_2 \bigr)$ we get
\begin{align*}
\MS(f_1,f_2)\,\z_S &= (\det \rmetric(x))^{-1/2} \,\sum_j \partial_{x_j}
\bigl(
(\det \rmetric(x))^{1/2} b_j \bigr) \, \z_S
\displaybreak[0]\\
&= d \Bigl( \sum_j (-1)^{j+1} (\det \rmetric(x))^{1/2}\, b_j \z[j]
\Bigr)\,.
\end{align*}
So we put\il{eta}{$\eta(f_1,f_2)$}
\be\label{etadef} \eta (f_1,f_2) = \sum_{i,j} (-1)^{1+j}
( \rmetric(x)^{-1})_{i,j} \Bigl( f_2 \,\partial_{x_i} f_1- f_1
\,\partial_{x_i}
f_2\Bigr) (\det \rmetric(x))^{-1/2} \z[j]\,.\ee
This satisfies assertions~\eqref{eta:diff} and~\eqref{eta:anti}, and is
independent of the choice of the coordinates.\medskip

Let $h\in G$. Since $h$ preserves the Riemannian metric, taking the
gradient commutes with left translation of functions on $S$:
$\grad\bigl( L(h^{-1}) f\bigr) = (\grad f)\circ h$. Furthermore
$\z_S \circ h = \z_S$. Hence
$\eta(f_1,f_2) \circ h = \eta\bigl( L(h^{-1}) f_1, L(h^{-1}) f_2\bigr)$.
This gives part~\eqref{eta:equi}. If $f_1$ and $f_2$ are eigenfunctions
of $\Dt$ with the same eigenvalue, then $\MS(f_1,f_2)=0$. This gives
part~\eqref{eta:closed}.
\end{proof}

\subsection{Differential forms with values in the principal series}

We combine the kernel function $r_\nu$ in \eqref{rnudef} and the
differential form $\eta$ in \eqref{etadef} in\il{omnu}{$\omv_\nu$}
\be\label{omnu} \omv_\nu(f; b, \cdot) = \eta\bigl(
f(\cdot),r_\nu(b;\cdot) \bigr)\,,\ee
for $f\in C^\infty(S)$, $\nu\in \CC$, $b\in \partial S$.
The dot in $f(\cdot)$ and $r_\nu(b;\cdot) $ refers to the variable in
$S$ used in \eqref{etadef} to get a $(\ds-1)$-form on $S$. This
differential form depends on an additional variable $b$ in $
\partial S$.
The coefficient function of each $\z[j]$ in this differential form is a
$C^\infty$-function on $(\partial S)\times S$, which is analytic in the
variable~$\xi$.

Now we assume that $f$ satisfies $\Dt f = ( \rho^2-\nu^2) \, f$.
Proposition~\ref{prop-eta} \eqref{eta:closed} shows that $\omv_\nu$ is
a closed differential form. For $h\in G$, $b\in \partial S$ and
$x\in S$:
\begin{align*}
\bigl(\omv_\nu(f&; b,\cdot) \circ h^{-1}\bigr) (x) = \eta \bigl( L(h)
f(\cdot),
L(h)\,r_\nu(b;\cdot\bigr) (x)
\bigr)&&\text{Prop.~\ref{prop-eta}\eqref{eta:equi}}
\displaybreak[0]\\
&= \eta\bigl( L(h) f(x) , \bigl((\tau_\nu(h^{-1})r_\nu(\cdot;x)
\bigr)(b) \bigr)&& \text{Prop~\ref{prop-Rnu}\eqref{Rnu:equi}}
\displaybreak[0]\\
&= J(h,b)^{-\rho-\nu} \eta\bigl( L(h) f(x), r_\nu(h\, b;x) \bigr)
&&\text{\eqref{taunu}}\displaybreak[0]\\
&= \bigl( \tau_\nu(h^{-1} ) \omv_\nu(L(h)f; \cdot,x ) \bigr) (b)\,.
\end{align*}
This gives the following result:
\begin{prop}\label{prop-omtrf}If $f\in C^\infty(S)$ satisfies
$\Dt f = (\rho^2-\nobreak\nu^2) f$, then the
$\V\om\nu(\partial S)$-valued differential form $\omv_\nu(f;b,x)$ is
closed, and satisfies for $h\in G$
\be \label{omv-trf}
\omv_\nu( f;b,\cdot)\circ h^{-1} \;(x) = \bigl(\tau_\nu(h^{-1})\,
\omv_\nu\bigl(L(h) f; \cdot,x\bigr)\bigr) (b)\,. \ee
\end{prop}

\subsection{The linear map from \intitle{ \A^0(\Gm;\nu) } to a
cohomology space}\label{sect6-linAtocoh}

We will integrate the differential form $\omv_\nu(f;b,x)$ in
\eqref{omnu} along $(\ds-1)$-cycles of the standard tessellation
$\tess$ to obtain cocycles. The interpretation of the map to cohomology
leads us to define spaces of semi-analytic vectors in the spherical
principal series.

\begin{prop}\label{prop-omfint}
Let $\nu\in \CC$, $f\in \A^0(\Gm;\nu)$, $b \in \partial S$,
$C\in \X^\tess_{\ds-1}$ and $X\in \X^\tess_\ds$.
\begin{enumerate}[label=$\mathrm{(\roman*)}$, ref=$\mathrm{\roman*}$]
\item \label{omfint:i}The integral
\[ I(f,C;b) = \int_{x\in C} \omv_\nu(f;b,x)
\]
is absolutely convergent. For each $\gm\in\Gm$ it satisfies
\be\label{int-bd0} \int_{x\in \gm^{-1} C} \omv_\nu(f;\cdot,x)=
\tau_\nu(\gm^{-1})
\int_{x\in C}\omv_\nu(f;\cdot,x)
\,. \ee
\item\label{omfint:nii} The map $I(f,C;\cdot)\colon \partial S\to\CC$ is
in
$\V \infty \nu(\partial S) \cap \V \om\nu  (\partial S\setminus (C\cap \Cu))$.
In particular, if $C$ does not contain any cusp,
$I(f,C;\cdot)\in \V \om \nu (\partial S)$.
\item \label{omfint:niii}
The linear extension of $I(f,\cdot;b)$ to the boundary
$\partial X = \sum_{B\in B(X)}B$ vanishes; i.e.,
\[ \sum_{B\in B(X)} I(f,B;b) = \sum_{B\in B(X)} \int_{x\in B}
\omv_\nu(f;b,x) = 0\,. \]
\end{enumerate}
\end{prop}

Regarding Proposition~\ref{prop-omfint} \eqref{omfint:nii}, we remark
that $C\cap \Cu$ contains at most one cusp. Thus, $I(f,C;\cdot)$ is
analytic on $\partial S$ except for at most one point.

\begin{proof}We use the description of tessellations in
Section~\ref{sect4}.

The integral $I(f,C;\xi) $ has an analytic integrand, since the cusp
form $f$ and the kernel function $r_\nu$ are analytic. So integration
over a compact $(\ds-1)$-cell $C$ is no problem, and the result is an
analytic function of the variable $\xi\in \partial S$.

We recall that $A^+_T=\bigl\{ \am(t) \;:\; t\geq T\}$, and that in
Definition~\ref{def-SY} we fixed the number $Y\geq 4$. If $C$ contains
a cusp, then $C\setminus\{ \cu\}$ has the form
$C=g_\cu D A^+_T \,\orgn$, where $D$ is a compact contractible subset
of $N$ and $T\geq Y$. Here we use that $\tess$ satisfies the condition
in Remark~\ref{refine-cond}. In view of \eqref{Rgg}, and part
\eqref{eta:equi} in Proposition~\ref{prop-eta} we can transform the
cusp $\cu$ to $\inft$, and also work with the cusp form
$x\mapsto f(g_\cu^{-1}\,x)$, which has quick decay at the cusp $\inft$.

We use a system of coordinates $(x_1,\ldots,x_{\ds-1})$ on $N$ given by
linear forms in the standard coordinates $(x^{(1)},x^{(2)})$ on $N$,
and moreover such that $x_{\ds-1}=0$ on $C$. With \eqref{omnu} and
\eqref{etadef}
\bad \int_{x\in D } & \int_{y=T}^\infty \omv_\nu(f;b,x)
= \int_{x\in D}\int_{y=T}^\infty (-1)^\ds \sum_{i=1}^{\ds-1} \bigl(
(\rmetric(x))^{-1})_{i,\ds-1} \\
&\qquad\hbox{} \cdot
   \Bigl( r_\nu(b;x) \,\partial_{x_i} f(x) - f(x) \,\partial_{x_i}
   r_\nu(b;x) \Bigr)
   \\
&\qquad\hbox{} \cdot
\det(\rmetric(x))^{-1/2} \, dx_1\, dx_2\, \ldots dx_{\ds-2}\, dy\,.
\ead
Here $\rmetric(x)$ describes the Riemannian structure in
$x=(x_1,\ldots , x_{\ds-1}, y)$ as discussed in~\eqref{Rmat'}, with $y$
in the role of $x_\ds$. The cusp form $f$ has quick decay at $\inft$.
Corollary~\ref{cor-grwt-hc} implies that the derivatives
$\partial_{x_i} f$ and $\partial_y f$ have quick decay at $\inft$ as
well.

Furthermore, $r_\nu(b;x) = \Ph_\nu(k_b^{-1} g_x)$ with the simple
function $\Ph_\nu$ in \eqref{Phinu}, $k_b\in K$ such that
$k_b \,\inft=b$, and $g_x\in G$ such that $g_x\,\orgn = x$. The
polynomial growth of $\Ph_\nu$ at the boundary implies that the
integral converges.

Let first $b\neq \inft$. Then $k_b\in K$ such that $k_b\,\inft = b$
satisfies $k_b\not\in M$, and hence $k_b^{-1}\,\inft\in N\, \nul$. As
$\nm(x^{(1)},x^{(2)})\am(y)$ varies through $C$ and $b$ is in a compact
set of $\partial S\setminus\{\inft\}$, then $r_\nu(b;x)$ has at most
polynomial growth, and the corresponding estimate is uniform in $b$.
The function $r_\nu(\cdot; x)$ is analytic, and hence the integral
$I(f;C;b)$ in~\eqref{omfint:i} converges absolutely and is analytic in
$b\in \partial S\setminus\{\inft\}$.

If $b$ varies through a neighborhood of $\inft$ in $\partial S$, then
still $r_\nu(b;x)$ is analytic in $(b,x)$ jointly, and the same holds
for all derivatives with respect to coordinates of $\xi$, all with
polynomial growth at the boundary. This shows that the integral is
continuous in $b$ on a neighborhood of $\inft$, and also that all its
derivatives of any order with respect to $b$ exist and are continuous
on a neighborhood of $\inft$. So $I(f,C;b)$ is in
$C^\infty(\partial S)$. We do not have information on the growth of the
derivatives, and hence cannot show analyticity at $b=\inft$.
\smallskip

Assertion \eqref{omfint:niii} is no problem if $X$ does not contain
cusps, since $\omv_\nu$ is a closed differential form, by
Proposition~\ref{prop-omtrf}. If $C$ contains the cusp $\cu$, then we
cut off $X$ at the horosphere $H_\cu(Z)$ for larger values of $Z$. That
adds $H_\cu(Z)\cap X$ as a boundary component. We use the estimates
discussed above to see that the contribution of this component tends to
zero as $Z$ tends to~$\infty$.

Relation~\eqref{int-bd0} follows from the transformation behavior in
\eqref{omv-trf}.
\end{proof}

\rmrk{Semi-analytic vectors in the spherical principal series}
Proposition \ref{prop-omfint} shows that to a given cusp form
$f\in \A^0(\Gm;\nu)$ we can associate a cocycle \il{psf}{$\ps^f$
cocycle associated to cusp form $f$}$\ps^f $ in the space
$Z^{\ds-1}\bigl( F^\tess_\bullet;\V\infty\nu(\partial S)\bigr)$
by
\be\label{psfdef} \ps^f(X) = I(f,X;\cdot)\qquad\text{ for each }X\in
\X^\tess_{\ds-1}\,.\ee
This cocycle has values in a submodule of $\V\infty\nu(\partial S)$ with
plenty of additional regularity: $\ps^f(X)\in \V\om\nu(U)$ for
$U=(\partial S)\setminus E$, where $E$ is a set with at most one point.

For each finite set $E$ of cusps of $\Gm$ we have the vector space
$ \V \om\nu\bigl( (\partial S) \setminus E)$.
The action of $\Gm$ induces
\be \tau_\nu(\gm) : \V \om\nu\bigl( (\partial S)
\setminus E\bigr)
\rightarrow \V \om\nu\bigl( (\partial S) \setminus \gm E\bigr)\,.\ee
We need to work with transformations in $\Gm$ to send cusps to cusps. We
put\il{Vom0}{$\V{\om^0}\nu$ semi-analytic vectors}
\be \label{Vom0}\V{\om^0}\nu = \lim_{\stackrel\longrightarrow E}
\V\om\nu\bigl(
(\partial S)\setminus E\bigr)\,,\ee
where $E$ runs over the finite sets of cusps of $\Gm$. The use of direct
limits ensures that we identify an analytic element
$f\in \V\om\nu\left(\partial S\right)$ with the restriction of $f$ to a
set $\left(\partial S\right)\setminus E$ for each finite set $E$ of
cusps. We call $\V{\om^0}\nu$ the $\Gm$-module of
\il{sav}{semi-analytic vectors}{\em semi-analytic vectors} in
$\V{}\nu(\partial S)$.
Here we use the submodule\il{Vsavs}{$\V\omi\nu$ smooth semi-analytic
vectors}
\be \V\omi\nu = \V{\om^0}\nu\cap \V\infty\nu(\partial S)\ee
of \il{ssav}{semi-analytic vectors, smooth}{\em smooth semi-analytic
vectors}. In this way, $\ps^f$ determines a class in the space
$H^{\ds-1}_\pb\bigl( \Gm;\V\omi\nu\bigr)$.

\rmrk{Recapitulation of principal series modules} Let us list the
principal series modules introduced up till now:
\begin{itemize}
\item The $G$-module $\V \om \nu(\partial S)$ of \emph{analytic
vectors}\,,
\item The $G$-module $\V \infty\nu(\partial S)$ of \emph{smooth
vectors}\,,
\item The $\Gm$-module $\V {\om^0}\nu$ of \emph{semi-analytic
vectors}\,,
\item The $\Gm$-module $\V {\om^0,\infty}\nu $ of \emph{smooth
semi-analytic vectors}.
\end{itemize}
\badl{Vovw} \xymatrix{ \V \om \nu(\partial S) \ar@{^(->}[r]
\ar@{^(->}[rd] & \V{\om^0,\infty}\nu(\partial S) \ar@{^(->}[r]
\ar@{^(->}[d]
& \V\infty\nu(\partial S)\\
& \V{\om^0}\nu(\partial S)
}
\eadl
The module $\V \om \nu(\partial S)$ is the space of global sections of
the sheaf $\V \om\nu$, and similarly for the smooth vectors. The
modules $\V {\om^0}\nu$ and $\V {\om^0,\infty}\nu$ are extension of the
module $\V\om\nu(\partial S)$. We have not defined these modules as
sections of a sheaf.

\begin{defn}\label{def-bdsing}
If $f\in \V{\om^0}\nu$, then we denote by \il{bsing}{$\bsing_\nu(\cdot)$
\text{ set of boundary singularities}}$\bsing_\nu(f)$ (the set of
\il{bs}{boundary singularity}\emph{boundary singularities} of $f$) the
minimal set $E$ of cusps such that
$f\in \V\om\nu\bigl( (\partial S)\setminus E\bigr)$.
So $\bsing_\nu(f) = \emptyset $ if and only if
$f\in \V \om\nu\bigl( \partial S\bigr)$.

We say that a cocycle
$\ps\in Z^{\ds-1}\bigl( F^\tess_\bullet; \V\omi\nu\bigr)$ satisfies the
\emph{boundary condition} \il{bdc}{$\bdc$ boundary condition}$\bdc$ if
for $C\in \X^\tess_{\ds-1}$
\be\label{bdc}
\bsing_\nu\bigl(\ps(C) \bigr) \subset \begin{cases}\emptyset &\text{ if
}C \subset S\,,\\
\{\cu\} &\text{ if }C \cap \Cu = \{\cu\}\,.
\end{cases}
\ee
By $Z^{\ds-1}\bigl( F^\tess_\bullet; \V\omi\nu\bigr)^\bdc$ we denote the
set of cocycles satisfying condition $\bdc$, and by
$H^{\ds-1}_\pb(\Gm;\V\omi\nu\bigr)^\bdc$ the space of those cohomology
classes that have a representative in
$Z^{\ds-1}\bigl( F^\tess_\bullet; \V\omi\nu\bigr)^\bdc$.
\il{H^bdc}{$ H^{\ds-1}_\pb(\Gm;\V\omi\nu\bigr)^\bdc$}
\il{Z^bdc}{$Z^{\ds-1}\bigl( F^\tess_\bullet; \V\omi\nu\bigr)^\bdc$}\end{defn}

So $\ps^f \in Z^{\ds-1}\bigl( F^\tess_\bullet; \V\omi\nu\bigr)^\bdc$ for
each cusp form $f\in \A^0(\Gm;\nu)$.

The above considerations lead to the first theorem relating cusp forms
to cohomology classes:
\begin{thm}\label{thm-btv}
Let $\nu \in \CC$. There is a linear map\il{btv}{$\btv_\nu$}
\be \btv_\nu: \A^0(\Gm;\nu) \rightarrow H^{\ds-1}_\pb\bigl( \Gm; \V \omi
\nu \bigr)^{\bdc}\,,\ee
associating to $f\in \A^0(\Gm;\nu)$ the cohomology class of $\ps^f$ (in
\eqref{psfdef}).
\end{thm}
\rmrk{Remarks}
\begin{enumerate}[label=$\mathrm{(\arabic*)}$, ref=$\mathrm{\arabic*}$]
\item $\btv_\nu$ is the zero map for most $\nu\in \CC$. It might be
non-zero only if $\rho^2-\nu^2\geq 0$. (See Proposition
\ref{prop-cusp-sp}.)
\item In Sections \ref{sect8}--\ref{sect12} we aim at
\begin{itemize}
\item a characterization of the image $\btv_\nu \A^0(\Gm;\nu)$ in
cohomological terms,
\item the determination of a set of $\nu\in \CC$ for which $\btv_\nu$ is
injective.
\end{itemize}
\end{enumerate}

%% file: ccro2-7-asp.tex

\bigskip

\def\flnm{ccro2-7-asp}

\section{The extension \intitle{\hat S} of \intitle{S}}\label{sect7}

In the previous sections we dealt with $S$ and its boundary
$\partial S$, each with its own real-analytic structure. In the next
section we will need to go beyond the boundary $\partial S$ in an
analytic way. The aim of this section is to view $S$ and $\partial S$
as embedded in a larger analytic manifold $\hat S$. We view this as a
generalization of the embedding $\uhp \cup \proj\RR\subset\proj\CC$ in
the case $\ds=2$, for which $S=\uhp$ and $\partial S=\proj\RR $.

With this embedding in place, functions on $S$ may extend analytically
to a neighborhood of $S$ in~$\hat S$. Lemmas \ref{lem-distan}
and~\ref{lem-xya} give such extensions for some distance-related
functions.

\subsection{Construction of \intitle{\hat S}} The space $S$ is
isomorphic to $N\times A$ as an analytic manifold. The corresponding
coordinates are the normalized horospherical coordinates
$(x^{(1)},x^{(2)},y)$ in $\RR^p \times \RR^q \times (0,\infty)$.

We define the space $S^-$ as $\RR^p\times \RR^q \times(-\infty,0)$. It
is isomorphic to $S$ as a manifold under the involution
\il{j}{$j:(x^{(1)},x^{(2)},y) \mapsto (x^{(1)},x^{(2)},-y)$}
\be \label{jinv} j: (x^{(1)},x^{(2)},y) \leftrightarrow
(x^{(1)},x^{(2)},-y)\,.\ee
The action of $G$ on $x\in S^-$ is given by $g\, x = jgj \,x$. We put
\il{ogn-}{$\orgn^- = (0,0,-1)\in S^-$}$\orgn^- = j\, \orgn $.

To make the set $\hat S = S\cup (\partial S) \cup S^-$ into an analytic
manifold we cover it by the sets \il{hatS}{$\hat S$ \text{ analytic
manifold containing $S$ and $\partial S$}}
\bad \hat S\setminus \{\inft\}&= S \cup (N\,\nul) \cup S^-\,,\\
\hat S \setminus \{\nul\}&= \wm S \cup (\wm N \nul)\cup \wm S^-= S \cup
\wm N \wm^{-1} \inft \cup S^-\,,
\ead
forming an atlas.

We provide $\hat S\setminus\{\inft\}$ with the standard analytic
structure of $\RR^\ds$ such that
\bad \bigl( x^{(1)}, x^{(2)},y) &\in \RR^p\times\RR^q\times \RR\text{
corresponds to }\\
&\quad
\begin{cases}\label{an-Sminf}
\nm(x^{(1)}, x^{(2)})\am(y)\, \orgn&\text{ if } y>0\,,\\
\nm(x^{(1)}, x^{(2)})\,\nul&\text{ if } y=0\,,\\
\bigl(j \nm(x^{(1)}, x^{(2)})\am(y) j \bigr)\,\orgn^-&\text{ if }y<0\,.
\end{cases}\ead
We call $\bigl( x^{(1)}, x^{(2)},y)$ \il{exthorco}{extended
horospherical coordinates}{\em extended horospherical coordinates}.

The transformation $\wm$ is an involution in $S$, $\partial S$ and $S^-$
that interchanges $\inft$ and $\nul$. The set
$\hat S\setminus \{\infty\} \cong \RR^\ds$ has the analytic structure
determined by the extended horospherical coordinates
$(x_1,\ldots,x_{\ds-1},y)$. The functions
$x_1\circ \wm, \ldots, x_{\ds-1}\circ \wm, y\circ \wm$ are defined on
$\hat S\setminus \nul$. They determine an analytic structure on
$\hat S\setminus \{\nul\}$. Restriction of these analytic structures to
the intersection $\left( \hat S \setminus
\{\infty\}\right) \cap \left( \hat S \setminus\{\nul\}\right) = \hat S\setminus
\{\infty,\nul\}$ gives two analytic structures on
$\hat S\setminus\{\inft,\nul\}$.

\begin{lem}[Relation between the analytic structures] \label{lem-as}
\mbox{ }
\begin{enumerate}[label=$\mathrm{(\roman*)}$, ref=$\mathrm{\roman*}$]
\item \label{as:0i}On $\left(\hat S \setminus \{\nul\} \right)\cap 
\left( \hat S \setminus\{\inft\}\right)$ both analytic structures are
isomorphic.
\item \label{as:S}Both analytic structures induce the standard analytic
structure on $S$ and on $S^-$.
\item \label{as:bd}The restrictions to $\partial S \setminus \{\inft\}$
of the analytic structures of $\partial S = K/M$ and of
$\hat S\setminus\{\inft\}$ coincide.
\item \label{as:bd0}The restrictions to $\partial S \setminus \{\nul\}$
of the analytic structures of $\partial S = K/M$ and of
$\hat S\setminus\{\nul\}$ coincide.
\end{enumerate}
\end{lem}
\begin{proof}
The analytic structure on $\hat S \setminus \{\inft\}$ is given in
\eqref{an-Sminf} by the analytic structure on $\RR^{p+q} \times \RR$.
On $\hat S\setminus \{\nul\}$ we have this structure, transformed
by~$\wm=\wm^{-1}$. We have to consider the effect of conjugation by
$\wm$ on the structure on $\hat S\setminus \{\inft,\nul\}$. We write
$n=\nm( x^{(1)},x^{(2)})$ and $a=\am(y)$, and use
\be \wm n a \wm\,\orgn = \wm n \wm a^{-1} \,\orgn = \nJ\bigl( \wm n \wm
\bigr) \,\Bigl(\am\bigl(\tJ(\wm n \wm )\bigr)\, a^{-1}
\Bigr)\,\orgn\,.\ee
The element $\wm n \wm$ depends analytically on $(x^{(1)}, x^{(2)})$,
and $a^{-1}= \am(y^{-1})$ depends analytically on $y\in \RR_{\neq 0}$.
This shows that the relation between both analytic structures on
$S \cup\left(\partial S\setminus\{\inft,\nul\}\right)$ is analytic. Use
the isomorphism $j$ to get the same for
$S^- \cup\left(\partial S\setminus\{\inft,\nul\}\right)$.
This gives \eqref{as:0i}.

Assertion \eqref{as:S} is clear from the use of the extended
horospherical coordinates.

The analytic structure of $\hat S \setminus\{\infty\}$ determines the
analytic coordinates $(x^{(1)}, x^{(2)}) $ on $N \, \nul $. An analytic
map to $K/M$ sends $n=\nm(x^{(1)},x^{(2)})$ to $\kI(n\wm)$. For the
inverse relation we take $k\in K \setminus M$. Then $k$ is in the big
cell $N\wm MAN$ of the Bruhat decomposition. So $k \,\infty $ is sent
to an element of $N\wm \,\inft= N\,\nul$. In this way we obtain
\eqref{as:bd}, and analogously \eqref{as:bd0} as well.
\end{proof}

The analytic manifold $\hat S$ is the one-point-compactification of
$\hat S\setminus \{\inft\} \cong \RR^\ds$. We call it the
\il{rflex}{extension}\emph{extension} of $S$ by reflection. On
$\hat S \setminus \{\orgn,\orgn^-\}$ we have \il{extpcCh}{extended
polar coordinates}{\em extended polar coordinates}
$(b,\tau) \in K/M\times (-1,1)$ corresponding to
\be\label{pce} \begin{cases} k_b\am(1/\tau)\,\orgn&
\qquad\text{for }\ \tau\in (0,1)\,,\\
k_b\,\inft & \qquad\text{for }\ \tau=0\,,\\
j k_b \am(-1/\tau)j\, \orgn^-& \qquad\text{for }\ \tau\in (-1,0)\,,
\end{cases}
\ee
with $k_b\in K$ such that $k_b\,\inft=b$. \il{tau}{$\tau$ coordinate on
$\hat S \setminus \{\orgn,\orgn^-\}$}
\begin{lem}\label{lem-epc}The extended polar coordinates are compatible
with the analytic structure of $\hat S$.
\end{lem}
\begin{proof}For $\tau=0$ the extended polar coordinates give the
standard analytic structure on $\partial S$, which is also the
structure induced by the analytic structure of $\hat S$ according to
\eqref{as:bd} and \eqref{as:bd0} in Lemma~\ref{lem-as}.

For $\tau>0$ we have the analytic map
$(b,\tau) \mapsto (x^{(1)},x^{(2)},y)$ determined by
\be \nm\bigl(x^{(1)},x^{(2)}\bigr) = \nJ\bigl( k_b\am(1/\tau)
\bigr)\,,\qquad y = \tJ\bigl( k_b\am(1/\tau) \bigr)\,.\ee
The inverse relation is determined by
$\orgn\neq \nm\bigl(x^{(1)},x^{(2)}\bigr)\, \am(y)  k = k_1 \am(1/\tau)$
with $\tau\in (0,1)$ and $k_1\in K$ determined up to
$k_1\mapsto k_1 m$, with $m\in M$, by the Cartan decomposition.
Application of $j: S \leftrightarrow S^-$ gives the correspondence on
$S^-\setminus\{\orgn^-\}$.
\end{proof}

\begin{prop}The space $\hat S = S\cup (\partial S)\cup S^-$ with the
analytic structure indicated above is a compact analytic manifold. The
group $G$ acts on it by analytic transformations.
\end{prop}

The compact manifold $\hat S$ can be realized as a sphere in $\RR^{d+1}$
provided with an analytic action of the group~$G$.
\begin{proof}After Lemmas \ref{lem-as} and~\ref{lem-epc} only the group
action is left to be checked.

The action of $G$ on $S=G/K$ is analytic with respect to the standard
analytic structure of the homogeneous space; this structure determines
the analytic structure of $\hat S$. The action of $G$ on $S^-$ is
derived from the analytic action on $S$, since the transformation $j$
is analytic.

On $\partial S = K/M \cong G/MAN$ the action is analytic for the
structure as homogeneous space. This standard structure determines the
analytic structure of $\hat S$. We still have to check that these
actions on the three subsets are the restriction of one analytic action
on $\hat S$. The action of $NA$ on $\hat S\setminus \{\inft\}$ is
clearly analytic, in particular on $(\partial S)\setminus \{\inft\}$.

Let $V$ be an open neighborhood of $n_0 \,\nul$, $n_0\in N$, in
$\hat S$. For $k$ in a sufficiently small neighborhood of $e$ in $K$,
the action on points of $V$ is determined by
$k n a \, \orgn = \nJ(kna)\am(\tJ(kna))\,\orgn$, which is analytic on
$V$. Together, this shows that locally on $\hat S\setminus \{\inft\}$
the action of $g$ is analytic for $g$ in a sufficiently small
neighborhood of $e$ in $G$.

The action of $N^- A =  \wm NA\wm^{-1}$ is analytic on
$\hat S \setminus \{\nul\}$. The action of $k$ near $e$ in $K$ is
analytic near points $n_0^- \,\inft$ with $n_0^-\in N^-$ analogously to
the previous case.

All together, this shows that the action of $g\in G$ is analytic at any
point $z\in \hat S$ for $g$ in a small neighborhood, which depends on
the point~$z$. Since $G$ is connected, this implies analyticity of the
action of $G$ on~$\hat S$.
\end{proof}

\subsection{Distance function} The length of a differentiable curve
$p:[t_1,t_2]\rightarrow S$ is given by the integral in \eqref{length}.
The Riemannian distance $\dist(z,x)$ is the minimal value of the length
of all paths from $z$ to $x$. This minimal length is attained for a
geodesic segment, which can be described as the curve of the form
$p(s) = g \am(e^s)\,\orgn$, $0\leq s \leq s_1$, with $z=g\,\orgn$ and
$x=g \am(e^{s_1})$. This leads to the formula $\dist(z,x) = s_1$, or
equivalently
\be \dist\bigl( g \, \orgn, \,g\,\am(y) \, \orgn\bigr) =\max( y,
y^{-1})\qquad
(g\in G,\; y>0)\,,\ee
with $y=e^{s_1}$. See also \eqref{distln}. The distance tends to
$\infty$ as $y $ tends to $0$ or to $\infty$. The distance function
$\dist$ makes $S$ into a metric space, and
$\dist(g\, z,\, g\, x) = \dist(z,x)$ for all $x,z\in S$. Moreover,
$\dist$ is an analytic function on
$\bigl \{ (z,x) \in S\times S\;:\; x\neq z\bigr\}$.

The following lemmas extend the function $z\mapsto - \log\dist(z,x)$
into $\hat S$, and use the involution $j$ in~\eqref{jinv}.

\begin{lem}\label{lem-distan}There is a real-valued analytic function
$\dtst$ \il{dtst}{$\dtst(x,z)$} on
$ \bigl\{(x,z)\;:\; x\in S,\; z\in \hat S\setminus\{x,jx\}\bigr\}$ such
that if $z$ is in the smaller set $ S\setminus\{x\}$, then
\be \dtst(x,z) = e^{-\dist(x,z)}\,.\ee

For fixed $x\in S$ the function $z\mapsto \dtst(x,z)$ has a simple zero
along $\partial S$ and is non-zero elsewhere.
\end{lem}
\begin{proof}For given $(x,z)$ in the prescribed domain take first
$p\in NA$ such that $p \,x= \orgn$. The element $p$ depends
analytically on $x$. The action of $K$ leaves $p\, x=\orgn$ fixed, and
can be used to bring $k p\, z$ on the line with normalized
horospherical coordinates $(0,0,y)$ with $-1<y<1$. The element
$p\in NA$ depends analytically on $ x$, and the element $k$ depends
analytically on $p\, x$ and $z$. So the value of $y$ in $(-1,1)$
depends analytically on $(x,z)$; we take $\dtst(x,z) = y$. We recall
that in the definition of $\hat S$ the coordinate function $y$, which
is positive on $S$, is zero along $N\,\nul \subset \partial S$
and is negative on $S^-$. The value $\dtst(x,z)$ is positive for
$z\in S\setminus\{x\}$, where $\dtst(x,z) = e^{-\dist(x,z)}$, and
negative for $z\in S^-$. The zero along $\partial S$ is a simple zero,
since letting $z =  k \am(1/\tau)$ move through $\partial S$ gives $y$
as a function of $\tau$ with positive derivative at $\tau=0$.

If $z\in S$, then $y\in (0,1)$, and
$\dist(x,z) = \dist(kp\, x,\,kp\, z) = \dist(\orgn, (0,0,y) ) = -\log y$.
\end{proof}

\begin{lem}\label{lem-xya}Let $x_1,x_2\in S$. The function
\[ (x_1,x_2,z) \,\mapsto\, e^{\dist(x_1,z) - \dist(x_2,z)} \]
defined for $z \in S \setminus \{x_1,x_2\}$, has values in the interval
$\bigl[ e^{-\dist(x_1,x_2) }, e^{\dist(x_1,x_2)} \bigr]$, and has a
positive analytic extension to the region
\[\bigl\{(x_1,x_2,z)\;:\; x_1,x_2\in S, \; z\in \hat S \setminus
\{x_1,x_2,jx_1,jx_2\}\bigr\} \,.\]
\end{lem}
\begin{proof}The triangle inequality implies the bounds.

The quotient $\dtst(x_2,z)/\dtst(x_1,z)$ of analytic functions is
analytic, except on the zero set of the denominator. The functions
$z\mapsto \dtst(x_1,z) $ and $z\mapsto \dtst (x_2,z)$ both have a
simple zero along $\partial S$. These zeros cancel each other in the
quotient.
\end{proof}

\begin{lem}\label{lem-ddv}Let $x_1,x_2\in S$ be fixed. For each $k\in K$
we write $x_j=  k n_j \am(t_j) $ with $t_j>0$ and $n_j\in N$. Then
\be \dist(x_1,\,k \am(t)\,\orgn) - \dist(x_2,\,k\am(t)\,\orgn) = \log
t_2-\log t_1+ \oh_{n_1,n_2}(t^{-1}) \quad\text{ as }t\uparrow\infty.\ee
(The implicit constants in the estimate depend on $n_1 $ and $n_2$.)
Moreover, with $z(t) = k \am(t)\,\orgn$ and
$z_j(t) = k n_j \am(t)\,\orgn$ for $j=1,2$, we have
\be \label{distrel} \dist( x_1, z(t) \bigr) - \dist\bigl( x_2,z(t)\bigr)
- \log(t_2/t_1)
\begin{cases} \leq \dist\bigl(z_1(t),z(t)\bigr)\,,\\
\geq -\dist\bigl( z_2(t),z(t)\bigr)\,.
\end{cases}\ee
\end{lem}
\begin{proof}The three points $z_1(t)$, $z_2(t)$ and $z(t)$ as defined
in the lemma tend to $k\,\inft$ as $t\uparrow\infty$. The invariance of
the distance function allows us to consider only $k=e$. Then we have
the situation sketched in Figure~\ref{fig-ddv}.
\begin{figure}\[\setlength\unitlength{.7cm}
\begin{picture}(5,7)(-.5,0)
\put(0,5){\circle*{.1}}
\put(2,5){\circle*{.1}}
\put(4,5){\circle*{.1}}
\put(2,1){\circle*{.1}}
\put(4,2){\circle*{.1}}
\put(-.5,0){\line(1,0){5}}
\put(0,5){\vector(0,1){2}}
\put(2,1){\vector(0,1){6}}
\put(4,2){\vector(0,1){5}}
\put(.1,5){$z(t)$}
\put(2.1,5){$z_1(t)$}
\put(4.1,5){$z_2(t)$}
\put(2.1,1){$x_1$}
\put(4.1,2){$x_2$}
\end{picture}\]
\caption{Sketch illustrating the proof of Lemma~\ref{lem-ddv}.\\
The vertical axis is the $y$-axis in horospherical coordinates. The
horizontal axis represents $N\,\nul$. }\label{fig-ddv}
\end{figure}
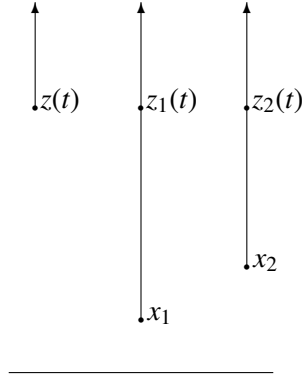
For $j=1,2$ we have
\[ \dist\bigl(x_j,z_j(t) \bigr)\leq\dist\bigl( x_j ,z(t)\bigr)\,. \]
To see this we note that $x_j$ and $z_j(t)$ are on the same geodesic
half-line to $\inft$. With \eqref{distln} we have
\[ \dist\bigl( x_j, z_j(t) \bigr) = \log\bigl( t/t_j\bigr)\,.\]
The geodesic segment from $x_j$ to $z(t)$ has the form
$\tau \mapsto \nm(x(\tau))\am(\tau ) \, \orgn$ with some function
$\tau\mapsto x(\tau)$. Its length is given by
\[ \int_{\tau=t_j}^t \sqrt{ Q(x(\tau),\tau)\,(x'(\tau),1) + \tau^{-2}
}\, d\tau \geq \log \bigl( t/t_j\bigr)\,, \]
since $Q(x,\tau)$ is positive definite; see \eqref{Rmat'}.

Together with the triangle inequality we obtain:
\be \label{distxz} \dist\bigl(x_j,z_j(t) \bigr)\leq\dist\bigl( x_j
,z(t)\bigr) \leq \dist\bigl(x_j,z_j(t) \bigr)+ \dist\bigl( z_j(t), z(t)
\bigr)\,.\ee
We have for $t\geq \max(t_1,t_2)$
\begin{align*}
\dist\bigl( x_1,z(t) \bigr) &\leq \dist\bigl( z_1(t_1), z_1(t)\bigr)+
\dist\bigl(z_1(t), z(t)\bigr)= \log(t/t_1)+ \dist\bigl(z_1(t),
z(t)\bigr)\,,\\
  \dist\bigl(x_2,z(t)\bigr) &\geq \dist(z_2(t_2),z_2(t)\bigr)-
  \dist\bigl( z(t), z_2(t)\bigr) \geq \log(t/t_2)\,.\end{align*}
Then
\[ \dist\bigl(x_1,z(t)\bigr)- \dist\bigl( x_2,z(t)\bigr)\leq
\log(t_2/t_1)
+ \dist\bigl( z_1(t),z(t)\bigr) \]
gives the first inequality in~\eqref{distrel}. Application of the same
approach to the difference
$\dist\bigl(x_2,x(z)\bigr) - \dist\bigl(x_1,z(t)\bigr)$ gives the other
inequality in~\eqref{distrel}.

Formula \eqref{Rmat} expressing the Riemannian structure in normalized
horospherical coordinates implies that
$\dist(z_j(t),z(t) )= \oh(t^{-1})$ for large values of~$t$, with an
implicit constant depending on $n_j$.
\end{proof}

%% file: ccro2-8-bgrms.tex

\bigskip

\def\flnm{ccro2-8-bgrms}

\section{Boundary germs}\label{sect8}
Theorem~\ref{thm-btv} has shown that there are maps from spaces of cusp
forms to cohomology groups of dimension $\ds-1$ with values in modules
of analytic functions on $\partial S$ with finitely many singularities.
To get hold of the images of the spaces of cusp forms it has turned out
to be useful to work with modules of analytic boundary germs.

The definition and study of these modules will take a number of
sections, leading to a reformulation of Theorem~\ref{thm-btv} in
Section~\ref{sect10}.

\subsection{Sheaves of boundary germs}\label{sect8-bg}For open
$U\subset S$ we put \il{EnuU}{$\E_\nu,\; \E_\nu(U)$}
\be \E_\nu(U) = \bigl\{ f\in C^\infty(U)\;:\; \Dt f = (\rho^2-\nu^2)\,
f\bigr\}\,. \ee
This defines $\E_\nu$ as a sheaf on~$S$.

Let $I$ be an open subset of $\partial S$.
The space of \il{bdgm}{boundary germ}\emph{boundary germs}
\il{Fnu}{$\F\nu, \; \F\nu(I)$}$\F\nu(I)$ on~$I$ is the direct limit
\be\label{Fnu} \F\nu(I) = \lim_{\stackrel \rightarrow U} \E_\nu \bigl( U
\cap S)\,,\ee
where $U$ runs through the open neighborhoods of $I $ in $\hat S$. So
$\F\nu$ is a sheaf on $\partial S$.

The action $L$ of $g\in G$, given by $L(g)f(x)=f(g^{-1}x)$, sends
functions $f$ on $I$ to functions $L(g)f$ on $g\,I$, and if $U$ is a
neighborhood of $I$ in $ \hat S $, then functions $f$ on $U \cap S$ are
sent to functions $L(g) f$ on the neighborhood $g\, U$ of $g\, I$
in~$ \hat S $. This induces a bijective linear map
$g^\ast : \F\nu(I) \rightarrow \F\nu(g\,I)$. This defines $\F\nu$ as a
$G$-equivariant sheaf. In particular, $ L$ is a representation of $G$
in $\F\nu(\partial S)$.

All elements of $\E_\nu(S)$ represent a unique element of
$\F\nu(\partial S)$. The embedding
$\E_\nu(S) \rightarrow \F\nu(\partial S)$ is compatible with the action
of $G$.

\rmrk{Analytic boundary germs}By an \il{abg}{boundary germ,
analytic}analytic boundary germ we mean a boundary germ that ``extends
across $\partial S$ analytically'' in the following way:

\begin{defn}\label{Womdef}Let $I$ be an open subset of $\partial S$. The
space
\il{Womnu}{$\W\om\nu\left(\partial S\right), \; \W\om\nu(I)$}$\W\om\nu(I)$
of \emph{$\nu$-analytic boundary germs on $I$} consists of the boundary
germs in $\F\nu(I)$ that are represented on $U \cap S$, for some
neighborhood $U$ of $I $ in $\hat S$, by a function $u$ of the form
\be \label{Audef} u(k \am(t) \, \orgn) = t^{-(\rho+\nu)}\, A_u( k
\am(t)\,\orgn) \qquad (k\in K, \; t>1 \text{ large})\,,\ee
where \il{uAu}{$A_u$}$A_u$ is an analytic function on $U$.
\il{abb}{analytic boundary behavior}
\end{defn}

We note that if $U$ is connected, the function $A_u$ on $U$ is uniquely
determined by $u$. Not every analytic function on some neighborhood $U$
of $\partial S$ in $\hat S$ can occur as a function $A_u$. The fact
that $\left( \Dt-\rho^2+\nu^2\right) u = 0$ on $U\cap S$ imposes a
partial differential equation for $A_u$ on $U\cap S$ with analytic
coefficients that extend to $U$ if $U$ is connected. The extended
differential operator is not elliptic on $U \cap \partial S$. The
condition that $A_u$ is analytic is an additional condition, which most
elements of $\F\nu(I)$ do not satisfy.

Suppose that the function $u\bigl( k \am(t) \,\orgn\bigr) =
t^{-(\nu+\rho)} \, A_u(k \am(t) \orgn)
$ represents an analytic boundary germ. Using that
$\lim_{t\uparrow\infty} k \am(t)\,\orgn = k \,\inft$, we obtain the
value
\[ A_u(k\,\inft)=\lim_{t\uparrow\infty} t^{\nu+\rho} u( k
\am(t)\,\orgn)\]
at $k\,\inft \in \partial S$.
If $\re\nu>-\rho$ the representatives $u$ of $\nu$-analytic boundary
germs are small near the boundary $\partial S$.
\smallskip

Definition~\ref{Womdef} is based on polar coordinates. It is useful to
have a formulation in horospherical coordinates as well.

\begin{prop}\label{prop-nuanal-hc}
Let $I$ be an open subset of $N \, \nul \subset \partial S$.
Let $U$ be a neighborhood of $I$ in $\hat S \setminus \{\inft\}$. Then
$u \in \E_\nu(U\cap S)$ represents a $\nu$-analytic boundary germ in
$\W\om\nu(I)$ if and only if there is an analytic function
\il{Bu}{$B_u$}$B_u$ on $U$ such that
\be \label{Anudef}u \bigl( n \am(y) \,\orgn\bigr) = y^{\rho+\nu} \,
B_u\bigl(n\am(y)\,\orgn\bigr)
\qquad n\in N,\; n\,\nul \in I, y\in (0,\e) \ee
for some $\e>0$.

There is a neighborhood $U_1\subset U$ of $I$ in $\hat S$ such that the
functions $A_u$ and $B_u$ are related by
\be B_u(z) = \DS(z) \,A_u(z)\,,\ee
where the function $\DS$ is non-zero and analytic. On $U_1\cap S$ it is
given by\il{DS}{$\DS(\cdot)$}
\be\label{DS}
\DS\bigl( k\am(t) \bigr) = e^{(\rho+\nu) \bigl( -\dist (\orgn,\, k
\am(t) \orgn) + \dist(\nJ(k\am(t))\,\orgn,\, k\am(t)\,\orgn)
\bigr)}\,.\ee
\end{prop}
We note that $t\uparrow\infty$ in \eqref{Audef}, whereas $y\downarrow 0$
in~\eqref{Anudef}.
\begin{proof}Let $n\in N$ and $k\in K$ vary subject to the condition
that $n\, \nul = k\, \infty$.
\[\setlength\unitlength{1.3cm}
\begin{picture}(2.5,1.6)(-1,-.4)
\put(-1,0){\line(1,0){2.5}}
\put(-.15,-.25){$n\,\nul$}
\put(-.84,.5){$n\am(y)\,\orgn$}
\put(.4,.5){$ k \am(t)\, \orgn$}
\thicklines
\put(0,0){\line(0,1){1.2}}
\qbezier(0,0)(0,.5)(.9,1.2)
\end{picture}\]
As $y\downarrow 0$, the point $n\am(y)\,\orgn$ approaches $n\, \nul$,
and as $t\uparrow \inft$ the point $k\am(t)\,\orgn$ approaches the same
point $n\,\nul$.

Let us compute $u$ at $k\am(t)\, \orgn$ in two ways. By \eqref{Audef}
\begin{align*}
u\bigl( k \am(t)\,\orgn \bigr) &= t^{-(\rho+\nu)} \, A_u\bigl( k \am(t)
\,\orgn \bigr)\,,\\
t&= e^{\dist( \orgn,\,\am(t)\,\orgn)}\,.
\end{align*}
We use \eqref{Anudef} and Proposition~\ref{prop-Id} to write
$k\am(t) \,\orgn = n_1 \am(y_1)\,\orgn$, with $n_1=\nJ(k\am(t))$ and
$y_1 = \tJ(k\am(t))$, to obtain
\begin{align*}
u\bigl( k \am(t) \orgn\bigr) &= u\bigl( n_1 \am(y_1) \orgn \bigr) =
y_1^{\rho+\nu}\, B_u\bigl(n_1 \am(y_1) \,\orgn\bigr)
= y_1^{\rho+\nu} \, B_u\bigl( k \am(t) \,\orgn\bigr)\,,\\
y_1 &= e^{-\dist(\orgn,\, \am(y_1)\,\orgn)} = e^{-\dist( k^{-1}
n_1\,\orgn,\, k^{-1} n_1\am(y_1)\,\orgn)}= e^{-\dist(
k^{-1}n_1\,\orgn,\,\am(t)\,\orgn)} \,.
\end{align*}
We consider the quotient of the factors:
\begin{align}\nonumber t^{-(\rho+\nu)}\bigm/y_1^{\rho+\nu}&
=\Bigl(e^{-\dist(\orgn,\,\am(t)\, \orgn)
+ \dist(k^{-1}n_1\,\orgn,\,\am(t)\,\orgn)} \Bigr)^{\rho+\nu}\\
\label{dtquot}&= \Bigl( \frac{\dtst(\orgn,\,\am(t)\,\orgn)}
{\dtst(k^{-1}n_1 \,\orgn,\am(t)\, \orgn)} \Bigr)^{ \rho+\nu }\,.
\end{align}
Lemma~\ref{lem-xya} shows that this function is analytic in
$k^{-1}n_1\,\orgn$ and in $\am(t)\,\orgn$, where $n_1$ depends on $t$
and $k$, and $k^{-1}n_1\, \orgn$ stays in $S$. Hence the factor extends
analytically as a function on a neighborhood of $U\cap\partial S$ in
$\hat S$. We denote this extension by $\DS$. The function $B_u$ extends
analytically as well. Conversely, if we know that $B_u$ extends to $U$,
then $A_u$ extends analytically.
\end{proof}

\rmrk{Example}The function
\be \label{exph} u\bigl( n \am(y) \, \orgn\bigr)=\Ph_\nu(n\am(y)k) = { y
}^{\rho+\nu}\ee
in \eqref{Phinu} (with $y\downarrow 0$) represents a $\nu$-analytic
boundary germ on $N \,\nul$, but not on $\partial S$.
Near points of $N\,\nul$ we have $B_u=1$, which clearly extends
analytically to $\hat S$. Near $\inft$ we get for $k\in M$ and $t>1$
\[ A_u( k\am(t)\,\orgn) = t^{\rho+\nu} u( \am(t)\,\orgn) =
t^{2(\rho+\nu)} \,.\]

\rmrk{Action of \intitle{G}}Let $u$ represent an element of
$\W\om \nu(I)$ for $I\subset \partial S$.
For each $g\in G$ the function $u_1=L(g^{-1})u: x\mapsto u(g \, x)$
represents an element of $\W \om \nu(g^{-1}I)$.

\begin{prop}\label{prop-aGA}{\rm (Action of $G$) }The action $L$ on the
representatives $u$ of $\nu$-analytic boundary germs corresponds to
$\tilde \tau_\nu$ on the corresponding analytic functions $A_u$ with
\badl{trfWt} \tilde \tau_\nu(g^{-1} ) A_u (z) &= \tilde J (g,
z)^{-(\rho+\nu)} \, A_u(g\, z)\,,\\
\tilde J_\nu( g,z ) &= \dtst( g^{-1}\,\orgn,\,z) \bigm/ \dtst(\orgn,\,z)
\,.
\eadl
The value $\tilde J(g,b)$ for $b\in \partial S$ is equal to the value
$J(g,b)$ of the factor of automorphy in the realization $\V\om\nu$ of
the principal series in \S\ref{sect5-prs-bd}.
\end{prop}
\begin{proof}For $g\in G$, let $u \in \E_\nu(U \cap S)$ represent an
analytic boundary germ, with $A_u $ analytic on $U$, and denote
$u_1=L(g^{-1})u$. If $u_1$ would represent a $\nu$-analytic boundary
germ as well, then the corresponding function $A_{u_1}$ would be given
for sufficiently large $t$ by
\bad A_{u_1}\bigl( k \am(t) \, \orgn \bigr)&= t^{(\rho+\nu)} u\bigl( g k
\am(t)\,\orgn\bigr) = t^{\rho+\nu} u\bigl( k_1 \am(t_1)
\,\orgn\bigr)
\\
& = t^{\rho+\nu} t_1^{-\rho-\nu} A_u\bigl( k_1 \am(t_1)
\,\orgn\bigr)
= (t_1/t)^{-(\rho+\nu)}\, A_u\bigl( gk \am(t)
\,\orgn\bigr)\,,\ead
with $k_1\in K$ and $t_1>1$ such that
$g k \am(t) \, \orgn = k_1 \am(t_1)\, \orgn$ (Cartan decomposition). We
have
\[ \log t_1 = \dist(\orgn, k_1\am(t_1)) = \dist ( g^{-1} \,\orgn,
k\am(t)\, \orgn)\,.\]

Lemmas \ref{lem-distan} and~\ref{lem-xya} show that
\[ t_1/t =e^{\dist ( g^{-1}\,\orgn ,\,k\am(t) \orgn)-\dist(\orgn,\,k
\am(t)\, \orgn) } = \dtst(\orgn, \,k \am(t)\,\orgn)\bigm/
\dtst(g^{-1}\,\orgn, \,k \am(t)
\,\orgn)
\]
extends analytically across $\partial S$. Since $A_u$ is analytic on
$U$, we conclude that $A_{L(g^{-1} )u}$ is analytic on a neighborhood
of $g^{-1} U \cap \partial S$.

We put $\tilde\tau_\nu(g^{-1}) A_u = A_{L(g^{-1})u}$. That leads to the
formula in \eqref{trfWt} for the factor of automorphy $\tilde J$.

The description as a quotient of analytic functions that are zero on
$\partial S$ is not handy to compute $\tilde J(g,b)$ for
$b\in \partial S$.
Instead, we use the analyticity of all quantities involved, to obtain
for $b\in \partial S\cap U$
\be \label{tldJ}\tilde J(g,b) ^{-(\rho+\nu)}\, \lim_{t\uparrow \infty}
A_u\bigl(p_2 \am(t)\,\orgn\bigr)
= \lim_{t\uparrow \infty} A_{u_1} \bigl(p_1 \am(t)\,\orgn\bigr)\ee
with $p_1,p_2\in G$ such that $p_1\,\inft= b$ and $p_2\,\inft=g\,b$. We
can take $p_1=k_1\in K$. Then $p_2 = g k_1$. We use
\[ A_{u_1}(k_1 \am(t)\,\orgn) = t^{\rho+\nu} \, u_1 (k_1\am(t)\,\orgn) =
t^{\rho+\nu}\, u( g k_1\am(t)\,\orgn)\,,\]

We write $gk_1 = k_2 n_2\am(t_2)$ with $k_2\in K$, $n_2\in N$ and
$t_2>0$, as a variant of Proposition~\ref{prop-Id}. Then
\begin{align*} u( g k_1 \am(t)\,\orgn) &= u\bigl( k_2 n_2
\am(t_2t)\,\orgn\bigr)= (t_2 t)^{-\rho-\nu} \, A_u\bigl( k_2 n_2
\am(t_2t),\orgn\bigr)= (t_2t)^{-\rho-\nu} A_u\bigl( gk_1\am(t)
\,\orgn)\,.\end{align*}
So
\[ A_{u_1}\bigl( p_1\am(t)\,\orgn) = (t_2)^{-\rho-\nu}\, A_u( g
k_1\am(t)\,\orgn)\,.\]
Taking the limit in~\eqref{tldJ} we find
\[ \tilde J(g,b) ^{-\rho-\nu} \,A_u(g\, b) = t_2^{-\rho-\nu} A_u(g\,
b)\,.\]
Hence $ \tilde J(g,b) = t_2 $. We have
$ k_2 n_2 \am(t_2) = k_2 \am(t_2) \, \bigl(\am(t_2)^{-1} n_2  \am(t_2) \bigr)$,
hence $\tilde J(g,b) = \tI(gk_1)$, which is just the value of $J(g,b)$
in~\eqref{Jdef}.
\end{proof}

\subsection{The restriction morphism}\label{sect8-restr}
The sheaf of $\nu$-analytic boundary germs $\W\om\nu$ turns out to be
isomorphic to the sheaf $\V\om\nu$ of analytic functions
on~$\partial S$, for general values of the spectral parameter. This
implies that the map $\btv_\nu$ in Theorem~\ref{thm-btv} from cusp
forms to $\V{\om^0,\infty}\nu$-valued cohomology classes can be
formulated in terms of $\nu$-analytic boundary germs. We will use this
formulation to get in \S\ref{sect10} an inverse of $\btv_\nu$.

\begin{defn}\label{def-res}For open subsets $I\subset \partial S$ the
\il{resmp}{restriction map}restriction map
\il{resnu}{$\rho_\nu$}$\rho_\nu:\W\om\nu(I) \rightarrow\V\om\nu(I)$
sends the boundary germ represented by $u$ to the analytic function
$A_u$ restricted to $\partial S$.
\end{defn} See Figure~\ref{fig-res}.
\begin{figure}[ht]
\[\setlength\unitlength{1cm}\begin{picture}(9,3)(-.8,-.5)
\qbezier(.5,0)(.5,1)(.8,2)
\put(.15,-.4){$u$}
\qbezier(3.5,0)(3.5,1)(3.8,2)
\qbezier(2.7,0)(2.7,1)(3,2)
\put(2.9,-.4){$A_u$}
\put(5.9,-.4){$\ph$}
\put(1.4,1){$\longrightarrow$}
\put(1.2,1.3){extend}
\put(4.6,1){$\longrightarrow$}
\put(4.4,1.3){restrict}
\thicklines
\qbezier(0,0)(0,1)(0.3,2)
\qbezier(3,0)(3,1)(3.3,2)
\qbezier(6,0)(6,1)(6.3,2)
\end{picture}\]
\caption{ Domains of $u$, $A_u$ and $\ph$ where $\ph$ is the image under
$\rho_\nu$ of the $\nu$-analytic boundary germ represented by $u$. The
thick line corresponds to an open part of $\partial S$.\\
The representative $u$ is defined on a region in $S$ near the boundary
$\partial S$. The function
$A_u\bigl( k \am(t) \, \orgn\bigr) = t^{\rho+\nu} u \bigl( k \am(t)\,\orgn)$
extends analytically into $\hat S$ across part of $\partial S$. The
restriction of $A_u$ to $\partial S$ is an analytic function $\ph$ on
an open subset of $\partial S$.
}\label{fig-res}
\end{figure}

\begin{thm}\label{thm-isoVW} The restriction map $\rho_\nu$ determines a
morphism $\W\om\nu \rightarrow \V\om\nu$ of $G$-equivariant sheaves,
with the action of $G$ on $\W\om\nu$ given by $L$ on representatives,
and the action $\tau_\nu$ of $G$ on $\V \om\nu$ in \eqref{taunu}.

If $\nu \in  \CC\setminus \frac12  \ZZ_{\leq -1}$, then $\rho_\nu$ is an
isomorphism of $G$-equivariant sheaves.
\end{thm}
This result generalizes \cite[Theorem 5.6]{BLZ13}.

\begin{proof}We first prove that $\rho_\nu$ gives a morphism of
$G$-equivariant sheaves. It will take the rest of this section to
establish the isomorphism $\W\om\nu \cong\V\om\nu$.

\rmrk{Morphism of \intitle{G}-equivariant sheaves} For $g\in G$, the
operator $L(g)$ on $u $ corresponds to $\tilde \tau_\nu(g)$ on $A_u$,
by Proposition~\ref{prop-aGA}. This proposition also implies that
$\tilde \tau_\nu(g)$ acts as $\tau_\nu(g)$ on the restriction to
$\partial S$.

\rmrk{Reduction of the statement to an isomorphism of stalks at one
point}The local structure of sheaves implies that morphisms of sheaves
are isomorphisms if the induced morphisms of the stalks at all points
of $\partial S$ are isomorphisms. The action of $G$ induces
isomorphisms of stalks at different points of $\partial S$.
Hence it suffices to prove that $\rho_\nu$ induces an isomorphism
$\bigl( \W\om\nu\bigr)_x \rightarrow \bigl( \V\om\nu\bigr)_x$ for one
point $x\in \partial S$. We choose $x=\nul$, for which we will use
Proposition~\ref{prop-nuanal-hc}.

Subsubsection~\ref{sect2-Lapl-horco} gives a description of the Laplace
operator in normalized horospherical coordinates. A computation shows
that $\Dt u = (\rho^2-\nu^2) u$ leads for $u=y^{\rho+\nu}B_u$ to the
relation
\be \label{deqAn}\Bigl( y \, \partial_y^2 +(2\nu+1) \partial _y +y\, L_1
+y^3 \, L_2 \Bigr) B_u= 0\,.\ee
This relation extends to the domain of $B_u$ by analyticity. We recall
that the operators $L_1$ and $L_2$ are second order differential
operators in the $x_j$ without constant term, with coefficients that
have at most total degree $2$ in the $x_j$. Actually, consultation of
\eqref{Dt}, \eqref{Dtp}, and \eqref{gRi} shows that
$L_2 = \sum_{i=p+1}^{p+q} \partial_{x_i}^2$.

\rmrk{Injectivity} Suppose that $U$ is an open neighborhood of $\nul$ in
$\hat S$, and that $u \in \E_\nu(U)$ represents a $\nu$-analytic
boundary germ. Then $B_u$ has a power series expansion that is
absolutely convergent on $[-\e,\e]^\ds$ for some $\e\in (0,1)$. We
write this expansion as\il{anc}{$a_n$ in \S\ref{sect8}}
\be\label{Aser} B_u(x,y) = \sum_{m\geq 0} a_m(x) y^m\,,\ee
where $x$ denotes $(x^{(1)},x^{(2)}) = (x_1,\ldots,x_{\ds-1})$.
Insertion into the differential equation
$\Dt B_u - (\rho^2-\nu^2) B_u=0$ yields
\be m(m+2\nu) a_m + L_1 a_{m-2} + L_2 a_{m-4}=0\,,\ee
for all $m\in \ZZ_{\geq 0}$. It is practical to take $a_m=0$ for $m<0$
to make the relation valid for all $m\in \ZZ$.

Since we assume that $\nu \not\in \frac12\ZZ_{\leq -1}$, we can use this
relation as a recursion relation. The relation does not impose
conditions on $a_0$, and shows that $a_1=0$. Recursively this implies
that $a_m=0$ for all odd $m$, and that \eqref{Aser} takes the form
\be \label{aser1} B_u(x,y) = \sum_{k\geq 0} a_{2k}(x) \, y^{2k}\,.\ee

Since $a_0(x) = B_u(x,0)$, the function $a_0$ is equal to the
restriction $\ph=\rho_\nu u$. We can use the relations to express each
$a_m $ in terms of $\ph$ and its derivatives. In this way we have shown
the injectivity of the map
$(\W\om\nu)_\nul \rightarrow (\V\om\nu)_\nul$, which implies the
injectivity of $\rho_\nu$ as a morphism of sheaves.
\medskip

\rmrk{Surjectivity} We start with $a_0(x) = \ph(x)$, analytic on the
region $B_\e$ indicated above. The recursion relation
\be \label{recur-k} a_{2k}= \frac{-1}{4k(k+\nu)} \bigl( L_1 a_{2(k-1)}
+ L_2 a_{2(k-2)}\bigr)\ee
describes all $a_{2k}$ in terms of $\ph$ and its derivatives with
respect to $x_1,\ldots,x_{p+q}$. Then we have a formal series
\be\label{Aserk} F =\sum_{k=0}^\infty a_{2k}(x) \, y^{2k}, \ee
which satisfies the differential equation for elements $B_u$. The
question is whether the series $F$ converges absolutely on some
neighborhood of $\nul$ in $\hat S$. If that is the case, then $F$ can
be taken as an analytic function $B_u$ such that $u= y^{\rho+\nu}B_u$
represents a boundary germ with restriction~$\ph$.

This reduces our task to showing that if $\ph$ is analytic on $B_\e$,
then the series in \eqref{Aserk} converges absolutely on some
neighborhood of $\nul\in \hat S$.
\smallskip

We use multi-indices in $\ZZ_{\geq 0}^{p+q}$, with
$\partial^\al = \partial_{x_1}^{\al_1} \partial_{x_2}^{\al_2}\cdots \partial _{x_{p+q}}^{\al_{p+q}}$,
$x^\al= x_1^{\al_1}\cdots x_{p+q}^{\al_{p+q}}$. Further
$|\al|:= \sum_{j=1}^{p+q}\al_j$,
$\al! := \al_1!\; \al_2!\;\cdots \al_{p+q}!$ and
$\max(\al) = \max(\al_1,\ldots,\al_{p+q}\bigr)$.
\il{mi}{$\al, \; |\al|,\;\al!,\; \max(\al)$ for multi-indices}

The description in \eqref{Dtp} and \eqref{gRi} shows that
$L_2 = \sum_{i=p+1}^{p+q} \partial_{x_i}^2$. The operator $L_2$ is a
linear combination of operators
$\partial_{x_i} \circ p_{i,k}(x) \circ \partial_{x_k}$ with
$i,k\in \{1,\cdots,p+q\}$. The polynomials have degree at most $2$ in
$x$, with coefficients determined by the linear forms $c_{j,i}$ in
\S\ref{sect2-Rie}. We rewrite
$\partial_{x_i} \circ p_{i,k}(x) \circ \partial _{x_k} = p_{i,k}(x) \partial_{x_k} \partial_{x_i} + \frac{\partial p_{i,k}}{\partial x_i} (x) \partial_k$.
In total, the operator $L_1$ is a linear combination of finitely many
operators of the form $x^\tau \partial^\s$, with multi-indices $\s$ and
$\tau$ satisfying $|\s|\leq 2$, $|\tau|\leq 2$.

The recursion relation \eqref{recur-k} can be used to write all
coefficients $a_{2k}$ as linear combinations of basis functions
$x^\bt \partial^\al \ph$, with $\al$ and $\bt$ running through the
multi-indices. It turned out to be helpful to give another form of the
formal series $F$, of which we want to show analyticity on a
neighborhood of $\nu$ in $\hat S$. We define the linear operator $\Lt$
by prescribing it on the basis functions.\il{Lt}{$\Lt$}
\badl{Lt-def} \Lt \bigl({ y }^{2h}\, x^\bt \, \partial^\al
\ph\bigr) &= \frac{-y^{2h+2} }{4(h+1)(h+1+\nu)} L_1 \bigl( x^\bt \,
\partial^\al
\ph\bigr)\\
&\qquad\hbox{}
+ \frac{-y^{2h+4} }{4(h+2)(h+2+\nu) } L_2\bigl( x^\bt \, \partial^\al
\ph\bigr)\,.\eadl

\begin{lem}\label{lem-H}The formal sum $F$ in \eqref{Aserk} is equal to
the formal sum
\be \label{Fdef1} H=\sum_{n=0}^\infty \Lt^n \ph \,,\ee
and for each $m\geq 0$ the function $a_{2m}$ is equal to the (finite)
sum of all terms in $H$ of degree $2m$ in $ y $.
\end{lem}
\begin{proof}The second statement implies that $F$ is equal to
$\sum_{n\geq 0} \Lt^n \ph$. We show it by recursion. The statement
holds for $m=0$, since each application of $\Lt$ increases the degree
in~$y$.

Let $S(n)$ denote the sum of the terms in $H$ of degree $2n$ in $ y$. By
\eqref{recur-k}
\be \label{asumS}
a_n = \frac{-1 }{4n(n+\nu)} \Bigl( \sum_{f\in S(n-1)}L_1 f
+ \sum_{g\in S(n-2)} L_2 g \Bigr)\,. \ee
We use that $L_1$ and $L_2$ are differential operators in $x$, and that
they do not change the dependence on~$ y $. If $f\in S(n-1)$ then
\[ \Lt f = \frac{-y^{2n}}{4n(n+\nu)} L_1\bigl( f/y^{2n-2}\bigr) +
\frac{-y^{2n+ 2 }}{4(n+1)(n+1+\nu)} L_2 \bigl( f/y^{2n-2}\bigr)\,.\]
The first of these two terms has degree $2n$ in $y$, and the second one
has degree $2n+2$. In an analogous way, we obtain that for
$g\in S(n-2)$ the function $\Lt g$ is the sum of two terms, only one of
which has degree $2n$. Hence the right-hand side of \eqref{asumS} gives
  the part of degree $2n$ in $y$ of
\[ \Lt\Bigl( \sum_{f\in S(n-1)}f + \sum _{S(n-2)} g \Bigr)\,.\]
Since all terms in $H$ of degree $2n$ in $y$ arise from applying $\Lt$
to terms of degree $2n-2$ or degree $2n-4$, this shows by induction the
second assertion in the lemma.
\end{proof}

This lemma shows that for the surjectivity it suffices to show the
convergence of the formal sum in~\eqref{Fdef1}. We start with a
holomorphic function $\ph$ on \[ B_\e=\bigl\{x\in \CC^{p+q}\;:\;
|x_i|\leq \e\text{ for }1\leq i \leq p+q\bigr\}\]
for some $\e\in (0,1)$. We want to find $\dt\in (0,\e]$, $\dt<1$, and
$\eta\in (0,1)$ such that $H$ converges absolutely for
$\bigl(x,y\bigr)$ in $ B_\dt\times [-\eta,\eta]$.

For $\z>0$ and for holomorphic functions $f$ on $B_\z$ we put
\be \|f\|_\z = \sup_{x\in B_\z} |f(x)|\,.\ee
\begin{lem} \label{lem-diffph}If $\ph$ is analytic on $B_\e$, then for
$\dt\in (0,\e)$
\be\label{pdphest} \|\partial^\al \ph \|_\dt \leq \al!\,\e^{-|\al|
}\,(1-\dt/\e)^{-|\al|-p-q}\, \|\ph\|_\e\,.\ee
\end{lem}
\begin{proof}We have an expansion
$\ph(x) = \sum_{\al\geq 0} c_\al x^\al$. Then
$|c_\al| \leq \e^{-|\al|} \, \|\ph\|_\e$.

For $\al\in \ZZ_{\geq 0}^{p+q}$
\begin{align*}
\partial^\bt \ph(x) &= \sum_{\al \geq \bt} c_\al x^{\al-\bt} \frac{
\al!}{(\al-\bt)!} = \sum_{\al\geq 0} c_{\al+\bt} x^\al
\frac{(\al+\bt)!}{\al!}\\
&\ll \sum_{\al\geq 0} \e^{-|\al+\bt|} \dt^{|\al|}
\frac{(\al+\bt)!}{\al!}\,\|\ph\|_\e = \e^{-|\bt|} \Bigl(
\prod_{j=1}^{p+q} \sum_{\al_j\geq 0}\Bigl(\frac \dt\e\Bigr)^{\al_j}
\frac{(\al_j+\bt_j)!}{\al_j!}\Bigr)\|\ph\|_\e\,.\end{align*}
We have used that $\bigl|x^\al| \leq \dt^{|\al|} $. For each factor in
the product we use that
\[\sum_{k\geq 0}(\dt/\e)^k (k+b)!/k! = (1- \dt/\e)^{-1-b} \, b!\,.\]
Then
\begin{align*}
\partial^\bt \ph(x)
&\ll \e^{-|\bt|} \Bigl( \prod_{j=1}^{p+q} \bigl(1-\dt/\e)^{-1-\bt_j}\,
\bt_ j! \Bigr)
\|\ph\|_\e = \e^{-|\bt|} \bigl(1-\dt/\e\bigr)^{-|\bt|-p-q}
\bt!\,\|\ph\|_\e\,.\qedhere
\end{align*}
\end{proof}

We use this lemma to compute $\Lt^n\ph$. First we consider the action of
$\Lt$ on one basis function
$ f_{n-1}=t^{2h(n-1)} \, x^{\bt(n-1)}\,\partial^{\al(n-1)} \ph$, with
$h(n-1) \in \ZZ_{\geq 0}$ and multi-indices $\al(n-1)$ and $\bt(n-1)$
in $\ZZ_{\geq 0}^{p+q}$. The image $\Lt f_{n-1}$ is a linear
combination of many basis functions, arising from the many simple
differential operators $c_{\tau,\s} x^\tau\partial^\s$ involved in
$L_1$ and $L_2$. That gives a finite number of terms, each of which we
can write as $c_n \, t^{2h(n)}\, x^{\bt(n)} \,\partial^{\al(n)} \ph$
with $h(n) \in \ZZ_{\geq 0}$ and $\al(n), \bt(n)\in \ZZ_{\geq 0}^{p+q}$.
We fix $C$ such that $|c_{\tau,\s}|\leq C$ for all $(\s,\tau)$. The
total number of terms is bounded by a large number $N$ obtained by
considering the number of terms $n_{\s,\tau}$ arising for each
$(\s,\tau)$. We use the abbreviations ${\bf a}(n) = |\al(n)|$,
${\bf b}(n) = |\bt(n)|$.

\begin{lem}In these notations:\il{aan}{${\bf a}(n) = |\al(n)|$ in
\S\ref{sect8}} \il{bbn}{${\bf b}(n) = |\bt(n)|$ in \S\ref{sect8}}
\begin{align}\label{he}h(n-1)+2 &\geq h(n) \geq h(n-1) +1 \,,\\
\label{ae}{\bf a}(n-1)&\leq {\bf a}(n) \leq {\bf a}(n-1)+2\,,\\
\label{be}{\bf b}(n) &\leq {\bf b}(n-1) +2\,,\qquad {\bf b}(n) \leq {\bf
a}(n)\,,\\
\label{cnest}
|c_n| &\leq \frac{2C \z_\nu \, d_n}{4 n\, \bigl|\nu+n\bigr|}\,,
\end{align}
where
\be \z_\nu = \max\Bigl\{ \frac{|\nu+1+m|}{|\nu+1+l|} \;:\; 0 \leq m \leq
l \leq 2n\Bigr\} \,,\ee
and
\be\label{de} |d_n| \leq \begin{cases}1&\text{ if }{\bf a}(n) = {\bf
a}(n-1)+2\,,\\
\max(\bt(n-1))&\text{ if } {\bf a}(n) = {\bf a}(n-1)+1\,,\\
\max(\bt(n-1))^2 & \text{ if } {\bf a}(n)= {\bf a}(n-1)\,.
\end{cases}
\ee
\end{lem}
\begin{proof}Definition~\eqref{Lt-def} implies that
$n\leq h(n) \leq 2n$. The positive quantity $\z_\nu$ allows us to
estimate $\frac1{|\nu+n|}{\bigl|\nu +1 + h(n-1)\bigr|} $ by
$ \frac{\z_\nu}{n}$. We recall that we work under the assumption that
$\nu \not \in \ZZ_{\leq -1}$.

The differential operator $c_{\tau,\s} x^\tau\partial^\s$ with
$|\s|,|\tau|\leq 2$, gives the factor $C$ in \eqref{cnest}, and
increases $|\al|$ and $|\bt|$ by at most $2$. If $|\bt(n)|>0$, a
differentiation with respect to some $x_i$ occurs. Then $|\al|$
increases less, $|\bt|$ decreases, and the differentiation gives a
factor estimated by $2d_n$.
\end{proof}

The function $\Lt^k \ph$ can be computed by making $k$ choices of
operators $f_j $ of the form $c_{\s,\tau} x^\tau \partial^\s$ for
$j=1,\ldots,k$. That gives at most $N^k$ terms
$f_k f_{k-1}\cdots f_1 \ph$ of the form
\be g=D \, t^{2h(k)} \, x^{\bt(k)} \, \partial^{\al(k) } \ph\,.\ee
We aim at an estimate that is valid for all the terms arising in this
way. We have $h(n) \geq n$ by \eqref{he}, and get
\be D \leq \Bigl( \frac {C \z_\nu}2 \Bigr)^k \frac{\prod_{n=1}^k d_n}{
k!\; \bigl| (\nu+1)_k \bigr|} \,, \ee
with the increasing Pochhammer symbol $(z)_n = \prod_{j=0}^{n-1} (z+j)$.

Let $e_1$ be the number of steps for which ${\bf a}(n)={\bf a}(n-1)+1$
and $e_2$ be the number of steps for which ${\bf a}(n) = {\bf a}(n-1)$.
Since ${\bf b}(n) \leq {\bf a}(n)$ by \eqref{be} we have
$\max(\bt)\leq \max(\al)$. Furthermore, $\max(\al) \leq |\al|$. We have
${\bf a}(k) = 2k-e_1-2e_2$ by \eqref{ae} and \eqref{de}.
\be \prod_{n=1}^ k d_n \leq |{\bf a}(k)|^{2e_2+e_1}\leq ( {\bf a}(k)+1)
\cdots (2k)
= \frac{(2k)!}{{\bf a}(k)!}\,. \ee

We obtain with Lemma~\ref{lem-diffph}
\begin{align*} \bigl\| g \bigr\|_\dt &\leq D \, |t| ^{2h(k) } \,
|x^\bt|\, \bigl\| \, \partial^{\al(k)} \ph \bigr\|_\dt\\
&\leq \frac{ (C \z_\nu/2)^k }{ k!\, |(\nu+1)_k|} \frac{(2k)!}{{\bf
a}(k)!} \eta^{2h(k)} \, \|x\|^{{\bf b}(k)} \,\al(k)! \, \e^{-{\bf a(k)}
}
(1-\dt/\e)^{-{\bf a}(k)-p-q} \, \|\ph\|_\e \,.\end{align*}

By Stirling's formula we have
$\frac{(2k)!}{k!\;(\nu+1)_k }\ll 2^{2k} \, k^{-1/2-\re\nu}$. The
influence of $\nu$ is in the implicit constant. We have
$\frac{\al(k)!}{{\bf a}!} \leq 1$, and $\|x\|^{{\bf b}(k) }\ll 1$ since
$0<\dt<1$. Using ${\bf a}(k) \leq 2k$, $h(k)\leq 2k$, and
$|t|\leq \eta$, we obtain
\[ \|g\|_\dt \ll |t|^{2k} \Bigl( 2C\z_\nu \e^2 (1-\dt/\e)^2 \Bigr)^ k\,
k^{-1/2-\re\nu} \,(1-\dt/\e)^{-p-q}\, \|\ph\|_\e\,. \]
The estimates that we carried out were wasteful for a single term.
However they are uniform over the $\oh(N^k)$ choices
$f_k f_{k-1}\cdots f_1$ of operators. With
$0<\eta < \bigl( 2 N C \z_\nu (\e-\dt) ^2\bigr)^{-1/2}$ the norms
$\|\Lt^k \ph \|_\dt$ have linear exponential decay in $k$ on the
product $[-\dt,\dt]^{p+q} \times [-\eta,\eta]$, which gives absolute
convergence of the series $F$ in \eqref{Aserk}, and hence analyticity.

This concludes the proof of the surjectivity of the restriction map
$\bigl( \W\om\nu\bigr)_\nul\rightarrow \bigl( \V\om\nu\bigr)_\nul$, and
ends the proof of Theorem~\ref{thm-isoVW}.
\end{proof}

%% file: ccro2-9-cohbg.tex

\bigskip

\def\flnm{ccro2-9-cohbg}

\section{Cohomology with values in boundary germs}\label{sect9}
We proceed under the assumption that $\nu \not\in  \frac12\ZZ_{\leq -1}$.

The isomorphism of $G$-equivariant sheaves
$ \rho_\nu: \W\om\nu \rightarrow \V\om\nu$ in Theorem~\ref{thm-isoVW}
provides us with isomorphisms between $G$-modules of analytic functions
on the boundary $\partial S$ with modules of $\nu$-analytic boundary
germs, on the condition $\nu \not\in \frac12\ZZ_{\leq -1}$.

The module $\V{\om^0}\nu $ in \eqref{Vom0} is isomorphic under
$\rho_\nu$ to
\be\label{Wom0} \W{\om^0}\nu = \lim_{\stackrel\rightarrow E}\W\om\nu
\bigl(
(\partial S) \setminus E\bigr)\,,\ee
where $E$ runs over the finite sets of cusps. The submodule
$\W{\om^0,\infty}\nu = \rho_\nu^{-1} \V{\om^0,\infty}\nu$ has no
independent description in the context of
$\W\om\nu$.\il{Womn0}{$\W{\om^0}\nu,\; \W{\om^0,\infty}\nu$}

Theorem~\ref{thm-btv} yields, under the condition
$\nu \not\in  \tfrac12 \ZZ_{\leq -1}$, a linear map
\be \btw_\nu : \A^0(\Gm;\nu) \rightarrow H^{\ds-1}_\pb\bigl(
\Gm;\W{\om^0,\infty}\nu\bigr)^\bdc \ee
by composing $\btv_\nu$ with the natural map corresponding to
$\rho_\nu^{-1}$.

In this section we will see that $\btw_\nu$ can be described directly by
integrals.

\subsection{Free space resolvent}\label{secta0-fsr}
We use results from the paper \cite{MiWa92} by Miatello and Wallach:
\begin{prop}{(\rm $K$-bi-invariant functions) }\label{prop-Q}
There is a family $\CC\ni \nu \mapsto Q_\nu$ of functions on
$G\setminus K$ such that \il{Qnu}{$Q_\nu, \; q_\nu(\cdot;\cdot)$}
\begin{enumerate}[label=$\mathrm{(\arabic*)}$, ref=$\mathrm{\arabic*}$]
\item\label{Q:holmer} $Q_\nu$ is holomorphic on $\re\nu>-\frac12$, with
a meromorphic extension to $\CC\setminus \frac12\ZZ_{\leq -1}$,
\item $Q_\nu(k_1 \am(t) k_2) = Q_\nu(\am(t))$ for all $t>1$,
\item\label{Q:CasQ} $ R(\Cas)  Q_\nu = -(\rho^2-\nu^2) Q_\nu$,
\item\label{Q:inf}
The function $ k\am(t)\,\orgn\mapsto Q_\nu\bigl(k\am(t)\bigr)$
represents a $\nu$-analytic boundary germ with restriction equal to the
constant function $1$ on~$\partial S$.
\item\label{Q:e}
There is $d_0>0$ such that
$Q_\nu(z) = \oh\bigl(\dist(z,\orgn)^{-d_0}\bigr)$ as
$z\rightarrow \orgn$. If $\ds\geq 3$, then we can take $d_0=\ds-2$.
\end{enumerate}
We put $q_\nu(x_1;x_2) = Q_\nu(g_1^{-1} g_2)$ for
$(x_1,x_2) \in S^2 \setminus \diag$, $g_j\in G$ such that
$x_j= g_j\,\orgn$.
\begin{enumerate}[label=$\mathrm{(\arabic*)}$, ref=$\mathrm{\arabic*}$]
\setcounter{enumi}{5}
\item $q_\nu(g x_1;gx_2)=q_\nu(x_1;x_2)$ for all $g\in G$,
\item $q_\nu(x_2;x_1) = q_\nu(x_1;x_2)$,
\item\label{Q:nu-anal}
$q_\nu(\cdot ,x_2) \in \E_\nu( S\setminus\{x_2\})$
represents an element of $\W \om\nu(\partial S)$,
\item\label{Q:Delta}$\Dt q_\nu(\cdot,x_2) = (\rho^2-\nu^2)\, q_\nu(\cdot;x_2)$
and $\Dt q_\nu(x_1;\cdot) = (\rho^2-\nu^2)\, q_\nu(x_1;\cdot)$,
\item \label{Q:resolv} if $f\in C_c^\infty(S)$, then
\be \label{resolve}\int_S q_\nu(x;y) \, \bigl( (\Dt-\rho^2+\nu^2)
f\bigr)(y)\; d\mu(y) = 2\nu c(\nu) f(x)\,,\ee
where $c$ is the $c$-function of Harish-Chandra.
\end{enumerate}
\end{prop}
 In the case of $\PSL_2(\RR)$ the function $Q_\nu$ is given explicitly 
 as $Q_{\nu+1/2,0}$ in \cite[(1.3b)]{BLZ15}. 

Theorem 1.1 in \cite{MiWa92} gives \eqref{Q:holmer}--\eqref{Q:CasQ},
except the restriction  of the singularities to $\frac12\ZZ_{\leq -1}$. 
See also \cite[Theorem 2.2]{MiWi99} and \cite[Proposition 3.2]{HHP19}.
The relation in \cite[Theorem 1.1, b)]{MiWa92} with the spherical
function reduces the position of the singularities to zeros and
singularities of the Harish-Chandra $c$-function. Theorem 1.2 in
\cite{MiWa92} implies part~\eqref{Q:inf}. Part (d) in Theorem 1.1
implies a more precise result than \eqref{Q:e}. In fact we can take
$d_0=\ds-2$ if $\ds\geq 3$, and any $d_0>0$ if $\ds=2$.

The resolvent relation \eqref{Q:resolv} follows from \cite[Lemma
2.2]{MiWa92}. Actually, this result of Miatello and Wallach is valid
for more general elements of $C^\infty(S)$ that satisfy a growth
condition at the boundary. For compactly supported functions we need
only local integrability, which is given by \cite[Lemma 2.1]{MiWa92}
for all values of $\nu$ at which $Q_\nu$ is defined.

The \il{cHC}{$c$ Harish-Chandra's $c$-function}$c$-function of
Harish-Chandra is given in \cite[(3.7)]{Hel18}. For rank one we get
\be \label{cHC} c(\nu) = \frac{ 2^{\rho-\nu} \,\Gf\bigl(
(p+q+1)/2\bigr)\, \Gf(\nu) } { \Gf\bigl( (p/2+1+\nu)/2\bigr)\,\Gf\bigl(
(p/2+q+\nu) /2\bigr)}\,,\ee
with $p$ and $q$ as in~\eqref{qq}. We see that it has a first order
singularity at $\nu=0$, and further it has zeros and poles in points of
$\frac12\ZZ_{\leq -1}$. The factor $\nu\,c(\nu)$ in \eqref{resolve} in
the proposition is holomorphic and non-zero for
$\nu \in \CC\setminus \frac12\ZZ_{\leq -1}$. It seems sensible to
impose the further restriction $\nu\not\in \frac12\ZZ_{\leq -1}$.
\medskip

We form the $(\ds-1)$-form\il{omW}{$\omw_\nu$}
\be \label{omWdef}\omw_\nu(f;x,\cdot) = \eta( f(.),
q_\nu(x,\cdot)\bigr)\,,\ee
on $S$ depending on $f\in C^\infty(S)$ and $x\in S$, analogously to the
definition of $\omv_\nu$ in \eqref{omnu}. See
Proposition~\ref{prop-eta} for $\eta(\cdot,\cdot)$. Like in
\eqref{omv-trf} we have for each $g\in G$:
\be \label{omW-trf}
\omw_\nu( f;x,\cdot) \circ g^{-1} = \omw_\nu( f|g^{-1}; gx,
\cdot)\qquad(g\in
G)\,.\ee

\begin{prop}\label{prop-psC}
Let $\nu \in \CC\setminus \frac12\ZZ_{\leq -1}$, and let $U \subset S$
be open. Let $D \subset U$ be a compact manifold with a contractible
interior $\mathring D$, and with a piecewise smooth boundary
$\partial D$.
Let $\partial D$ be oriented according to the outward normal.

For each $h\in \E_\nu(U)$ and $x\in S \setminus  D $
\be \int_{ \partial D} \omw_\nu(h;x,\cdot) =
\begin{cases} 2\nu\, c(\nu)\, h(x) &\text{ if } x\in \mathring D\,,\\
0&\text{ if }x\in S\setminus D \,.
\end{cases}\ee
\end{prop}
\begin{proof}

Let $f\in C_c^\infty(S)$. With~\eqref{Q:Delta} in
Proposition~\ref{prop-Q} and~\eqref{MS} we obtain, on all of
$S\setminus\{x\}$,
\begin{align*}
q_\nu(x,\cdot)\, (\Dt-\rho^2+\nu^2) f = q_\nu(x,\cdot)\, \Dt f - f \,\Dt
q_\nu(x,\cdot) = - \MS\bigl( f, q_\nu(x,\cdot)
\bigr)\,.
\end{align*}
Hence \eqref{resolve} takes the form
\badl{-MS-int} 2\nu\, c(\nu) \, f(x) &= - \int_{y\in S} q_\nu(x,y)\,
\bigl((\Dt-\rho^2+\nu^2) f \bigr)(y)
\, d\mu(y) \\
&= - \lim_{\e\downarrow 0} \int_{y,\; \e \leq \dist(y,x) \leq R }
q_\nu(x,y)\, \bigl((\Dt-\rho^2+\nu^2) f \bigr)(y)\, d\mu(y) \,,
\eadl
for each $R$ such that the support of $f$ is contained in $\bigl\{ y\in
S \;:\; \dist(y,x) < R\bigr\}$. Part~\eqref{eta:diff} of
Proposition~\ref{prop-eta} allows us to formulate this as
\bad 2\nu\, c(\nu)\, f(x) &=-I(\nu;x,R;f)
+ \lim_{\e\downarrow 0} I(\nu;x,\e;f)\,,\\
I(\nu;x,r;f) &= \int_{y\in H(x,r)} \eta(f,q_\nu(x,\cdot))\,,
\ead
with the $(\ds-1)$-dimensional hypersurface $H(x,r)=\bigl\{y \in
S\;:\;\dist(y,x)=r\bigr\}$ oriented according to the normal directed
away from $x$. See Figure~\ref{fig-psC} a). The hypersurface $H(x,R)$
is outside the support of $f$. Hence
\be 2\nu\, c(\nu)\, f(x) = \lim_{\e\downarrow 0} I(\nu;x,\e;f)\,.\ee
\begin{figure}[t]

\begin{center}
\includegraphics[width=6cm]{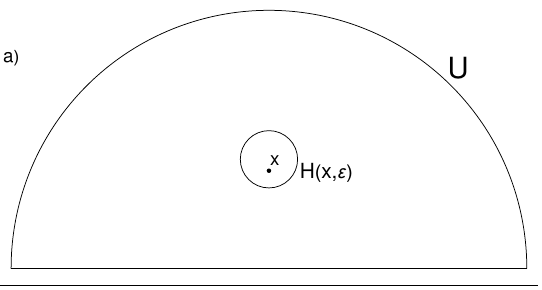}
\medskip\\
\includegraphics[width=5.5cm]{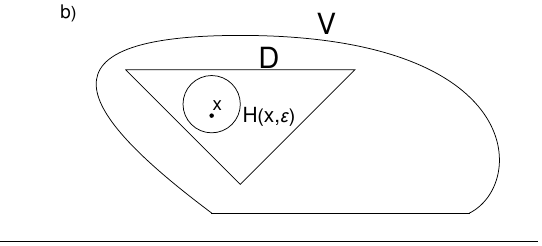}
\qquad \includegraphics[width=5.5cm]{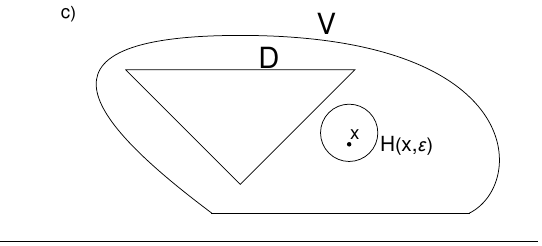}
\end{center}
\caption{Illustrations for the proof of
Proposition~\ref{prop-psC}.}\label{fig-psC}
\end{figure}

Under the assumptions in the proposition we take $f$ equal to $h$ on a
neighborhood $V$ of $D  $ such that $V\subset U$, and extend it to an
element $f\in C^\infty_c(S)$. Then
\[
(\Dt f)(x) = (\rho^2-\nu^2) f(x)
\]
for all $ x  \in V$. Now the differential form $\eta(f,q_\nu(x,\cdot))$
is closed on $V \setminus\{x\}$.

If $x\in \mathring D$ the integral $I(\nu;x,\e;f)$ does not depend on
$\e$, as long as $H(x,\e)$ stays inside $ D$. $\bigl($See
Figure~\ref{fig-psC} b).$\bigl)$ Deforming $H(x,\e)$ into $\partial D$
gives one part of the assertion, taking into account the minus sign
in~\eqref{-MS-int}. If $x$ is outside $D$ (case c) in the figure), then
the integral $\int_D \eta\bigl( f,q_\nu(x,\cdot)\bigr)$
is zero, since then the differential form is closed on $D$.\end{proof}

\subsection{Cohomology with values in analytic boundary
germs}\label{sect2-coh-abg}
Theorem~\ref{thm-isoVW} implies that if
$\nu \not \in \frac12\ZZ_{\leq -1}$, then the linear map $\btv_\nu$ can
be lifted to a map from cusp forms to a cohomology space with values in
boundary germs. Here we give an integral transformation describing this
map, under the same assumption $\nu\not\in \frac12\ZZ_{\leq -1}$.

\begin{lem}\label{lem-rmph}Let
$\nu \in \CC \setminus\frac12\ZZ_{\leq -1}$. The restriction morphism
$ \rho_\nu: \W\om\nu\rightarrow \V\om\nu$ induces the following
relations:
\begin{enumerate}[label=$\mathrm{(\roman*)}$, ref=$\mathrm{\roman*}$]
\item \label{rmph:Rq} The kernel functions $r_\nu$ in \eqref{rnudef} and
$q_\nu$ in Proposition~\ref{prop-Q} satisfy
\be \rho_\nu q_\nu(\cdot;x) = r_\nu(\cdot;x)\qquad (x\in S)\,. \ee
\item\label{rmph:omrel}The differential forms $\omv_\nu(f; \cdot,x)$ in
\eqref{omnu} and $\omw_\nu(f;\cdot,x)$ in \eqref{omWdef} satisfy
\be \bigl(\rho_\nu \omw_\nu(f;\cdot,x)\bigr)(y) = \omv_\nu(f;y,x)\,.\ee
\end{enumerate}
\end{lem}
\begin{proof}
For part~\eqref{rmph:Rq} we check in the definitions that for $t\geq 1$
\[ r_\nu(\inft;\orgn) = \Phi_{0,\nu}(k^{-1})=1\,,\qquad q_\nu( \am(t)
\, \orgn, \orgn) \;\sim\; t^{\rho+\nu}\,.\]
Hence $\rho_\nu q_\nu( \cdot,\orgn) (\inft) = r_\nu( \infty;\orgn)$.

Each point $(b,x)\in \left(\partial S \right)\times S$
can be written as $(g\,\inft,g\, \orgn)$, by taking $g=kan\in KAN$:
first choose $k\in K$ such that $k\,\inft=b$, \ and then choose
$an\in AN$ such that $an\,\orgn = k^{-1}\, x$. To complete the proof of
part \eqref{rmph:Rq} we use Theorem~\ref{thm-isoVW} and~\eqref{Rgg}:
\bad q_\nu ( \cdot,\cdot)\bigl( |g\times |g\bigr)(\am(t)\, \orgn,\orgn)
&= q_\nu( g \am(t)\, \orgn,g\, \orgn)\,,\\
\bigl( q_\nu(\cdot,\cdot) \bigr)\bigl( |_\nu g \times|g\bigr)
(\inft,\orgn) &= \bigl((\rho_\nu^{-1} r_\nu
)(\cdot;g\,\orgn)\bigr)(g\,\inft)\,,
\\
r_\nu(\cdot,\cdot) \bigl( |_\nu g\times |g)(\inft,\orgn) &=r_\nu( g\,
\inft;g\, \orgn)\,.
\ead

Part~\eqref{rmph:omrel} considers $(\ds-1)$-forms in the
variable~$x\in S$, with values in $\W\om\nu\left(\partial S\right)$ and
$\V\om\nu\left(\partial S\right)$,
respectively. The restriction map involves only the parameter $y$.
\end{proof}

A consequence of this result is an explicit description of the map from
cusp forms to cocycles. The image of this map has even a more precise
description.

\begin{prop}\label{prop-omfintW}Let
$\nu \in \CC\setminus \frac12\ZZ_{\leq - 1}$. The linear map
\il{btw}{$\btw_\nu$ }\be \label{btw}\btw_\nu : \A^0(\Gm;\nu)
\rightarrow H^{\ds-1}_\pb \bigl( \Gm;\W{\om^0,\infty}\nu\bigr)^\bdc \ee
assigns to $f\in \A^0(\Gm;\nu)$ the cohomology class of the cocycle
determined by
\be \label{Iintf}\X^\tess_{\ds-1} \ni X \mapsto I_X(f;y) = \int_{x\in X}
\omw_\nu(f;y,x)\,.\ee
\begin{enumerate}[label=$\mathrm{(\roman*)}$, ref=$\mathrm{\roman*}$]
\item\label{omfintW:i} The integral $I_X(f;y)$ is well-defined for
$y\in S \setminus X$ and satisfies
\begin{enumerate}[label=$\mathrm{(\alph*)}$, ref=$\mathrm{\alph*}$]
\item\label{omfintW:ia} For each $\gm\in \Gm$
\be\label{int-bd0W} \int_{x\in \gm^{-1} X} \omw_\nu(f; y,x)= \int_{x\in
X}\omw_\nu(f;\gm y,x)
\,. \ee
\item\label{omfintW:ib} If $X$ contains no cusps (equivalently, if
$X\in \X^{\tess,Y}_{\ds-1}$), then $b \mapsto I_X(f;b)$ represents an
element of $\W \om \nu\left(\partial S\right)$. If
$X\cap \Cu = \{\cu\}$, then $b \mapsto I_X(f;b)$ is in
$\W{\om^0,\infty}  \nu$ and represents an element of
$\W\om\nu \left( \left(\partial S\right)\setminus \{\cu\} \right)$.
\item \label{omfintW:ic} For all $X\in \X^\tess_{\ds-1}$ the integral is
in $\E_\nu(S\setminus X)$.
\end{enumerate}
\item\label{omfintW:ii}
Let $Z\in \X_\ds^\tess$ have boundary
$\partial Z = \sum_{X\in B( Z)} \e_X\,X$ with $\e_X\in \{1,-1\}$
oriented by the outward pointing normals. Then the function
\be \label{bd-eq}
y \mapsto \sum_{X\in B(Z)} \int_{x\in \e_X X} \omw_\nu(f;y,x) =
\begin{cases} 0 &\text{ if }x\in S\setminus Z\,,\\
2\nu \,c(\nu)\,f(x) &\text{ if }x\in \mathring Z\,.
\end{cases}
\ee
on $S \setminus \partial Z $ represents the zero $\nu$-analytic boundary
germ. By $\mathring Z$ we denote the interior of $Z$.
\end{enumerate}
\end{prop}
\begin{proof}
Cells $X\in \X^{\tess,Y}_{\ds-1}$ are compact, and the integration over
$X$ poses no convergence problems. By Proposition~\ref{prop-Q}
\eqref{Q:nu-anal} the integrand represents an analytic boundary germ.
It is defined on $S \setminus X$. If the cell $X\in \X^\tess_{\ds-1}$
contains a cusp $\cu$, then the convergence has to be checked with a
limit procedure, like in the discussion leading up to
Proposition~\ref{prop-omfint}. Then the integral represents a
$\nu$-analytic boundary germ on
$\left(\partial S\right) \setminus \{\cu\}$.
This gives part (i)(b). Part (i)(c) is clear, and part (i)(a) follows
from~\eqref{omW-trf}.\medskip

Part \eqref{omfintW:ii} follows directly from Proposition~\ref{prop-psC}
if $Y \in \X^{\tess,Y}_\ds$. If $Y$ contains a cusp $\cu$ we first
apply Proposition~\ref{prop-psC} to a truncated version of $Y$, and
then use a limit procedure based on the quick decay at $\cu$ of $f$ and
its derivatives.
\end{proof}

For $(\ds-1)$-cells $X$ containing a cusp $\cu$ the integral in
\eqref{Iintf} represents a $\nu$-analytic boundary germ on
$(\partial S)\setminus\{\cu\}$. This implies that for
$k\,\infty\neq \cu$ the integral $h=I_X(f;\cdot)$ satisfies
$h\bigl(k\am(t)\,\orgn) = \oh(t^{-\rho-\nu})$ as $t\uparrow\infty$. We
need also to understand how $h(z)$ behaves as $z\rightarrow \cu$
through $S\setminus X$.

We use the norm on $\RR^{p+q}$ given by
\il{nrmN}{$\|\cdot\|$}$\|\xi\|^2 = \sum_{j=1}^{p+q}\xi_j^2$, and recall
that $A_Y^+= \bigl\{ \am(t)\;:\; t\geq Y\bigr\}$.

\begin{prop}\label{prop-eh}
Let $\nu \in \CC \setminus \frac 12\ZZ_{\leq -1}$, $\re\nu>-\rho$, and
suppose that $X\in \X^\tess_{\ds-1}$ contains the cusp $\cu$. Then
there is a $(\ds-2)$-dimensional compact subset $C\subset N$ such that
$X\cap S= g_\cu C A_Y^+$. For $\e>0$ let
\il{Deps}{$D_\e$}$D_\e = \bigl\{ \nm(\xi) \;:\; 
\xi\in \RR^{p+q},\; \|\xi-x\|< \e${ for all }$x\in \RR^{p+q}$\text{
such that }$ \nm(x) \in C \bigr\}$.

For each cusp form $f\in \A^0(\Gm;\nu)$ the integral $h=I_X(f;\cdot)$ in
\eqref{Iintf} satisfies the following bounds.
\begin{enumerate}[label=$\mathrm{(\alph*)}$, ref=$\mathrm{\alph*}$]
\item \label{eh:polc} The function $h$ is bounded on
$S \setminus g_\cu D_\e  A^+_{a} $ for each $\e>0$ and for each
$a\in (0,Y)$.
\item\label{eh:vertl}
$h(g_\cu n \am(t) \, \orgn) = \oh\bigl(t^{-\rho-\nu}\bigr)$ as
$t\uparrow\infty$, uniformly for $n$ in compact sets contained in
$N \setminus C$
\end{enumerate}
\end{prop}
We note that if the cusp form $f$ is non-zero, then
Proposition~\ref{prop-cusp-sp} implies that $\nu$ satisfies the
stronger condition $\rho^2-\nu^2\geq0$. It is worthwhile to keep $\nu$
as general as we can. Then we can use Theorem~\ref{thm:main1} to conclude
that certain cohomology spaces are zero if $\nu$ does not satisfy
$\rho^2-\nu^2\geq 0$.

The proof takes the remainder of this subsection.

\rmrk{Transformation to \intitle{\inft} and notation for coordinates}
Without loss of generality we can work with the cusp~$\inft$. Then $X$
has in horospherical coordinates the description $C \times [Y,\infty)$
for some compact part $C$ of a $(\ds-2)$-dimensional subvariety of $N$.
We use $w=\nm(x)\am(y)\,\orgn$ for the integration variable in
$I_X(f,\cdot)$.
For $h$ we use $ h\bigl( \nm(\xi) \am(t) \,\orgn)$.

\rmrk{The kernel function}The function $q_\nu$ (introduced in
Proposition~\ref{prop-Q})
satisfies
\be q_\nu(g_1\,\orgn, g_2\,\orgn) = Q_{\nu}(g_1^{-1}g_2) \,,\ee
and depends only on $\dist( g_1\,\orgn,g_2\orgn)$. The function
$Q_{\nu}\bigl(k_1\am(t)k_2\bigr)$ is equal to $t^{-\rho-\nu}$ times an
analytic function of $t^{-2}$ on a neighborhood of
$\frac 1t\in (-1,1)$. If the distance is large, then
\badl{qnuas}q_\nu\bigl(z;w\bigr)&= e^{-(\rho+\nu)\dist(z,w)}\,
A_\nu\bigl( e^{-2\dist(z,w)}\bigr)\\
&=
e^{-(\rho+\nu)\dist\bigl(z,w\bigr)}\Bigl(1+\oh(e^{-2\dist(z,w)})\Bigr)\qquad
\text{as } \dist(z,w) \rightarrow\infty\,, \eadl
with an analytic function $A_\nu$ on a neighborhood of $0$ in $\CC$,
that satisfies $A_\nu(0)=1$.

As $t\downarrow 1$ the function $Q_\nu(\am(t))$ is $(t-1)^{2-d}$ times
an analytic function in $t-1$. So the kernel function $q_\nu$ can be
estimated by $\dist(z,w)^{1-c }$ with $c=\ds-1$. In the case $\ds=2$
there is also a logarithmic term in the expansion as $t\downarrow 0$.
Then we can use any positive value of $c$ in the exponent.

\rmrk{First derivative of the kernel function}In the integral
representation
\be \label{h-int} h(z)= \int_{w\in X}\eta\bigl( f(w), q_\nu(z,w\bigr)
\bigr)\ee
not only the function $q_\nu(z,w)$ occurs, but also its first
derivatives with respect to $w$. See~\eqref{etadef}.

\begin{lem}\label{lem-ek}
Let $x\in \RR^{ p+q}$ and $y>0$, and let $z\neq \nm(x)\am(y)\,\orgn$. Then
\bad
\partial_{x_j} \dist\bigl(z,\nm(x)\am(y)\bigr)
&\ll y^{-1} + \bigl( 1+\|x^{(1)}\|\bigr)y^{-2}&\quad& 1\leq j \leq
p\,,\\
\bigl| \partial_{x_j}\dist\bigl(z,\nm(x)\am(y)\bigr)\bigr| &\leq
y^{-2}&&j\geq p+1\,,\\
\bigl| \partial_{y}\dist\bigl(z,\nm(x)\am(y)\bigr)\bigr| &\leq y^{-1}\,.
\ead
\end{lem}
\begin{proof}
Let $z\neq \orgn$. The  tangent  vectors $\bigl( \partial_{x_j}\bigr)_\orgn$ and
$\bigl(\partial _y\bigr)_\orgn$  at $\orgn$
form an orthonormal basis of
$\Ta\orgn S$. Let $v_z\in \Ta\orgn S$ be the unit vector pointing to
$z$. Then
\[\partial_y\dist(z,\nm(x)\am(y)\,\orgn)\bigr|_{x=0,y=1}= -\rmetric\bigl(v_z, (\partial_y)_\orgn\bigr) \in [-1,1]\,.\]
The same holds for $(\partial_{x_j}\bigr)_\orgn$.

For general $z\neq \nm(x)\am(y)\,\orgn$ the derivatives $m_j$ in
\eqref{mj} of the distance function have also values in $[-1,1]$. With
the inverse relations in \eqref{ir} we arrive at the estimates in the
lemma.
\end{proof}

\rmrk{General approach to the estimations} In the integral
\be\label{hintX} h(z) = \int_{(n,y)\in C\times[Y,\infty)} \eta\bigl(
f(n\am(y)\,\orgn), q_\nu(z,n \am(y)\,\orgn)\bigr)\ee
the integrand is linear in $f$ and its first derivatives in
horospherical coordinates, and also linear in $q_\nu(z, \cdot)$ and its
first derivatives. Like in the proof of Proposition~\ref{prop-omfint}
we use a bound \il{CfN}{$C_{f,N}$}$C_{f,N} \,y^{-N}$ for the influence
of $f$ for each large $N$. To keep track of uniformity of estimates we
do not use $\oh(\cdot)$ and $\ll$, but indicate the constants
explicitly.

For each $z\in S\setminus X$ there exists
\il{w0}{$w_0=w_0(z)$}$w_0=w_0(z) \in X$ such that
$\dist(z,w_0) \leq \dist(z,w)$ for all $w\in  X $. We
put\il{Bzy}{$B(z,y)$}
\be \label{Bzy} B(z,y) = C_B\cdot\begin{cases} \sup_{n\in
C}e^{-(\rho+\nu)\dist( z,n\am(y)\,\orgn)}& \text{ if } \dist(z,w_0)
\geq 1\,,\\
\sup_{n\in C} \dist(z, n\am(y)\,\orgn)^{-c}&\text{ if }\dist(z, w_0)
<1\,,
\end{cases}\ee
with $c=\ds-1$ if $\ds\geq 3$, and $c>1$ if $\ds=2$. (The choice of $c$
takes into account the first order derivatives of the kernel function.)
The constant \il{CB}{$C_B$}$C_B$ takes into account the implicit
constants in the estimates of the kernel function, and depends on
$\re\nu$, on the volume of $C$, and on constants in the application
of~\eqref{etadef}. In this way we can estimate the integral
in~\eqref{hintX} by
\be \label{hint} \bigl|h(z) \bigr|\leq C_{f,N} Y^{-N} \int_{y=Y}^\infty
B(z,y)\, dy\,.\ee
The factor $Y^{-N}$ arises from the influence of the function $f$.

\rmrk{On the distance function}The triangle inequality implies that
\[ \dist\bigl(\nm(\xi)\am(t)\,\orgn,\, \nm(x)\am(y)\,\orgn\bigr)
\leq \dist\bigl( \nm(\xi) \am(t)\,\orgn,\, \nm(\xi)\am(y)\,\orgn\bigr) +
\dist\bigl( \nm(\xi)\am(y)\,\orgn,\,\nm(x)\am(y)\,\orgn\bigr)\,.\]
The first term on the right is easy. With \eqref{distln}
\be \dist\bigl( \nm(\xi) \am(t)\,\orgn,\, \nm(\xi)\am(y)\,\orgn\bigr) =
\bigl|\log(y/t)\bigr|\,.\ee

The `horizontal distance' in the second term is the length of a path
that we do not know precisely. This path can be parametrized as
\[ p(\tau) = \nm\bigl(\Xi(\tau)\bigr)\am\bigl(\eta(\tau)\bigr)\,\orgn\]
where $\Xi(\tau)= x+\tau(\xi-x)$, and where $\eta$ is a $C^\infty$
function $[0,1]\rightarrow (0,\infty)$ that satisfies
$\eta(0) = \eta(1) = y$. The length of this path is given by the
integral in \eqref{length}, which with help of \eqref{Rmat} can be
written as
\be\label{dst-int} \int_{\tau=0}^1 \sqrt{
\frac{\|\xi^{(1)}-x^{(1)}\|^2}{\eta(\tau)^2} + \frac{
Q(\Xi(\tau)^{(1)})[\xi-x]}{\eta(\tau)^4} +
\frac{\eta'(\tau)^2}{\eta(\tau)^2 } } \,d\tau\,.\ee
We denote by $Q(x)[\cdot]$
the positive definite quadratic form on $\RR^{p+q}$ with $p\times q$
block matrix
\be \begin{pmatrix} C(x^{(1)}) C(x^{(1)})^t& - C(x^{(1)})\\-C(x^{(1)})^t
& I_q\end{pmatrix}\,. \ee
The matrix elements are at most quadratic in the coordinates of
$x^{(1)}$. If $q=0$, then the middle term under the square root
in~\eqref{dst-int} vanishes.\smallskip

The integral in \eqref{dst-int} leads directly to the following bounds.
\begin{lem}\label{lem-dilub}Let $\xi, x\in \RR^{p+q}$, $y\geq 0$,
then\il{bsl}{$b_l(\xi,x), \, b_s(\xi,x)$}
\bad \dist\bigl(\nm(\xi)\am(y)\,\orgn,\,\nm(x)\am(y)\,\orgn\bigr)
& \leq \bigl( y^{-1} + y^{-2} b_l(\xi,x) \bigr) \bigl\|
\xi-x\bigr\|\,,\\
&\geq \bigl( \eta_m ^{-1} + \eta_m^{-2} \, b_s(\xi,x) \bigr)
\bigl\|\xi-x\bigr\|\,,
\ead
where
\bad b_l(\xi,x) &= \max_{0\leq \tau\leq1} \sup_{\xi\neq 0}
\frac{\bigl|Q(\Xi(\tau)^{(1)})[\xi-x]\bigr|}{\|\xi-x\|}\,,\\
b_s(\xi,x) &= \min_{0\leq \tau\leq1} \inf_{\xi\neq 0}
\frac{\bigl|Q(\Xi(\tau)^{(1)})[\xi-x]\bigr|}{\|\xi-x\|}\,,
\ead
and where $\eta_m = \max _{0\leq \tau\leq 1} \eta(\tau)$ for the
function $\eta$ describing the geodesic segment between
$\nm(\x)\am(y)\,\orgn$ and $\nm(x)\am(y)\,\orgn$.
\end{lem}

\begin{lem}\label{lem-d3e}
With the notations introduced above:
\begin{align*}
\dist\bigl( \nm(\xi) &\am(t)\,\orgn,\nm(x)\am(y) \,\orgn\bigr)\\
&\geq \max\Bigl( \bigl|\log(y/t) \bigr|, 2\log(\eta_m/y) , \,
\|\xi-x\|\, \bigl( \eta_m^{-1}+ b_s(\xi,x) \eta_m^{-2} \bigr)\Bigr)\,.
\end{align*}
\end{lem}
\begin{proof}The triangle inequality implies that
$\dist\bigl( \nm(\xi) \am(t)\,\orgn,\nm(x)\am(y) \,\orgn\bigr) $ is
larger than $
\dist\bigl(\nm(\xi) \am(t)\,\orgn, \nm(\xi)\am(y) \,\orgn\bigr)
= \bigl|\log(y/t)\bigr|$ and also larger than
$ \dist\bigr(\nm(\xi)\am(y) , \nm(x)\am(y)\bigr)$. In the shortest path
realizing the latter distance, the function $\eta$ goes up from $y$ to
$\eta_m$ and then down again to~$y$. Working only with the last term
under the square root in~\eqref{dst-int} we get the second lower bound.
Lemma~\ref{lem-dilub} gives the third lower bound.
\end{proof}

\begin{lem}\label{lem-estd}We use the notations discussed above. Let
$w=\nm(x)\am(y)\,\orgn$ with $y\geq Y$, and let
$z=\nm(\xi)\am(t)\,\orgn$.
\begin{enumerate}[label=$\mathrm{(\alph*)}$, ref=$\mathrm{\alph*}$]
\item \label{estd:a} If
$\dist( \nm(\xi)\am(t)\,\orgn, \nm(x) \am(y)\,\orgn) \leq 1$, then
\be \label{esth}
e^{-1}t \leq y \leq e t\quad\text{ and} \quad \eta_m \geq \|\xi-x\|
\,(1+e^{-1/2} b_s(\xi,x)/y\bigr) \,. \ee
\item \label{estd:b} If the conditions in~\eqref{esth} hold, then
\be \dist( \nm(\xi)\am(t)\,\orgn, \nm(x) \am(y)\,\orgn) \geq e^{-1/2}
\|\xi-x\| y^{-1}\,. \ee
\end{enumerate}
\end{lem}
\begin{proof}The assumption that the distance is at most $1$ implies by
Lemma~\ref{lem-d3e}
\[ e^{-1}t \leq y \leq et\,,\quad \eta_m \leq e^{1/2}y\,, \quad \eta_m
\geq \|\xi-x\|\bigl(1+ b_s(\xi,x)/\eta_m\bigr) \,. \] This implies
directly the first two statements in \eqref{esth}. Furthermore, with
$\eta_m \leq e^{1/2} y$ we get
\[ \eta_m \geq \|\xi-x\| \bigl( 1+ b_s(\xi,x) /\eta_m \bigr) \geq
\|\xi-x\|\,\bigl( 1+ e^{-1/2} b_s(\xi,x) y^{-1}\bigr) \,.\]
This implies \eqref{esth}.

With Lemma~\ref{lem-d3e}
\[\dist( \nm(\xi)\am(t)\,\orgn, \nm(x) \am(y)\,\orgn) \geq \|\xi-x\|
\eta_m\,. \]
If the conditions in~\eqref{esth} hold, then we proceed with
\[ \geq\|\xi-x\| (e^{1/2} y)^{-1}\,.\qedhere \]
\end{proof}
\medskip

\begin{proof}[Proof of part \eqref{eh:polc} of
Proposition~\ref{prop-eh}] Let $z = \nm(\xi)\am(t)\,\orgn \in D_\e$ for
a given $\e>0$, with $D_\e$ as defined in the proposition.

Suppose first that the distance $\dist(z,w) \geq 1$ for all $w\in X$.
Then we have in \eqref{Bzy}
\[ B(z,y) \leq C_B\, 1 = C_B\,.\]
With \eqref{hint} we have for each $N\geq 2$
\be \bigl|h(z) \bigr| \leq C_{f,N} \int_{y=Y}^\infty y^{-N} \, C_B\, dy
= \frac{Y^{1-N}}{N-1} \, C_{f,N}\, C_B\,.\ee

Suppose next that the minimal distance $\dist(z,w_0(z)\bigr)$ of $z$ to
$X$ is smaller than~$1$. By Lemma~\ref{lem-estd} this is relevant for
$t/e\leq y \leq et$, and
\[ \dist(z, \nm(x)\am(y)\,\orgn) \geq e^{-1/2} \,\bigl\|\xi-x\bigr\| \,
y^{-1} \geq e^{-1/2} \,\e y^{-1}\,.\]
In \eqref{Bzy} we can take
\be B(z,y) \geq C_B \, e^{c/2}\, \e^{-c}\, y^c\,.\ee
We get the following contribution to a bound for $h(z)$:
\be \label{intpt} C_{f,N} \int_{y=\max(Y,t/e)}^{et} y^{-N}\, C_B \,
e^{c/2} \, \e^{-c}\, y^c \,dy \leq C_{f,N} C_B e^{c/2}
\e^{-c}\,\frac{Y^{1+c-N}}{N-c+1} \ee
for $N>2+c$. (Since $1+c-N<0$ we could replace $\max(Y,t/e)$ by $Y$.)
On the remainder of $[Y,\infty)$ we integrate $C_{f,N} y^{-N} C_B$. This
gives
\be \bigl| h(z) \bigr| \leq C_{f,N}\, C_B \Bigl( \frac{Y^{1-N}}{N-1} +
e^{c/2} \e^{-c} \frac{Y^{1+c-N}}{N-c-1}\Bigr)\,,\ee
for any choice of $N>c+1$.

Hence $h(z)$ is bounded on $S\setminus D_\e A^+_{Y}$.\end{proof}

\begin{proof}[Proof of part \eqref{eh:vertl} of
Proposition~\ref{prop-eh}] We consider
$z=\nm(\xi) \am(t) \,\orgn \not\in X$ with $t\geq eY$.

If the minimal distance $\dist(z,w_0) \leq 1$, then this happens for
$t/e\leq y \leq et$. We put $\e=\min_{x\in C} \|\xi-x\|$. Like in
\eqref{intpt} this gives the contribution
\be \label{e-intm}C_{f,N} C_B e^{c/2}
\e^{-c}\,\frac{(t/e)^{1+c-N}}{N-c+1}\ee
to a bound of $h(z)$, where we can take any $N>2+c+ \rho+\re\nu$.

For $y\in [Y,\infty] \setminus [t/e,et]$ we use
\be B(z,y) = C_B\, e^{-(\rho+\re\nu)\bigl|\log y/t \bigr|}\,.\ee
In this way we arrive at the following contributions, with
$N > 1+\rho+\re\nu$.
\begin{align*} \int_{y=et}^\infty &C_{f,N} y^{-N} \, C_B\,
(y/t)^{-\rho-\re\nu}\, dy \leq C_{f,N} C_B\, t^{\rho+\re\nu}
\frac{(et)^{1-N-\rho-\re\nu}}{N+\rho+\re\nu-1}\\
&\ll  e^{1-N-\rho-\re\nu} C_{f,N}\, C_B \, t^{1-N}\,,\displaybreak[0]\\
\int_{y=Y}^{t/e} &C_{f,N}\, C_B\, y^{-N+\rho+\re\nu}\,
t^{-\rho-\re\nu}\, dy\\
&\leq e^{N-1-\rho-\re\nu}\, C_{f,N}\, C_B\,
\frac{Y^{1-N-\rho-\re\nu}}{N-\rho-\re\nu-1}\, t^{-\rho-\nu}\,.
\end{align*}
Combining these estimates with \eqref{e-intm}, we obtain the bound
in~\eqref{eh:vertl}.\end{proof}

\subsection{Spaces of \intitle{\nu}-semi-analytic boundary
germs}\label{sect-spacesgerms}
\begin{defn}By a \il{Siedm}{Siegel domain}Siegel domain for $\Gm$ at the
cusp $\cu$ we mean a set of the form\il{Sie-cu}{$\Sie_\cu(C,T)$}
$ \Sie_\cu(C,T) = g_\cu C  A^+_T \subset S$ with some compact subset
$C\subset N$ and $ A^+_T= \bigl\{ \am(t)\;:\; t\geq T \bigr\}$ for some
$T>0$.
\end{defn}

For $\gm\in \Gm$ the translate $\gm \, \Sie_\cu(C,T)$ is a Siegel domain
at the cusp $\gm \,\cu$. For each $X\in \X^\tess_{\ds-1}$ containing
the cusp $\cu$, the set $X\cap S$ is contained in a suitably chosen
Siegel domain at~$\cu$.

\begin{defn}\label{exc-nbh}Let $E$ be a finite set of cusps of $\Gm$. We
call a neighborhood $U$ of $\left( \partial S\right) \setminus E$ in
$\hat S$ an \il{excnbh}{excised neighborhood}\emph{excised
neighborhood} of $\left( \partial S\right) \setminus E$ if it is of the
form
\be\label{Uen}
U' \setminus \bigcup_{\cu\in E} \Bigl( \Sie_\cu (C_\cu,T_\cu) \cup j
\Sie_\cu(C_\cu,T_\cu)\Bigr)
\ee
with a usual neighborhood $U'$ of $ \partial S$ in $\hat S$, and Siegel
domains (as defined above) that depend on the cusp.
(We recall that $j$ is the involution in $\hat S$ that fixes
$\partial S$ and interchanges $S$ and~$S^-$.)
\end{defn}

An excised neighborhood is a standard neighborhood of $\partial S$ from
which we have cut out relatively small sets in $S$ and $S^-$. See
Figure~\ref{fig-sau1} for an illustration in case $\ds=2$. (Adapted
from \cite[Figure 9.1]{BLZ15}.
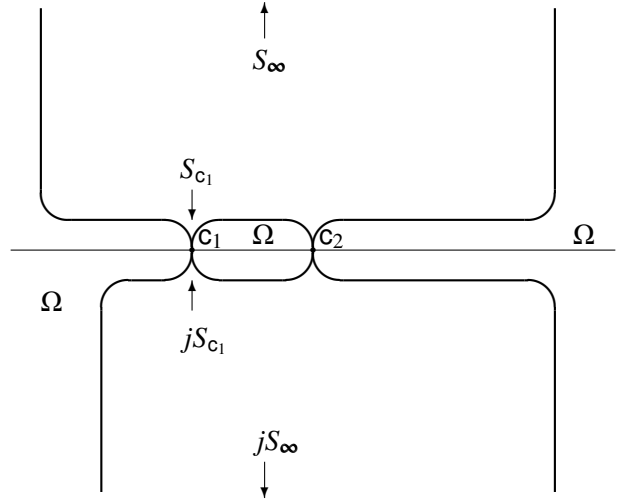
\begin{figure}[ht]
\[\setlength\unitlength{.8cm}
\begin{picture}(10,8)(0,-4)
\put(0,0){\line(1,0){10}}
\put(3.1,.1){$\cu_1$}
\put(5.1,.1){$\cu_2$}
\put(4,.1){$\Om$}
\put(.5,-1){$\Om$}
\put(9.3,.1){$\Om$}
\put(3,0){\circle*{.1}}
\put(5,0){\circle*{.1}}
\put(3,1){\vector(0,-1){.5}}
\put(2.8,1.2){$S_{\!\cu_1}$}
\put(3,-1){\vector(0,1){.5}}
\put(2.8,-1.6){$jS_{\!\cu_1}$}
\put(4,3){$S_{\!\inft}$}
\put(4.2,3.5){\vector(0,1){.6}}
\put(4,-3.3){$jS_{\!\inft}$}
\put(4.2,-3.5){\vector(0,-1){.6}}
\thicklines
\put(2.5,0){\oval(1,1)[r]}
\put(3.5,0){\oval(1,1)[l]}
\put(4.5,0){\oval(1,1)[r]}
\put(5.5,0){\oval(1,1)[l]}
\put(2,-1){\oval(1,1)[tl]}
\put(3.5,.5){\line(1,0){1}}
\put(3.5,-.5){\line(1,0){1}}
\put(2,-.5){\line(1,0){.5}}
\put(1.5,-1){\line(0,-1){3}}
\put(1,1){\oval(1,1)[lb]}
\put(1,.5){\line(1,0){1.5}}
\put(.5,1){\line(0,1){3}}
\put(8.5,1){\oval(1,1)[rb]}
\put(9,1){\line(0,1){3}}
\put(5.5,.5){\line(1,0){3}}
\put(8.5,-1){\oval(1,1)[rt]}
\put(9,-1){\line(0,-1){3}}
\put(5.5,-.5){\line(1,0){3}}
\end{picture}
\]
\caption{An excised neighborhood~$U$ of $\partial S
\setminus\{\cu_1,\cu_2,\inft\}$ in the case $\ds=2$. By $S_{\!\cu}$ we
indicate the Siegel domain $\Sie_\cu(C_\cu,T_\cu)$ in
Definition~\ref{exc-nbh}. }\label{fig-sau1}
\end{figure}

In the context of $\PSL_2(\RR)$ and of the embedding
$\uhp \subset \proj\CC$ the name \emph{rounded neighborhood} is better
than the name \emph{excised neighborhood}; see \cite[Figures 9.1 and
9.2]{BLZ15}. Here we keep the older terminology.

\begin{defn}\label{def-Woe}\il{Woexi}{$\W{\om^0,\exc}\nu\supset \W{\om^0,\excg}\nu
\supset \W{\om^0,\infty,\excg}\nu$}

We define the following $\Gm$-modules of semi-analytic boundary germs
contained in $\W{\om^0}\nu$.
\begin{enumerate}[label=$\mathrm{(\roman*)}$, ref=$\mathrm{\roman*}$]
\item\label{Woe:exc} $\W{\om^0,\exc}\nu$ is the submodule of those
boundary germs $[u]\in \W{\om^0}\nu$ that can be represented by a
function $u\in \E_\nu\bigl( U \cap S\bigr)$ for an excised neighborhood
$U$ of $(\partial S)\setminus E$ for some finite set $E$ of cusps.

\item \label{Woe:excg}$\W{\om^0,\excg}\nu$ is the submodule of the
boundary germs $[u]\in \W{\om^0,\exc}\nu$ represented by
$u\in \E_\nu\bigl( U \cap S\bigr)$ as in \eqref{Woe:exc}, that satisfy
\begin{enumerate}[label=$\mathrm{(\alph*)}$, ref=$\mathrm{\alph*}$]
\item \label{Woe:bdd}$u$ is bounded on $U \cap S$,
\item
\label{Woe:vertl}$u\bigl( g_\cu n \am(t)\,\orgn) = \oh(t^{-\rho-\nu})$
as $t\uparrow\infty$ uniformly for $n$ in compact sets such that
$g_\cu n \am(t) \,\orgn \in U$ for all sufficiently large values
of~$t$.
\end{enumerate}
\item \label{Woe:inf}$\W{\om^0,\infty,\excg}\nu$ is the submodule of the
$[u]\in \W{\om^0,\excg}\nu$ such that $\rho_\nu [u] \in \V{\om^0}\nu$
extends as an element of $\V\infty\nu(\partial S)$.
\end{enumerate}\il{Voexi}{$\V{\om^0,\exc}\nu\supset \V{\om^0,\excg}\nu
\supset \V{\om^0,\infty,\excg}\nu$} The restriction map gives
corresponding $\Gm$-modules
$\V{\om^0,\ast}\nu = \rho_\nu \W{\om^0,\ast}\nu$.
\end{defn}

We use Definition~\ref{def-bdsing} of boundary singularities for these
modules as well. For $h\in \W{\om,\infty,\excg}\nu$ the assertion
$\bsing_\nu(h)=\emptyset$ is equivalent to the statement
$f\in \W\om\nu(\partial S)$.

Propositions \ref{prop-omfintW} \deqref{omfintW:i}{omfintW:ib}
and~\eqref{prop-eh} imply that integration of cusp forms leads to
elements of $\W{\om^0,\excg}\nu$.
\begin{prop}
\label{prop-btimg}Let $\nu \in \CC$ satisfy the conditions
$\nu \not\in \frac12\ZZ_{\leq-1}$ and $\re\nu>-\rho$. Then the linear
map in \il{btwe}{$\btw_\nu$}\eqref{btw} is a linear map
\be\label{btimg} \btw_\nu: \A^0(\Gm;\nu) \rightarrow H^{\ds-1}_\pb\bigl(
\Gm; \W{\om^0,\infty,\excg}\nu\bigr)^{\bdc}\,.\ee
\end{prop}

\subsubsection{Recapitulation of
\intitle{\Gm}-modules}\label{sect9-recap}
We have introduced two sheaves on $\partial S$: The sheaf $\V\om\nu$ of
real-analytic functions in Definition \ref{def-Vomsh}, and the sheaf of
$\nu$-analytic boundary germs $\W\om\nu$ in Definition~\ref{Womdef}.
With these sheaves we defined the $\Gm$-modules $\V{\om^0}\nu$ and
$\W{\om^0}\nu$, in \eqref{Vom0} and \eqref{Wom0}.
Theorem~\ref{thm-isoVW} gives, for $\nu \not\in \frac12\ZZ_{\leq -1}$,
an isomorphism of $G$-equivariant sheaves
$\rho_\nu:\W\om\nu \rightarrow \V\om\nu$. We have defined several more
modules, partly under the assumption
$\nu \not\in \frac12\ZZ_{\leq -1}$. We give an overview in
Table~\ref{tab-recap}.
\begin{table}[ht]
\[ \xymatrix{ \V{\om^0}\nu
&& \W {\om^0}\nu\ar[ll]^{\rho_\nu}_\cong & \text{\eqref{Vom0},
\eqref{Wom0}}\\
\V{\om^0,\exc}\nu = \rho_\nu \W{\om^0,\exc} \nu \ar@{^(->}[u]&&
\W{\om^0,\exc}\nu \ar[ll]^{\rho_\nu}_\cong \ar@{^(->}[u] &
\text{Defn.~\ref{def-Woe}}\\
\V{\om^0,\excg}\nu = \rho_\nu \W{\om^0,\excg} \nu \ar@{^(->}[u]&&
\W{\om^0,\excg}\nu \ar[ll]^{\rho_\nu}_\cong \ar@{^(->}[u] &
\text{Defn.~\ref{def-Woe}}\\
\V{\om^0,\infty,\excg}\nu = \V{\om^0\excg}\nu\cap \V\infty\nu(\partial
S)
\ar@{^(->}[u] \ar[rr]_{\rho_\nu^{-1}}^= && \W{\om^0,\infty,\excg}\nu
\ar@{^(->}[u]
&\text{Defn.~\ref{def-Woe}}\\
\V\om\nu(\partial S) \ar@{^(->}[u]
&& \W \om\nu(\partial S) \ar[ll]^{\rho_\nu}_\cong \ar@{^(->}[u]\\
} \]
\caption{Overview of $\Gm$-modules, for
$\nu \in \CC\setminus \frac12\ZZ_{\leq -1}$.}\label{tab-recap}
\end{table}

%% file: ccro2-10-coh-cf.tex

\bigskip

\def\flnm{ccro2-10-coh-cf}

\section{From cocycles to cusp forms} \label{sect10}

Proposition~\ref{prop-btimg} shows that the image
$\btw_\nu \A^0(\Gm;\nu)$ is contained in the cohomology space
$H^{\ds-1}_\pb\bigl( \Gm; \W{\om^0,\infty,\excg}\nu\bigr)^{\bdc}$. We
turn to the construction of a linear map in the opposite direction.

\begin{thm}\label{thm-alW}Let $\nu\in \CC\setminus\frac12\ZZ_{\leq0}$,
$|\re\nu|<\rho$.

There is a linear map \il{alwnu}{$\alw_\nu$}
\be \alw_\nu: H^{\ds-1}_\pb\bigl( \Gm; \W{\om^0,\excg}\nu\bigr)^\bdc
\rightarrow \A^0(\Gm;\nu)\ee
such that $\alw_\nu \circ \btw_\nu$ is the identity on $\A^0(\Gm;\nu)$,
for the linear map $\btw_\nu$ in Proposition~\ref{prop-btimg} composed
with the natural map
\[H^{\ds-1}_\pb\bigl( \Gm;
\W{\om^0,\infty,\excg}\nu\bigr)^{\bdc}\rightarrow H^{\ds-1}_\pb\bigl(
\Gm; \W{\om^0,\excg}\nu\bigr)^{\bdc}\,.\]
\end{thm}
This theorem is a direct consequence of Propositions \ref{prop-cohE},
\ref{prop-cE}, and~\ref{prop-pi}.

\subsection{Modules of functions on \intitle{S} representing
\intitle{\nu}-analytic boundary germs}\label{sect10-reprs}
We recall that semi-analytic boundary germs $[u]\in\W{\om^0,\excg}\nu$
are represented by functions $u\in \E_\nu( U \cap S)$ for some open
neighborhood $U$ in $\hat S$ of
$(\partial S) \setminus \bsing_\nu \bigl([u]\bigr)$. The function
$t \mapsto u\bigl( k \am(t) ,\orgn)$ should extend analytically as a
function of $t^{-1}$ on a neighborhood of $ t^{-1}=0$ 
for those $k\in K$ 
for which
$k\,\inft \not\in \bsing_\nu([u])$, 
in such a way that
there is an analytic function $A_u$ on $U$ such that
$u\bigl(k\am(t) \,\orgn\bigr) 
= t^{-(\rho+\nu) }\, A_u\bigl( k \am(t)\,\orgn)$ for all
$k\am(t)\,\orgn \in U \cap S$. The relation $\Dt u = (\rho^2-\nu^2) u$
on $U\cap S$ is equivalent to $\tilde\Dt_\nu A_u=0$ on $U$, with a
differential operator $\tilde\Dt_\nu$ that we described in terms of
normalized horospherical coordinates in~\eqref{deqAn}.

In \eqref{exph} we saw an example in which we can choose the open set
$U\subset \hat S$ in such a way that it contains~$S$. For
$[u]\in \W{\om^0,\excg}\nu$ the set $U$ is fairly large. Then $U$ is an
excised neighborhood of
$(\partial S) \setminus \bsing_\nu\bigl([u]\bigr)$,
for which $U\cap S$ is the union of a compact set and a finite number of
Siegel domains. For our purpose it is convenient to work with
representatives of $\nu$-analytic boundary germs that are defined on
the whole symmetric space~$S$.

\begin{defn}\label{def-GN}
Let $\G{\om^0,\excg}\nu$ be the space of $f\in C^\infty(S)$ for which
there are a finite set $E\subset \Cu$ and an excised neighborhood $U$
of $(\partial S) \setminus E$ such that
\begin{enumerate}[label=$\mathrm{(\alph*)}$, ref=$\mathrm{\alph*}$]
\item $f$ represents a $\nu$-analytic boundary germ in
$\W\om\nu\left( (\partial S) \setminus E\right)$,
\item \label{GN:bdd}$f$ is a bounded function on~$S$,
\item \label{GN:dac} for each $\cu\in E$
\be \label{def-GN-gc}
f\bigl( g_\cu n \am(t) \, \orgn\bigr) = \oh \bigl(
t^{-(\nu+\rho)}\bigr)\qquad \text{ as } t\uparrow\infty \ee
uniformly for $n$ in compact sets in $N$.
\end{enumerate}

Further we define
\begin{align*} \G\om\nu &= \bigl\{ f\in C^\infty(S) \;:\; \text{$f$
represents an element of $\W\om\nu(\partial S)$}\bigr\}\subset
\G{\om^0,\excg}\nu \,,\\
\N {\om^0,\excg}\nu &= \bigl\{ f\in \G{\om^0,\excg}\nu\;:\; \text{$f$
represents the zero boundary germ}\bigr\}\,,\\
\N\om\nu &= \G\om\nu\cap \N{\om^0,\excg}\nu\,.
\end{align*}
\il{Greprnw}{$\G\om\nu \subset \G{\om^0,\excg}\nu$}
\il{Nreprnw}{$\N\om\nu \subset \N{\om^0,\excg}\nu$}
\end{defn}
So $\G\om\nu$ is the case $E=\emptyset$ in the definition of
$\G{\om^0,\excg}\nu$. The support of elements of $\N{\om^0,\excg}\nu$
is the union of a compact subset of $S$ and finitely many Siegel
domains. The space $\N{\om^0,\excg}\nu$ depends on $\nu$, since $\nu$
occurs in the growth conditions at the cusps in $E$. The space
$\N\om\nu $ is equal to $C_c^\infty(S)$, and it does not depend
on~$\nu$.

\begin{prop}\label{prop-seqs}The spaces in Definition~\ref{def-GN} are
$\Gm$-modules, related by
\badl{exseq-NGW} \xymatrix{ 0 \ar[r] & \N\om\nu\ar@{^(->}[d] \ar[r] & \G
\om\nu \ar@{^(->}[d]\ar[r] & \W\om\nu(\partial S)\ar@{^(->}[d] \ar[r]
&0\\
0 \ar[r] & \N{\om^0,\excg}\nu \ar[r] & \G{\om^0,\excg}\nu \ar[r]
& \W{\om^0,\excg}\nu \ar[r] &0 }
\eadl
in which the rows are exact.
\end{prop}
\begin{proof}We show the surjectivity in the lower exact sequence. The
other statements are clear.

Any element of $\W{\om^0,\excg}\nu$ has a representative $f_0$ in
$\E_\nu\left( U \cap S\right)$ for some excised neighborhood $U$ of
$(\partial S)\setminus E$ for some finite set of cusps. It also
satisfies the conditions in part~\eqref{Woe:excg} in
Definition~\ref{def-Woe}, which are similar to the conditions in parts
\eqref{GN:bdd} and~\eqref{GN:dac} in Definition~\ref{def-GN}, but
concern only points in an excised neighborhood.

We have to extend a modification of $f_0$ to all of $S$. Near a cusp
$\cu \in E$ we have the situation sketched in Figure~\ref{fig-Sie}.
\begin{figure}[ht]
\[\setlength{\unitlength}{1cm}
\begin{picture}(2.4,3)(-1.2,2)
\put(-.7,3.9){$\Sie_\cu(C_\cu,Y)$}
\put(-.1,2.2){$Q$}
\thicklines
\put(-1,3){\line(1,0){2}}
\put(-1,3){\line(0,1){2}}
\put(1,3){\line(0,1){2}}
\put(-1,3){\line(-1,-1){1}}
\put(1 ,3){\line(1,-2){.5}}
\end{picture}
\]
\caption{The Siegel domain $\Sie_\cu(C_\cu,Y)$ transported to $\inft$.}
\label{fig-Sie}
\end{figure}
If $C_\cu$ is chosen suitably, then
$\Dt f_0 = (\rho^2-\nobreak\nu^2) f_0$ outside $\Sie(C_\cu,Y)$ and
outside the compact set $Q$. For each $\cu \in E$ we take a bounded
open set $A_\cu\subset N$ containing the compact set $C_\cu$, and we
take a compact set $D_\cu \subset N$ containing $A_\cu$. Transporting
$\cu$ to $\inft$ we have the situation sketched in
Figure~\ref{fig-Siedms}.
\begin{figure}[ht]\[\setlength{\unitlength}{1cm}
\begin{picture}(7.6,3)(-3.3,2)
\put(-.7,4.3){$\Sie_\cu(C_\cu,Y)$}
\put(-3.2,4.3){$\Sie_\cu(D_\cu\setminus A_\cu$}
\put(-2.5,3.9){$,Y)$}
\put(1.4,4.3){$\Sie_\cu(D_\cu\setminus A_\cu$}
\put(2,3.9){$,Y)$}
\put(-.1,2.2){$Q$}
\put(-3.4,2.5){$1\geq \ch\geq 0$}
\put(1.8,2.5){$0\leq \ch\leq 1$}
\thicklines
\put(-3.3,3){\line(1,0){6.6}}
\put(-1,3){\line(0,1){2}}
\put(1,3){\line(0,1){2}}
\put(-1,3){\line(-1,-1){1}}
\put(1,3){\line(1,-2){.5}}
\put(-1.3,3){\line(0,1){2}}
\put(-3.3,3){\line(0,1){2}}
\put(1.3,3){\line(0,1){2}}
\put(3.3,3){\line(0,1){2}}
\put(-1,3){\line(-1,-1){1}}
\put(1,3){\line(1,-2){.5}}
\put(-1.3,3){\line(-1,-1){1}}
\put(1.3,3){\line(1,-2){.5}}
\put(-3.3,3){\line(-1,-1){1}}
\put(3.3,3){\line(1,-2){.5}}
\end{picture}
\]
\caption{Inclusion of various Siegel domains.}\label{fig-Siedms}
\end{figure}
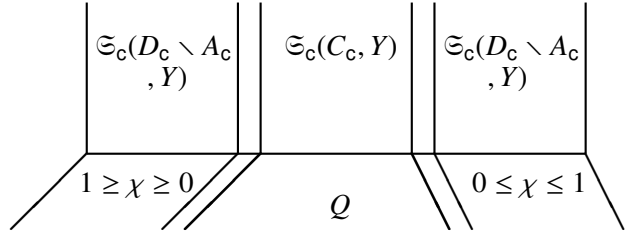
Similarly, the compact set $Q \subset \SY$ is embedded into larger
compact sets. In this way we arrive at a transition zone in $U\cap S$
on which we can multiply $f_0$ by a smooth function $\ch$ going from
$1$ to zero. Then $f=\ch\, f_0$ extended by zero on the remaining
region is a smooth function representing the same $\nu$-analytic
boundary germ as $f_0$. Moreover the estimates in \eqref{GN:bdd} and
\eqref{GN:dac} in Definition~\ref{def-GN} are valid on each Siegel
domain $\Sie_\cu(D_\cu,Y)$ with $\cu\in \Cu$.

We can add to $f$ any smooth function with compact support not
intersecting the Siegel domains $\Sie_\cu(C_\cu,Y)$. The resulting
function represents the same boundary germ as $f_0$.
\end{proof}

\begin{defn}For any function $f:S\rightarrow \CC$ the set of
\il{sngset}{$\nu$-singularity}\emph{$\nu$-singularities}
\il{sing}{$\sing_\nu(\cdot)$}$\sing_\nu(f)$ of $f$ is the complement of
the maximal open set $W\subset S$ such that $f\in \E_\nu(W)$.
\end{defn}
This concept of singularity is very wide. Any function that is not an
eigenfunction of the Laplace operator with eigenvalue $\rho^2-\nu^2$ on
a neighborhood of $x\in S$ has a $\nu$-singularity at~$x$.

\subsection{Construction of invariant eigenfunctions from cocycles}\label{sect10-constr} We
turn to the construction that will allow us to prove
Theorem~\ref{thm-alW}. Actually, the construction that we describe is
stronger than what we need for the present purpose. This construction
forms the essential step from cocycles back to automorphic forms. The
method is analogous to the approach in \cite{BLZ15} for automorphic
forms on cofinite discrete subgroups of $\PSL_2(\RR)$. First cocompact
discrete subgroups are handled in \S7.1. The presence of cusps needs
more work, carried out in \S12.2.

\rmrk{Discussion} If we deal with a cocycle $\ps$ coming from a cusp
form, then Proposition~\ref{prop-psC} tells us what to do. On a union
of fundamental domains $Z$ we integrate $\ps$ over $\partial Z$.
We get back the cusp form on the interior of~$Z$, and we can extend the
result to $S$ in a $\Gm$-invariant way. See Figure~\ref{fig-illZ}.
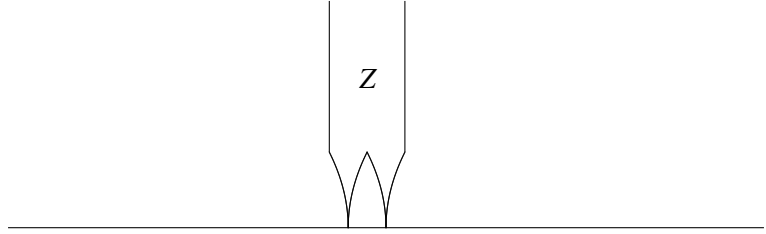
\begin{figure}[ht]
\[ \setlength\unitlength{1cm}
\begin{picture}(10,3)(-5,0)
\put(-5,0){\line(1,0){10}}
\put(-.35,1.85){$Z$}
\put(-.75,3){\line(0,-1){2}}
\qbezier(-.75,1)(-.5,.5)(-.5,0)
\qbezier(-.5,0)(-.5,.5)(-.25,1)
\qbezier(-.25,1)(0,.5)(0,0)
\qbezier(0,0)(0,.5)(.25,1)
\put(.25,3){\line(0,-1){2}}
\end{picture}\]
\caption{Two-dimensional illustration of the approach of the proof of
Proposition~\ref{prop-psC}. A union $Z$ of fundamental
domains.}\label{fig-illZ}
\end{figure}

We have to modify this idea, to be able to apply it to cocycle that
might not come from cusp forms.

The difference between $\nu$-analytic boundary forms and their
representative functions on $S$ is confusing. It is difficult to work
with cocycles with values in $\nu$-analytic boundary forms. More
convenient would be cocycles with values in the space $\G\om\nu$ of
representatives, which are actual functions on~$S$. This is not
possible with cocycles, however we can find cochains with values in
$\G\om\nu$ that represent the cocycle with which we start. If we have
an arbitrary cocycle it may happen that the value of its
representatives on $\partial Z$ has singularities in the interior of
$Z$, and the application of Proposition~\ref{prop-psC} breaks down.

A cocycle that satisfies the condition `exc' has representatives on
$\partial Z$ that have no singularities in an excised neighborhood of
$\partial S \setminus E$, where the finite set $E$ of cusps depend on
the components of $\partial Z$.
It turns out that if we take a very wide cycle $C$ around $Z$ a suitable
representative of the cocycle can be chosen such that its value on
$\partial C$
has no singularities on the interior of $Z$. The choice of $C$ has to
take into account the set $Z$ and the choice of the representative of
the cocycle. See Figure~\ref{fig-illZC}.
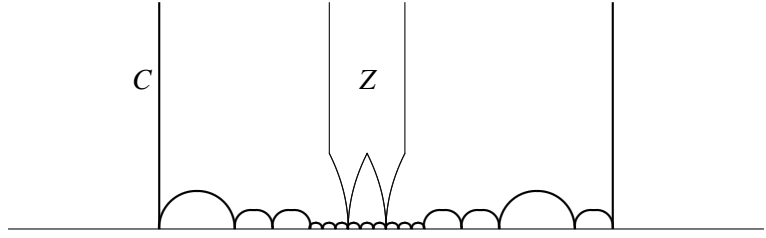
\begin{figure}[ht]
\[ \setlength\unitlength{1cm}
\begin{picture}(10,3)(-5,0)
\put(-5,0){\line(1,0){10}}
\put(-.35,1.85){$Z$}
\put(-3.35,1.85){$C$}
\put(-.75,3){\line(0,-1){2}}
\qbezier(-.75,1)(-.5,.5)(-.5,0)
\qbezier(-.5,0)(-.5,.5)(-.25,1)
\qbezier(-.25,1)(0,.5)(0,0)
\qbezier(0,0)(0,.5)(.25,1)
\put(.25,3){\line(0,-1){2}}
\thicklines
\put(-3,0){\line(0,1){3}}
\put(3,0){\line(0,1){3}}
\put(-2.5,0){\oval(1,1)[t]}
\put(-1.75,0){\oval(.5,.5)[t]}
\put(-1.25,0){\oval(.5,.5)[t]}
\put(-.927,0){\oval(.167,.167)[t]}
\put(-.75,0){\oval(.167,.167)[t]}
\put(-.583,0){\oval(.167,.167)[t]}
\put(-.427,0){\oval(.167,.167)[t]}
\put(-.25,0){\oval(.167,.167)[t]}
\put(-.083,0){\oval(.167,.167)[t]}
\put(.427,0){\oval(.167,.167)[t]}
\put(.25,0){\oval(.167,.167)[t]}
\put(.083,0){\oval(.167,.167)[t]}
\put(.75,0){\oval(.5,.5)[t]}
\put(1.25,0){\oval(.5,.5)[t]}
\put(2,0){\oval(1,1)[t]}
\put(2.75,0){\oval(.5,.5)[t]}
\end{picture}\]
\caption{Two-dimensional illustration of the approach of the proof of
Proposition~\ref{prop-psC}. A cycle around $Z$.
\\
Figure copied from \cite[p 75]{BLZ15}.} \label{fig-illZC}
\end{figure}

Our main task is to adapt the two-dimensional reasoning in \cite{BLZ15}
to the present $\ds$-dimensional context.

\begin{prop}\label{prop-cohE}Let
$\nu \in \CC\setminus \frac12\ZZ_{\leq -1}$. There is an explicit
construction $\ps\mapsto u_\ps$ that provides for each cocycle
$\ps\in Z^{\ds-1}(F^\tess_\bullet;\W{\om^0,\exc}\nu\bigr)^\bdc$ a
function $u_\ps \in \E_\nu(S)$ such that
\begin{enumerate}[label=$\mathrm{(\alph*)}$, ref=$\mathrm{\alph*}$]
\item \label{cohE:lin}$u_\ps$ depends linearly on $\ps$,
\item \label{cohE:idp} $u_\ps$ depends only on the cohomology class of
$\ps$ in $Z^{\ds-1}(F^\tess_\bullet;\W{\om^0,\exc}\nu\bigr)^\bdc$,
\item \label{cohE:inv} $u_\ps|\gm = u_\ps$ for all $\gm\in \Gm$.
\end{enumerate}
\end{prop}
\begin{proof}After a geometrical preparation we will describe the
construction, and prove it properties.

\rmrk{\intitle{R}-neighborhoods of \intitle{(\ds-1)}-cells} Before
proving the proposition, we consider the influence of the condition
`exc' on cocycles. We define for each $C\in \X^\tess_{\ds-1}$ the
$R$-neighborhood $N_R(C)$ for each positive real
parameter~$R$.\il{NCR}{$N_R(C)$}

If $C\in \X^{\tess,Y}_{\ds-1}$ we take
\be N_R(C)= \bigl\{ x\in S\;:\; \dist(x,C) \leq R\bigr\}\,.\ee
This is a compact set, since $C$ is a compact set in $S$.

If $C\not\in \X^{\tess,Y}_{\ds-1}$ then $C$ contains a unique cusp
$\cu$, and $C$ is contained in some Siegel domain $\Sie_\cu(B,T)$ with
$B$ compact in $N$ and $T>0$. We define
\[ B_R = \bigl\{ \nm(x) \;:\; \bigl| x_i - x'_i\bigr|\leq R\text{ for }
1\leq i \leq p+q \text{ for all }\nm(x')\in B\bigr\}\]
in horospherical coordinates $x=(x^{(1)},x^{(2)}) \in \RR^{p+q}$, with
$n(x') = n(x_1', \cdots, x_{p+q}')$.
We take
\be N_R(C) = \Sie_\cu\bigl(B_R,(T^{-1}+R)^{-1}\bigr)\,.\ee

These $R$-neighborhoods satisfy the following statements:
\begin{lem}
For any $C_0 \in \X^{\tess,Y}_{\ds-1}$ there are only finitely many
$C\in \X^\tess_{\ds- 1}$ such that $N_R(C_0) \cap N_R(C)$ is non-empty.

If $C_0 \in \X^{\tess}_{\ds-1}\setminus \X^{\tess,Y}_{\ds-1}$, then
there is a cusp $\cu\in C_0$. For all $\gm\in \Gmm\cu$ we have
$\cu \in \gm C_0$. For any $C_0 \in \X^\tess_{\ds-1}$ there are only
finitely many $C\in \X^\tess_{\ds-1}$ such that
\be N_R(C_0) \cap N_R(C) \cap S \neq\emptyset\,.\ee
\end{lem}
\begin{proof}If $C_0 \in \X^{\tess,Y}_{\ds-1}$ then $N_R(C_0)$ is a
compact set, and intersects only finitely many $\Gm$-translates of the
fundamental domain $D(p_0)$ in \eqref{Dp} on which we based the
tessellation. Each $\Gm$-translate of $D(p_0)$ intersects only finitely
many $(\ds-1)$-cells in $\X^\tess_{\ds-1}$.

If $C_0$ contains a cusp $\cu$, then $N_R(C_0)$ is not compact, since it
contains a region in $S$ with $\cu$ in its closure. The construction of
$N_R(C_0)$ takes care that it intersects only finitely many $\ds$-cells
containing $\cu$.
\end{proof}

\rmrk{Construction}Let
$\ps\in Z^{\ds-1}\bigl( F^\tess_\bullet; \W {\om^0,\excg}\nu\bigr)^\bdc$.
The cocycle is determined by its values $\ps(C_j)$ on a finite set of
$(\ds-1)$-cells $C_1,\ldots, C_n$ that generate $\X^\tess_{\ds-1}$.

We use the exact sequences in~\eqref{exseq-NGW} to choose
representatives $\tilde \ps(C_j)  \in \G\om\nu$ for $j=1,\ldots,n$. If
$C_j \in \X^{\tess,Y}_{\ds-1}$ we can take
$\tilde\ps(C_j)\in \G\om\nu$. If $C_j$ contains a cusp $\cu$, then
$\tilde\ps(C_j) \in \G{\om^0,\excg}\nu$ with
$\bsing_\nu\bigl( \tilde\ps(C_j)\bigr) \subset \{\cu\}$. We extend
$\tilde\ps$ in a $\Gm$-equivariant way to a $\CC[\Gm]$-linear map
$\tilde\ps: F_{\ds-1}^\tess\rightarrow \G{\om^0,\excg}\nu$. This is
only a cochain, not a cocycle.

We arrange that $\sing_\nu\bigl(\tilde\ps(C_j)\bigr) \subset N_R(C_j)$
for $1\leq j \leq n$ by choosing $R$ sufficiently large. This implies
$\sing_\nu\bigl(\tilde\ps(C)\bigr) \subset N_R(C)$ for all
$C\in \X^{\tess}_{\ds-1}$, by $\Gm$-equivariance.\medskip

To construct a function in $\E_\nu(S)$ we consider first a compact
region $Z$ that is contained in the union $Z_1$ of finitely many
$\ds$-cells $Y \in \X^\tess_\ds$. The boundary $\partial  Z_1$ is a
cycle in $ \ZZ[\X^\tess_{\ds-1}]$. We take a much larger union
$\tilde Z \supset Z_1$ of finitely many $\ds$-cells
$Y \in \X^\tess_\ds$ such that its boundary
$\partial \tilde Z = \sum_i \e_i D_i \in \ZZ[\X^\tess_{\ds-1}]$
satisfies for all $i$
\be Z \cap N_R(D_i) \cap S = \emptyset\,.\ee
Hence
\be \sing_\nu \bigl( \tilde\ps(\partial\tilde Z ) \bigr) \subset
\bigcup_i \sing_\nu\bigl( d \tilde \ps(D_i)\bigr)\ee
does not intersect $Z\cap S$. The function
\be \label{utZtp} u_{\tilde Z,\tilde\ps} = \frac1{2\nu
c(\nu)}\tilde\ps(\partial\tilde Z) = \frac1{2\nu c(\nu)} \sum_i
(d\tilde\ps)(D_i)\ee
on $Z\cap S$ determines an element of $\E_\nu\bigl( \mathring Z\bigr)$.
We do not need factors in $\{1,-1\}$ to fix the orientation. The
notation of the $\ds$-cells $D_i$ is automatically positive, and
$\partial \tilde Z$ gets its orientation from $\tilde Z$. We note that
on $Z\cap S$ the functions $\tilde \ps(D_i)$ are representatives of
$\ps(D_i)$. The division by $\nu\, c(\nu)$ causes no problems.
See~\eqref{cHC}.

If $\tilde Z'$ is another $(\ds-1)$-cycle satisfying the same
conditions, then $\partial \tilde Z' - \partial \tilde Z$ is the
boundary of a sum of $\ds$-cells far away from $Z$, and on $Z\cap S$
\[ u_{\tilde Z',\tilde\ps}-u_{\tilde Z,\tilde \ps} =
\frac1{2\nu\,c(\nu)} \tilde \ps\bigl( \partial (\tilde Z'-\tilde
Z)\bigr)
= \frac1{2\nu\,c(\nu)}\ps\bigl( \partial (\tilde Z'-\tilde
Z)\bigr)=0\,\]
since $\ps$ is a cocycle. This shows that we can write
$u_{Z,\tilde \ps}$, independently of the choice of~$\tilde Z$.

If $Z_1\supset Z$ is a union of more $\ds$-cells, then a suitable choice
of $\tilde Z_1$ works fine for $Z$, and we have
\[ u_{\tilde Z_1,\tilde\ps} = u_{Z,\tilde \ps}\qquad\text{on }Z\cap
S\,.\]
This implies that we can define
\be u_{\tilde \ps}(z) = \lim_Z u_{Z,\tilde \ps}(z)\,,\ee
where $Z$ runs over larger and larger sums of finitely many $\ds$-cells.
This defines $u_{\tilde\ps} \in \E_\nu (S)$.

\rmrk{Properties}We have still to show that the constructed
eigenfunction $u_{\tilde\ps}$ does not depend on the choice of the
cochain $\tilde\ps$ and that it has properties
\eqref{cohE:lin}--\eqref{cohE:inv}.

Let $\gm\in \Gm$. For $z\in S$ we take $Z$ such that it contains $\gm z$
and $z$, and we take $\tilde Z$ suitable for~$Z$.
\begin{align*}
2\nu\, c(\nu)&\, u_{\tilde \ps}(\gm z)= \tilde \ps(\partial \gm \tilde
Z)(\gm z)
= \tilde \ps (\gm \partial \tilde Z) |\gm^{-1} \, (z)\\
&= \tilde \ps(\partial \tilde Z) (z) = 2\nu\, c(\nu)\,
u_{\tilde\ps}(z)\,.
\end{align*}
So $u_{\tilde\ps} \in \E_\nu(\Gm\backslash S)$.

Two different choices of cochains $\tilde \ps $ representing $\ps$
differ by a $(\ds-1)$-cochain $b$ with values in $\N{\om^0,\excg}\nu$.
In the situation of \eqref{utZtp} this gives on $Z\cap S$ the zero
function. So we can write
\be u_{\tilde\ps} = u_\ps\,.\ee
This gives a $\Gm$-invariant function $u_\ps$ in $\E_\nu( S)$. This is
part~\eqref{cohE:inv}. For assertion \eqref{cohE:lin} we note that we
can take lifts of cocycles in a linear way.

Let
$\ps=d\eta \in B^{\ds-1}(F^\tess_\bullet; \W{\om^0,\excg}\nu) \cap Z^{\ds-1}(F^\tess_\bullet; \W{\om^0,\excg}\nu)^\bdc$.
Choose $\tilde\eta(D) \in \G{\om^0,\excg}\nu$ with image
$\eta(D)\in \W{\om^0,\excg}\nu$ for representatives of the $\Gm$-orbits
in $X^\tess_{\ds-2}$, and extend $\tilde \eta$ as a cochain. Then
$d\tilde \eta $ can play the role of $\tilde \ps$. Since
$\tilde\ps(\partial C) = d\tilde\ps  (C) = d^2\tilde\eta(C)=0$,
we conclude that $u_{\tilde\ps}=0$, and hence $u_\ps=0$. This gives
assertion~\eqref{cohE:idp}.
\end{proof}

The construction gives for classes
$[\ps] \in H^{\ds-1}_\pb\bigr(\Gm;\W{\om^0,\exc}\nu\bigr)^{\bdc}$ a
specific element of $\E_\nu(\Gm\backslash S)$ without condition on the
growth at the cusps. In Theorem~\ref{thm-alW} we deal with cocycles
with values in the smaller module $\W{\om^0,\infty,\excg}\nu$. So there
is further work to do. Furthermore, we have no guarantee that the
construction might not give zero functions.

\subsection{Cocycles and cusp forms}\label{sect10-coc-cf}

\begin{prop}\label{prop-cE}Let
$\nu \in \CC\setminus \frac12\ZZ_{\leq -1}$, $|\re\nu|<\rho$. If
$\ps\in Z^{d-1}\bigl(F^\tess_\bullet;\W{\om^0,\excg}\nu\bigr)^\bdc$,
then $u_\ps\in \A^0(\Gm;\nu)$.
\end{prop}
\begin{proof}The additional assumption concerning the cocycle $\ps$
implies that we can use condition~\eqref{GN:dac} in
Definition~\ref{def-GN}. By Definition~\ref{def-cf} we need quick decay
for $u\in\A^0(\Gm;\nu)$ at the cusps. Actually, any decay of size
$\oh\bigl(t^{\rho-|\re\nu|-\e}\bigr)$, $\e>0$, suffices by
Lemma~\ref{lem-cusp-crit}.

For a given cusp $\cu$ it suffices to give an estimate of $u_\ps$ on a
Siegel domain $\Sie_\cu(C,T)$, where $C$ contains a fundamental domain
of $\Gmm\cu\backslash N$. We take $\tilde Z$ as the boundary of a union
of $\ds$-cells covering $\Sie_\cu(C,T)$ such that
$N_R(D_i) \cap \Sie_\cu(C,T)= \emptyset$ for all the finitely many
$(\ds-1)$-cells occurring in $\tilde Z$.

For boundary components $D_i$ for which $\cu\not\in D_i$ the
contribution $\tilde\ps(D_i)$ represents a $\nu$-analytic boundary germ
on a neighborhood $U_i$ of $\cu $ in $\partial S$. Then we have
$\tilde\ps(D_i) \bigl( g_\cu \nm(x) \am(t) \bigr) 
= t^{-\rho-\nu} A\bigl( g_\cu \nm(x) \am(t) \bigr) $ for all
$ g_\cu \nm(x) \am(t) \in \Sie_\cu(C,T)$, for some analytic function on
a neighborhood of $U_i$ in $\hat S$. The analytic function $A$ is
bounded on $U_i$, and
\be \lim_{t\uparrow\infty} t^{\rho+\nu} \tilde\ps(D_i) \bigl( g_\cu
\nm(x) \am(t) \bigr) = A_u( g_\cu \,\inft)= A_u(\cu)\,. \ee
This gives an estimate of the desired strength of the contribution of
$D_i$.

The remaining boundary components $D_i$ contain the cusp~$\cu$. The
growth condition \eqref{GN:dac} in Definition~\ref{def-GN} gives for
the contribution of these components the behavior
$\oh\bigl(t^{-(\rho+\nu)}\bigr)$ as well.

To satisfy the condition in \label{appl3110} Lemma~\ref{lem-cusp-crit}
we need $-\rho-\re\nu < \rho-\bigl|\re\nu\bigr|$. Under the additional
condition $|\re\nu|<\rho$ we conclude that $u_\ps$ is a cusp form.
\end{proof}

\begin{prop}\label{prop-pi}Suppose that $\nu \in \CC$ satisfies
$\nu\not\in \frac12\ZZ_{\leq -1}$ and $|\re \nu|<\rho$.

Let $u\in \A^0(\Gm;\nu)$ and let the $(\ds-1)$-cocycle $\ps$ be
determined by $\ps(C)  = I_C(u;\cdot)$ for all $C\in \X^\tess_{\ds-1}$,
as in \eqref{Iintf}. Then $u_\ps=u$.
\end{prop}
\begin{proof}We consider a $\ds$-cell $X\in \X^{\tess,Y}_\ds$. Let
$V_\e$ be the open set of points $z\in X$ with distance to $\partial X$
larger than $\e$. By taking $\e$ sufficiently small, we arrange that
$V_\e$ has a non-empty interior, which is open.

The equality $I_C(u;y)=\int_{x\in C} \omw_\nu(u; y,x) $ and the
inclusion $\sing\bigl( I_C(u;\cdot) \bigr) \subset C$ hold for each
boundary component $C\in B(X)$. We take a lift $\tilde\ps $ of $\ps$
like in the proof of Proposition~\ref{prop-cohE}, and now can take the
parameter $R$ smaller than $\e$. Then $\tilde\ps(C)(y) = \ps(C)(y)$ for
$y\in \mathring V_\e$. This gives for $y\in V_\e$
\bad 2\nu \, c(\nu) & u_\ps(y) = \tilde \ps(\partial X)(y) =
\ps(\partial X)(y)\\
&= \int_{x\in \partial X}\omw_\nu(u;y,x) = 2\nu\,c(\nu) \, u(y)\,,
\ead
with use of Proposition~\ref{prop-psC} in the last step. Since
$\nu \, c(\nu) \neq 0$ under the condition on~$\nu$, we conclude that
$u_\ps=u$ on the non-empty open set $\mathring V_\e$. By analyticity,
$u_\ps=u$ on~$S$.
\end{proof}

%% file: ccro2-11-inj.tex

\bigskip

\def\flnm{ccro2-11-inj}

\section{Injectivity of the map from cocycles to cusp
forms}\label{sect11}
Theorem~\ref{thm-alW} provides a linear map $\alw_\nu$ from
$H^{\ds-1}_\pb(\Gm;\W{\om^0,\excg}\nu\bigr)^\bdc$ that is only a
one-sided inverse to the linear map $\btw_\nu$ in
Proposition~\ref{prop-btimg}.

In the case $\ds=2$ the map $\alw_\nu$ is a bijection. See
\cite[Proposition 12.6]{BLZ15}. For $\ds\geq 3$ we do not know yet
whether the kernel of $\alw_\nu$ is zero or not. In this section we
describe the kernel as the image of a cohomology space with values in a
specific space of representatives of $\nu$-analytic boundary forms.
This leads, for $\nu$ satisfying $|\re\nu|<\rho$ and
$\nu\not\in \frac12\ZZ_{\leq-1}$, to the result in
part~\eqref{mnthm:nsp} of Theorem~\ref{mainthm} stated in the
introduction. The less complicated statement in \eqref{mnthm:nugrl}
will be considered in the next section.\smallskip

In Subsection~\ref{sect11-co-fs} we transform the construction of
$u_\ps$ in Subsection~\ref{sect10-constr} into a description as a
locally finite sum. This description is used in
Subsection~\ref{sect11-kernel} to relate the subspace of
$H^{\ds-1}_\pb\bigl( \Gm;\W{\om^0,\infty,\excg}\nu\bigr)^\bdc$ that we
have to divide out, to the image of the cohomology space
$H^{\ds-1}_\pb(\Gm;\G{\om^0,\excg}\nu)$. With help of sheaf cohomology
we will show that
$H^{\ds-1}_\pb(\Gm;\G{\om^0,\excg}\nu) = H^{\ds-1}_\pb
\bigl(\Gm;  \G{\om^0,\excg}\nu \cap \E_\nu(S)\bigr)$. This leads to
completion of the proof of part~\eqref{mnthm:nsp} of
Theorem~\ref{mainthm}.

\subsection{Coinvariants and infinite sums}\label{sect11-co-fs} The
construction in the proof of Proposition~\ref{prop-cohE} of the cusp
form $u_\ps$ associated to a cocycle
$\ps\in Z^{\ds-1}\bigl( F^\tess_\bullet;\W{\om^0,\excg}\nu\bigr)^\bdc$
is based on a lift
$\tilde\ps\in C^{\ds-1}\bigl( F^\tess_\bullet;\G{\om^0,\exc^0}\nu\bigr)$,
and on evaluation of $\tilde\ps$ on $(\ds-1)$-cycles. The next
proposition gives an alternative formulation, which will be important
later on in this section.

\begin{prop}\label{prop-ups-sum}Let
$\nu \in \CC\setminus \frac12\ZZ_{\leq-1}$, $ \re\nu>-\rho$. Let
$\ps\in Z^{\ds-1}\bigl( F^\tess_\bullet; \W{\om^0,\excg}\nu)^{\bdc}$
and let
$\tilde\ps \in C^{\ds-1}\bigl( F^\tess_\bullet; \G{\om^0,\excg}\nu)^\bdc$
be a lift as in the proof of Proposition~\ref{prop-cohE}. Then
\be \label{ups-sum}
u_\ps (x) = \sum_{D\in \X^\tess_d} d\tilde \ps(D)(x)\ee
represents $u_\ps$ as a locally finite sum of elements of
$\G{\om^0,\excg}\nu$.
\end{prop}
The cancellation of the $\nu$-singularities is remarkable. Since
$u_\ps\in \E_\nu(S)$ we have $\sing_\nu(u_\ps)=\emptyset$. However, the
choice of the lift $\ps$ implies that in general
$\sing_\nu\bigl( d\tilde\ps(D)\bigr) \neq \emptyset$, and hence that
singularities of neighboring cells $D$ have to cancel each other.
\begin{proof} Let $\tilde \ps$ be a cochain representing $\ps$, like in
the proof of Proposition~\ref{prop-cohE}. Since $d\ps(D)=0$ for each
$\ds$-cell, we have $d\tilde\ps(D) \in \N{\om^0,\excg}\nu$.

We recall that the union of the $\ds$-cells is equal to $S^\ast$, with
overlap only in boundary components. We consider a finite union
$U_E=\bigcup_{D\in E} D$ with a finite subset $E \subset \X^\tess_d$.
For suitable choice of $E$ we can use
$\tilde Z = \sum_{D\in E}\partial D$ as a $(\ds-1)$-dimensional set
surrounding a given set $Z$, as in the proof of
Proposition~\ref{prop-cohE}. Then
$u_\ps = \frac1{2\nu\, c(\nu)}\, \tilde\ps\bigl( \partial  U_E  \bigr)$
on $Z$. If we add more and more $\ds$-cells to $E$ we get a
representation of $u_\ps$ as a sum over larger and larger sets~$D$.
Since the sets $\sing(\tilde\ps(D))$ have only finitely many non-empty
pairwise intersections, the sum $\sum_{D \in E} d\tilde\ps(D)(x) $
becomes locally constant in $x$ as $E\subset \X^\tess_d$ expands, and
the infinite sum in~\eqref{ups-sum} converges absolutely on $S$ and
represents $u_\ps$.
\end{proof}

The group $\Gm$ acts freely on the set $\X^\tess_{\ds}$ of $\ds$-cells,
with finitely many orbits. We can generate these orbits by
$\fd_0 \in \X^{\tess,Y}_{\ds}$ and $\fd_\cu$ where $\cu$ runs over a
set \il{RGm}{$R(\Gm)$ representatives of $\Gm\backslash \Cu$}$R(\Gm)$
of representatives of $\Gm\backslash \Cu$. So we can write the
assertion \eqref{ups-sum} as
\be \label{ups-sum1} u_\ps = \sum_{\gm\in \Gm}L(\gm) \Bigl(
d\tilde\ps(\fd_0) + \sum_{\cu\in R(\Gm)} d\tilde\ps(\fd_\cu)
\Bigr)\,.\ee
By Proposition~\ref{prop-H0d} the space
$ H^\ds_\pb(\Gm;\N{\om^0,\excg}\nu)$ is isomorphic to the space of
coinvariants $(\N{\om^0,\excg}\nu)_\Gm$, which is
$\N{\om^0,\excg}\nu\bigm/ \left( 1-L(\Gm) \right)\N{\om^0,\excg}\nu$,
where we denote by $\left( 1-L(\Gm) \right)\N{\om^0,\excg}\nu$ the
space generated by the elements $L(\gm)v-v$, where $\gm$ runs through
$\Gm$ and $v$ runs through $\N{\om^0,\excg}\nu$. The coinvariant
associated to the cohomology class of $d\tilde\ps$ can be represented
by
\be\label{dsum} c\bigl( d\tilde \ps) =d\tilde\ps(\fd_0) + \sum_{\cu\in
R(\Gm)}d\tilde\ps(\fd_\cu)\,.\ee
Other choices of generating $\ds$-cells lead to the same quantity modulo
the space $\left( 1-L(\Gm) \right)\N{\om^0,\excg}\nu$.

\begin{prop}\label{prop-upxcoinv}Let
$\nu \in \CC\setminus \frac12\ZZ_{\leq-1}$, $\re\nu>-\rho$. Let
$\ps\in Z^{\ds-1}\bigl( F^\tess_\bullet; \W{\om^0,\excg}\nu)^{\bdc}$
and let
$\tilde\ps \in C^{\ds-1}\bigl( F^\tess_\bullet; \G{\om^0,\excg}\nu)^\bdc$
be a lift as in the proof of Proposition~\ref{prop-cohE}.

The cusp form $u_\ps$ is zero if and only if the cohomology class of
$d\tilde\ps$ is equal to the trivial coinvariant
$\left( 1-L(\Gm) \right)\N{\om^0,\excg}\nu$.
\end{prop}
\begin{proof}We construct a partition of unity on $S$. This is a
function $\ch\in C^\infty(S)$ such that
$\sum_{\gm\in \Gm}\ch(\gm\, x)=1$ for all $x\in S$, as a locally finite
sum.

We choose $a>Y$. Then
$S_{\!\Gm}(a) = S \setminus \bigcup_{\cu\in \Cu}B_\cu(a)$ is a
neighborhood of $\SY$. We choose $\ch_0 \in C_c^\infty(S)$ such that
\be \sum_{\gm\in \Gm} \ch_0( \gm\, x)=1\qquad \text{for }x\in
S_{\!\Gm}(a)\,.\ee

For each $\cu \in R(\Gm)$ we take $\kp_\cu \in C_c^\infty(N)$ such that
$\sum_{\ld\in \Ld_\cu} \kp_\cu(\ld^{-1}n)=1$ for $n\in N$, and then put
for $g_\cu n \am(t) \,\orgn\in B_\cu(Y)$
\be \ch_\cu \bigl(g_\cu n \am(t) \,\orgn) = \kp_\cu(n) \,.\ee
Then
\be \sum_{\gm\in \Gmm\cu} \ch_\cu(\gm^{-1} x)=1\qquad\text{for }x\in
B_\cu(Y)\,.\ee
We note that $\Ld_\cu$ is a lattice in $N$, and that
$\Gmm \cu = g_\cu \Ld_\cu g_\cu^{-1}\subset\Gm  $. We define
$\ch_\cu=0$ on $S\setminus B_\cu(Y)$.

We glue $\ch_0$ and the functions $\ch_\cu$ together with a function
$\thv \in C^\infty(S)$ determined by
\be \thv(x) = 1\text{ for }x\in \SY\,,\quad \thv(g_\cu n \am(t) \,\orgn)
= \thv_0(t) \text{ for } g_\cu n \am(t)\,\orgn \in B_\cu(Y)\,,\ee
with $\thv_0 \in C^\infty\bigl([Y,\infty)\bigr)$ satisfying
$\thv_0\geq 0$, $\thv_0(Y)=1$ and $\thv_0(t) =0$ for $t\geq a$, where
$a>y$ has been chosen at the start of the proof. Then we obtain
$\ch\in C^\infty(S)$ by
\bad \ch(x) &= \ch_0(x) & &\text{for } x\in \SY\,,\\
&= \thv(x) \,\ch_0(x) + \sum_{\cu \in \Gm\backslash \Cu} \bigl(1-
\thv(x) \bigr)\ch_\cu(x)
&& \text{for } x \in \gm B_\cu(Y),\; \gm\in \Gm\,.
\ead
The resulting function $\ch$ is in $\N{\om^0,\excg}\nu$, and
$\sum_{\gm\in \Gm} L(\gm) \ch=1$, as a locally finite sum.

We put
$f= d\tilde \ps(\fd_0) + \sum_{\cu \in R(\Gm)} d\tilde\ps(\fd_\cu)$.
With \eqref{ups-sum1} and \eqref{dsum} we have
\be \label{upsf}u_\ps = \sum_{\gm\in \Gm} L(\gm) f\,.\ee
We apply the following computation, with locally finite sums. We use
that left translation of products satisfies
$L(\gm)(a\cdot b) = \bigl( L(\gm) a\bigr)\cdot
\bigl( L(\gm) b\bigr)$.
\begin{align}\nonumber
f - \ch \cdot u_\ps &= f\cdot \sum_{\gm\in \Gm} L(\gm^{-1}) \ch
&&\text{partition of unity}\\
\nonumber
&\qquad\hbox{} - \ch\cdot \sum_{\gm\in \Gm} L(\gm)
f&&\text{with \eqref{upsf}}
\displaybreak[0]\\
\nonumber
&= \sum_{\gm\in \Gm} f\cdot L(\gm^{-1})\ch - \sum_{\gm\in \Gm} L(\gm)
\Bigl( f \cdot L(\gm^{-1})\ch \Bigr)
\displaybreak[0]\\
&= \sum_{\gm\in \Gm} \Bigl( 1 - L(\gm)\Bigr) \Bigl(f \cdot
L(\gm^{-1})\ch \Bigr)\,.
\end{align} The function $ f \cdot \bigl(L(\gm^{-1})\ch
\bigr)$ can be non-zero only if
$\supp(f) \cap \bigl( \gm^{-1} \supp (\ch) \bigr)\neq \emptyset$. This
happens only for finitely many $\gm\in \Gm$. Hence $f-\ch\cdot u_\ps$
is an element of $\left( 1-L(\Gm)\right) \N{\om^0,\excg}\nu$.

Under the assumption that $u_\ps=0$ the function
$f=d\tilde\ps(\fd_0) + \sum_{\cu\in R(\Gm)} d\tilde\ps(\fd_\cu)$
represents the trivial coinvariant in $(\N{\om^0,\excg}\nu)_\Gm$.
Conversely, if $f$ represents the trivial coinvariant, then the sum of
its $\Gm$-translates vanishes.
\end{proof}

\subsection{The kernel of the map from cohomology to
cocycles}\label{sect11-kernel}
Under the conditions $\nu \in \CC\setminus\frac12\ZZ_{\leq -1}$,
$|\re\nu|<\rho$ we now have the following scheme of linear maps
\badl{Hdiag} \xymatrix@C=.5cm{ \A^0(\Gm,\nu) \ar[r]^-{\btw_\nu}&
H^{\ds-1}_\pb\bigl (\Gm;\W{\om^0,\infty,\excg}\nu\bigr)^{\bdc}
\ar@{^(->}[d]^\ii\\
& H^{\ds-1}_\pb\bigl( \Gm;\W{\om^0,\excg}\nu\bigr)^{\bdc}
\ar[ul]_{\alw_\nu} \ar@{^(->}[d]
\\
H^{\ds-1}_\pb\bigl( \Gm;\G{\om^0,\excg}\nu \bigr)\ar[r] &
H^{\ds-1}_\pb\bigl( \Gm;\W{\om^0,\excg}\nu\bigr)
\ar[r]& H^\ds_\pb\bigl( \Gm;\N{\om^0,\excg}\nu\Bigr)
}
\eadl
See Proposition~\ref{prop-btimg} and Theorem~\ref{thm-alW}. The upper
vertical map is the natural map associated to the injective linear map
\il{i}{$\ii$}
$\ii:\W{\om^0,\infty,\excg}\nu \rightarrow \W{\om^0,\excg}\nu$.
Proposition~\ref{prop-pi} states that the composition of the three maps
in the triangle starting at $\A^0(\Gm;\nu)$ is the identity. This
implies that $\btw_\nu$ and the composition $\ii\circ \btw_\nu$ are
injective.

The bottom row is part of the long exact sequence associated to the
short exact sequence
\[ 0 \rightarrow \N{\om^0,\excg}\nu \rightarrow \G{\om^0,\excg}\nu
\rightarrow \W{\om^0,\excg}\nu \rightarrow 0\,.\]
See Remark~\ref{sect4-sexsq}. The space
$H^\ds_\pb\bigl( \Gm;\N{\om^0,\excg}\nu\Bigr) $ is isomorphic to the
space of coinvariants $(\N{\om^0,\excg}\nu)_\Gm$; see
Proposition~\ref{prop-H0d}. So a class in
$H^{\ds-1}_\pb\bigl( \Gm;\G{\om^0,\excg}\nu \bigr)^\bdc$ represented by
$\ps\in Z^{\ds-1}\bigl ( F^\tess_\bullet;\W{\om^0,\excg}\nu\bigr)^\bdc$
is in the kernel of $\alw_\nu$ if its image $[d\tilde\ps]$ in
$H^\ds_\pb\bigl( \Gm;\N{\om^0,\excg}\nu\bigr) $ vanishes. See
Proposition~\ref{prop-upxcoinv}. The part of the long exact sequence in
the diagram \eqref{Hdiag} shows that this happens if $[\ps]$ is in the
image of $H^{\ds-1}_\pb\bigl( \Gm;\G{\om^0,\excg}\nu \bigr)$ under the
natural map associated to
$\G{\om^0,\excg}\nu \rightarrow \W{\om^0,\excg}\nu$. In this way, we
arrive at the following result.

\begin{prop}\label{prop-qIwe}If
$\nu \in \CC\setminus\frac12\ZZ_{\leq -1}$, $|\re\nu|<\rho$ then the
kernel of the surjective map
$\alw_\nu: H^{\ds-1}_\pb\bigl( \Gm;\W{\om^0,\excg}\nu\bigr)^\bdc
\rightarrow\A^0(\Gm;\nu)$ is equal to the
intersection\il{Iwe1}{$I^{W,\excg}_\nu$}
\badl{Iwe1} I^{W,\excg}_\nu &:=H^{\ds-1}_\pb\bigl( \Gm;
\W{\om^0,\excg}\nu\bigr)^\bdc\\
& \quad\hbox{}\cap \im\Bigl( H^{\ds-1}_\pb \bigl(
\Gm;\G{\om^0,\excg}\nu\bigr) \rightarrow H^{\ds-1}_\pb \bigl(
\Gm;\W{\om^0,\excg}\nu\bigr)\Bigr)\,.\eadl
\end{prop}
In the case $\ds=2$ we have
$H^1_\pb \bigl( \Gm;\G{\om^0,\excg}\nu\bigr) =\{0\}$ by
\cite[Proposition 12.5]{BLZ15}.
\medskip

We recall that $\N\om\nu = C^\infty_c(S)$, and that $\N{\om^0,\excg}\nu$
is the space of functions $f\in C^\infty(S)$ with support contained in
the union of a finite number of Siegel domains and a compact subset of
$S$ satisfying the growth conditions in parts \eqref{GN:bdd}
and~\eqref{GN:dac} in Definition~\ref{def-GN}.

We abbreviate
\il{Dtnu}{$\N\Dt\nu\subset \N{\om^0}\nu$}$\Dt_\nu = \Dt - \rho^2+\nu^2$,
and define $\N\Dt\nu$ as the image $\Dt_\nu \G{\om^0,\excg}\nu$. The
elements of $\N\Dt\nu$ have support in the union of a finite number of
Siegel domains and a compact subset of $S$. The differentiation will in
general destroy the growth conditions that elements of
$\G{\om^0,\excg}\nu$ satisfy.

The kernel of $\Dt_\nu:\G{\om^0,\excg}\nu \rightarrow \N\Dt\nu$ is the
$\Gm$-module \il{Eomex}{$\E^{\om^0,\excg}\nu$}$\E^{\om^0,\excg}_\nu 
= \E_\nu(S)\cap \G{\om^0,\excg}\nu$. We have the exact sequence
\be \label{exEGN}
0 \longrightarrow \E^{\om^0,\excg}_\nu\longrightarrow \G{\om^0,\excg}\nu
\stackrel{\Dt_\nu}\longrightarrow \N\Dt\nu\longrightarrow 0\,.\ee

\begin{lem}\label{lem-iDt}Let
$\nu \in \CC\setminus \frac12\ZZ_{\leq -1}$, $|\re\nu|<\rho$. For each
$h\in C^\infty_c(S)$ there are $f\in C^\infty(S)$ such that
$\Dt_\nu f = h$.
\end{lem}
\begin{proof}Since $Q_\nu$ is locally integrable, \cite[Lemma
2.1]{MiWa92}, the function
\[ f_1 (x) = \int_{y\in S} q_\nu(x;y) h(y)\, d\mu(y) \]
is well-defined.
We test against $\ph \in C^\infty_c(S)$. So
$\Dt_\nu \ph \in C_c^\infty(S)$, and with use of \eqref{resolve} in
Proposition~\ref{prop-Q},
\begin{align*}
\int_{x\in S} (\Dt_\nu \ph)(x) \, f_1(x)\, d\mu(x)
&= \int_{x\in S} (\Dt_\nu \ph)(x) \int_{y\in S} q_\nu(x;y) \, h(y)\,
d\mu(y) \, d\mu(x)
\displaybreak[0]\\
&= \int_{y\in S} \int_{x\in S} q_\nu(x;y) \Dt_\nu \ph(x) \, d\mu(x)\,
h(y)\, d\mu(y)
\displaybreak[0]\\
& = 2\nu\, c(\nu) \int_{y\in S} \ph(y) \, h(y)\, d\mu(y)\,.
\end{align*}
So $f_1$ is a weak solution of the differential equation
$\Dt_\nu f = 2\nu \,c(\nu)\, h$. Since $\Dt$ is elliptic and
$\nu\,c(\nu)\neq 0$ the equation $\Dt_\nu f = h$ has a smooth solution.
\end{proof}

\subsection{Sheaf cohomology} We recall Proposition~\ref{prop-pcshc},
which states that parabolic cohomology spaces with values in a
$\Gm$-module $V$ are isomorphic to sheaf cohomology spaces on
$\Gm\backslash \Sast$ with values in a locally constant sheaf
associated to $V$. By
$\Q[\N\Dt\nu]= \Gm\backslash \left( \N \Dt\nu \times \Sast\right)$ we
denote the locally constant sheaf associated to~$\N\Dt\nu$.

Let \il{Cinfsh}{$\C^\infty,\; \C^\ast$}$\C^\infty$ denote the sheaf of
$C^\infty$-functions on $\Gm\backslash S$, and let $j$ denote the
embedding $\Gm\backslash S \rightarrow \Gm\backslash \Sast $. We put
$\C^\ast = j_\ast \C^\infty : U \mapsto \C^\infty\bigl( j^{-1}U\bigr)$.

\begin{lem}\label{lem-ordst}The restriction of $\Q[\N\Dt\nu]$ to
$\Gm\backslash S$ is isomorphic to the sheaf $\C^\infty$.
\end{lem}
\begin{proof}It suffices to check that the stalks at all points
$ x \ \in \Gm\backslash S$ are isomorphic. By
$\pi :\Sast \rightarrow \Gm\backslash \Sast$ we denote the natural
projection. Since $\Gm$ is torsion-free, each point $x\in S$ has small
open neighborhoods $U_0$ that are relatively compact in $S$ and for
which $\gm \,U_0 \cap U_0 = \emptyset$ for all
$\gm\in \Gm\setminus\{e\}$. Then $\bigsqcup \gm\, U_0$ is open in $S$
and the image $U=\pi U_0 = \pi\bigl(\bigsqcup \gm\, U_0 \bigr) $ is a
neighborhood of $x\in \Gm\backslash S$.

Each element $h\in \N\Dt\nu$ is a $C^\infty$-function on $S$ with
support contained in the union of a compact set and a finite number of
Siegel domains. Taking the compact set large so that it contains $U_0$,
we see that the sections in $\Q[\N\Dt\nu](U_0)$ are precisely the
sections of $\C^\infty$ on $U_0$. Hence for $p\in \Gm\backslash S$ we
have equality of stalks
\[ \Q[\N\Dt\nu]_p = (\C^\infty)_p\,.\qedhere\]
\end{proof}

\begin{lem}\label{lem-cuspst}Let
$p\in (\Gm\backslash \Sast)\setminus(\Gm\backslash S) = \Gm\backslash \Cu$.
Then
\begin{enumerate}[label=$\mathrm{(\alph*)}$, ref=$\mathrm{\alph*}$]
\item \label{cuspst:Q}the stalk $\Q[\N\Dt\nu]_p$ is zero;
\item \label{cuspst:C}the stalk $(\C^\ast)_p$ is non-zero, and each germ
in $(\C^\ast)_p$ is given by a function in $C^\infty(\Gm\backslash S)$
with support in $\bigcup_{\cu\in p} B_\cu(Y)$.
\end{enumerate}
\end{lem} In this statement, $p$ denotes a class in $\Gm\backslash \Cu$.
So $p$ is a point on the quotient $\Gm\backslash \Sast$, and $\cu\in p$
runs through the cusps in the class $p$.
\begin{proof}
The extended horoballs $B^\ast_\cu(a)$ with $a>Y$ form a neighborhood
basis of $ \cu$ in $\Sast$. Each germ in $(\C^\ast)_p$ can be
represented by a $\Gm$-invariant $C^\infty$-function $f$ on the union
$\bigcup_{\cu\in p} B_\cu(a)$ for some $a>Y$. Let $\cu\in p$. The
restriction of $f$ to $B_\cu(a)$ is $\Gm_\cu$-invariant, and
multiplication by a suitable $N_\cu$-invariant cut-off function $\ch$
gives another representative $\ch f$ with compact support in the
interior of $B_\cu(a)$. This determines a $\Gm$-invariant function in
$C^\infty(S)$ with support in the union $\bigcup_{\cu\in p} B_\cu(Y)$.
This gives part~\eqref{cuspst:C}.

If we start with a non-zero element $h\in \Q[\N\Dt\nu]_p$ its
restriction to $B_\cu(a)$ cannot be periodic, since the support of $h$
is contained in a Siegel domain at $\cu$. This gives
part~\eqref{cuspst:Q}.
\end{proof}

We recall that for a sheaf $A$ and a subsheaf $B$ we get a presheaf
$U \mapsto A(U)/ B(U)$ that may not be a sheaf. The quotient sheaf
$A\bigm/ B$ is the sheaf associated to this presheaf. See for instance
\cite[Chap.~I, \S1, Proposition-Definition 1.2]{Harts}.

A sheaf $F$ on $X$ is a skyscraper sheaf at $p\in X$ if for some vector
space $V$
\[ F(U)=V \quad\text{ if }p\in U,\quad F(U)=\{0\}\quad\text{ if }
p\not\in U\,.\]

\begin{prop}\label{prop-exseq}The quotient sheaf
$\,\C^\ast \bigm/ \Q[\N\Dt\nu]$ is a finite direct sum of sky\-scraper
sheaves on the finitely many points of
$\Gm \backslash \Cu \subset \Gm\backslash \Sast$.
\end{prop}
\begin{proof}If $\Gm$ has only one $\Gm$-orbit of cusps, then Lemmas
\ref{lem-ordst} and~\ref{lem-cuspst} show that we get a sky-scraper
sheaf as the quotient. Suppose that there are more than one $\Gm$-orbit
of cusps. The description of the support of representatives in
part~\eqref{cuspst:C} of Lemma~\ref{lem-cuspst} shows that we can split
up representatives of the quotient into representatives corresponding
to the different points of $\Gm\backslash \Cu$.
\end{proof}
\begin{cor} \label{cor-vcoh}If $\ds\geq 3$, then
$H^i\bigl( \Gm\backslash \Sast, \Q[\N\Dt\nu]\bigr)\cong
H^i_\pb\bigl( \Gm;\N \Dt\nu\bigr) =\{0\}$ for $i \geq 2$.
\end{cor}
\begin{proof}
We use the long exact sequence of cohomology spaces on
$\Gm\backslash \Sast$ associated to the short exact sequence
\be 0 \longrightarrow \Q[\N\Dt\nu] \longrightarrow \C^\ast
\longrightarrow \Y\longrightarrow 0\label{Qexseq} \ee
with $\Y =   \C^\ast\bigm/ \Q[\N\Dt\nu]$.

The sheaves $\C^\ast$ and $\Y$ are soft, which means that all germs in
all stalks are represented by global sections. See \cite[Chap.~A,
\S4]{GR79}. The sheaf $\C^\infty$ of smooth functions is soft, and the
definition of the direct image implies that $\C^\ast$ is soft. A finite
direct sum of skyscraper sheaves is soft as well.

For soft sheaves, the sheaf cohomology spaces vanish in all dimensions
$i\geq 1$. In the long exact sequence associated to~\eqref{Qexseq} we
have for $i\geq 2$ the part
\[ H^{i-1}\bigl( \Gm\backslash \Sast; \Y\bigr)
\rightarrow H^{i}\bigl( \Gm\backslash \Sast; \Q[\N\Dt\nu]\bigr)
\rightarrow H^{i}\bigl( \Gm\backslash \Sast; \C^\ast\bigr)\,.\]
The two outer spaces are zero, hence
$H^{i}\bigl( \Gm\backslash \Sast; \Q[\N\Dt\nu]\bigr)$ vanishes as well.
We use Proposition~\ref{prop-pcshc} for the relation with parabolic
cohomology spaces.
\end{proof}

\rmrk{Remarks}Taking $i=\ds$ in the corollary we get
$H^\ds_\pb\bigl( \Gm; \N\Dt\nu\bigr)=\{0\} $. With
Proposition~\ref{prop-H0d} we conclude that the space of coinvariants
$(\N\Dt\nu)_\Gm$ is zero. This is in contrast with the fact that
$H^\ds_\pb\bigl( \Gm; \N{\excg}{} \bigr) = (\N{\excg}{})_\Gm$ is
non-zero. Otherwise the construction of $u_\ps$ in
Proposition~\ref{prop-cohE} would not work.

A consequence for the case $\ds\geq 3$ is that the exact sequence in
\eqref{exEGN} implies that
\be \label{HisoEG}H^{\ds-1}_\pb\bigl( \Gm;\G{\om^0,\excg}\nu\bigr) \cong
H^{\ds-1}_\pb \bigl( \Gm;\E^{\om^0,\excg}_ \nu\bigr)\,.\ee

\subsection{The image of spaces of cusp forms in parabolic cohomology
spaces}\label{sect11-icf}
We have the morphisms of $\Gm$-modules
\be \E^{\om^0,\excg}_\nu \rightarrow \G{\om^0,\excg}\nu\,,\quad\text{
and } \G{\om^0,\excg}\nu\rightarrow\W{\om^0,\excg}\nu\,.\ee
The former morphism is injective and the latter one surjective. The
composition is injective, and we can consider $\E^{\om^0,\excg}_\nu$ as
a submodule of $\W{\om^0,\excg}\nu$. Similarly, we have
$\E^{\om^0,\infty,\excg}_\nu$ as a submodule of
$\W{\om^0,\infty,\excg}\nu$.

\begin{prop}\label{prop-dsc}Let
$\nu \in \CC\setminus \frac12\ZZ_{\leq -1}$, $|\re\nu|<\rho$. We have
the direct sum decompositions of vector spaces
\badl{dsW} H^{\ds-1}_\pb \bigl( \Gm;\W{\om^0,\infty,\excg}\nu\bigr)^\bdc
&= \btw_\nu \A^0(\Gm;\nu) \oplus H^{\ds-1}_\pb \bigl(
\Gm;\E^{\om^0,\infty,\excg}_\nu\bigr)^\bdc \,,\\
H^{\ds-1}_\pb \bigl( \Gm;\W{\om^0,\excg}\nu\bigr)^\bdc &= \ii\,\btw_\nu
\A^0(\Gm;\nu) \oplus H^{\ds-1}_\pb \bigl(
\Gm;\E^{\om^0,\excg}_\nu\bigr)^\bdc\,,
\eadl
where $\A^0(\Gm;\nu) \cong \btw_\nu \A^0(\Gm;\nu)  \cong 
\ii\,\btw_\nu \A^0(\Gm;\nu) $.
\end{prop}
\begin{proof}The isomorphism in \eqref{HisoEG} and
Proposition~\ref{prop-qIwe} show that the kernel of $\alw_\nu$ in the
diagram~\ref{Hdiag} is equal to the image of
$H^{\ds-1}\bigl( \Gm;\E^{\om^0,\excg}_\nu\bigr)$ in
$H^{\ds-1}_\pb\bigl( \Gm;\W{\om^0,\excg}\nu)$ intersected with
$H^{\ds-1}_\pb\bigl( \Gm;\W{\om^0,\excg}\nu)^\bdc$. The condition `bdc'
concerns the position of the boundary singularities, and can be applied
to $\E^{\om^0,\excg}_\nu$ as well as to $\W{\om^0,\excg}\nu$. Together
with Proposition~\ref{prop-pi} this gives the second decomposition.
Pulling back along $\ii$ in diagram~\eqref{Hdiag} we obtain the first
decomposition.
\end{proof}

\begin{proof}[Proof of part~\eqref{mnthm:nsp} of Theorem~\ref{mainthm}]
The recapitulation in \S\ref{sect9-recap} shows that for
$\nu \in \CC\setminus \frac12\ZZ_{\leq -1}$
\[\V{\om^0,\infty,\excg}\nu \subset \V {\om_0}\nu =
\mathop{\lim}_{\longrightarrow}\V\om\nu\left( (\partial S) \setminus
E\right)\,.\]
In this direct limit $E$ runs over finite sets of cusps. Elements of
$\E^{\om^0,\infty,\excg}_\nu$ represent boundary germs in
$\W\om\nu\left( \partial S \setminus E\right)$ for such a finite set
$E$. We indicate by $\rho_\nu \E^{\om^0,\infty,\excg}_\nu$ the space of
restrictions of these boundary germs. (See Definition~\ref{def-res} and
Theorem~\ref{thm-isoVW} for the restriction maps $\rho_\nu$.)

Application of the restriction map $\rho_\nu$ to \eqref{dsW} leads for
$\nu \in \CC\setminus\frac12\ZZ_{\leq-1}$ to
\bad \btv_\nu &= \ii \btw_\nu: \A^0(\Gm,\nu) \rightarrow
H^{\ds-1}_\pb\bigl( \Gm;
\V{\om^0,\infty,\excg}\nu\bigr)^\bdc\qquad\text{(injective)}\,,\\
H^{\ds-1}_\pb\bigl( \Gm;& \V{\om^0,\infty,\excg}\nu\bigr)^\bdc =
\btv_\nu\A^0(\Gm,\nu) \oplus H^{\ds-1}_\pb\bigl( \Gm; \rho_\nu
\E^{\om^0,\infty,\excg}_\nu\bigr)^\bdc \,.
\ead
This gives the injectivity of $\bt_\nu$ in Theorem~\ref{mainthm}, and
the assertion in part~\eqref{mnthm:nsp}.
\end{proof}

%% file: ccro2-12-gef.tex

\bigskip

\def\flnm{ccro2-12-gef}

\section{Global eigenfunctions of the Laplace operator} \label{sect12}

We defined in \eqref{exEGN} the space
$\E^{\om^0,\excg}_\nu\subset \E_\nu(S)$ as the kernel of
$\Dt_\nu: \G{\om^0,\excg}\nu \rightarrow \N\Dt\nu$. The subspace
$\E^{\om^0,\infty,\excg} $ is characterized by the condition
$\rho_\nu f\in \V\infty\nu(\partial S)$.
To complete the proof of the main result, Theorem~\ref{mainthm}, stated
in the introduction, we prove in this section the following result.

\begin{thm}\label{thm-E} If $\nu\in \CC\setminus \frac12\ZZ$ and
$|\re\nu|<\rho$, then $\E^{\om^0,\excg}_\nu=\{0\}$.
\end{thm}

\begin{proof}[Proof of part \eqref{mnthm:nugrl} of
Theorem~\ref{mainthm}] Immediate from Theorem \ref{thm-E}, together
with part \eqref{mnthm:nsp} of Theorem~\ref{mainthm}.
\end{proof}

We will use only part of the properties of functions in
$\E^{\om^0,\excg}_\nu$, and will actually prove
\begin{prop}\label{prop-E}
Let $\nu\in \CC\setminus \frac12\ZZ$ and $|\re\nu|<\rho$. If
$f\in \E_\nu(S)$ is bounded and represents a $\nu$-analytic boundary
germ on $\partial S \setminus E$ for some finite subset
$E\subset \partial S$, then $f=0$.
\end{prop}
This implies Theorem~\ref{thm-E}, and will be a consequence of
Proposition~\ref{prop-Edist}.

\subsection{Poisson transformation}\label{sect-Poi-tr}Central in this
section is the Poisson transformation. It sends functions
$\ph \in C^\infty(\partial S)$ to elements of $\E_\nu(S)$:
\be \poiss\nu \ph(g\,\orgn)
= \int_{b\in \partial S} r_\nu(b; g\,\orgn) \, \ph(b)\, db = \int_{k\in
K }\Ph_\nu\bigl( k^{-1} g\bigr)\, \ph(k\,\infty)\, dk \,.\ee
This transformation makes sense more generally: we may replace $\ph$ by
a distribution on $\partial S$, or even by a linear form on
$C^\om(\partial S)$.

\rmrk{Distributions}We recall that a distribution on $\partial S$ is a
linear form on $C^\infty(\partial S)$ that is continuous for the
topology determined by the supremum norm of all derivatives on compact
sets.
(Since $\partial S$ is compact we can work with the supremum norms of
derivatives on $\partial S$.) A function $\ph\in C^\infty(\partial S)$
determines the distribution $\ph:\ch \mapsto \langle \ph,\ch\rangle = 
\int_{k\in K} \ph(k\,\infty)\,\ch(k\,\infty)\, dk$. By
\il{C-inf}{$C^{-\infty}(\cdot)$} $C^{-\infty}(\partial S) $ we indicate
the space of distributions on $\partial S$.

The action $\tau_\nu$ on $C^\infty(\partial S)$ satisfies
$\bigl\langle \tau_\nu(g) \ps , \ph\bigr\rangle =
\bigl\langle \ps , \tau_{-\nu}(g^{-1})\ph\bigr\rangle$. This relation
extends to distributions. Providing $C^{-\infty}(\partial S)$ with the
action $\tau_{-\nu}$ we obtain the dual $\V{-\infty}{-\nu}(\partial S)$
of the representation $\V\infty\nu(\partial S)$.

For a distribution $\bt$ on $\partial S$ we use the notation $\bt(f)$ if
$f\in C^\infty(\partial S)$. If we have a function given by some
expression $E$, then we use the notations
$\bt\left( b \mapsto E(b) \right)$ and $\bt\bigl( E(\cdot)\bigr)$.

\rmrk{Poisson transformation}The \il{Ptrf}{Poisson
transformation}Poisson transform $\poiss\nu\bt$ is the function on $S$
obtained by applying $\bt$ to
$b\mapsto r_\nu(b;x)$:\il{poiss}{$\poiss\nu$}
\be \poiss\nu\bt(x) = \bt\bigl( r_\nu(\cdot; x) \bigr)\,.\ee
Property \tlref{Rnu:eiggrow}{Rnu:eigen} in Proposition~\ref{prop-Rnu}
implies that $\poiss\nu \bt \in\E_\nu(S)$. The Poisson transformation
intertwines the action $\tau_{-\nu}$ with the action $L$ on
$\E_\nu(S)$. Lewis has shown \cite[Theorem 5.3]{Le78} that the image
$\poiss \nu C^{-\infty}(\partial S)$ is equal to the space of all
elements $f\in \E_\nu(S)$ with polynomial growth at the boundary. He
works under the condition that $\nu \not\in \ZZ$ and that $\ld=-i\nu$
is ``simple''. He discusses that these two conditions lead to the
requirements $\nu \not\in \ZZ$ and $\nu \not \in -p/2+\ZZ_{\leq 0}$. So
the use of this theorem forces us to add $\nu \not\in \ZZ_{\geq 0}$ to
the conditions on the spectral parameter.

The \il{supp}{support}support \il{sp}{$\supp(\cdot)$}$\supp(\ph)$
of a function $\ph \in C^\infty(\partial S)$ is the closure of the set
$x\in \partial S$ such that $\ph(x) \neq 0$. The support of a
distribution $\bt$ on $\partial S$ consists of the points
$x\in \partial S$ such that $\bt(\ph) \neq 0$ for some
$\ph \in C^\infty(\partial S)$ with $x\in \supp \ph$.

A large part of this section will be needed to establish the following
result.
\begin{prop}\label{prop-supp-sing}Let $\nu\in \CC\setminus \frac12\ZZ$
and $|\re\nu|<\rho$. Let $\bt\in C^{-\infty}(\partial S)$. The set
$\partial S$ is the disjoint union of the sets
\[ \supp(\bt) \qquad\text{ and }\qquad \bigl\{ x\in \partial S\;:\;
\poiss\nu \bt \in
(\W\om\nu)_x\bigr\}\,.\]
\end{prop}
The statement that $\poiss\nu \bt$ is an element of the stalk
$(\W\om\nu)_p$ means that $\poiss\nu\bt$ represents a $\nu$-analytic
boundary germ on a neighborhood of $p$ in~$\partial S$.

Proposition~\ref{prop-E} will follow from
Proposition~\ref{prop-supp-sing} and the following result:
\begin{prop}\label{prop-Edist}Let $\nu\in \CC$ satisfy $\rho+\re\nu>0$
and $\nu \not\in \frac12\ZZ_{\leq-1}\cup \ZZ_{\geq 0}$. Suppose that
the support of $\bt\in C^{-\infty}(C)$ consists of finitely many
elements of $\partial S$.
If $\poiss\nu\bt$ is bounded on $S$, then $\bt=0$.
\end{prop}
We give a proof in the next subsection.

\subsection{Distributions on \intitle{N\,\nul}}Functions
$\ph \in C^\infty(\partial S)$ determine a function $\ph^\prj$ on
$N\,\nul\cong \RR^{p+q}$ by
\il{prjm}{$\ph^\prj$}$\ph^\prj(x) = \ph\bigl( \kI(\nm(x)\wm)\,\inft)$.
If $\ph$ has support in $N\,\nul$, then integration over $\partial S$
can be formulated as integration over
$(\partial S)\setminus\{\inft\} = N\,\nul \cong \RR^{p+q}$
\be \int_{\partial S} \ph(b) \, db = \int_{\RR^{p+q}} \ph^\prj(x) \, a_0
\, \bigl((1+c\|x^{(1)}\|^2)^2+4c\|x^{(2)}\|^2\bigr)^{-2\rho} \, dx\,.
\ee
with the standard Lebesgue measure $dx$ on $\RR^{p+q}$, and where
$a_0>0$ is chosen in accordance with the total volume $1$ of
$\partial S$ for the invariant measure \il{db}{$db$}$db$ on $K/M$.
(Compare \cite[(3) p~1, Theorem 1.14 p~52]{Hel70}.)
We have seen in Proposition~\ref{prop-Id} that
\be (1+c\|x^{(1)}\|^2)^2+4c\|x^{(2)}\|^2 = \tJ\bigl(
\wm\nm(x)\bigr)^{-2}\,.\ee
In this way we can work with functions on $\partial S$ in terms of
functions on $N\,\nul$. In the context of $\PSL_2(\RR)$ this is called
the \emph{projective model} in \cite[\S2.1]{BLZ15}.

Helgason gives in \cite[Corollary 1.18 p 65]{Hel70} an explicit formula
for the Poisson kernel on $N\,\nul\cong \RR^{p+q}$. Helgason works with
$\tI\bigl( \am(y)^{-1}k\bigr)$, and we take
$\tJ\bigl( k^{-1}\am(y) \bigr)$. For $x=(x^{(1)},x^{(2)})\in \RR^{p+q}$
and $y>0$ we obtain\il{rnu-pl}{$r_\nu^\prj;\; r_\nu^\lnm$}
\be\label{Poiss-expl} r_\nu^\prj \bigl( x; \, \am(y)\, \orgn)
= y^{\rho+\nu} \Bigl( \frac{ (1+c\|x^{(1)}\|^2)^2+ 4 c \|x^{(2)}\|^2} {
(y^2+c\|x^{(1)}\|^2)^2+ 4 c \|x^{(2)}\|^2} \Bigr)^{(\rho+\nu)/2}\,.\ee
The positive constant $c$ is not normalized here in the same way as in
\cite{Hel70}; see the proof of Proposition~\ref{prop-Id}.
\smallskip

Although conceptually nice, the projective model is less convenient than
the \emph{line model} of the functions \il{lnm}{$\ph^\lnm$}
$\ph^\lnm(x) =\bigl((1+c\|x^{(1)}\|^2)^2+
4 c \|x^{(2)}\|^2\bigr)^{-(\rho+\nu)/2} \ph^\prj(x)$. In the line model
the kernel of the Poisson transformation takes the form
\bad\label{Poiss-lnm}
r_\nu^\lnm(x; \am(y) \,\orgn) &= y^{\rho+\nu} \Bigl( \bigl(
y^2+c\|x^{(1)}\|^2)^2+ 4c \|x^{(2)}\|^2\Bigr)^{-(\rho+\nu)/2}\\
& =\tJ \bigl( \wm \nm(x) \am(y) \bigr)^{\rho+\nu} \,.\ead
See \eqref{tJexpl}.

If $n, n_1\in N$, then
\be \label{rnulnm-sub}
r_\nu^\lnm(n;n_1\am(y)\,\orgn) = r_\nu^\lnm(n_1^{-1} n ;\am(y)\, \orgn)\,.
\ee

\begin{proof}[Proof of Proposition~\ref{prop-Edist}] Let $\bt$ be a
distribution with finite support. With a partition of unity we can
write it as a sum of distributions with support consisting of one
point. The action of $K$ enables us to consider distributions with
$\supp(\bt)=\{\nul\}$.

In this way we have a distribution $\bt$ on $\RR^{p+q} \cong N \,\nul$
with support $\{0\}$. Such distributions are a linear combination of
finitely many derivatives of the $\dt$-distribution
$\dt: \ph \mapsto \ph(0)$. We use multi-indices
$\al=(\al_1,\ldots,\al_{p+q})$ to describe derivatives as
$\partial^\al =
\partial_{x_1}^{\al_1}\cdots \partial_{x_{p+q}}^{\al_{p+q}}$,
and put $|\al|= \sum_{j=1}^{p+q}\al_j$. We put $N(x,y)=
N(x^{(1)},x^{(2)},y)= \bigl(y^2+c \|x^{(1)}\|^2\bigr)^2+ 4 c
\|x^{(2)}\|^2$ and write $\rho+\nu = 2b$. By \eqref{rnulnm-sub} we have
\be \poiss\nu \bigl( \partial ^\al \dt \bigr)
\bigl(\nm(-x)\am(y)\,\orgn\bigr)
=(-1)^{|\al|} \partial^\al \frac {y^{2b}}{N(x,y)^b} \,.\ee
Since $\partial ^\al \dt $ is a non-zero distribution, its Poisson
transform is non-zero, and is analytic on~$S$. So
$\poiss\nu \bigl( \partial ^\al \dt \bigr)$ is not identically zero
near $\orgn$. Hence the polynomial $p_\al\bigl( x,y\bigr)$ such that
\be (-1)^{|\al|}\partial^\al \frac{y^{2b}}{N(x,y)^b} = y^{2b} p_\al(x,y)
N(x,y)^{-b-|\al|} \ee
is not the zero polynomial. In each differentiation step the total
degree of $p_\al$ increases at most by $3$.

If $p_\al(0,y) \neq 0$ the function
$\poiss\nu \bigl( \partial ^\al \dt \bigr) 
\bigl(\nm(-x)\am(t)\,\orgn\bigr)$ is given by
$p_\al(0,y) y^{-2b-4|\al|}$ with $3|\al|+\re(-2b-4|\al|)<0$, since the
total degree of $p_\al$ is smaller than $3|\al|$, and since $N(x,y)$
has degree $4$ in~$y$. It is unbounded as $y\downarrow 0$. If
$p_\al(0,y)=0$ we take a curve
\be \label{vcxy} y \mapsto \bigl( a_1 y^{d_1},\ldots, a_{p+q}
y^{d_{p+q}}\bigr)\ee
with $d_j=\frac12$ for $1\leq j \leq p$ and $d_j=1$ for
$p+1\leq j \leq p+q$, and $a_j$ such that $p_\al$ is not identically
zero along this curve by similar degree arguments. Then the Poisson
transform is unbounded along this curve. So the Poisson transforms of
the simple derivatives $\partial^\al\dt$ are unbounded.

So let us consider a finite linear combination
$\bt = \sum_i c_i \partial^{\al(i)}\dt$. Its Poisson transform at
$\nm(x)\am(y)\,\orgn$ is the linear combination
\[ \poiss\nu \bt=\sum_i c_i y^{2b} p_{\al(i)}(x,y) N(x,y)^{-b-|\al(i)|}\,.\]
We can write $\poiss \nu \bt\bigl( \nm(-x)\am(y)\,\orgn)$ as
$ y^{2b} P(x,y)\, N(x,y)^{-b-M}$, with $M\in \ZZ_{\geq 0}$, and a
polynomial $P$ with total degree smaller than $M$. This new function
cannot be identically zero if $\bt$ is not the zero linear combination,
by Lewis's theorem. We again construct a curve of the same form as in
\eqref{vcxy} through $\RR^{p+q}\times(0,\infty)$ approaching
$0\in \RR^{p+q+1}$ along which $\poiss\nu\bt$ is unbounded.

So a bounded function in $\E_\nu(S)$ cannot be of this form.
\end{proof}

\subsection{Approximation of distributions} We still have to prove
Proposition~\ref{prop-supp-sing}. It can be localized to a neighborhood
of $\nul$ in~$N\,\nul$. First we discuss the approximation of
distributions by functions.

The convolution of functions $\ph,\ps\in C_c^\infty(N)$ is defined by
\be \ps\ast \ph (n) = \int_N\ps(nn_1^{-1}) \ph(n_1)\, dn_1 = \int_N
\ps(n_1) \ph(n_1^{-1} n ) \,dn_1\,.\ee

The linear map $\ph \mapsto \ps\ast \ph$ is continuous for the supremum
norm of all derivatives of $\ph$ up to a given order. For
$\ch\in C^\inft(N)$ we have
\be\int_N \ch(n) (\ps\ast \ph)(n)\, dn = \int_N (\check \ps \ast
\ch)(n)\, \ph(n)\, dn\,,\ee
with $\check \ps(n)= \ps(n^{-1})$. This shows how to deal with the
convolution of $\ps\in C_c^\infty(N)$ with a distribution $\bt$ with
compact support:
\be\label{convdist} (\ps \ast \bt) : \ch \mapsto \bt\bigl( \check \ps
\ast \ch\bigr)\,. \ee
The distribution $\check\ps \ast\bt $ is given by a function in
$C_c^\infty(N)$; see eg.~\cite[17-10, Thm~7.5]{Tre67}.

A Dirac sequence with shrinking support is a sequence
$\bigl( \ps_m \bigr)_m$ of elements in $C_c^\infty(N)$ such that
$\ps_m\geq 0$, and $\int_N \ps_m(n)\, dn=1$ for all $m$, for which the
supports $\supp\ps_m$ form a shrinking sequence of compact sets with
intersection $\{e\}$. See \cite[Chap VIII, \S3]{Lng93}. We can choose
the Dirac sequence such that $\check \ps_m = \ps_m$. The important
property is that $\ps_m \ast \ch$ converges uniformly to
$\ch \in C^\infty_c(N)$, and the same holds for all derivatives
(under right differentiation).

\begin{lem}\label{lem-conv}Let $\bt$ be a distribution on $N$ with
compact support. If $\ps_m$ is a Dirac sequence with shrinking support,
then
\begin{enumerate}[label=$\mathrm{(\alph*)}$, ref=$\mathrm{\alph*}$]
\item \label{conv:lim} For each $\ch\in C^\infty(N)$
\[ \bt(\ch) = \lim_{m\rightarrow \infty }\int_{n\in N}(\ps_m\ast\bt)(n)
\, \ch(n)\, dn\,.\]

\item \label{conv:supp} Let $n\in N$. Then $n\in \supp(\bt)$ if and only
if $n\in \supp(\ps_m\ast \bt)$ for all sufficiently large $m$.

\item \label{conv:anal} If $\poiss\nu \bt$ represents an element of the
stalk $(\W\om\nu)_x$, then $\poiss\nu(\ps_m\ast \ch)$ represents an
element of this stalk for all sufficiently large~$m$.
\end{enumerate}
\end{lem}
We note that part~\eqref{conv:supp} gives an equivalence, whereas part
\eqref{conv:anal} gives only one implication.
\begin{proof}
With the relation
\[\int_N (\ps_m\ast \bt)(n)\ch(n)\, dn = \int_N \bt(n_1)\, (\check
\ps_m(n_1)\ast\ch)(n_1)\, dn_1\,,\]
part \eqref{conv:lim} follows directly from the convergence of
$\ps_m\ast \ch$ to $\ch$.

For part~\eqref{conv:supp} we consider $\ch\in C_c^\infty$. Part
\eqref{conv:lim} shows that $\bt(\ch) \neq 0$ if and only if
$(\ps_m\ast \bt)(\ch)\neq0$ for all large values of $m$.

For part \eqref{conv:anal} suppose that $\poiss\nu \bt( n\am(y)\,\orgn)$
is $\nu$-analytic on an open set $I\subset N\,\nul$ containing $x$.
Then there is an analytic function $B$ on an open set $U\subset \hat S$
containing $I$ such that
\be \bt \bigl( r^\lnm_\nu(\cdot, \nm(\xi)\am(y)\,\orgn) \bigr) =
y^{\rho+\nu} B(\nm(\xi)\am(y)\,\orgn)
\ee
for $\nm(\xi)\am(y)\,\orgn\in S \cap U$. Since the support of $\ps_m$
shrinks to $e$, we have for all large values of $m$
\be \poiss\nu(\ps_m\ast \bt)(\nm(\xi)\am(t)\,\orgn) = y^{\rho+\nu}
\int_N \ps_m( n_1) B\bigl( n_1^{-1} \nm(\xi)\am(y)\,\orgn\bigr)\, dn_1
\ee
for all $\nm(\xi)\am(y)\,\orgn\in S \cap U_m$ for some neighborhood
$U_m$ of $x$ in $\hat S$.
\end{proof}

\subsection{Proof of Proposition~\ref{prop-supp-sing}} To prove
Proposition~\ref{prop-supp-sing} we have to show that for all
$b\in \partial S$
\be \label{supp-sing}
b\not \in \supp(\bt)
\; \Leftrightarrow\; \poiss\nu \bt \text{ represents an element of
}(\W\om\nu)_b \,.\ee
We first consider the forward implication.
\begin{lem}\label{lem-supp-anal}Let $\nu \in \CC\setminus\frac12\ZZ$,
and $\bt\in \V{-\infty}\nu(\partial S)$. If
$b\in (\partial S)\setminus \supp(\bt)$, then $\poiss \nu\bt$
represents an element of the stalk $(\W\om\nu)_b$.
\end{lem}
\begin{proof}If $\supp(\bt) = \partial S$ the lemma holds trivially. So
let $\supp(\bt)$ be a compact subset of $\partial S$ not containing a
given point $b$. The Poisson transform
\[ \poiss\nu \bt( k\am(t)\,\orgn) = \bt\Bigl( b_1 \mapsto
r_\nu(b_1;k\am(t)
\,\orgn)\Bigr)\]
depends only on the values of $r_\nu(b_1;k\am\,\orgn)$ for
$b_1\in \supp(\bt)$. We write
\[r_\nu(b_1;k\am(t)\,\orgn) = t^{-\rho-\nu} A(b_1;k\am(t)\,\orgn)\,,\]
since the factor $t^{-\rho-\nu}$ is left out of consideration in the
definition of $\nu$-analytic boundary germs. So for $k\in K$ such that
$k\,\infty$ varies through a neighborhood of $b$ we have
\begin{align*}\poiss\nu \bt( k\am(t)\,\orgn) & =\bt\Bigl( b_1 \mapsto
t^{-\rho-\nu} A(b_1, k \am(t)\,\orgn) \Bigr)\\
&= t^{-\rho-\nu} \,\bt\Bigl( b_1 \mapsto A(b_1,k\am(t)\,\orgn) \Bigr)\,.
\end{align*}
The function $A(b_1, x)$ extends analytically from $x\in S $ near $b$ to
a neighborhood of $b$ in $\hat S$. So the Poisson transform
$\poiss\nu \bt$ represents a $\nu$-analytic boundary germ on a
neighborhood of $b$ in $\partial S$.\end{proof}

We are left with the converse implication in \eqref{supp-sing}.

\rmrk{Reduction to functions} We can localize $\bt$ such that
$\supp(\bt) \subset N\,\nul$, and $b=\nul$. Parts \eqref{conv:anal} and
\eqref{conv:supp} in Lemma~\ref{lem-conv} implies that there are
functions $\ph\in C_c^\infty(N\,\nul)$ approximating $\bt$ such that
the converse implication in \eqref{supp-sing} is implied by the
analogous implication with $\bt$ replaced by
$\ph \in C_c^\infty(N\,\nul)$. So we have to prove:
\begin{lem}\label{lem-anal-supp} Let $\nu\in \CC\setminus \frac12\ZZ$,
$|\re\nu|<\rho$. If $\poiss\nu \ph$ represents a $\nu$-analytic
boundary germ on an open set $I\subset N\,\nul$, then $\ph(b) = 0$ for
all $b\in I$.
\end{lem}
\begin{proof}
With the action of $N$ we can transform any $b\in I$ to $b=\nul$. Then
Lemma~\ref{lem-anal-nul} below implies that $\ph(\nul)=0$. We can carry
out the transformation of $b$ to $\nul$ for any $b\in I$. Hence each
$b\in I$ is not in $\supp \ph$.\end{proof}

\begin{lem}\label{lem-anal-nul}Let $\nu\in \CC\setminus \frac12\ZZ$ and
$|\re\nu|<\rho$. If
\[ y \mapsto y^{-\rho-\nu}\,\poiss\nu \ph( \am(y)\,\orgn)\]
extends analytically to $y=0$, then $\ph(\nul)=0$.
\end{lem}
The proof requires auxiliary lemmas, and is completed at the end of this
subsection.

\rmrk{Explicit Poisson integral}To prove Lemma~\ref{lem-anal-nul} we use
the explicit formula in \eqref{Poiss-lnm}, which is based on
$N\,\nul \cong \RR^{p+q}$.
\badl{pph1} \poiss\nu \ph(\am(y)\,\orgn) &= y^{\rho+\nu} \int_{x\in
\RR^{p+q}}\ph\bigl( x^{(1)},x^{(2)}\bigr)\\
&\qquad\hbox{} \cdot
\Bigl( \bigl( y^2+c\|x^{(1)}\|^2\bigr)^2+4
c\|x^{(2)}\|^2\Bigr)^{-(\rho+\nu)/2}\, dx\,.\eadl
We have $\ph\in C_c^\infty(\RR^{p+q})$ and $y>0$. Convergence is
guaranteed by the compactness of $\supp(\ph)$. The integral is analytic
in $y\in (0,\infty)$.

The transformation of variables $x^{(1)}\mapsto y x^{(1)}$,
$x^{(2)} \mapsto y^2 x^{(2)}$, implies $dx \mapsto y^{2\rho} \, dx$.
This leads to the alternative explicit formula
\badl{pph2} \poiss\nu \ph(\am(y)\,\orgn) &= y^{\rho-\nu} \int_{x\in
\RR^{p+q}}\ph\bigl( yx^{(1)},y^2x^ {(2)}\bigr) \\
&\qquad\hbox{} \cdot
\Bigl( \bigl( 1+c\|x^{(1)}\|^2\bigr)^2+4
c\|x^{(2)}\|^2\Bigr)^{-(\rho+\nu)/2}\, dx\,.\eadl

For our purpose it is important to leave the factor $y^{\rho+\nu}$ out
of consideration. We study the following
function\il{Ap}{$A(\ph,\rho+\nu,y), \; A_\dt(\ph;\rho+\nu,y)$}
\badl{fctA} A(\ph&;\rho+\nu, y) = y^{-\rho-\nu}\, \poiss \nu \ph \bigl(
\am(y)\,\orgn\bigr)\\
&= \int_{x\in\RR^{p+q}}\ph\bigl( x^{(1)},x^{(2)}\bigr)
\Bigl( \bigl( y^2+c\|x^{(1)}\|^2\bigr)^2+4
c\|x^{(2)}\|^2\Bigr)^{-(\rho+\nu)/2}\, dx
\\
&= y^{-2\nu} \int_{x\in\RR^{p+q}}\ph\bigl( yx^{(1)},y^2x^ {(2)}\bigr) \\
&\qquad\hbox{} \cdot
\Bigl( \bigl( 1+c\|x^{(1)}\|^2\bigr)^2+4
c\|x^{(2)}\|^2\Bigr)^{-(\rho+\nu)/2}\, dx
\eadl
modulo the subspace \il{an}{$\anal$}$\anal$ of functions on $(0,\infty)$
that extend analytically to a neighborhood of $[0,\infty)$ in~$\RR$.

We split up the Poisson integral by integration, separately over the
region
\be \label{ell-def} \ell(\dt) = \bigl\{ x\in \RR^{p+q}\;:\;
\|x^{(1)}\|\leq \dt \text{ and } \|x^{(2)}\|\leq \dt^2\bigr\}\,,\ee
and over its complement in $\RR^{p+q}$, for $\dt>0$.

\begin{lem}\label{lem-Idt}For $\dt>0$
put\il{Idt}{$I_\dt(\ph;\rho+\nu,y)$}
\bad I_\dt(\ph; \rho+\nu,y) &= \int_{x\in \ell(\dt)} \ph(x^{(1)},
x^{(2)}) \\
&\qquad\hbox{} \cdot
\Bigl( \bigl( y^2+c\|x^{(1)}\|^2\bigr)^2+4
c\|x^{(2)}\|^2\Bigr)^{-(\rho+\nu)/2}\, dx\,.\ead
Then $\nu \mapsto I_\dt(\ph;\rho+\nu,y)$ is holomorphic on $\CC$, and
\be y^{-\rho-\nu} \poiss \nu(\ph)\bigl( \am(y)\, \orgn) \equiv \,
I_\dt(\ph;\rho+\nu,y)
\bmod \anal\,.\ee
\end{lem}
\begin{proof}
We have
\be \poiss \nu \ph(\am(y)\,\orgn) = y^{\rho+\nu} I_\dt(\ph;\rho+\nu,y) +
y^{\rho+\nu} A_\dt(\ph;\rho+\nu,y)\,,\ee
where
\bad A_\dt(\ph;\rho+\nu,y) &= \int_{\|x^{(1)}\|> \dt \text{ or }
\|x^{(2)}\|> \dt^2}\ph\bigl( x^{(1)},x^{(2)}\bigr)\\
&\qquad\hbox{} \cdot
\Bigl( \bigl( y^2+c\|x^{(1)}\|^2\bigr)^2+4
c\|x^{(2)}\|^2\Bigr)^{-(\rho+\nu)/2}\, dx\,.\ead
The compact support of $\ph$ takes care of convergence of the integrals
for $I_\dt$ and $A_\dt$, and implies that both depend holomorphically
on $\nu \in \CC$. Furthermore, the quantity
$ \bigl( y^2+c\|x^{(1)}\|^2\bigr)^2+4 c\|x^{(2)}\|^2$ is larger than
$y^4$. Hence the integral for $A_\dt$ depends analytically on
$y\in \RR$ and $A_\dt(\ph;\rho+\nu,y) \in \anal$.
\end{proof}

\begin{lem}\label{lem-I-estm}Let $\re\nu>-\rho$. Suppose that
$\ph (x) = \oh \bigl(\|x\|\bigr)$ on its domain. If $0<\dt\leq y\leq 1$, then
\bad I_\dt\bigl( \ph;\rho+\nu,y) &\ll \dt^{2\rho+1} y^{-2(\rho+\nu)}\\
&\ll y^{-2\nu+1}\qquad\text{ if } \dt=y\,.\ead
\end{lem}
\begin{proof}On the domain of integration the integrand is estimated by
\[\oh \left(
\bigl(\|x^{(1)}\|+\|x^{(2)}\|\bigr)y^{-2(\rho+\nu)}\right)\,.\]
The integral of $x^{(1)}$ over a ball in dimension $p$ is, up to a
factor, the integral of $r_1=\bigl\|x^{(1)}\|$ against the measure
$r_1^{p_1 -1}\, dr_1$, and similarly for the ball in $\RR^q$. In this
way, we obtain:
\begin{align*}
\int_{\ell(\dt) }& \bigl( \|x^{(1)}\|+\| x^{(2)}\|\bigr)\, dx\\
&\ll \int_{r_1=0}^\dt \int_{r_2=0}^{\dt^2} \, (r_1+r_2)
r_1^{p-1}r_2^{q-1}\, dr_1\,dr_2\\
&= \int_{r_1=0}^\dt r_1^p \, dr_1\, \int_{r_2=0}^{\dt^2} r_2^{q-1}\,
dr_2
+ \int_{r_1=0}^\dt r_1^{p-1}\, dr_1 \,\int_{r_2=0}^{\dt^2} r_2 ^q\, dr_2
\\
&\ll \dt^{p+1}\dt^{2q} + \dt^p \dt^{2(q+1)} \ll \dt^{2\rho+1}\,.
\end{align*}
This gives the estimate
\[ I_\dt \ll \dt^{2\rho+1} y^{-2(\rho+\nu)} \ll y^{1-2\nu}\,. \]
In the last step we use $\dt\leq y$.
\end{proof}

We can compute $I_\dt(\ph;\rho+\nu,y)$ explicitly if $\ph$ is polynomial
on the domain of integration. To carry this out for $\ph(x)=1$ on
$\ell(\dt)$, we first consider the
integral\il{PI}{$\PI{n,\dt}(h,\cdot)$}
\be \label{PIdef}\PI{n,\dt}(h,F) : t\mapsto \int_{x\in\RR^n ,\,
\|x\|\leq \dt} F(t^2+h\|x\|^2)\, dx \ee
for $F$ analytic on $(0,\infty)$. We take $\dt>0$ and
$n\in \ZZ_{\geq 1}$ and denote by $c_n$ the positive factor such that
\il{cn}{$c_n$}$dx= c_n \, r^{n-1}\, dr $ for $r=\|x\|$. To avoid long
notations we indicate the power function $u\mapsto |u|^{-w}$ as
$|\cdot|^{-w}$.

\begin{prop}\label{prop-PI}{\rm The integral $\PI{n,\dt}$ for some
choices of $F$.}
\begin{enumerate}[label=$\mathrm{(\roman*)}$, ref=$\mathrm{\roman*}$]
\item \label{PI:wp} Let $ F(\xi) = \xi^w$ with $\re w>0$.
\begin{enumerate}[label=$\mathrm{(\alph*)}$, ref=$\mathrm{\alph*}$]
  \item\label{PI:gen} If $ w-\frac n2 \not \in \ZZ_{\leq 0}$ then
\badl{PIw-gen} \PI{n,\dt}\bigl( h,|\cdot|^{-w} \bigr)(t) &= \frac{ c_n
\,\Gf\bigl( \frac n2\bigr)\,\Gf\bigl( w-\frac n2\bigr)}{2\,
h^{n/2}\,\Gf(w)} t^{n-2w}\\
&\qquad\hbox{} +\frac{ c_n\,\dt^{n-2w}\, h^{-w}}{n-2w } \hypg21\Bigl( w,
w-\frac n2, w-\frac n2+1; -\frac{t^2}{h \dt^2}\Bigr)\,.
\eadl
\item \label{PI:w2}If $w=\frac n2$ then
\be \PI{n,\dt}\bigl( h,|\cdot|^{-w} \bigr)(t) =- c_n h^{-n/2} \log t +
\bigl( c_n h^{-n/2}\,\log\dt+(\mathrm{cst})\bigr)\,,\ee
where $(\mathrm{cst})$ is a real constant, not depending on $\dt$.
\item \label{PI:w2<}
If $w=\frac n2-k$ with $k\in \ZZ\cap [1,n/2-1]$ then
\bad \PI{n,\dt}\bigl(h,|\cdot|^{-w} \bigr)(t) &= - \frac{c_n\,(-1)^k
\Gf\bigl( \frac n2\bigr)}{h^{n/2} \, k!\; \bigl( \frac n2-k-1\bigr)!}
\, t^{2k}\, \log t\\
&\qquad\hbox{}
+ \frac{c_n \,(-1)^k \Gf\bigl( \frac n2\bigr)}{ h^{n/2} \, k!\;\bigl(
\frac n2-k-1\bigr)!}\bigl( \log \dt +(\mathrm{cst})\bigr)
\, t^{2k}\,.
\ead
\end{enumerate}
\item\label{PI:anal}If $F\in \anal$, then $\PI{n,\dt}(h,F) \in \anal$.
\end{enumerate}
\end{prop}

\begin{proof}We rewrite \eqref{PIdef} as
\be \PI{n,\dt}(h,F)(t)= c_n \int_{r=0}^\dt F(t^2+hr^2)\, r^{n-1}\, dr\,.
\ee
Applied with $F(\xi)=\xi^{-w}$ we get
\be\label{PI2F1} = \frac{ c_n \dt^n}{n} \hypg21\Bigl( w, \frac n2;\frac
n2+1; - \frac{\dt^2 h}{t^2}\Bigr)\,.\ee
This hypergeometric function can be written as a linear combination of
hypergeometric functions that are holomorphic at $0$. We use the
relation in \cite[Chap II, 2.9]{MOT53} expressing $u_1$ in terms of
$u_3$ and $u_4$ (in the notations used there). This relation writes the
holomorphic function $\PI{n,\dt}$ in $w$ as the sum of two meromorphic
terms. The singularities occur at
$w\in \frac{n}2+ \ZZ_{\leq 0}\cup \ZZ_{\leq 0}$. We use that
$\hypg21 (a,0;c;z) = 1$, and arrive at \eqref{PIw-gen}.

If $w\in\frac n2+\ZZ_{\leq 0}$ we have to work with another basis of the
solutions of the hypergeometric differential equation. The expression
in \eqref{PI2F1} is holomorphic in $w$ for $w\in \CC$. Figuring out how
the singularities in both terms in the right-hand side cancel each
other, we arrive at \tlref{PI:wp}{PI:w2} and \tlref{PI:wp}{PI:w2<}.

In part \eqref{PI:anal} we get the integral
\[ \int_{x\in \RR^n,\,\|x\|\leq \dt} F( t^2+h \|x\|^2) \, dx =c_n
\int_{r=0}^\dt F(t^2+h r^2) r^{n-1}\, dr\]
of an analytic function on a compact region, which results in an
analytic function of $t>0$. If $F$ extends analytically to a
neighborhood of $0$ in $\RR$, then the integral is analytic in $t$ on a
neighborhood of $0$.
\end{proof}

In the three cases in part \eqref{PI:wp} of this proposition the
integral $\PI{n,\dt}$ is the sum of two terms. The first term is a
multiple of a simple function of $t$, which is not analytically
extendable to $t=0$. The second term is analytically extendable to a
neighborhood of $t=0$. (We note that $k$ is positive in
\tlref{PI:wp}{PI:w2<}.)
In each of the three cases, the first term does not depend on the
truncation parameter~$\dt$.

\begin{proof}[Proof of Lemma~\ref{lem-anal-nul} for \intitle{q=0}]In
this case $\rho=\frac p2$, and $|\re\nu|<\frac p2$.

We write the function $\ph \in C_c^\infty(\RR^p)$ on the region
$\|x\|\leq \dt$ as the constant function $\ph(0)$ and
$\ph_1(x) = \ph(x)-\ph(0)$. The Poisson integral in \eqref{pph1}
simplifies to
\bad y^{-\rho-\nu}&\poiss\nu\ph\bigl( \am(y)\,\orgn) = \int_{x\in \RR^p}
\ph(x)
\,\bigl( y^2+c\|x\|^2\bigr)^{-\rho-\nu}\, dx\\
&\;\equiv\; I_\dt(\ph(0) + \ph_1;\rho+\nu,y) \bmod\anal
\ead
by Lemma~\ref{lem-Idt}. We take $0\leq \dt \leq y\leq 1$. The main term
is given by $\PI{p,\dt}(c,|\cdot|^{-p/2-\nu})$. So in
Proposition~\ref{prop-PI} \eqref{PI:wp} we use $n=p$ and
$w=\frac p2+\nu$. Hence $w-\frac n2=\nu$. The parts \eqref{PI:w2} and
\eqref{PI:w2<} occur for values of $\nu$ in $\frac12\ZZ_{\leq 0}$ which
we have to exclude anyhow, since we use Lewis's theorem in
\S\ref{sect-Poi-tr}. The main term is
\be \ph(0) \Bigl( b_0 y^{-2\nu} +\dt^{-2\nu} a_0(t^2/\dt^2)\Bigr)\,,\ee
with $b_0\in \RR$ and $a_0\in \anal$. The dependence on $\dt$ is
explicit.

Lemma~\ref{lem-I-estm} gives
\be\label{er-est} I_\dt( \ph_1;\rho+\nu,y) = \oh
\bigl(\dt^{p+1}y^{-p-2\nu}\bigr)\,. \ee
The assumption that $\poiss\nu\ph(\am(y))\, \in \anal$ extends
analytically to $y=0$ implies that modulo $\anal$ with
$0<\dt\leq y\leq 1$
\bad 0&\equiv \poiss \nu \ph\bigl( \am(y)\, \orgn)\\
&\equiv I_\dt(\ph(0)+\ph_1;\rho+\nu;y)&&\text{by Lemma~\ref{lem-Idt}}\\
&\equiv \ph(0) b_0 y^{-2\nu}+I_\dt(\ph_1;\rho+\nu;y)
&\quad&\text{by Proposition~\ref{prop-PI}}\\
&=\ph(0) b_0 y^{-2\nu}+ \oh \bigl( y^{1-2\nu}) &&\dt=y\text{ in
Lemma~\ref{lem-I-estm}}\,.
\ead
For $\nu \not\in \frac12 \ZZ$ the non-analytic factor $y^{-2\nu}$ cannot
cancel with an element of $\anal$. As $y\downarrow0$ it is larger than
the $\oh$-term. Since $b_0\neq0$ we need that $\ph(0)=0$.
\end{proof}

\begin{proof}[Proof of Lemma~\ref{lem-anal-nul} for \intitle{q\geq 1}]
The idea of the proof is the same as for $q=0$, except that now the
operators in \eqref{PIdef} have to be applied twice.

We start with $\PI{q,\dt^2}(4c,|\cdot|^{-(\rho+\nu)/2} \bigr)(t) $ for
$t>0$. We have $w=\frac{\rho+\nu}2 $ and $n=\frac q2$. We use $\dt^2$,
  in view of the condition $\|x^{(2)}\|\leq \dt^2$ in the definition of
$\ell(\dt)$ in~\eqref{ell-def}. So $w-\frac n2=\frac p4+\frac\nu 2$.
Since we work under the condition $\nu \not\in \frac12\ZZ_{\leq 0}$ the
exceptional values in part \tlref{PI:wp}{PI:gen} of Proposition
\ref{prop-PI} are avoided, and we get
\be \PI{q,\dt^2}(4c,|\cdot|^{-(\rho+\nu)/2} \bigr)(t) \equiv b_1
t^{-\nu-p/2} \bmod \anal\,,\ee
with a nonzero quantity $b_1$.

We apply Proposition~\ref{prop-PI} \tlref{PI:wp}{PI:gen} to
$I_{p,\dt}(c,F)$ with $F(\xi) = \xi^{-p/2-\nu}$. So $w-n=\nu$ and we
get
\be \PI{p,\dt}\bigl(c,|\cdot|^{-p/2-\nu}\bigr)(y) \equiv b_2
y^{-2\nu}\bmod \anal\,,\ee
with $b_2\neq0$ if $\nu \not\in \ZZ_{\leq 0}$. Part \eqref{PI:anal}
shows that the other term from the first stage, which is in $\anal$, is
sent to an element of $\anal$. Now we can proceed like in the case
$q=0$.\end{proof}

%% file: ccro2-13-concl.tex

\bigskip

\def\flnm{ccro2-13-concl}

\section{Concluding remarks}\label{sect13}

\rmrk{Comparison with \cite{BLZ15}} The case of cofinite discrete
subgroups of $\PSL_2(\RR)$ acting on the complex upper half-plane is
treated in \cite{BLZ15}. We followed its approach as far as possible.
We restricted the generalization to symmetric spaces of non-compact Lie
groups of real rank one in several ways:
\begin{itemize}
\item We did not consider cocompact discrete groups.
\item We simplified our work by assuming that the cofinite discrete
subgroup $\Gm$ has no torsion.
\item We considered only cusp forms, whereas in \cite{BLZ15} more
general invariant eigenfunctions were considered.
\item We did not use mixed parabolic cohomology; as in \cite[Definition
10.1]{BLZ15}. It seemed simpler to work with cohomology spaces
satisfying the condition~``$\bdc$''; see~\eqref{bdc}. The integrals in
Propositions \ref{prop-omfint} \eqref{omfint:i} and \ref{prop-omfintW}
give directly cocycles that satisfy this condition.
\end{itemize}

Theorem B in \cite{BLZ15} states that spaces of Maass cusp forms are
isomorphic to several mixed parabolic cohomology spaces, with values in
various modules,  without  a condition ``bdc'' that we use here.
Proposition 13.1 in \cite{BLZ15} may be called `separation of singularities'. It implies
 that 
if the boundary singularities of an element $f\in \W{\om^0}\nu$ form a
finite set $E$, we can decompose  $f$ as   $\sum_{e\in E} f_e$ with
$f_e\in \W{\om^0}\nu$ with $\bsing_\nu(f_e) \subset \{e\}$. 
 In the context of groups of real rank we did not establish a similar result. 
Hence the
concept of mixed parabolic cohomology groups in \cite{BLZ15} had to be
replaced by the condition ``$\bdc$'' determining a subspace of
parabolic cohomology spaces related to $\W{\om^0}\nu$ and
$\V{\om^0}\nu$. Various steps are much more complicated than
in~\cite{BLZ15}. Compare, for instance, the proof of the isomorphism
$\W\om \nu \cong \V\om\nu$ in Theorem~\ref{thm-isoVW} with the proof of
\cite[Theorem 5.6]{BLZ13}, or the discussion in \S\ref{sect11} with
\cite[Proposition 5.3 and Lemma 5.4]{BLZ13}.

\rmrk{The spectral parameter} The spectral parameter $\nu$ parametrizes
the characters of $A$, the spherical principal series representations,
and the eigenvalues of the Laplace operator.

Proposition~\ref{prop-cusp-sp} states that cusp forms can occur only for
$\nu \in i\RR \cup [-\rho,\rho]$. Proposition~\ref{prop-irr-ps} cites a
result of Kostant on the values of $\nu$ for which the spherical
principal series is irreducible.

The main result Theorem~\ref{mainthm} concerns values of $\nu\in \CC$
that satisfy $|\re\nu|<\rho$ and $\nu \not\in \frac12\ZZ_{\leq-1}$,
with the nicest result for $\nu \not\in \frac12\ZZ$. These restrictions
arise in the following way.
\begin{itemize}
\item The linear map $\btv_\nu$ in Theorem~\ref{thm-btv} for the space
$\A^0(\Gm;\nu)$ of cusp forms to a parabolic cohomology space can be
constructed for all $\nu\in \CC$. It is based on the Poisson kernel
$r_\nu$ in Proposition~\ref{prop-Rnu}. Of course, $\btv_\nu$ is the
zero map if $\nu$ does not satisfy $\rho^2-\nu^2\geq 0$.
\item The $\nu$-analytic boundary germs in Section~\ref{sect8} can be
defined for all $\nu \in \CC$. They are useful for our purpose only if
we have the isomorphism in Theorem~\ref{thm-isoVW}, for which
$\re\nu \not\in  \frac12 \ZZ_{\leq -1}$ is required.

\item In Section~\ref{sect9} we construct a map $\btw_\nu$ from spaces
of cusp forms to cohomology spaces with values in a module of
$\nu$-analytic boundary germs. We need the kernel function $q_\nu$ in
Proposition~\ref{prop-Q}, which,  like Theorem~\ref{thm-isoVW},  requires
$\nu\not\in \frac12\ZZ_{\leq -1}$.

\item Proposition~\ref{prop-eh} shows that under the additional
condition $\re\nu>-\rho$ the map $\btw_\nu$ has image in a module of
$\nu$-analytic boundary germs with a special behavior at some cusps. We
need the condition $\re\nu>-\rho$  in the proof of Proposition~\ref{prop-cohE} to characterize cusp forms with help of Lemma~\ref{lem-cusp-crit}.

\item Section~\ref{sect10} constructs cusp forms from given cocycles. In
the proof of Proposition~\ref{prop-cohE} we have to show that the
eigenfunction of $\Dt$ that we have constructed is a cusp form. 
We use Lemma~\ref{lem-cusp-crit} to show this,
now needing $\re\nu<\rho$ as well.

The conditions $|\re\nu |<\rho$ and $\nu \not\in \frac12\ZZ_{\leq -1}$
 allow us to prove in Section~\ref{sect11}
part~\eqref{mnthm:nsp} in Theorem~\ref{mainthm}.

\item To show part \eqref{mnthm:nugrl} of the main result we use the
Poisson transformation in Section~\ref{sect12}.  We use a  result of Lewis in
\cite{Le78}  that  needs $\nu \not\in \frac12\ZZ_{\leq -1}\cup \ZZ_{\geq 0}$.
We have worked under the condition $\nu \not\in \frac12\ZZ$, thus
avoiding complicated case distinctions. If one is able to extend
Lewis's result to $\nu=0$ it might be worthwhile to attempt to try to
show that Proposition~\ref{prop-cohE} works for $\nu=0$ as well.
\end{itemize}

\rmrk{Allowing torsion}The restriction to torsion-free discrete groups
$\Gm$ of motions in $S$ is practical. Several technical complications
are avoided. By going over to vector-value cusp forms, our results
probably can be extended to more general discrete groups that contain a
torsion-free subgroup of finite index.

\rmrk{Transfer operator} For $\PSL_2(\ZZ)$ a parabolic cocycle $\ps$ can
be determined by the value on $S= \matr0{-1}10$, which gives an
analytic function on $\proj\RR \setminus \{0,\infty\}$. The cocycle
relations give rise to a three term relation satisfied by the
restriction of $\ps_S$ to the interval $(0,\infty)$. In the study of
the associated transfer operator, see Mayer \cite{May91}, it is
important that this relation is contracting on the functions
on~$(0,\infty)$.

In the case $\ds\geq 3$ removal of a finite number of cusps from
$\partial S$ leaves us with a connected set. It remains to be seen
whether our result can help in the study of discretization of the
geodesic flow on~$S$.

%% file: ccro2-ind.tex

\def\flnm{ccro2-ind}
\newcommand\ind[2]{\item #1\quad\ #2}
\renewcommand\il[1]{\pageref{i-#1}}

\section*{Index}
\begin{multicols}{2}
\raggedright
\begin{trivlist}\footnotesize
\ind{analytic}{\il{anal}}
\ind{--- boundary behavior}{\il{abb}}
\ind{augmentation}{\il{augm}}
\ind{automorphic form}{\il{autf}}
\indexspace
\ind{boundary germ}{\il{bdgm}}
\ind{---, analytic}{\il{abg}}
\ind{boundary of cell}{\il{bdic}}
\ind{boundary operator}{\il{bdop}}
\ind{boundary singularity}{\il{bs}}
\ind{Bruhat decomposition}{\il{Bd}}
\indexspace
\ind{Cartan decomposition}{\il{Cd}}
\ind{Cartan involution}{\il{Ci}}
\ind{Casimir element}{\il{Ce}}
\ind{cell}{\il{cell}, \il{icA}}
\ind{chain}{\il{chn}}
\ind{chain map}{\il{chmp}}
\ind{coboundary}{\il{cob}}
\ind{coboundary map}{\il{cbm}}
\ind{cochain}{\il{coch}}
\ind{cocycle}{\il{coc}}
\ind{cohomology space}{\il{cohsp}}
\ind{---, parabolic}{\il{pcs}, \il{pbcs}}
\ind{coinvariants}{\il{coinv}}
\ind{contractible}{\il{contr}}
\ind{cusp}{\il{cusp}}
\ind{cusp form}{\il{cf}}
\indexspace
\ind{excised neighborhood}{\il{excnbh}}
\ind{extension $\hat S$ of $S$}{\il{rflex}}
\indexspace
\ind{homotopic}{\il{hmtp}}
\ind{homotopy}{\il{hmtpy}}
\ind{horoball}{\il{horb}}
\ind{---, extended}{\il{ehb}}
\ind{horosphere}{\il{hors}}
\ind{horospherical coordinates}{\il{horsco}}
\ind{---, extended}{\il{exthorco}}
\ind{---, normalized}{\il{nhorsco}}
\indexspace
\ind{interior of cell}{\il{iciX}}
\ind{invariants}{\il{inv}}
\ind{invariant eigenfunction}{\il{inveif}}
\ind{invaiant tessellation}{\il{Gmitess}}
\ind{Iwasawa decomposition}{\il{Id}}
\indexspace
\ind{Jacobian matrix}{\il{Jm}}
\indexspace
\ind{Killing form, normalized}{\il{Kf}}
\ind{$K$-finite}{\il{Kfin}}
\indexspace
\ind{Laplace operator}{\il{Laop}}
\ind{lattice}{\il{latt}}
\indexspace
\ind{Poisson transformation}{\il{Ptrf}}
\ind{polar coordinates}{\il{polco}}
\ind{---, extended}{\il{extpcCh}}
\ind{polynomial growth}{\il{pg1}}
\ind{--- at a cusp}{\il{pgc1}}
\ind{principal series, spherical}{\il{ps}}
\indexspace
\ind{restriction map}{\il{resmp}}
\indexspace
\ind{quick decay}{\il{qd1}}
\ind{--- at a cusp}{\il{qdc1}}
\indexspace
\ind{rank one, real}{\il{ro}}
\ind{resolution}{\il{resol}}
\ind{---, acyclic}{\il{ac-res}}
\indexspace
\ind{Siegel domain}{\il{Siedm}}
\ind$\nu$-{singularity}{\il{sngset}}
\ind{spectral parameter}{\il{spparm}}
\ind{semi-analytic vector in princ.\ series}{\il{sav}}
\ind{---, smooth}{\il{ssav}}
\ind{suitably large, for the parameter $Y>0$}{\il{asl}}
\ind{support}{\il{supp}}
\ind{symmetric space}{\il{ssp}}
\ind{---, extended}{\il{extsp}}
\indexspace
\ind{tessellation}{\il{tes2}, \il{tessdef}}
\ind{---, standard}{\il{tessst}}
\indexspace
\ind{Weyl group}{\il{Wg}}
\end{trivlist}
\end{multicols}


\renewcommand\ind[2]{\item $#1$\quad\ #2
}
\section*{List of notations}
\begin{multicols}{2}
\raggedright
\begin{trivlist}\footnotesize
\ind{ \A(\Gm;\nu)\text{ space of automorphic forms}}{\il{Ausp}}
\ind{\A^0(\Gm;\nu)\text{ space of cusp forms}}{\il{A0}}
\ind{\Ad(g) : \XX\mapsto g\XX g^{-1}}{\il{Ad}}
\ind{\anal}{\il{an}}
\ind{A \subset G \text{ split torus}}{\il{A}}
\ind{A^+_T= \bigl(\am(t) \;:\; t\geq T\}}{\il{A+T}}
\ind{A^+=\bigl\{\am(t)\;:\; t>1\bigr\}}{\il{A+}}
\ind{A_u, \text{ repres. of $\nu$-anal. boundary
germ}}{\il{uAu}}
\ind{A(\ph;\rho+\nu,y),\; A_\dt(\ph;\rho+\nu,y)}{\il{Ap}}
\ind{{\bf a}(m) = |\al(m)| \text{ in \S\ref{sect8}}}{\il{aan}}
\ind{\alie=\Lie(A)}{\il{alie}}
\ind{\am(t)\in A}{\il{am}}
\ind{\ad(\YY): \XX \mapsto[\YY,\XX]}{\il{ad}}
\ind{a_n\text{ coefficient in \S\ref{sect8}}}{\il{anc}}
\indexspace
\ind{B(\cdot,\cdot) \text{ normalized Killing form} }{\il{B}}
\ind{B(X) \text{ set of boundary components}}{\il{BX}}
\ind{B_\cu(Y) \text{ horoball}}{\il{BcuY}}
\ind{B_u \text{ anal. fct. corresp. to $u$}}{\il{Bu}}
\ind{B^\ast_\cu(Y) \text{ extended horoball}}{\il{BacuY}}
\ind{B^i(F^\tess_\bullet;\cdot)}{\il{Ccb}}
\ind{B(z,y)}{\il{Bzy}}
\ind{{\bf b}(m) = |\bt(m)| \text{ in \S\ref{sect8}}}{\il{bbn}}
\ind{\bdc \text{ boundary condition on cocycles}}{\il{bdc}}
\ind{\bsing_\nu(\cdot) \text{ set of boundary
singularities}}{\il{bsing}}
\ind{b_l(\xi,x), \, b_s(\xi,x)}{\il{bsl}}
\ind{db \text{ invariant measure on }\partial S}{\il{db}}
\indexspace
\ind{\C^\infty,\; \C^\ast}{\il{Cinfsh}}
\ind{\Cas\text{ Casimir element}}{ \il{Cas}}
\ind{\Cu \text{ set of cusps for $\Gm$}}{\il{Cu}}
\ind{C(x^{(1)})}{\il{Cx1}}
\ind{C^i(F^\tess_\bullet;\cdot)}{\il{Cch}}
\ind{C^\infty(G)_K\text{ $K$-finite functions}}{\il{Gi_K}}
\ind{C^{-\infty}(\cdot)}{\il{C-inf}}
\ind{C_{f,N}}{\il{CfN}}
\ind{C_B}{\il{CB}}
\ind{\cu \text{ general notation for a cusp}}{\il{cu}}
\ind{c(\cdot)\text{ Harish-Chandra's $c$-function}}{\il{cHC}}
\ind{c \text{ explicit constant in Iwasawa decomposition}}{\il{c}}
\ind{c_n}{\il{cn}}
\ind{c_{j,i}\text{ linear form on $\RR^p$}}{\il{cji}}
\indexspace
\ind{\DS}{\il{DS}}
\ind{D(p) \text{ Dirichlet fundamental domain}}{\il{Dp}}
\ind{D_\e}{\il{Deps}}
\ind{\div}{\il{div}}
\ind{\dist(\cdot,\cdot)\text{ distance function on $S$}}{\il{dist}}
\ind{\ds=\dim S}{\il{d}}
\ind{d^i\text{ coboundary map}}{\il{di}}
\indexspace
\ind{\E_\nu(S)^\Gm \text{ space of $\Gm$-invariant} \text{
eigenfunctions}}{\il{EnuGm}}
\ind{\E_\nu \text{ sheaf of eigenfunctions of $\Dt$}}{\il{EnuU}}
\ind{\E^{\om^0,\excg}_\nu}{\il{Eomex}}
\ind{\orgn = e K \in S}{\il{orgn}}
\ind{\orgn^- =j\,\orgn}{\il{ogn-}}
\ind{e \text{ unit element}}{\il{e}}
\indexspace
\ind{\F\nu, \; \F\nu(I) \text{ boundary germs}}{\il{Fnu}}
\ind{\fd_\cu \text{ fundamental domain for $\Gmm \cu$ in
$B^\ast_\cu(Y)$}}{\il{fdcu}, \il{fdcua}}
\ind{\fd(\Ld) \text{ fundamental domain for $\Ld$ in $N$}}{\il{fdld}}
\ind{\fd_0 \text{ fundamental domain for $\Gm$ in $\SY$}}{\il{fd0}, \il{fd0a}}
\ind{F^\pb_\bullet = \CC\bigl( \Cu ^{1+\bullet} \bigr) \text{ cuspidal
resolution}}{\il{Fpb}}
\ind{F^\tess_\bullet \text{ resolution based on a tessellation of
$\Sast$}}{\il{Ftb}}
\ind{F^{\tess,Y}_\bullet \text{ resolution based on a tessellation of
$\SY$}}{\il{FtessY}}
\indexspace
\ind{\G\om\nu \subset \G{\om^0,\excg}\nu}{\il{Greprnw}}
\ind{G}{\il{G}}
\ind{\glie=\Lie(G)}{\il{glie}}
\ind{\glie_c= \CC\otimes_\RR \glie}{\il{gliec}}
\ind{\glie_{j \al} \text{ root space for $\glie,\alie)$}}{\il{glieal}}
\ind{\grad}{\il{grad}}
\ind{g_\cu \in G \text{ satisfies }g_\cu\,\infty = \cu}{\il{gcu}}
\ind{\rmetric }{\il{Riemat}}
\indexspace
\ind{\H_0 \in \alie}{\il{H0}}
\ind{H^\nu=H^{1,\nu\al} \text{ spherical principal series}}{\il{Hnu}}
\ind{H^\nu_0 \supset H^\nu_\infty \supset H^\nu_\om \supset H^\nu_K
\text{ spaces in $H^\nu$}}{\il{Hnuio0K}}
\ind{H_\cu(Y) \text{ horosphere}}{\il{HcuY}}
\ind{H^i(F^\tess_\bullet;\cdot)}{\il{Coh}}
\ind{H^i_\pb(\Gm;\cdot) \text{ parabolic coh. space}}{\il{Hipb}}
\ind{H^{\ds-1}_\pb(\Gm;\V\omi\nu\bigr)^\bdc}{\il{H^bdc}}
\indexspace
\ind{I^{W,\excg}_\nu}{\il{Iwe1}}
\ind{I_\dt(\ph;\rho+\nu,y)}{\il{Idt}}
\ind{\ii}{\il{i}}
\indexspace
\ind{J(g,x)\text{ factor of automorphy}}{\il{J}}
\ind{j : (x^{(1)},x^{(2)},y) \mapsto (x^{(1)},x^{(2)},-y)}{\il{j}}
\indexspace
\ind{K\subset G \text{ maximal compact}}{\il{K}}
\ind{\klie=\Lie(K)}{\il{klie}}
\ind{\kJ(g)\,,\;\kI(g) \text{ components in Iwasawa decomp.}}{\il{kIJ}}
\indexspace
\ind{\Lt\text{ operator in \S\ref{sect8}}}{\il{Lt}}
\ind{L\text{ action by left translation/differentiation}}{\il{L}}
\indexspace
\ind{\MS(f_1,f_2)\text{ Maass--Selberg form}}{\il{MS}}
\ind{M\subset K\text{ centralizes } A}{\il{M}}
\indexspace
\ind{\N\om\nu,\;\N{\om^0,\excg}\nu}{\il{Nreprnw}}
\ind{\N\Dt\nu\subset \N{\om^0}\nu}{\il{Dtnu}}
\ind{N\subset G \text{ maximal unipotent}}{\il{N}}
\ind{N_\cu = g_\cu N g_\cu^{-1}}{\il{Ncu}}
\ind{N_R(C)}{\il{NCR}}
\ind{\nlie=\Lie(N)}{\il{nlie}}
\ind{\nJ(g),\; \nI(g) \text{ components in $N$ in Iwasawa
dec.}}{\il{nIJ}}
\ind{\nm\bigl( x^{(1)},x^{(2)}\bigr) \in N}{\il{nm}}
\ind{\nm[x]}{\il{nm[]}}
\indexspace
\ind{\poiss\nu}{\il{poiss}}
\ind{\PI{n,\dt}(h,\cdot)}{\il{PI}}
\ind{\pr \text{ projection }\Sast \rightarrow \Gm\backslash
\Sast}{\il{pr}}
\ind{p = \dim(N/Z_N)}{ \il{pp}}
\ind{\ph^\prj}{\il{prjm}}
\indexspace
\ind{\Q[V]}{\il{QV}}
\ind{Q_\nu, \; q_\nu(\cdot;\cdot)}{\il{Qnu}}
\ind{Q(x',y)}{\il{Qxay}}
\ind{q=\dim(Z_N)}{\il{qq}}
\indexspace
\ind{R\text{ action by right translation/differentiation}}{\il{R}}
\ind{r_\nu(\cdot,\cdot) \text{ kernel function on }(\partial S)\times
S}{\il{rnu}}
\ind{r_\nu^\prj(\cdot;\cdot),\; r^{\lnm}(\cdot;\cdot)}{\il{rnu-pl}}
\ind{R(\Gm)\text{ syst. of repr. of }\Gm\backslash \Cu}{\il{RGm}}
\indexspace
\ind{\Sie_\cu(C,T)\text{ Siegel domain at $\cu$}}{\il{Sie-cu}}
\ind{S = G/K \text{ symmetric space}}{\il{S}}
\ind{\hat S \text{ analytic manifold containing $S$ and
$\partial S$}}{\il{hatS}}
\ind{\SY \text{ $=S$ minus all horoballs $B_\cu(Y)$}}{\il{SY}}
\ind{\Sast \text{ extended symmetric space}}{\il{Sast}}
\ind{\sing_\nu(\cdot) \text{ set of singularities}}{\il{sing}}
\ind{\supp(\cdot)}{\il{sp}}
\ind{\partial S}{\il{bdS}}
\indexspace
\ind{\tess \text{ general symbol for a tessellation of
$\Sast$}}{\il{tess}}
\ind{\sttess \text{ standard tessellation of $\Sast$}}{\il{sttess}}
\ind{\tess(\cu) \text{ tessellation of a horoball}}{\il{tesscu}}
\ind{\tess(D,Y)}{\il{tess-DY}}
\ind{\Ta x X \text{ tangent space of $X$ at $x$}}{\il{Ta}}
\ind{\tJ(g)\,,\; \tI(g) \text{ determine comp. in $A$ in Iw.
decomp.}}{\il{tIJ}}
\indexspace
\ind{\Utess(\Ld) \text{ $\Ld$-invariant tessellation of
$N$}}{\il{UtessLd}}
\ind{\UU_j \in \mlie}{\il{UUi}}
\ind{U(\glie)\text{ universal enveloping algebra of }\glie}{\il{Ug}}
\indexspace 
\ind{\V\om\nu\text{ $G$-equivariant sheaf on $\partial S$}}{\il{Vomsh}}
\ind{\V\om\nu(\partial S),\; \V\infty\nu(\partial S) \text{ anal.\ and
smooth vectors}}{\il{VS}}
\ind{\V{\om^0}\nu \text{semi-analytic vectors}}{\il{Vom0}}
\ind{\V\omi\nu \text{ smooth semi-analytic vectors}}{\il{Vsavs}}
\ind{\V{\om^0,\exc}\nu\supset \V{\om^0,\excg}\nu \supset
\V{\om^0,\infty,\excg}\nu}{\il{Voexi}}
\ind{\VV i \in \plie}{\il{Vi}}
\ind{V^\Gm \text{ $\Gm$-invariants}}{ \il{V^Gm}}
\ind{V_\Gm \text{ $\Gm$-coinvariants}}{ \il{V_Gm}}
\indexspace
\ind{\W\om\nu(\partial S),\; \W\om\nu(I) \text{ $\nu$-analytic boundary
germs}}{\il{Womnu}}
\ind{\W{\om^0}\mu,\; \W{\om^0,\infty}\nu}{\il{Womn0}}
\ind{\W{\om^0,\exc}\nu\supset \W{\om^0,\excg}\nu \supset
\W{\om^0,\infty,\excg}\nu}{\il{Woexi}}
\ind{\WW i \in \klie}{\il{Wi}}
\ind{\wm\in K\;:\; \wm a \wm^{-1} = a^{-1} \text{ for all }a\in
A}{\il{wm}}
\ind{w_0=w_0(z)}{\il{w0}}
\indexspace
\ind{\XX_i\in \nlie}{\il{XXi}}
\ind{\X^\tess_i \text{ set of $i$-cells}}{\il{Xtessi}, \il{XtessiA}}
\ind{\X^{\tess,Y}_i}{\il{tessY}, \il{tessYa}}
\ind{x^{(1)}\,, x^{(2)}\text{ coordinate vectors on $N$}}{\il{x12}}
\indexspace
\ind{Y}{\il{Y}}
\indexspace
\ind{Z_X \text{ center of group }X}{\il{Z}}
\ind{Z^i(F^\tess_\bullet;\cdot)}{\il{Coc}}
\ind{Z^{\ds-1}\bigl( F^\tess_\bullet; \V\omi\nu\bigr)^\bdc}{\il{Z^bdc}}
\indexspace
\indexspace
\ind{\al \text{ simple root for $(\glie,\alie)$}}{\il{al}}
\ind{\alw_\nu\text{ linear map from cohomology to cusp
forms}}{\il{alwnu}}
\indexspace
\ind{\btv_\nu,\;\btw_\nu} {\il{btv}, \il{btw}, \il{btwe}}
\indexspace
\ind{\Gm\text{ discrete subgroup of $G$}}{\il{GM}}
\ind{\Gmm\cu = \bigl\{\gm\in \Gm \;:\; \gm \, \cu = \cu\}}\il{Gmcu}
\indexspace
\ind{\Dt \;=\; - \Dt_+\text{ Laplace operator}}{\il{Dt0}, {\bf\il{Dt}}}
\ind{\dtst(x,z) \text{ logarithm of distance function}}{\il{dtst}}
\indexspace
\ind{\e \text{ augmentation}}{\il{e1}}
\indexspace
\ind{\z_S \text{ invariant $\ds$-form on $S$}}{\il{ztS}, {\bf\il{ups}}}
\ind{\z[j]\,\; 1\leq j \leq \ds : \text {$(\ds-1)$-forms $S$}}{\il{upsj}}
\indexspace
\ind{\eta(f_1,f_2) \text{ $(\ds-1)$-form on $S$}}{\il{eta}}
\indexspace
\ind{\tht\text{ Cartan involution}}{\il{tht}}
\indexspace
\ind{d\mu\text{ invariant measure on $S$}}{\il{dmu}}
\indexspace
\ind{\rho = \frac12 p + q}{\il{rho}}
\ind{\rho_\nu \index{ restriction
$\W\om\nu \rightarrow\V\om\nu$}}{\il{resnu}}
\indexspace
\ind{\tau \text{ coordinate on }\hat S}{\il{tau}}
\ind{\tau_\nu\text{ representation of $G$ in $\V\om\nu$}}{\il{taunu}}
\indexspace
\ind{\Ph_\nu \in H^\nu}{\il{Phinu}}
\indexspace
\ind{\Ps_\nu : H^{\nu \al} \rightarrow \V{}\nu}{\il{Psnu}}
\ind{\ps^f \text{cocycle associated to cusp form $f$}}{\il{psf}}
\indexspace
\ind{\omv_\nu (f;\xi,\cdot) \text{ $\V\om\nu(\partial S)$-valued
differential} \text{ form} }{\il{omnu}}
\ind{\omw_\nu(f;\xi,\cdot) \text{ $\W\om\nu(\partial S)$-valued
differential form}}{\il{omW}}
\indexspace
\indexspace 
\ind{\partial_i \text{ boundary map}}{\il{parti}}
\ind{\mathring X \text{ interior of $i$-cell $X$}}{\il{intX}}
\ind{\al, \; |\al|,\;\al!,\; {\mathrm mx}(\al)\text{ for
multi-indices}}{\il{mi}}
\ind{\nul,\, \inft \text{ points in $ \partial S$}}{\il{inftnul}}
\ind{\|\cdot\|\text{ norm on $\RR^n$}}{\il{nrmN}}
\end{trivlist}
\end{multicols}